\documentclass[11pt]{report}
\usepackage{geometry}
\usepackage[utf8]{inputenc}
\usepackage[T1]{fontenc}
\usepackage[english]{babel}
\usepackage{amsmath}
\usepackage{amssymb}
\usepackage{amsthm}
\usepackage{newlfont}
\usepackage{enumerate}
\usepackage{bbm}
\usepackage{mathrsfs}
\usepackage{cases}
\usepackage{newtxtext,newtxmath}
\usepackage{graphicx}

\newtheorem{teo}{Theorem}[section]
\newtheorem*{teo*}{Theorem}

\newtheorem{lemma}[teo]{Lemma}
\newtheorem{prop}[teo]{Proposition}
\newtheorem{cor}[teo]{Corollary}
\theoremstyle{remark}
\newtheorem{rem}[teo]{Remark}

\newtheorem{cla}{Step}
\newcommand{\overbar}[1]{\mkern 1.5mu\overline{\mkern-1.5mu#1\mkern-1.5mu}\mkern 1.5mu}
\newcommand{\R}{\mathbb{R}}
\newcommand{\C}{\mathbb{C}}
\newcommand{\Z}{\mathbb{Z}}
\newcommand{\N}{\mathbb{N}}
\newcommand{\Sf}{\mathbb{S}}

\newcommand{\de}{\,\mathrm{d}}
\newcommand{\F}{\mathcal{F}}

\newcommand{\Q}{\mathcal{Q}}
\newcommand{\A}{\mathcal{A}}
\newcommand{\D}{\mathcal{D}}
\newcommand{\Le}{\mathcal{L}}
\newcommand{\E}{\mathcal{E}}
\newcommand{\M}{\mathcal{M}}
\newcommand{\K}{\mathcal{K}}
\newcommand{\Res}{\mathcal{R}}
\newcommand{\Ne}{\mathbf{N}}
\newcommand{\n}{\mathbf{n}}

\DeclareMathOperator*{\esssup}{ess\,sup}
\DeclareMathOperator*{\essinf}{ess\,inf}

\numberwithin{equation}{chapter}

\begin{document}

\begin{titlepage}
\begin{center}

\vspace*{.06\textheight}
{\scshape\LARGE Università degli Studi di Torino\\
Politecnico di Torino\par}
\vfill
\large \textit{Ph.D. in Pure and Applied Mathematics\\ XXX cycle}\\[0.3cm] 
\vfill
\textsc{\Large Doctoral Thesis}\\[0.5cm] 
\vspace{1.5cm}
\hrule 
\vspace{0.4cm}
{\huge \bfseries Nonlinear variational problems\\ with lack of compactness\par}\vspace{0.4cm} 
\hrule
\vspace{1.5cm} 
\includegraphics[scale=0.13]{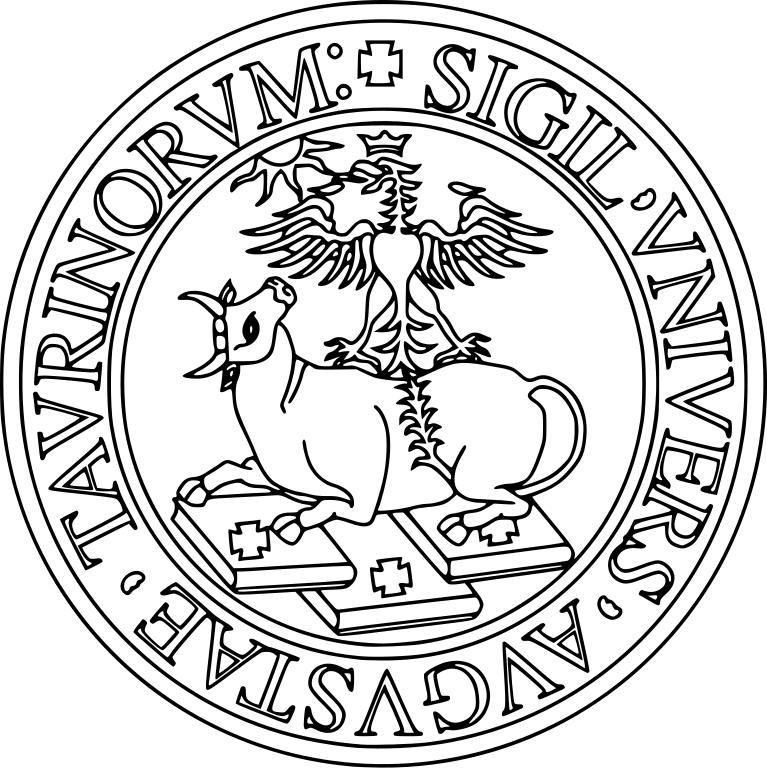} 
\quad \quad \quad \quad
\includegraphics[scale=0.13]{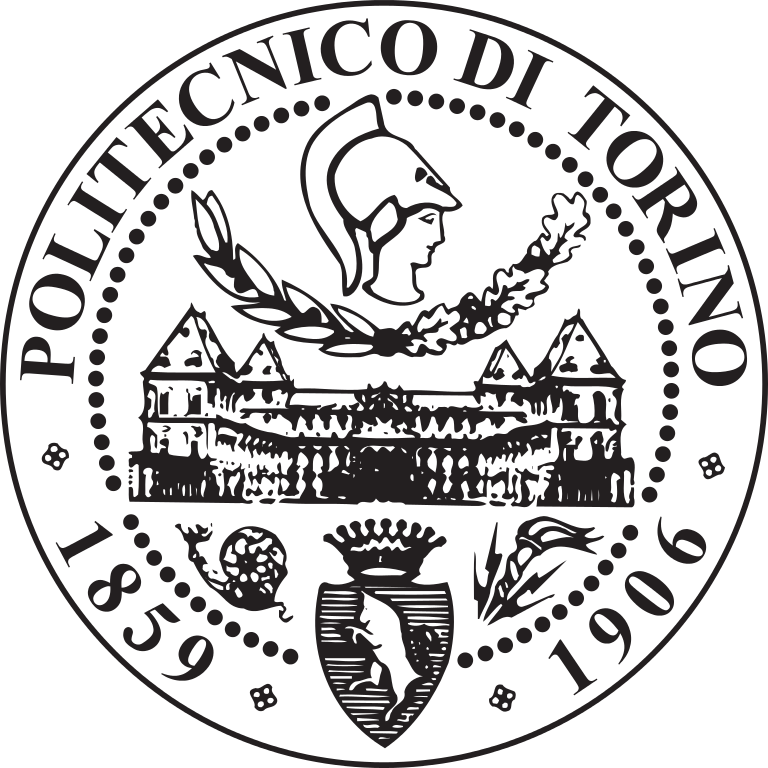} 
\vfill 
\begin{minipage}[t]{0.4\textwidth}
\begin{flushleft} \large
\emph{Supervisor:}\\
Prof. Paolo Caldiroli 
\vspace{5pt}

\end{flushleft}
\end{minipage}
\begin{minipage}[t]{0.4\textwidth}
\begin{flushright} \large
\emph{Candidate:} \\
Gabriele Cora   
\end{flushright}
\end{minipage}\\[3cm]
 
\vfill

\large \textit{A.A. 2014/15-2015/16-2016/17}\\[0.3cm] 

\vfill


\end{center}
\end{titlepage}

\renewcommand{\thepage}{\roman{page}}

\tableofcontents

\chapter*{Introduction}

\addcontentsline{toc}{chapter}{Introduction}

In the last century, variational methods have provided to be powerful and versatile tools for solving mathematical problems deriving from various fields, from analysis to geometry to physics. From the rigorous foundation of the theory, that can be dated back to the works of Weierstrass, Arzelà, Fréchet, Hilbert, Lebesgue, and many others famous mathematicians, the calculus of variations has imposed itself as the preferable way to tackle problems which admit a variational structure, and techniques more and more powerful and refined have been developed. For a short historical account about this subject we refer, e.g., to the preface of the first edition of the book by Struwe \cite{Struwe}, and the references therein.  

Suppose that a problem can be expressed as an equation of the form
\[
F(u) = 0,
\] 
where $F$ is an operator defined on some subset $A$ of a Banach space, with values in another Banach space.
We say that the problem admits a variational structure, or it is of variational form, if the operator $F$ coincides with the Fréchet differential of a functional $E$, that is, if it can be equivalently expressed as a so-called Euler-Lagrange
equation
\[
DE(u) = 0.
\]
In particular, the direct methods of the calculus of variations consist in finding solutions of the Euler-Lagrange equation as minimizers of the functional $E$.

In this thesis we deal with two different classes of variational problems:
\vspace{2pt}

\noindent 1) the problem of closed curves with prescribed curvature, or $H$-loop problem;
\vspace{2pt}

\noindent2) the study of the nodal solutions of the fractional Brezis-Nirenberg problem.
\vspace{2pt}

In both cases we deal with nonlinear equations (an ODE system for problem 1, and an elliptic equation for problem 2) which admit a variational structure. Nevertheless, in both cases, a lack of compactness occurs and this constitutes a strong obstruction for the application of the direct methods of the calculus of variations or even standard variational methods. Therefore, we are lead to use more refined techniques or also to follow different approaches, like the Lyapunov-Schmidt method and blow-up analysis. This will allow us not only to obtain existence and multiplicity results, but also qualitative properties of the solutions. 

Let us describe the two problems with some more detail.

\vspace{5pt}

The $H$-loop problem consists in finding closed regular curves in $\R^2$, parametrized by a function $U:\R \to \R^2$, whose curvature $\K(U)$ coincides at every point with a prescribed function $H: \R^2 \to \R$. 
There are multiple reasons justifying the mathematical interest towards this kind of problem. 

On one hand, its resolution is quite challenging. The problem is, in its nature, strongly dependent from the qualitative properties of the prescribed curvature function $H$. Even if it possesses a nice variational structure,
the associated energy functional is not coercive, even considering some additional constraint. Indeed, most of the results existing in the literature and in this thesis deal with curvatures which are, in some sense, small perturbations of the constant curvature.
Moreover, the smallness assumptions are not sufficient by themselves to recover existence of $H$-loops as minimizers of the associated energy, and both the sign of $H$ and the magnitude of its oscillations can either help to recover existence or pose an obstruction to it. The reason is that, in its generality, the problem lack of invariances which could help to recover compactness.

Secondly, the $H$-loop problem is strictly related to the $\mathcal{H}$-surface problem, that is, the study of the existence of parametrized hypersurfaces in $\R^N$ having prescribed mean curvature $\mathcal{H}$. As one can imagine, its resolution is an even more challenging and complex task. Nevertheless, we believe that the study of the $H$-loop problem could shed some light on the higher dimensional phenomena, and develop techniques which can be adapted to the higher dimensional case. In some cases, the two problems are deeply connected: that is the case of the $\mathcal{H}$-cylinder problem in $\R^3$. In other cases, surprisingly enough, even taking prescribed curvature $H$ and mean curvature $\mathcal{H}$ sharing the same structure, the results in low and high dimensions turn out to be completely different.

Finally, the $H$-loop problem has applications to physics, too: indeed, it is equivalent to the study of helicoidal trajectories of a nonrelativistic charged particle in a nonconstant magnetic field, which is of interest for several areas, as condensed matter theory, accelerator physics, magnetobiology, and plasma physics (see e.g. \cite{luque}). 

The $H$-loop problem is the content of Chapter \ref{hloopchapter}. The results displayed therein will be submitted in a couple of papers, one in collaboration with P. Caldiroli. 

\vspace{5pt}

As regards the fractional Brezis-Nirenberg problem, it consists in the following semilinear elliptic problem: 
\[
\begin{cases}
(-\Delta)^s u = |u|^{2^*_s-2}u + \lambda u & \text{ in }\Omega\\
u = 0 & \text{ in }\R^N \setminus \Omega,
\end{cases}
\]
where $\Omega \subset \R^N$ is a bounded domain, $s \in (0,1)$ is such that $N > 2s$, $\lambda \in \R$, $2^*_s= 2N/(N-2s)$ is the fractional critical Sobolev exponent and $(-\Delta)^s$ is the fractional Laplace operator, or $s$-Laplacian.

This is the fractional counterpart of the classical Brezis-Nirenberg problem, which is a typical example of variational problem with a global lack of compactness, due to the critical Sobolev embedding. Nevertheless, through a deep study of the variational structure and of the energetic levels, Brezis and Nirenberg proved existence of positive solutions. Then, several other mathematicians gave their contribution to the problem, recovering existence, multiplicity and asymptotic results both for positive and for sign-changing solutions. However, also in the case of $\Omega$ being the unitary ball in $\R^N$, several open questions still remain, and in particular it is still not completely clear under which conditions on the parameter $\lambda$ sign-changing solutions in dimensions $3\leq N \leq 6$ exist. 

In the last decades, a lot of attention has been drawn to problems involving the integro-differential operator $(-\Delta)^s$. On one hand, it has been proposed as an alternative way to model phenomena belonging to very different fields, as economics, biology, and physics. In some contexts, its presence gives rise to unexpected behaviors, which completely differ from what is known for their classical counterpart, as in the case of the $s$-minimal surfaces. 
Moreover, being intrinsically nonlocal, most of the known standard techniques do not straightforward apply and a complete theory is far to be known. 
This is even more true when dealing with sign-changing solutions of fractional partial differential equations, since the presence of nonlocal terms of interaction between their positive and negative parts poses a serious obstruction to the application of even the more common tool, as the maximum principle.  

In view of the previous discussion, the study of the qualitative properties of sign-changing solutions of the fractional Brezis-Nirenberg problem, aside being a challenging task, can both shed some light on the open questions for the classical Brezis-Nirenberg equation, and in general be useful to develop new techniques which allow to deal with sign-changing solutions of fractional problems. 

In Chapter \ref{FBNchapter} we present our contribution to the study of the qualitative properties of the least-energy sign-changing solutions of the fractional Brezis-Nirenberg problem in the ball $B_R$. The content of Chapter \ref{FBNchapter} gave rise to a submitted paper \cite{coraiaco}, written jointly with A. Iacopetti.

\renewcommand{\thepage}{\arabic{page}}
\setcounter{page}{1}

\begin{chapter}{The $H$-loop problem}\label{hloopchapter}
\begin{section}{Introduction}
Given a parametrization $U \in C^2(\R; \R^2)$ of a regular planar curve, its (signed) curvature is defined as 
\[
\K(U) = \frac{i\dot U \cdot \ddot U}{|\dot U|^3}.
\]
Then, given a function $H: \R^2 \to \R$, we call $H$-loop a solution of the following nonlinear problem:
\begin{equation}\label{generalHloop}
\begin{cases}
U \in C^2(\R; \R^2), & \dot U \neq 0\\
\exists\ T>0 \text{ such that } U(t + T) = U(t) & \forall\ t\in \R\\
\K(U) = H(U) & \forall\ t\in \R.
\end{cases}
\end{equation}
When the curvature $H$ is assumed to be constant, the problem is trivially solved. Indeed, when $H \neq 0$, \eqref{generalHloop} is satisfied exactly by the parametrizations of the circles of radius $1/|H|$, either clockwise or anticlockwise oriented, depending on the sign of $H$, while for $H =0$ no periodic solution exists. On the other hand, when general curvatures are considered, the situation drastically changes. 
In recent years several papers have been devoted to the study of the existence and of the qualitative properties of the solutions of the $H$-loop problem, with different assumptions on the prescribed curvature function and different approaches.

In \cite{novval} Novaga and Valdinoci treated the case of periodic curvatures. They proved that simply assuming that $H$ is periodic and bounded, it can be approximated in the $L^1$ sense by a $C^\infty$ periodic curvature $\tilde H$ for which $\tilde H$-loops exist. On the other hand, as they pointed out, the $L^1$ norm did not seem to be very well suited for this problem. The natural question then was whether the same result holds if the $L^1$ norm is replaced by the $L^\infty$ one. Indeed, in the subsequent paper by Goldman and Novaga \cite{goldnov}, it was proved that if $H$ is a periodic $C^{0,\alpha}$ function of zero average on the unitary cell which satisfies
\begin{equation}\label{goldnovcond}
\int_E H(p) \de p \leq (1-\Lambda) P(E, Q) \quad \forall E \subset Q = [0,1]\times[0,1],
\end{equation}
for some $\Lambda >0$, with $P(E, Q)$ being the relative perimeter of $E$ in $Q$, then there exists a sequence of $(H+ \varepsilon_k)$-loops, with $\varepsilon_k \to 0$.
Both results are obtained through a Geometric Measure Theory approach and, as far as we know, they are the state of the art when dealing with periodic curvatures. Nevertheless, as a consequence of the applied techniques, they seem to lack in providing informations about the topological properties of the $H$-loops. 

Another studied class of prescribed curvatures is the asymptotically constant one. 
In \cite{musina} Musina has shown, through a variational approach involving the Nehari manifold, that existence of $H$-loop is assured for every curvature of class $C^2$ satisfying
\begin{equation}\label{musinacond}
\begin{cases}
&\sup_{p \in \R^2}|(\nabla H(p)\cdot p)p|<1,\\
&\exists\ H^\infty>0 \text{ such that } H(p) = H^\infty + o (|p|^{-1})\quad \text{ as }\quad|p|\to + \infty. 
\end{cases}
\end{equation}

We would like to cite also the works of Caldiroli, Guida  \cite{calguida} and Guida, Rolando \cite{symmguida}. In the former, the case of curvatures which are small $C^2$ perturbations of a constant and posses a nondegenerate critical point is treated, while in the latter is considered the case of symmetric $C^2$ curvatures either exhibiting some homogeneity or satisfying a uniform growth condition along radial directions. Moreover, we recall the work of Kirsch and Laurain \cite{obstr} for a nonexistence result; if $H$ is a positive $C^{0,1}$ function always increasing in a given direction, then neither immersed nor embedded $H$-loop can exist. To conclude, we would like to cite the work of Bethuel, Caldiroli, Guida \cite{BCG} and the references therein for an overview about the $H$-loop and the $\mathcal{H}$-surfaces problem, and the recent work of Musina and Zuddas \cite{musinahyp} for an existence result about closed and embedded curves of prescribed geodesic curvature in the hyperbolic plane $\mathbb{H}^2$.

As one can readily see, the classes of $H$-loop problems considered are treated with different techniques and with different assumptions on the curvature function. 
The aim of the first Section of this Chapter is to provide a general and malleable method which allows to recover existence of solutions for the $H$-loop problem.
We focus on the class of asymptotically periodic curvatures, that is, that can be written in the form
\begin{equation}\label{asinteprclass}
H(p) = H_0 + H_1(p) + H_2(p), \quad p \in \R^2, 
\end{equation}
with $H_0 \in \R$, and $H_1, H_2$ respectively a periodic and an asymptotically constant function. 

The procedure, inspired by the work of Caldiroli \cite{caldiiso} on the $\mathcal{H}$-bubble problem, can be divided in two steps. 
Consider the functional
\[
\F(u) = \int_0^T F(u, \dot u) \de t,
\]
with $u$ belonging to the Sobolev space of $T$-periodic mappings in $W^{1,1}_{loc}(\R; \R^2)$ and the Lagrangian $F$ satisfying suitable assumption, such that $\F$ can be seen as a generalized length functional. 
As a first step, we see that if $H$ satisfies quite general conditions, the associated energetic functional (whose precise definition and properties we expose in Subsection \ref{varaiationalsubsect}) belongs to the aforementioned class of variational integrals.

Secondly, we look for minimizers of $\F$ constrained to the set
\[
\{u \in W^{1,1}_{loc}(\R; \R^2) \ |\ u \text{ is }T-\text{periodic and }\A(u) = \tau\},
\]
where $\A$ is the signed area of the bounded components of $\R^2 \setminus u([0,T])$ (see Subsection \ref{varaiationalsubsect} for its formal definition).
The area constrained minimization of $\F$ is a challenging task in itself and existence of minimizers is not assured. In addition, we point out that it is also related to the existence of weighted isoperimetric regions in $\R^2$; we refer to the introduction of Section \ref{existsection} for further details. 
Nevertheless, it is possible to provide conditions on $H$ (and consequently on $F$) such that area constrained minimizers exist for every value of $\tau \in \R$. 

As a consequence, for every $\tau$ we recover solutions of 
\begin{equation}\label{Kloopprob2}
\begin{cases}
u \in C^2(\R; \R^2), & \A(u) = \tau,\\
\exists\ T>0 \text{ s.t. }u(t + T) = u(t), & \forall t \in \R,\\
\K(U) = H(U) - \lambda & \forall t \in \R,
\end{cases}
\end{equation}
where the additional term $\lambda$ appears as a Lagrange multiplier due to the constrained minimization.
Therefore, existence of $\tilde H$-loop is provided, where $\tilde H = H -\lambda$ is still a curvature belonging to the class \eqref{asinteprclass}, and can be seen as an $L^\infty$ perturbation of the prescribed curvature $H$.  
Finally, we exhibit a connection between the set of admissible Lagrange multipliers and the energetic level of the minimizers for fixed $\tau$, thus recovering a perturbative results in the spirit of \cite{goldnov}.

We stress out that, in general, it holds that $\lambda \neq 0$. 
As an example, let us consider the case of constant curvatures; if $H \equiv H_0\neq 0$, as already pointed out the only solutions of \eqref{generalHloop} are the circles of radius $1/|H_0|$ and any center, which have signed area equal to $\text{sign}(H_0)\pi/H_0^2$. As a consequence, for every $H_0 \in \R\setminus \{0\}$ there exists a unique $\tau \in \R\setminus\{0\}$ such that \eqref{generalHloop} with $H = H_0$ admits a solution of area $\tau$, and vice-versa.
Nevertheless, the associated energy functional for every $\tau$ admits an area constrained minimizer, which solves problem \eqref{Kloopprob2}. Thus, unless $\tau = \text{sign}(H_0)\pi/H_0^2$ it must be $\lambda \neq 0$.

Our first result is about $\Z^2$-periodic curvatures. Let us denote
\[
[H]:= \int_{[0,1]^2} H(p) \de p,
\]
the mean over the unitary square of the function $H$. Then we have the following.
\begin{teo}\label{mainteo3}
Let $H: \R^2 \to \R$ be such that the following are satisfied:  
\begin{subnumcases}{\label{k1hp}}
H \in C^{0, \alpha}(\R^2; \R), \quad \text{ for some }\alpha \in (0,1),\\
H(p + n) = H(p), \quad \forall p \in \R^2, n \in \Z^2,\\
|H - [H]|_\infty  < 2 \sqrt{2}. 
\end{subnumcases}
Then for every $\tau \in \R \setminus \{0\}$ there exists $(u_\tau, \lambda) \in C^{2}(\R; \R^2) \times \R $ which satisfies \eqref{Kloopprob2}.
\end{teo}

Secondly, we treat the case of asymptotically constant curvatures. 
\begin{teo}\label{mainteo4}
Let $H: \R^2 \to \R$ be such that the following are satisfied:  
\begin{subnumcases}{\label{k2hp}}
H \in C^{0,\alpha}(\R^2; \R), \quad \text{ for some }\alpha \in (0,1),\\
H(p) \to H^\infty \in \R   \quad \text{as }|p|\to + \infty,\\
|H- H^\infty|_{(2,1)} < \left(\frac{2}{\pi}\right)^{\frac{3}{2}},\\
\exists\  \omega_+, \omega_- \subset \Sf^1 \text{open s.t. }\pm (H(sp)-H^\infty) < 0 \quad \forall p \in \omega_\pm, \forall s>0 \label{conecond}.
\end{subnumcases}
Then for every $\tau \in \R$ there exists $(u_\tau, \lambda) \in C^{2}(\R; \R^2) \times \R $ which satisfies \eqref{Kloopprob2}.
\end{teo}
Here we denoted as
\[
|H|_{(2,1)}: = \int_0^{+\infty} \frac{H^*(t)}{\sqrt{t}}\de t
\]
the standard norm in the Lorentz space $L(2,1)$, where $H^*$ is the symmetric decreasing rearrangement of $H$.

Finally, a result for asymptotically periodic curvatures is given. As far as we know, these have never been treated before. 
\begin{teo}\label{mainteo5}
Let $H_1, H_2: \R^2 \to \R$ satisfy assumptions \eqref{k1hp} and \eqref{k2hp}, respectively. Moreover, assume that
\begin{equation}\label{commonbound}
\frac{\sqrt{2}}{4}|H_1-[H_1]|_\infty + \left(\frac{\pi}{2}\right)^{\frac{3}{2}}|H_2- H_2^\infty|_{(2,1)} < 1.
\end{equation}
Then for every $0< \tau_0 < \tau_1$ there exists $\varepsilon_0 \in (0,1]$ such that for every $\varepsilon \in (0, \varepsilon_0]$ and for every $|\tau| \in [\tau_0, \tau_1]$, there exists $(u_{\varepsilon, \tau}, \lambda) \in C^2(\R/\Z; \R^2) \times \R$ solution of Problem \eqref{Kloopprob2} with $H=\varepsilon H_1 + H_2$.
\end{teo}

To conclude, we present our perturbative result which, up to now, applies only to the periodic and asymptotically constant case.
\begin{teo}\label{mainteo6}
Let $H$ be such that either \eqref{k1hp} and $[H]=0$ or \eqref{k2hp} and $H^\infty=0$ are satisfied. Then the following holds:
\begin{enumerate}[i)]
\item There exist two sequences, $(H_j) \subset \R$ with $|H_j| \to + \infty$, and $(\tau_j) \subset \R \setminus \{0\}$ with $\tau_j \to 0$, being such that there exists a $H_j + H$-loop of area $\tau_j$.  
\item There exist two sequences, $(H'_k) \subset \R$ with $|H'_k| \to 0$, and $(\tau'_k) \subset \R \setminus \{0\}$ with $|\tau_k| \to +\infty$, being such that there exists a $H_k + H$-loop of area $\tau_k$.
\end{enumerate}
\end{teo}
In the case of periodic curvatures, we recover exactly the result of \cite{goldnov}. To be precise, our assumptions are stronger, since the bound on the $L^\infty$ norm of $H$ implies \eqref{goldnovcond}, but the converse does not hold. Moreover, \eqref{goldnovcond} allows $H$ to take large negative values. On the other hand, in some cases we recover additional properties for the solutions of \eqref{Kloopprob2}: taking $|\tau| \in [\tau_0, \tau_1]$ and $\varepsilon \in (0, \varepsilon_0]$, with $\tau_0, \tau_1, \varepsilon_0$ and $H_1$ as in Theorem \ref{mainteo5}, the resulting $(\varepsilon H_1-\lambda)$-loops given by Theorem \ref{mainteo4} have the topological type of the circle, that is, they are Jordan curves (see Proposition \ref{simplemin}).

As for the asymptotically constant case, conversely from \eqref{musinacond} no assumption on the magnitude of the oscillations of $H$ is made, and indeed $H$ is not needed to be differentiable at all. We point out that \eqref{conecond} is not necessary and can be considerably weakened. For instance, it is sufficient that $H$ has constant sign on the tails of two cones to still recover existence of minimizers for every $\tau$. On the other hand, if only one set $\omega$ with the properties of either $\omega_+$ or $\omega_-$ exists, then existence of minimizers is assured for, respectively, either positive or negative values of the area. 
\vspace{5pt}

As already pointed out, once existence of an $H$-loop is established it is natural to try to understand its topological properties, and in particular if it is an embedded or an immersed curve.
In the celebrated paper of Alexandrov \cite{alexandrov} it is shown that the only closed compact hypersurfaces of class $C^2$ with constant mean curvature which are embedded in $\R^N$ are the spheres. On the other hand, Wente \cite{wente} showed that this result fails to be true if we remove the embeddedness condition; indeed, he constructed immersed tori of constant mean curvature which are not embedded in any $\R^N$ with $N\geq 3$.
Nevertheless, when $N=2$ the result of Alexandrov still holds true, also for immersed hypersurfaces (indeed, this can be seen also with elementary tools): as we already pointed out before, the only closed compact curve of class $C^{2}$ with constant curvature is the circle. 

According to \cite[Theorem 0.2]{musina}, if the curvature $H$ is taken to be radial, positive and non-increasing as a function of the distance form the origin, then every embedded $H$-loop is a circle (and such result is sharp with respect to the growth assumption).
In other words, in the two dimensional case the result of Alexandrov applies to a wider class of curvatures. A natural question then is whether \cite[Theorem 0.2]{musina} remains true removing the embeddedness assumption, or if this condition is sharp.

In Section \ref{immersedloops} we answer to such question, proving the existence of denumerably many immersed closed curves whose curvature belongs to the class of radially symmetric functions in the form
\begin{equation}\label{curvclassintro}
H(p) = 1 + \frac{A}{|p|^\gamma} + O\left( \frac{1}{|p|^{\gamma+\beta}}\right), \text{ for }p \in \R^2 \text{ with large }|p|,
\end{equation}
with $A \in \R\setminus\{0\}$, $\gamma >1$ and $\beta \geq 0$. 

To be more precise, we obtain the following.
\begin{teo}\label{mainteo0}
Let $H \in C^2(\R^2; \R)$ be a radially symmetric function in the form
\[
H(p) = h(|p|) \quad \text{ with }\quad  
h(s) = 1 + \frac{A}{s^{\gamma}} + \frac{\tilde h(s)}{s^{\gamma + \beta}} \quad \text{ when $s$ is large}, 
\]
with $A \in \R\setminus\{0\}$, $\gamma>1$, $\beta \geq0$ and $\tilde h \in C^2((0, +\infty);\R)$ being such that the following are satisfied:
\[
\begin{aligned}
&\begin{cases}
\tilde h (s) \text{ is bounded} &  \text{if }\beta >1,\\
\tilde h (s) = B + o(s^{\beta-1}), \quad B \in \R &\text{if }\beta \in (0, 1], \\
\tilde h(s) = B + o(s^{-1}), \quad B \neq -A & \text{if } \beta = 0,
\end{cases}\\
&|h''(s)| \leq  \frac{C}{s^{\gamma +1 + \min\{1, \beta\}}} \quad\text{ for some }C>0,  \text{ when $s$ is large}.
\end{aligned}
\]
Then there exist infinitely many closed immersed curves of class $C^{2,\alpha}$ whose curvature coincide with $H$ at every point. 
\end{teo} 
The above immersed curves are constructed through an application of the Lyapunov-Schmidt reduction method, together with a fine variational argument, following the scheme presented in the work of Caldiroli and Musso \cite{calmusso}. Here, embedded tori with curvature in the form \eqref{H0} are constructed, although with very different assumptions on the parameters. Indeed, in \cite{calmusso} both a restriction on the sign of $A$ and a smallness condition on $\gamma$ appear, that is, $A<0$ and $\gamma \in (0,2)$. 
On the other hand, the condition $\gamma >1 $, which appears in our case, is consistent with the result of Wei and Yan \cite{weiyan}, who constructed solutions of the nonlinear Schr\"oedinger equation with potential belonging to the class \eqref{H0}. 
As for the conditions on $\tilde h$ and $\beta$, they are needed to rule out interferences between the part of the energy related to $|p|^{-\gamma}$ and the one related to $|p|^{-\gamma-\beta}$ and up to this point we don't know whether they are purely technical or not.

We point out that, also considering the model case
\[
H(p) = 1 + \frac{A}{|p|^\gamma} \quad \text{ for large } |p|, 
\]
the function $H$ does not satisfy, in general, neither assumptions \eqref{musinacond} nor assumptions \eqref{k2hp}, since also large values of $|A|$ are allowed. Moreover, while the above assumption where global, here there are no restrictions on the behavior of $H$ near the origin. 

The content of this Chapter is organised as follows. The weighted isoperimetric regions problem, together with the proofs of Theorems \ref{mainteo3}-\ref{mainteo6}, is the content of Section \ref{existsection}, while in Section \ref{immersedloops} we present the construction of the immersed loops with prescribed radial curvature, thus proving Theorem \ref{mainteo0}. 

\vspace{5pt}

\textbf{Notation.} Let $T>0$ be fixed. In this Chapter we denote by
\[
|u|_{p;T} = \left(\int_0^T |u|^p \de t\right)^{\frac{1}{p}}
\]
the usual $L^p(\R/T\Z; \R^2)$ norm of $T\Z$-periodic functions, for $p \in [1, \infty)$.

Given $k \in \N$, $\alpha \in (0,1]$ we denote by 
\[
C^{k, \alpha}_T:= C^{k, \alpha}(\R / T\Z; \R^2)= \left\{u \in C^{k, \alpha}(\R; \R^2) \ |\ u(t + T) = u(t)\  \forall t\right\},
\]
the Banach space of $T$-periodic $C^{k, \alpha}$-functions, endowed with the norm
\[
\|u\|_{k, \alpha;T}:=\sum_{i=1}^k\left|\frac{\de^i u}{\de t^i}\right|_{\infty;T} + \left[\frac{\de^k u}{\de t^k}\right]_{\alpha;T},
\]
which is equivalent to the standard $C^{k,\alpha}$ norm. Here, we denoted as
\[
\left|u\right|_{\infty;T}:= \sup_{t \in [0,T]}\left| u \right|, \quad \text{ and }\quad\left[u\right]_{\alpha;T}:= \sup_{0\leq t_1<t_2\leq T}\frac{\left|u(t_2) - u(t_1)\right|}{|t_2-t_1|^{\alpha}},
\]
the standard $L^\infty$ norm and the H\"older / Lipschitz seminorm of $T\Z$-periodic functions, respectively.
In order to ease the notation, in Section \ref{immersedloops} we use the same notations for the spaces $C^{k,\alpha}(\R / T\Z; \R)$ and their seminorms.

We define as
\begin{equation}\label{Hspace}
\begin{aligned}
&W_T^{1,1} := W^{1,1}(\R / T\Z; \R^2) = \{ u \in W^{1,1}_{loc}(\R; \R^2) \ |\ u(t+T) = u(t) \ \forall t \in \R\},\\
&H_T^1 := H^1(\R / T\Z; \R^2) = \{ u \in H^{1}_{loc}(\R; \R^2) \ |\ u(t+T) = u(t) \ \forall t\in \R\},
\end{aligned}
\end{equation}
the Sobolev spaces of $T\Z$-periodic functions, endowed with the norms
\begin{equation}\label{barinorm}
\begin{aligned}
&\|u\|_{W_T^{1,1}} : = \int_{0}^T |u|\de t + \int_0^T |\dot u|\de t,\\
&\|u\|_{H_T^1} := \sqrt{|[u]|^2 + \int_0^T |\dot u|^2 \de t},
\end{aligned}
\end{equation}
respectively. We denoted by $[u] = \frac{1}{T}\int_0^Tu \de t $ the mean of $u$ over the period, and we recall that $\|\cdot\|_{H_T^1}$  is equivalent to the standard norm in $H^1$ thanks to the Poincaré-Wirtinger inequality
\begin{equation}\label{PWineq}
\frac{4\pi^2}{T}\int_0^T |u - [u]|^2 \de t \leq \int_0^T |\dot u|^2 \de t, \quad \forall u \in H_T^1. 
\end{equation}

Finally, we identify $\R^2$ and the set of constant parametrizations
\[
\{ u: \R \to \R^2 \ | \ u \equiv c \in \R^2\}.
\]
In addition, with a slight abuse of notation, we also identify $\R^2$ with $\C$ in the standard way. In particular, $ip$ denotes the counterclockwise rotation of $\frac{\pi}{2}$ radians of the vector $p \in \R^2$, and 
\[
e^{\frac{2\pi}{T} i t} = \left(\cos \frac{2\pi}{T} t, \sin \frac{2\pi}{T}t\right), \
\] 
is a $T$-periodic parametrization of the unitary circle counterclockwise oriented.

Through all the Chapter, we denote by $C$ a generic positive constant, whose value may vary even in the same line.

\begin{subsection}{The variational structure}\label{varaiationalsubsect}
In this Subection we present an alternative though equivalent version of the $H$-loop problem. Then we introduce the associated energy functional, and we prove that every $H$-loop is a critical point of such functional.

The $H$-loop problem in the form \eqref{generalHloop} is in general difficult to tackle, since it involves a fully nonlinear equation and because of the presence of the unknown period $T$. 
Nevertheless, through a standard construction (see e.g. \cite{BCG}) it can be expressed in a more manageable form.
Since $U$ is of class $C^2$ and $\dot U \neq 0$, we can perform a standard reparametrization by arc length. Then \eqref{generalHloop} can be equivalently formulated as
\[
\begin{cases}
\exists\ T>0 \text{ s.t. }U \in C^2_T,\\
|\dot U| = \text{const.},\\
\K(U) = H(U), & \forall t \in \R.
\end{cases}
\]
After a simple computation (recall that $2\ddot U \cdot \dot U = \frac{\de}{\de t}|\dot U|^2$) we obtain the following system of second order equations
\[
\begin{cases}
\exists\ T>0 \text{ s.t. }U \in C^2_T\setminus\R^2,\\
(\ddot U  - |\dot U|H(U)i \dot U) \cdot \dot U = 0,  &\forall t \in \R\\
(\ddot U  - |\dot U|H(U)i \dot U) \cdot i \dot U = 0,  &\forall t \in \R.
\end{cases}
\]
Therefore, noticing that, for every $t \in \R$, $\{\dot U(t), i \dot U(t)\}$ is an orthogonal basis of $\R^2$, we easily infer that the $H$-loop problem is equivalent to 
\begin{equation}\label{hloopversion2}
\begin{cases}
\exists T>0 \text{ s.t. }U \in C^2_T\setminus\R^2,\\
\ddot U = \left(\frac{1}{T}\int_0^T|\dot U|\de t\right) H(U) i \dot U & \forall\ t\in \R.
\end{cases}
\end{equation}
As one can see, the condition $|\dot U| = \text{const.}$ has been absorbed in the quasilinear second order equation, which has the additional property of being invariant with respect to transformations in the form $ U(t) \mapsto u(t)= U(\lambda t)$ with $\lambda > 0$. As a consequence, we can always reduce ourselves to the space of periodic $C^{2}$-functions with fixed period.

Fix $T'>0$; the final form of the $H$-loop problem is
\begin{equation}\label{hloop}
\begin{cases}
u \in C^2_{T'}\setminus\R^2,\\
\ddot u = \left(\frac{1}{T'}\int_0^{T'}|\dot u|\de t\right) H(u) i \dot u & \forall\ t\in \R
\end{cases}
\end{equation}

\vspace{5pt}

Let us introduce the variational setting of the problem. 
Let be $H\in C^0(\R^2; \R)$, and suppose that there exists a vector field $Q_H$ which satisfies
\begin{equation}\label{divergass}
\begin{cases}
Q_H \in C^{0,1}_{loc}(\R^2; \R^2) &\\
\text{div }Q_H = H & \text{for a.e. }p \in \R^2
\end{cases}
\end{equation}
We define the functional $\E_H: W^{1,1}_T\to \R$ as
\begin{equation}\label{energhia}
\E_H(u):= L(u) + \A_H(u),
\end{equation}
where 
\begin{equation}\label{functionals}
L(u) = \int_0^T|\dot u|\de t \quad \text{ and } \quad \A_H(u):= \int_0^T Q_H(u) \cdot i \dot u \de t.
\end{equation}
Thanks to \eqref{divergass} we readily infer that $L, \A_H: W^{1,1}_T \to \R$. 

A continuous solution of 
\[
\text{div }Q_H = H,\quad \text{for a.e. } p \in \R^2,
\]
always exists; consider, for instance,
\[
Q_H(p) = \frac{1}{2}\left( \int_0^{p_1}H(s, p_2) \de s, \int_0^{p_2}H(p_1, s) \de s\right).
\]
On the other hand, in general it is neither unique nor is Lipschitz continuous, therefore it does not satisfy \eqref{divergass} (see e.g. \cite{BBrezis} for the case of periodic datum $H$). Nevertheless, the functional \eqref{energhia} is always defined for continuous curvatures, and its behavior does not really depend on $Q_H$ but only on $H$. Indeed, as a consequence of \cite{JurkNonne} we have that the following generalized divergence theorem holds.
\begin{prop}\label{gendivprop}
Let $G= (G_1, G_2) \in C^0(\R^2; \R^2)$ be a continuous vector field. Define
\begin{equation}\label{generdiv}
\text{div }G(p)= 
\begin{cases}
\frac{\partial G_1}{\partial p_1}(p) + \frac{\partial G_2}{\partial p_2}(p) & \text{if } G\text{ is differentiable in } p=(p_1, p_2)\\
0 & \text{otherwise}
\end{cases}
\end{equation}
and suppose that $\text{div }G \in C^0(\R^2; \R)$. 
Then for every rectifiable closed curve $u \in W_T^{1,1}$ it holds
\begin{equation}\label{winddiv}
\int_0^T G(u) \cdot i \dot u \de t = \int_{\R^{2}}\omega_u(p)\text{div }G(p) \de p,
\end{equation}
where $\omega_u(p)$ is the winding number of the closed path $u([0,T])$ evaluated in $p \in \R^2$. 
\end{prop}

When $H\equiv 1$, we can take $Q_H(p) = \frac{1}{2}p$. We obtain
\begin{equation}\label{areafunct}
\A (u):= \A_1(u) = \frac{1}{2}\int_0^T u \cdot i\dot u \de t. 
\end{equation}
Then by \eqref{winddiv} we get
\begin{equation}\label{areaintr}
\A(u) = \int_{\R^2}\omega_u(p)\de p,
\end{equation}
which shows that $\A$ is indeed the classical (signed) area functional. In view of that, $\A_H$ too can be interpreted as the area weighted by $H$ of the bounded components of $\R^2 \setminus u([0,T])$.

It is useful to recall the classical divergence formula (that can also be recovered by \eqref{winddiv} together with \eqref{areaintr}). When $u \in W_T^{1,1}$ parametrizes a Jordan curve, i.e., when $u([0,T])$ is a closed and simple curve, then it holds
\begin{equation}\label{divergencejord}
\A_H(u) = \int_0^1 Q_H(u) \cdot i \dot u \de t = \text{sign }\left(\A(u)\right) \int_{\mathcal{B}_u}H(p) \de p,
\end{equation}
where $\mathcal{B}_u$ is the bounded component of $\R^2 \setminus u([0,T])$.

In the following we see that solutions of \eqref{hloop} are indeed critical points for \eqref{energhia}. As is known, the length functional $L$ belongs to $C^{1}(W^{1,1}_T\setminus \R^2;\R)$ and it holds
\[
L'(u)[\varphi] = \int_0^T\frac{\dot u}{|\dot u|}\cdot \dot \varphi \de t, \quad u\in W^{1,1}_T\setminus \R^2,\ \varphi \in W^{1,1}_T.
\] 

As for the anisotropic area functional $\A_H$, the following holds. 
\begin{lemma}\label{Qfuncreg}
Let be $H \in C^0(\R^2; \R)$ and suppose that there exists $Q_H: \R^2 \to \R^2$ satisfying \eqref{divergass}. Then $\A_H \in C^1(W^{1,1}_T;\R)$ and for every $u,\varphi \in W_T^{1,1}$ the following holds:  
\begin{equation}
\A_H(u+\varphi)-\A_H(u) = \int_0^T\int_0^1 H(u+\sigma \varphi)\varphi \cdot i (\dot u + \sigma \dot \varphi) \de \sigma \de t,\label{Qfuncreg1}
\end{equation}
\begin{equation}\label{Qfuncreg2}
\A_H'(u)[\varphi] = \int_0^T H(u)\varphi \cdot i\dot u \de t,
\end{equation}
\begin{equation}\label{Qfuncreg3}
\frac{\de}{\de s} \A_H(su) = s \int_0^T H(su)u \cdot i\dot u \de t.
\end{equation}
\end{lemma}
\begin{proof}
Let $u, \varphi \in W_T^{1,1}$. Since by assumption $Q_H \in C^{0,1}_{loc}(\R^2; \R^2)$, we have that $DQ_H$ is defined almost everywhere. Moreover, since $W_T^{1,1} \subset L^\infty$, we get that 
\[
\esssup_{t \in [0,T], \sigma \in [0,1]}|D Q_H(u + \sigma \varphi)|<+\infty.
\]
Therefore, we get that
\[
\begin{aligned}
&\left(t \mapsto \frac{\de}{\de t}\left(Q_H(u + \sigma \varphi)\cdot i \varphi\right) \right)\in L^1([0,T]; \R), & \forall \sigma \in [0,1]\\
&\left(\sigma \mapsto \frac{\de}{\de \sigma}\left( Q_H(u + \sigma \varphi) \cdot i (\dot u + \sigma \dot \varphi)\right)\right) \in L^1([0,1]; \R), & \forall t \in [0,T].
\end{aligned}
\]
As a consequence, we are allowed to apply the fundamental theorem of calculus; recalling that $u, \varphi$ are periodic functions, we obtain
\begin{equation}\label{eq:proofreg}
\begin{aligned}
\A_H(u+\varphi) - \A_H(u) =& \int_0^T Q_H(u+ \varphi) \cdot i(\dot u + \dot \varphi) - Q_H(u) \cdot i \dot u \de t  \\
=& \int_0^T \int_0^1 \frac{\de}{\de \sigma} Q_H(u+ \sigma \varphi) \cdot i(\dot u + \sigma \dot \varphi)\de \sigma \de t  \\
=& \int_0^T\int_0^1 DQ_H(u + \sigma \varphi)\varphi \cdot i (\dot u + \sigma \dot \varphi) + Q_H(u+ \sigma \varphi) \cdot i \dot \varphi \de \sigma \de t,
\end{aligned}
\end{equation}
and
\begin{equation}\label{eq:proofreg10}
\begin{aligned}
0 &= \int_0^1 (Q_H(u + \sigma \varphi)\cdot i \varphi)_{|t=T} - (Q_H(u + \sigma \varphi)\cdot i \varphi)_{|t=0}\de \sigma\\
 &= \int_0^1\int_0^T \frac{\de}{\de t}\left(Q_H(u + \sigma \varphi)\cdot i \varphi\right) \de t \de \sigma\\
&=\int_0^T\int_0^1 Q_H(u+\sigma \varphi) \cdot i \dot \varphi \de \sigma \de t + DQ_H(u+ \sigma \varphi) (\dot u + \sigma \dot \varphi) \cdot i \varphi \de \sigma \de t.
\end{aligned}
\end{equation}

From \eqref{eq:proofreg}, \eqref{eq:proofreg10} and the algebraic relation
\[
Mv \cdot i w - Mw \cdot i v = tr(M) v \cdot i w, \quad \forall M \in \R^{2 \times 2}, v, w \in \R^2,
\]
we infer
\[
\A_H(u+ \varphi) - \A_H(u) = \int_0^T\int_0^1 H(u+\sigma \varphi)\varphi \cdot i (\dot u + \sigma \dot \varphi) \de \sigma \de t,
\] 
hence \eqref{Qfuncreg1} is proved.

Fix $\varepsilon_0 >0$ and let be $\varepsilon \in (0, \varepsilon_0)$ and $u, \varphi \in W_T^{1,1}$ with $\|\varphi\|_{W_T^{1,1}}=1$. By \eqref{Qfuncreg1} we have that
\[
\frac{\A_H(u+ \varepsilon \varphi) - \A_H(u)}{\varepsilon} = \int_0^T\int_0^1 H(u+\sigma \varepsilon \varphi)\varphi \cdot i (\dot u + \sigma \varepsilon \dot \varphi) \de \sigma \de t.
\] 

Since $H$ is uniformly continuous on every compact subset of $\R^2$, we readily get that there exists a constant $C>0$ which does not depend neither on $\varepsilon$ nor $\sigma$ such that
\[
\begin{aligned}
&|K(u+\sigma \varepsilon \varphi)\varphi \cdot i (\dot u + \sigma \varepsilon \dot \varphi)|\leq C|\varphi|(|\dot u| + |\dot \varphi|)\in L^1([0,T]\times[0,1]; \R),\\
&\lim_{\varepsilon \to 0}K(u+\sigma \varepsilon \varphi)\varphi \cdot i (\dot u + \sigma \varepsilon \dot \varphi) = K(u)\varphi \cdot i \dot u, \quad \quad\text{a.e. }(t, \sigma)\in [0,T]\times[0,1].
\end{aligned}
\]
Then, by Lebesgue's dominated convergence Theorem we get
\[
\A_K'(u)[\varphi]=\lim_{\varepsilon \to 0}\frac{\A_K(u+\varepsilon \varphi) - \A_K(u)}{\varepsilon} = \int_0^1 K(u) \varphi \cdot i \dot u \de t,
\]
as desired.

Next we show that $\A_H \in C^1(W_T^{1,1}; \R)$. Let $(u_k) \subset W_T^{1,1}$ be such that $u_k \to u $ in $W_T^{1,1}$. Recalling that the immersion $W^{1,1}_T \hookrightarrow C^0_T$ is continuous, we get that $u_k \to u$ uniformly. In addition, we have
\[
\begin{aligned}
\left|\A'_H(u_k)[\varphi] - \A'_H(u)[\varphi]\right| &=\left|\int_0^T H(u_k)\varphi \cdot \dot u_k \de t-\int_0^T H(u)\varphi \cdot \dot u \de t \right|  \\
&\leq \int_0^T |H(u_k) - H(u)||\varphi||\dot u_k| \de t + \int_0^T\! |H(u)||\varphi||\dot u_k - \dot u|\de t\\
&\leq C\left(\int_0^T |H(u_k) - H(u)||\dot u_k| \de t + \int_0^T\! |H(u)||\dot u_k - \dot u|\de t\right)\|\varphi\|_{W^{1,1}_T}.
\end{aligned}
\]
Then, taking the limit as $k \to +\infty$, both terms in the right-hand side goes to zero and the desired property easily follows. 

As for \eqref{Qfuncreg3}, it is a straightforward consequence of \eqref{Qfuncreg2}. Indeed, 
\[
\frac{\de}{\de s}\A_H(su) = \A_H'(su)[u] = s\int_0^T H(su) u \cdot i \dot u \de t.
\]
The proof of the Lemma is concluded. 
\end{proof}

Thanks to Lemma \ref{Qfuncreg}, a simple computation shows that if $u$ is a solution of \eqref{hloop}, then it satisfies
\[
\begin{cases}
|\dot u| = \text{const.}& \text{a.e. }t \in \R\\
\E'_H(u)[\varphi] = 0 & \forall \varphi \in W^{1,1}_T.
\end{cases}
\]
Indeed, also the converse holds true. 
\begin{lemma}\label{solreg}
Let $H \in C^0(\R^2; \R)$ and let $u \in W^{1,1}_T$ satisfy 
\begin{equation}\label{asdjalskdj}
\begin{cases}
|\dot u| = \text{const.} & \text{a.e. }t \in \R\\
\E'_H(u)[\varphi] = 0 & \forall \varphi \in W^{1,1}_T.
\end{cases}
\end{equation}
Then $u \in C^{2}_T$ and is a solution of \eqref{hloop}.
\end{lemma}
\begin{proof}
Combining assumptions \eqref{asdjalskdj} we infer that
\[
\int_0^T\dot u \cdot \dot \varphi + |\dot u|H(u) i \dot u \cdot \varphi \de t=  0.
\]
Moreover, since $ |\dot u| = \text{const.}$ implies that $ u \in W^{1, \infty}(\R/T\Z; \R^2)$, and since $H$ is uniformly continuous on every compact subset of $\R^2$, we readily get that $f: = |\dot u| H(u) i\dot u\in L^\infty(\R; \R^2)$. As a consequence, for every $\varphi \in C^1_c([0,T]; \R^2)$ it holds that
\[
- \int_0^T\dot u \cdot \dot \varphi \de t = \int_0^T|\dot u|H(u)i \dot u \cdot \varphi \de t = \int_0^T f \cdot\varphi \de t,
\]
namely, $f$ is the weak derivative of $\dot u$, which implies that $\ddot u = f$ a.e. and $u \in W^{2,\infty}(\R / T\Z; \R^2)$.
By Sobolev-Morrey embedding we infer that $u \in C^{1,\alpha}_T$ for any $\alpha \in (0,1)$, thus in particular $f$ is continuous. In addition, since $\dot u \in W^{1,\infty}(\R /T\Z; \R^2)$ we can apply the fundamental theorem of calculus. We obtain 
\[
\dot u(t) = \dot u(t_0) + \int^t_{t_0}f(s) \de s, \quad \forall t,t_0 \in \R,
\]
which, thanks to the continuity of $f$, implies that $\dot u \in C^1_T$ and $\ddot u = f$ pointwise. Therefore $u \in C^2_T$ and is a classical solution.
\end{proof}

\end{subsection}
\begin{subsection}{Applications}
To conclude the introduction to the $H$-loop problem, we would like to explore in further details the possible applications. 

To begin with, we see that by solving the $H$-loop problem we automatically obtain solutions of the $\mathcal{H}$-cylinder problem with mean curvature $\mathcal{H}$ depending on just two variables. 
Let $\tilde {\mathcal{C}}$ be the cylinder in $\R^3$ of infinite length and circular section of radius one.  Let $\n$ be its rotational axis, and introduce a system of coordinates being such that the third axis coincides with $\n$.
Then the function $\tilde U : \R^2\setminus \{0\}\to \R^3$ given by
\[
\tilde U(\theta, r) = (\cos \theta,  \sin \theta, \log r),
\] 
where we used polar coordinates $(\theta, r) = \left( \text{arg} (p),|p|\right)$, is a parametrization of $\tilde{\mathcal{C}}$. Moreover, it satisfies the conformality condition in polar coordinates, that is
\[
\frac{\partial \tilde U}{\partial r} \cdot \frac{\partial \tilde U}{\partial \theta} = 0 = r^2\left|\frac{\partial \tilde U}{\partial r}\right| - \left|\frac{\partial \tilde U}{\partial \theta}\right|.
\]

Given a function $\mathcal{H}: \R^3 \to \R$, an $\mathcal{H}$-cylinder is defined as a parametrizable surface $\mathcal C$ which admits a $C^2$ conformal parametrization $U: \R^2 \to \R^3$ diffeomorphic to $\tilde U$, having mean curvature which coincides to $\mathcal H$ at every point. 
In general, the mean curvature of a surface $\mathcal{C}$ on $\R^3$ can be expressed by means of the first and second fundamental forms as 
\begin{equation}\label{Hcyleq}
2 \mathcal{H} = \frac{1}{EG-F^2}(GL- 2FM + EN).
\end{equation}
In particular, if $\mathcal{C}$ admits a conformal parametrization $U= U(x,y)$, the mean curvature equation can be expressed as
\begin{equation}\label{Hcyleq2}
\Delta U = 2\mathcal{H}(U) \frac{\partial U}{\partial x} \wedge \frac{\partial U}{\partial y}
\end{equation}
where $\wedge$ is the standard wedge product.
For further details about the analytical formulation of the mean curvature equation, and for a proof of the relation between \eqref{Hcyleq} and \eqref{Hcyleq2}, we refer to \cite{BCG}.

As a consequence of the previous discussion, the $\mathcal{H}$-cylinder problem consists in finding a parametrization $U$ diffeomorphic to $\tilde U$ and such that
\begin{equation}\label{hcylprob}
\begin{cases}
&\frac{\partial \tilde U}{\partial x} \cdot \frac{\partial \tilde U}{\partial y} = 0 = \left|\frac{\partial \tilde U}{\partial x}\right| - \left|\frac{\partial \tilde U}{\partial y}\right|\\
&\Delta U = 2\mathcal{H}(U) \frac{\partial U}{\partial x} \wedge \frac{\partial U}{\partial y} 
\end{cases}
\end{equation}

Suppose now that $\mathcal{H}$ is continuous and satisfies 
\begin{equation}\label{cylindercond}
\frac{\partial \mathcal{H}}{\partial \n}(P) = 0, \quad \forall P \in \R^3, 
\end{equation}
for some direction $\n \in \R^3$, which, up to rotation, we suppose to be $\n = (0,0,1)$.
Thus $\mathcal{H}$ only depends on two variables, and in particular $\mathcal{H}(P_1, P_2, P_3) = \mathcal{H}(P_1, P_2, 0)$ for every $P = (P_1, P_2, P_3) \in \R^3$. 
Therefore, for $p = (p_1, p_2) \in \R^2$, we define $H(p_1, p_2 ) := 2 \mathcal{H}(p_1, p_2, 0)$.

We look for particular parametrizations in the form 
\begin{equation}\label{paramcyleq}
U(\theta, r) = (u_1 (\theta), u_2(\theta), \log r),
\end{equation}
where $u = (u_1, u_2)$ is a $T$-periodic $C^2$ function. Putting together \eqref{cylindercond} and \eqref{paramcyleq}, the $\mathcal{H}$-cylinder problem \eqref{hcylprob} can be written in polar coordinates as
\[
\begin{cases}
\exists\ T>0 \text{ s.t. } u \in C^2_T,\\
|\dot u| = 1,\\
\ddot u = H(u) i \dot u & \forall t \in \R,
\end{cases}
\]
which is indeed an equivalent formulation of the $H$-loop problem (see \eqref{hloopversion2}). In particular, since $U$ has to be diffeomprhic to $\tilde U$, we have that every embedded $H$-loop give rise to a corresponding $\mathcal{H}$-cylinder $\mathcal{C}$ with $\mathcal{H}$ satisfying \eqref{cylindercond}. 

\vspace{5pt}

Secondly, we show that the $H$-loop problem is equivalent to the problem of the existence of helicoidal trajectories for a nonrelativistic charged particle in a oriented magnetic field.
We say that a trajectory $P: \R \to \R^3$ is helicoidal if there exists a unitary vector $\n \in \R^3$ such that the component $P_\parallel$ in the direction of $\n$ describe a uniform right motion, while its projection $P_\perp$ on a plane orthogonal to $\n$ describe a periodic closed curve of period $T$.
From classical physics we know that, in the presence of an external magnetic field $B: \R^3 \to \R^3$, the motion of a particle of charge $e$ and mass $m$ is driven by the Lorentz force, and its trajectory satisfies the differential equation
\begin{equation}\label{mageq}
m \ddot P = e \dot P \wedge B(P).
\end{equation}

Suppose that the magnetic field satisfies the following assumptions
\[
\begin{cases}
&\exists\ \n \in \R^3 \text{ s.t. }B(P) = b(P) \n\\
&\frac{\partial b}{\partial \n}(P) = 0, \quad \forall P \in \R^3, 
\end{cases}
\]
that is, $b$ is constant in the direction $\n$. Then it depends only on the components of $P$ belonging to a plane orthogonal to $\n$, and we get that $b: \R^2 \to \R$. With respect to this assumptions \eqref{mageq} can be rewritten as
\[
m \ddot P = eb(P_\perp)\dot P\wedge \n.
\]
For simplicity, we can assume that $\n = (0, 0, 1)$. Then we get the following system of equations
\begin{equation}\label{mageq2}
\begin{cases}
m \ddot P_\parallel = 0\\
m \ddot P_\perp = - e b(P_\perp) i \dot P_\perp.
\end{cases}
\end{equation}

The parallel component obviously describes a uniform right motion. Moreover, a simple computation shows that if \eqref{mageq2} is satisfied, then $|\dot P_\perp| = v$ for some $v \in \R_+$. Then it is sufficient to relabel $H = \frac{-e}{mv}b$ to infer that, in order to have helicoidal trajectories, the following must be satisfied
\[
\begin{cases}
\exists\ T>0 \text{ s. t. } P_\perp \in C^2_T,\\
|\dot P_\perp |= v,\\
\ddot P_\perp =|\dot P_\perp|H(u) i \dot P_\perp, 
\end{cases}
\]
which is exactly an equivalent formulation of the $H$-loop problem.

\end{subsection}
\end{section}

\begin{section}{Isoperimetric regions in $\R^2$ and the $H$-loop problem}\label{existsection}

This Section is devoted to the proof of Theorem \ref{mainteo3}-\ref{mainteo6}. To ease the notation, we fix the period to be $T = 1$ and we drop the subscripts when denoting the functional spaces and their respective norms. Keeping that in mind and arguing as we did to recover \eqref{hloop}, we can reformulate \eqref{Kloopprob2} equivalently 
as
\begin{equation}\label{Kloopprob}
\begin{cases}
u \in C^{2},\\
\ddot u = L(u) (H(u)-\lambda)i \dot u & \forall t \in \R.
\end{cases}
\end{equation}
As already mentioned, solutions of \eqref{Kloopprob} rise as area constrained minimizers of functionals belonging to a general class of variational integrals.

Let us introduce precisely such variational integrals. Consider the functional 
\begin{equation}\label{varfunctional}
\mathcal{F}(u) := \int_0^1 F(u, \dot u) \de t, 
\end{equation}
where the Lagrangian $F: \R^2\times \R^2 \to \R$ satisfies, for every $p, q, q_1, q_2 \in \R^2$,
\begin{subnumcases}{\label{Fhyp}}
F \in C^0(\R^2 \times \R^2; \R); \label{F:reg}\\
F(p,\theta q_1 + (1-\theta)q_2) \leq \theta F(p, q_1) + (1-\theta)F(p, q_2) \quad \forall \theta \in [0,1];\label{F:conv}\\
F(p, \lambda q) = \lambda F(p,q) \quad \forall \lambda >0; \label{F:hom}\\
\exists\ m_1, m_2>0 \text{ such that }m_1 |q| \leq F(p,q) \leq m_2 |q|. \label{F:ell}
\end{subnumcases}
Under these assumptions, $\F$ can be seen as a generalized length functional. As a model case we can consider the classical length functional $L$ defined in \eqref{functionals}, corresponding to the case $F(p,q) = |q|$, but more general situations can be treated; for instance, taking
\[
F(p,q) = \sqrt{g(p)q \cdot q}
\]
where, the matrix $g(p) \in \R^{2\times2}$ is positive-definite for every $p\in \R^2$, we recover the Riemannian length.

In \eqref{areafunct} we introduced the area functional 
\[
\mathcal{A}(u) := \frac{1}{2} \int_0^1 u \cdot i\dot u \de t.
\]
Both $\F$ and $\A$ turn out to be well defined on the Sobolev space $W^{1,1}$ of $1$-periodic mapping defined in \eqref{Hspace}, recalling that it holds $W^{1,1} \subset L^\infty(\R; \R^2)$.
Since the Lagrangians of $\F$ and $\A$ are both positively homogeneous of degree one with respect to $q$, they turn out to be invariant under Lipschitz reparametrizations of curves, and this agrees with the geometrical meaning of such 
functionals.

Recall that the classical isoperimetric inequality in $\R^2$ holds, i.e.
\begin{equation}\label{classicineq}
S\sqrt{\left|\A(u) \right|} \leq L(u), \quad \forall u \in W^{1,1}, 
\end{equation}
where $S:= \sqrt{4\pi}$ is the sharp constant and the equality is attained exactly by circles in $\R^2$ of any center and radius. 

Thanks to \eqref{F:ell} and \eqref{classicineq}, also $\F$ satisfies an isoperimetric-like inequality, that is
\begin{equation}\label{isoperlike}
S_F \sqrt{|\A(u)|} \leq \F(u) \quad \forall u \in W^{1,1},
\end{equation}
where $S_F \in (m_1S, m_2S)$. The existence of extremals arises as a natural question, and appears immediately to be a challenging task. Indeed, because of the dependence of $F$ from $p$, in general \eqref{isoperlike} is not invariant with respect to translations and dilations in the form $u \mapsto \lambda u + p_0$ with $\lambda >0$. In particular, differently from \eqref{classicineq}, the problem of the existence of minimizers of \eqref{isoperlike} can not be reduced to the existence of area constrained minimizer, that is, satisfying $\A(u) = \tau$ for some $\tau \in \R$. Therefore, depending on the value of $\tau$, minimizers of various shape and with different properties can arise. 

Justified by that, let us define
\begin{equation}\label{areaconst}
\M_\tau := \{ u \in W^{1,1} \ | \ \mathcal{A}(u) = \tau\},
\end{equation}
and let us study the minimization problem for $\tau \in \R$ fixed, that is, we look for minimizers of
\begin{equation}\label{minprob}
S_F(\tau) := \inf_{v \in M_\tau} \mathcal{F}(v),
\end{equation}
Under the generality of assumptions \eqref{Fhyp}, finding solutions of \eqref{minprob} is still far to be obvious. Indeed, although $\F$ turns out to be weakly lower semicontinuous in $W^{1,1}$ and the area constraint weakly closed (see Subsection \ref{odvisub}), $\F$ is not in general coercive, and lack of compactness of minimizing sequences can occur. In fact, there are both conditions on the Lagrangian that assure existence of minimizers, and conditions which lead to nonexistence (see Subsection \ref{subsectproofth2}). 

Moreover, minimizers of \eqref{minprob} can be seen as particular cases of weighted isoperimetric regions in $\R^2$.
Given two positive lower semicontinuous functions $f,g: \R^N \to \R_+$, a subset $E \subset \R^N$ is called a weighted isoperimetric set if it satisfies
\[
P_f(E) = \inf\{P_f(F) \ |\ |F|_g = \tau\},
\] 
for a fixed $\tau \in \R$. Here, the weighted volume and perimeter are defined as
\[
P_f(E) := \int_{\partial^M E}f(x) \de \mathscr{H}^{N-1}(x) \quad \text{ and }\quad |E|_g := \int_E g(x) \de \mathscr{H}^{N},
\]
where $\partial^M E$ denotes the essential boundary of $E$. 

In recent years, several papers have been dedicated to the study of weighted isoperimetric sets, also because of their close connection with the isoperimetric sets on Riemannian manifolds; see for instance \cite{cinpra, cinpra2, defrapra, digio, MorPra} and the references therein.
The problem is usually tackled through a Geometric Measure Theory approach. In our setting, although we consider only bounded and connected sets whose boundary admits a parametrization in $W^{1,1}_{loc}(\R; \R^2)$, and only the unweighted area functional, more general length densities are allowed, since they can also depend on the derivative of the parametrization.

First of all we treat the purely periodic length density case. Our result is the following.
\begin{teo}\label{mainteo1}
Assume that $F$ satisfies \eqref{Fhyp} and 
\begin{equation}
F(p+n,q) = F(p,q), \quad \forall\ p, q \in \R^2,\ n \in \Z^2. \label{F:per}
\end{equation}
Then for every $\tau \in \R$ there exists $u_\tau \in \M_{\tau} \cap C^{0,1}$ such that $\F(u_\tau) = S_{F}(\tau)$. 
\end{teo}

Secondly, we consider the case of length densities $F$ which are asymptotically periodic. In this case a sufficient condition for the existence of area constrained minimizers appears. 
\begin{teo}\label{mainteo2}
Let $F$ satisfy \eqref{Fhyp}. Moreover, suppose that there exists $F_{\infty}:\R^2 \times \R^2 \to \R$ such that 
\begin{subnumcases}{\label{Fasper2}}
F_\infty \in C^0(\R^2 \times \R^2; \R) \label{Fper:reg}\\
F_\infty(p+n, q) = F_\infty(p,q) & $\forall n \in \Z^2$\\
\forall q \in \R^2,\ F(p,q) - F_{\infty}(p, q) \to 0 & as  $|p| \to + \infty$. \label{Fper:asint}
\end{subnumcases}
For every $\tau \in \R$ it holds that 
\begin{equation}\label{energyorder}
S_{F}(\tau) \leq S_{F_\infty}(\tau).
\end{equation}
In addition, if for a fixed $\tau \in \R$ is satisfied
\begin{equation}\label{asexcond}
S_{F}(\tau) < S_{F_{\infty}}(\tau),
\end{equation}
then there exists $u_\tau \in \M_{\tau} \cap C^{0,1}$ such that $\F(u_\tau) = S_{F}(\tau)$. 
\end{teo}
Notice that, thanks to \eqref{Fper:asint}, also $F_\infty$ satisfies \eqref{Fhyp}, which implies that minimizers for $S_{F_\infty}(\tau)$ always exist thanks to Theorem \ref{mainteo1}. Furthermore, we point out that condition \eqref{asexcond} is not necessary. Indeed, we construct two examples where it is not satisfied; in the former, minimizers of $S_F(\tau)$ exist and coincide with the minimizers of $S_{F_{\infty}}(\tau)$, while in the latter no minimizer exists.
Moreover, also a sufficient condition for \eqref{asexcond} to hold for every $\tau \in \R$ is provided. Hence, also for the asymptotically periodic case, we have existence of area constrained minimizers for every value of the area. 

As a further remark, we notice that Theorem \ref{mainteo2} covers also the case of asymptotically null perturbations of the classic length functional, i.e. such that $F_\infty(p, q) = |q|$.

This Section is organized as follows: in Subsection \ref{odvisub} we present some known results about variational integrals, while Subsection \ref{subsectth1} and \ref{subsectproofth2} are devoted to the proof of Theorem \ref{mainteo1} and Theorem \ref{mainteo2}, respectively. 
In the following, the aforementioned results are applied in order to recover solutions of \eqref{Kloopprob}. In particular, in Subsection \ref{loopsub} we derive Theorem \ref{mainteo3} and Theorem \ref{mainteo4}, while Subsection \ref{mixedsub} is dedicated to the proof of Theorem \ref{mainteo5}. Finally, in Subsection \ref{isoperfunctsubsect} we study the properties of the isoperimetric function, thus deriving Theorem \ref{mainteo6}.

\begin{subsection}{One dimensional variational integrals}\label{odvisub}
In this Subsection we present some known semicontinuity and continuity results, together with a characterization of minimizers of variational integrals whose Lagrangian is homogeneous of degree two.
All the results are well known, and contained e.g. in \cite{BGH}. We present them here for the sake of completeness, stated and proved in a form suitable for our purposes. 

We recall that, with respect to our setting, an (autonomous) variational integral is defined as a functional $\F:W^{1,1} \to \R$ in the form 
\[
\F(u) = \int_0^1 F(u, \dot u) \de t, 
\]
where $F:\R^2 \times \R^2 \to \R$ is the associated Lagrangian.

The first result we present is a continuity result for linear Lagrangians. As a matter of fact, the linearity with respect to the second set of variables is a condition both necessary and sufficient for the continuity of variational integrals with respect to the weak convergence in $W^{1,1}$ and $H^1$.

\begin{prop}[{\cite[Proposition 3.4]{BGH}}]\label{linearcont}
A variational integral $\mathcal{F}(u)$ associated to a Lagrangian $F \in  C^0(\R^2 \times \R^2; \R)$ is sequentially continuous with respect to the weak convergence in $W^{1,1}$ if and only if $F$ is linear with respect to $q$, i.e. it can be written in the form 
\begin{equation}\label{linearLag}
F(p, q) = A(p) + B(p)\cdot q,
\end{equation}
with $A, B \in C^0(\R^2; \R)$. 
\end{prop} 

The next result shows that the convexity is a sufficient condition for the semicontinuity. This is a generalization of the Tonelli's semicontinuity theorem due to Ioffe. 

The following technical Theorem is required.
\begin{teo}[{\cite[Theorem 2.12]{BGH}}]\label{L1boundth}
Let $\mathcal{C}$ be a subset of $L^1(\Omega)$. Then the following claims are equivalent:
\begin{enumerate}[i)]
\item $\mathcal{C}$ is sequentially weakly compact in $L^1(\Omega)$;
\item the functions $u \in \mathcal{C}$ are equibounded in $L^1(\Omega)$ and the set function 
\[
E \mapsto \int_E|u|\de x, \quad E \subset \Omega, \quad u \in \mathcal{C},
\] 
are equiabsolutely continuous;
\item the functions $u \in \mathcal{C}$ are uniformly integrable, i.e. the integrals
\[
\int_{\{x \in \Omega \ |\ |u(x)|>c\}}|u(x)|\de x
\]
tend to zero as the positive number $c$ tends to $+\infty$, uniformly for $u \in \mathcal{C}$;
\item there exists a function $\Theta: (0, +\infty) \to \R$ (that can be taken as convex and increasing) such that 
\[
\begin{aligned}
&\lim_{t \to +\infty}\frac{\Theta(t)}{t} = +\infty,\\
&\sup_{u \in \mathcal{C}}\int_\Omega \Theta(|u|)\de x < +\infty.
\end{aligned}
\]
\end{enumerate}
\end{teo}

\begin{teo}[{\cite[Theorem 3.6]{BGH}}]\label{lscconvex}
Let $F \in  C^0(\R^2 \times \R^2; \R)$  be a continuous Lagrangian such that 
\begin{enumerate}[i)]
\item $F$ is non-negative;
\item $F(p, q)$ is convex with respect to $q$.
\end{enumerate}
The functional $\F: W^{1,1} \to \R$ defined as in \eqref{varfunctional} is weakly lower semicontinuous in $W^{1,1}$, i.e. if $(u_k) \subset W^{1,1}$ is such that $u_k \rightharpoonup u$ in $W^{1,1}$, then 
\[
\F(u) \leq \liminf_{k \to +\infty}\F(u_k).
\]
\end{teo}
\begin{proof}
Let $(u_k) \subset W^{1,1}$ be a sequence such that $u_k \rightharpoonup u$ in $W^{1,1}$. Our aim is to extimate the quantity $c := \liminf_{k \to + \infty}\F(u_k)$. 

Up to subsequences, we can assume that $c = \lim_{k \to +\infty}\F(u_k)$. 
Moreover, taking a further subsequence if necessary, thanks to the Sobolev embeddings we have that 
\[
\begin{aligned}
& u_k \to u \quad \text{ in }L^1([0,1]; \R^2),\\
& u_k \to u \quad \text{ a.e. }t \in [0,1],\\
& \dot u_k \rightharpoonup \dot u \quad \text{ in }L^1([0,1]; \R^2).
\end{aligned}
\]

Since $\dot u_k \rightharpoonup \dot u$ in $L^1([0,1]; \R^2)$ we can apply Theorem \ref{L1boundth}. In particular, a careful analysis of its proof shows that the function $\Theta$ can be taken such that some additional properties are satisfied. In particular we infer that there exists $\theta: [0, + \infty) \to [0, + \infty)$ which is convex, strictly increasing, and such that
\[
\lim_{s \to +\infty}\frac{\theta(s)}{s} = + \infty \quad \text{ and }\quad \sup_{k \in \N}\int_0^1 \theta(|\dot u_k|)\de t \leq 1.
\]  

Let us define $H(s):= \sqrt{s \theta(s)}$. A simple computation shows that $H:[0, + \infty) \to [0, +\infty)$ is strictly increasing and it holds $\lim_{s \to + \infty}\frac{H(s)}{s} = + \infty$. Moreover, we get that $ \lim_{s \to + \infty}\frac{\theta(s)}{H(s)} = + \infty$. In addition, it is well defined $H^{-1}:[0, +\infty) \to [0, +\infty)$, which is strictly increasing too.

Let $(\xi_k)$ be the sequence defined as $\xi_k(t) := H(|\dot u_k(t)|)$. Using the definition of $H$ and the properties of $\theta$ we get that $(\xi_k) \subset L^1([0,1]; \R)$. 
Moreover, let be $\phi(s) := \theta(H^{-1}(s))$. Notice that also $\phi: [0, + \infty) \to [0, +\infty)$ is strictly increasing and satisfy $\lim_{s \to \infty} \frac{\phi(s)}{s} = + \infty$. 
Furthermore, it holds that
\[
\sup_{k \in \N}\int_0^1 \phi(\xi_k) \de t = \sup_{k \in \N}\int_0^1 \theta(|\dot u_k|) \de t \leq 1.
\]
Hence $(\xi_k)$ and $\phi$ satisfy the hypothesis of Theorem \ref{L1boundth} and we conclude that, up to subsequences, there exists $\xi \in L^1([0,1]; \R)$ such that $\xi_k \rightharpoonup \xi$ in $L^1([0,1]; \R)$. 

Applying the Lemma of Mazur to the sequence $((\dot u_k, \xi_k)) \subset L^1([0,1]; \R^2) \times L^1([0,1]; \R)$, we get that there exists a strictly increasing sequence $(N_k) \subset \N$ being such that $\lim_{k \to + \infty}N_k = + \infty$, and some coefficients ${\alpha_{i,k}}\subset [0, + \infty)$ which satisfy 
\begin{equation}\label{convsemicont3}
\begin{aligned}
&\sum_{i = N_k+1}^{N_{k+1}}\alpha_{i,k} = 1 \quad \forall k\in \N;\\
&\mu_k(t) := \sum_{i = N_k+1}^{N_{k+1}}\alpha_{i,k}\dot u_k(t) \to \dot u \quad \text{ in }L^1([0,1]; \R^2);\\
&\eta_k(t) := \sum_{i = N_k+1}^{N_{k+1}}\alpha_{i,k}\xi_k(t) \to \xi \quad \text{ in }L^1([0,1]; \R).
\end{aligned}
\end{equation}
In particular we have that $\mu_k \to \dot u$ and $\eta_k \to \xi$ for a.e. $t \in [0,1]$.

Let $t_0 \in [0,1]$ be such that $u_k(t_0) \to u(t_0)$, $\mu_k(t_0) \to \dot u(t_0)$ and $\eta_k(t_0) \to \xi(t_0)$ at the same time.  
To begin with, we define the sequence $(\varepsilon_k) \subset [0, +\infty)$ as 
\[
\varepsilon_k := \max\{|u_i(t_0) - u(t_0)| \ |\ i=N_k +1,\ldots, N_{k+1}\},
\] 
and we readily see that $\varepsilon_k \to 0$ as $k \to +\infty$. 
Then, we define the sequence $(\lambda_k(t_0)) \subset \R$ as 
\begin{equation}\label{convsemicont4}
\lambda_k(t_0) := \sum_{i = N_k+1}^{N_{k+1}} \alpha_{i,k}F(u_i(t_0), \dot u_i(t_0)). 
\end{equation}
Notice that thanks to assumption $i)$ we have that $\lambda_{k}(t_0)\geq 0$ for every $k \in \N$. 
Finally, we introduce the sets
\[
\mathcal{A}_k : = \{(\mu, \eta, \lambda) \in \R^2 \times \R \times \R \ |\ \eta = H(|\mu|),\ \exists s\in \R^2 \text{ s.t.} |u(t_0)-s|< \varepsilon_k,\ \lambda \geq F(s, \mu)\}.
\]
Notice that for every $i = N_k+1, \ldots, N_{k+1}$ it holds that
\[
(\dot u_i (t_0), \xi_i(t_0), F(u_i(t_0), \dot u_i(t_0))) \in \mathcal{A}_k.
\]
As a consequence, recalling \eqref{convsemicont3} and \eqref{convsemicont4}, we get that $(\mu_k (t_0), \eta_k(t_0), \lambda_k(t_0))$ is contained in the convex bulk of $\mathcal{A}_k$. 
Hence, by means of the Theorem of Caratheodory on the convex bulk, we infer that there exist five points in $\mathcal{A}_k$, which we denote by $(\mu_{j,k}, \eta_{j,k}, \lambda_{j,k})$ for $j = 1,\ldots,5$, and some constants $\beta_{j,k}$ being such that
\begin{equation}\label{convsemicont5}
\begin{aligned}
&\mu_{j,k}\in \R^2, \quad \eta_{j,k}\geq 0, & &\lambda_{j,k}\geq 0,\quad \beta_{j, k} \geq 0, \\
&\sum_{j=1}^5\beta_{j,k}=1, & &\sum_{j=1}^5 \beta_{j,k}\mu_{j,k} = \mu_k(t_0), \\
&\sum_{j=1}^5\beta_{j,k}\eta_{j,k} = \eta_k(t_0), & &\sum_{j=1}^5\beta_{j,k}\lambda_{j,k}= \lambda_k(t_0).
\end{aligned}
\end{equation}
Moreover, by definition of $\mathcal{A}_k$, for every $j=1,\ldots,5$ there exists $s_{j,k} \in \R^2$ being such that $
|u(t_0)-s_{j,k}|\leq \varepsilon_k$ and $\lambda_{j,k} \geq F(s_{j,k}, \mu_{j,k})$. 

Let us denote $\{1,\ldots,5\}$ as $J \cup J^c$, where the subset $J$ is such that $j \in J$ if and only if $|\mu_{j,k}| \not \to + \infty$ as $k \to + \infty$. 
Notice that, by definition of $\mathcal{A}_k$, as $k \to + \infty$ it holds 
\[
\sum_{j=1}^5\beta_{j,k}H(|\mu_{j,k}|) = \sum_{j=1}^5\beta_{j,k}\eta_{j,k} = \eta_k(t_0) \to \xi(t_0) < + \infty
\]
Since $H$ is strictly increasing, we infer that if $j \in J^c$ then $\beta_{j,k} \to 0$ as $k \to + \infty$, and also that $J \neq \emptyset$, since otherwise we will get a contradiction with \eqref{convsemicont5}. 

Up to subsequences, we can assume that there exists $\mu_j \in \R^2$ such that $\mu_{j,k} \to \mu_j$ when $j \in J$ and there exists $\beta_j$ such that $\beta_{j,k}\to \beta_j$  as $k \to + \infty$. 
Moreover we have that 
\[
\eta_k(t_0) = \sum_{j=1}^5 \beta_{j,k} \eta_{j,k} \geq \sum_{j\in J^c}\beta_{j,k}|\mu_{j,k}|\frac{H(|\mu_{j,k}|)}{|\mu_{j,k}|}
\]
hence, thanks to the properties of $H$, we also get that $\beta_{j,k}|\mu_{j,k}| \to 0$ when $j \in J^c$. 

Taking the limit as $k \to +\infty$ in \eqref{convsemicont5} and taking into account the previous remarks we get that
\[
\sum_{j \in J}\beta_j = 1, \quad \sum_{j \in J}\beta_{j}\mu_j = \dot u (t_0).
\]
Since $F$ is continuous and satisfies assumptions $i)$ and $ii)$, this leads to 
\begin{equation}\label{convsemicont9}
\begin{aligned}
F(u(t_0), \dot u(t_0)) &= F\left(u(t_0), \sum_{j \in J}\beta_j \mu_j\right) \leq \sum_{j \in J}\beta_{j}F(u(t_0), \mu_j)\\
&= \lim_{k \to +\infty}\sum_{j \in J}\beta_{j,k}F(s_{j,k}, \mu_{j,k}) \leq \lim_{k \to + \infty}\sum_{i=1}^5 \beta_{j,k}F(s_{j,k}, \mu_{j,k})\\
& \leq \lim_{k \to + \infty}\sum_{j=1}^5\beta_{j,k}\lambda_{j,k} = \lim_{k \to + \infty}\lambda_k(t_0),
\end{aligned}
\end{equation}
where we used the definition of $\mathcal{A}_k$ and that $s_{j,k} \to u(t_0)$ as $k \to + \infty$. 

Fix $\varepsilon >0$. There exists $\overline k$ being such that for every $k > \overline k$ it holds that $\F(u_i) \leq c + \varepsilon$ for every $i =N_k+1, \ldots, N_{k+1}$. 
Then, taking \eqref{convsemicont9} into account and applying the Lemma of Fatou we obtain that
\[
\F(u) \leq \int_{0}^1 \lim_{k \to + \infty} \lambda_k(t)\de t \leq \liminf_{k \to + \infty}\sum_{i = N_k+1}^{N_{k+1}}\alpha_{i, k}\int_0^1F(u_i, \dot u_i)\de t \leq c+ \varepsilon,
\]
hence the theorem follows thanks to the arbitrariness of $\varepsilon$. 
\end{proof}

We conclude the discussion of the continuity and semicontinuity properties of variational integrals with the following continuity result. 
\begin{lemma}\label{strongcont}
Let be $F \in C^0(\R^2 \times \R^2; \R)$ such that \eqref{F:conv} and \eqref{F:ell} are satisfied. Then the functional
\[
\F(u) = \int_0^1F(u, \dot u) \de t
\]
is continuous with respect to the strong convergence in $W^{1,1}$. 
\end{lemma}
\begin{proof}
Let $(u_k) \subset W^{1,1}$ being such that $u_k \to u$ in $W^{1,1}$. Recall that up to subsequences this implies that $u_k \to u$ a.e. By \eqref{F:conv} and \eqref{F:ell} we get that 
\[
\begin{aligned}
|\F(u_k) - \F(u)| &\leq \int_0^1 |F(u_k, \dot u_k - \dot u)| \de t + \int_0^1 |F(u_k, \dot u) - F(u, \dot u)| \de t \\
&\leq m_2 \|u_k - u \|_{W^{1,1}} + \int_0^1 |F(u_k, \dot u) - F(u, \dot u)| \de t
\end{aligned}
\]
As $k \to + \infty$ both terms in the right-hand side goes to zero: the first one trivially, the second one as a consequence of the Lebesgue dominated convergence theorem together with \eqref{F:ell}. Thus the Lemma is proved. 
\end{proof}

Finally, we present a characterization of minimizers of variational integrals whose Lagragian is positively homogeneous of degree two with respect to $q$. Notice that in particular it applies to $Q(p, q) = F^2(p, q)$ when $F$ satisfies \eqref{Fhyp}.

Let $\varepsilon_0 > 0$ and $I_0:= (-\varepsilon_0, \varepsilon_0)$. We say that a function $\xi: [0,1]\times I_0 \to [0,1]$ is an admissible parameter variation if the following properties hold: 
\begin{equation}\label{admissiblevar}
\begin{cases}
&\text{for every }\varepsilon \in I_0\text{, the function }t \to \xi(t, \varepsilon)\text{ is a diffeomorphism of class }C^1 \\
&\text{for every }\varepsilon \in I_0\text{ it holds }\xi(0, \varepsilon) = 0\text{ and }\xi(1, \varepsilon) = 1\\ 
&\text{for every }t \in [0,1] \text{ it holds that }\xi(t, 0) = t\\
&\text{for every }\varepsilon \in I_0\text{ we have that }\frac{\partial \xi}{\partial \varepsilon}(\cdot , \varepsilon) \in C^1([0,1]; \R). 
\end{cases}
\end{equation}

\begin{prop}[{\cite[Proposition 1.14, Remark 3]{BGH}}]\label{intenergy}

Let $Q: \R^2\times \R^2 \to \R$ be such that
\begin{equation}\label{quasinhp}
\begin{cases}
Q(p, q) \in C^0 (\R^2 \times \R^2; \R); \\
\exists \ m_2 >0 \text{ s.t. } Q(p, q) \leq m_2 |q|^2, &\forall (p, q) \in \R^2\times \R^2,\\
Q(p, \lambda q) = \lambda^2 Q(p, q), &\forall (p, q) \in \R^2 \times \R^2, \lambda >0.
\end{cases}
\end{equation}

Let $\mathcal{C} \subset H^1$ be a class invariant under admissible parameter variations, i.e. such that if $u \in \mathcal{C}$, then for every admissible parameter variation $\xi$ and for every $\varepsilon \in I_0$ we have that $v( \cdot, \varepsilon) := u(\xi(\cdot, \varepsilon)) \in \mathcal{C}$.
If $u \in \mathcal{C}$ minimizes in $\mathcal{C}$ the variational integral 
\[
\Q (u) = \int_0^1 Q(u, \dot u ) \de t, 
\]
then there exists $h \in \R$ such that $Q(u(t), \dot u(t)) = h$ for a.e. $t \in [0,1]$. 
\end{prop}
\begin{proof}
Fix $\varphi \in C^\infty_c([0,1]; \R)$ and define $\mu: [0,1]\times \R \to \R$ as 
\begin{equation}\label{mudef}
\mu(x, \varepsilon) := x - \varepsilon \varphi(x).
\end{equation}
It is easy to see that there exists $\varepsilon_0 >0$ being such that 
\[
\frac{\partial \mu}{\partial x}(x, \varepsilon) >0 \quad \forall \varepsilon \in (-\varepsilon_0, \varepsilon_0).
\]
Let be $I_0 := (-\varepsilon_0, \varepsilon_0)$. The restriction of $\mu$ to $[0,1] \times I_0$ (that we still denote as $\mu$) is such that $\mu:[0,1]\times I_0 \to [0,1]$ and satisfies $\mu(0, \varepsilon) = 0$, $\mu(1, \varepsilon) = 1$ for every $\varepsilon \in I_0$ and $\mu(x, 0) = x$ for every $x \in [0,1]$.  
As a consequence, the function $\xi: [0,1]\times I_0 \to [0,1]$ defined as $\xi(\cdot, \varepsilon) := \mu^{-1}(\cdot, \varepsilon)$ for every $\varepsilon \in I_0$, satisfies conditions \eqref{admissiblevar} and is an admissible parameter variation. 

Let us define the family of functions $v(t, \varepsilon) := u(\xi(t, \varepsilon))$. Since $\mathcal{C}$ is invariant under admissible parameter variations, then $v(\cdot, \varepsilon) \in \mathcal{C}$ for every $\varepsilon \in I_0$. Moreover, let us consider the function $\psi: I_0 \to \R$ given by
\begin{equation}\label{parvar1}
\psi(\varepsilon) := \int_0^1 Q\left(v(t, \varepsilon), \frac{\partial }{\partial t}(v(t, \varepsilon))\right)\de t.  
\end{equation}
Since it holds that $t = \xi(t, 0)$ (and consequently that $1 = \frac{\partial \xi}{\partial t}(t, 0)$), a simple computation shows that $\Q(u) = \psi(0)$. Therefore, since $u$ minimizes $\Q$ in $\mathcal{C}$, we get that $\psi'(0) = 0$. 

On the other hand, since by definition $t = \mu(\xi(t, \varepsilon), \varepsilon)$, differentiating with respect to $t$ we obtain
\[
1 = \frac{\partial \mu}{\partial x}(\xi(t, \varepsilon))\frac{\partial \xi}{\partial t}(t, \varepsilon).
\]
As a consequence, performing the change of variables $\xi(t, \varepsilon) = x$ in \eqref{parvar1} and recalling \eqref{mudef} we infer that 
\begin{equation}\label{robba20}
\begin{aligned}
\psi(\varepsilon) &= \int_0^1 Q\left(u(\xi(t, \varepsilon)), \frac{\partial u}{\partial t}(\xi(t, \varepsilon))\frac{\partial \xi}{\partial t}(t, \varepsilon)\right) \de t\\
&= \int_0^1 Q(u(x), \dot u(x) (1 - \varepsilon \varphi'(x))^{-1})(1 - \varepsilon \varphi'(x)) \de x \\
&= \int_0^1 Q(u(x), \dot u(x))(1 - \varepsilon \varphi'(x))^{-1} \de x,
\end{aligned}
\end{equation}
where in the last equality we used \eqref{quasinhp} together with the fact that $1-\varepsilon \varphi' >0$ for every $x \in [0,1]$. 
Since $u \in H^1$, it holds that
\[
\left|\frac{Q(u, \dot u) \varphi'}{(1 - \varepsilon \varphi')^2}\right|\leq \frac{m_2|\varphi'|_\infty}{(1 - \varepsilon_0 |\varphi'|_\infty)^2} |\dot u|^2\in L^1([0,1]; \R^2),
\]
hence we are allowed to differentiate with respect to $\varepsilon$ in \eqref{robba20}. We infer that
\[
\psi'(\varepsilon)=\int_0^1\frac{Q(u, \dot u)\varphi'}{(1-\varepsilon \varphi')^2} \de x,
\]
therefore
\begin{equation}\label{parvar2}
0 = \psi'(0) = \int_0^1 Q(u, \dot u) \varphi' \de x.
\end{equation}
Since \eqref{parvar2} holds true for every $\varphi \in C^\infty_c([0,1]; \R)$; the conclusion follows as a straightforward application of DuBois-Reymond's Lemma (see e.g. \cite[Lemma 1.8]{BGH}).
\end{proof}

\end{subsection}

\begin{subsection}{Proof of Theorem \ref{mainteo1}}\label{subsectth1}

Without loss of generality, we can reduce to the case $\tau >0$. 
On one hand, let $\tau <0$ and $u \in \M_{\tau}$. We define $v(t) := u(-t)$ and the Lagrangian $\hat F(p,q) := F(p,-q)$. Then $\hat F$ still satisfies \eqref{Fhyp},\eqref{F:per} and it holds that
\[
\A(v) =-\tau \quad and \quad \hat \F(v) := \int_0^1 \hat F(v, \dot v) \de t = \F(u), 
\]
thus
\[
S_{F}(\tau) = S_{\hat F}(-\tau).
\]
As a consequence it is sufficient to prove existence of area constrained minimizer for $-\tau>0$ to automatically obtain existence of area constrained minimizer for $\tau <0$. 

When $\tau = 0$, we easily get that the set of minimizers of $S_{F}(0)$ coincides with $\R^2$, that is, the set of constant functions in $W^{1,1}$. Indeed, if $u \equiv c$ for a certain $c \in \R^2$, then $\dot u \equiv 0$ and we have that $u \in \M_{0}$ and $\F(u) = 0$. Since \eqref{F:ell} implies that $S_{F}(0) \geq 0$, we get that $\F(u) = S_{F}(0) = 0$. As a consequence any minimizer $v \in \M_{0}$ of $S_{F}(0)$ must satisfy $\F(v) = 0$. Hence by \eqref{F:ell} we infer that $|\dot v| = 0$ a.e. in $[0,1]$, i.e. $v$ is a constant.   

We divide the proof of the case $\tau >0$ in three Lemmas. 
First of all, we notice that since $W^{1,1}$ is not a reflexive space, a bounded minimizing sequence is not necessarily weakly convergent. 
To overcome this difficulty, we first prove the following. Let us define the space
\begin{equation}\label{quasinormal}
H_F^1 := \{ u \in H^1 \ |\ \exists\ h \geq 0 \text{ s.t. }F(u, \dot u) = h \text{ for a.e. } t \in [0,1]\}, 
\end{equation}
which is the subspace of the so-called quasinormal curves in $H^1$. 
\begin{lemma}\label{infWHLip}
For every $\tau>0$ it holds that 
\[
S_{F}(\tau) = \inf_{v \in \M_{\tau} \cap H^1} \F(v) = \inf_{v \in \M_{\tau} \cap H_F^1}\F(v).
\]
\end{lemma}
\begin{proof}
Fix $\tau >0$. Thanks to \eqref{F:ell}, the following chain of inclusion holds:
\[
H_F^1 \subset C^{0,1} \subset H^1 \subset W^{1,1}. 
\]
Then we infer 
\begin{equation}\label{infWHLip0}
S_{F}(\tau) \leq \inf_{v \in \M_{\tau} \cap H^1}\F(v) \leq \inf_{v \in \M_{\tau} \cap C^{0,1}}\F(v) \leq \inf_{v \in\M_{\tau} \cap H^1_F}\F(v).
\end{equation}

To begin with, we are going to prove that 
\begin{equation}\label{infWHLip1}
S_F(\tau) = \inf_{\M_\tau \cap C^{0,1}}\F(v).
\end{equation}
First of all, it holds that $\M_\tau	\cap C^{0,1}$ is dense in $\M_\tau$. 
Fix $u \in \M_\tau$. By standard density results, there exists a sequence $(u_k) \subset C^{0,1}$ such that $u_k \to u$ in $W^{1,1}$. Define
\[
\lambda_k : = \sqrt{\frac{\tau}{|\A(u_k)|}}.
\]
As a consequence of Proposition \ref{linearcont}, it holds $\lim_{k \to + \infty}\lambda _k = 1$. 

Let us consider the rescaled sequence $(v_k)$ defined as $v_k := \lambda_k u_k$. Since
\[
\A(v_k) = \lambda_k^2\A(u_k) = \tau,
\]
we have that $(v_k) \subset \M_\tau \cap C^{0,1}$. In addition, we get that $v_k \to u$ in $W^{1,1}$. Indeed
\[
\|u-v_k\|_{W^{1,1}} \leq \|u-u_k\|_{W^{1,1}} + |1-\lambda_k|\|u_k\|_{W^{1,1}} \to 0,
\]
where we used that $(u_k)$ is equibounded in $W^{1,1}$. 

Let now $\varepsilon >0$ be fixed. By definition of infimum, there exists $u \in \M_\tau$ such that 
\[
\F(u) \leq S_F(\tau) + \frac{\varepsilon}{2}.
\]
As we have seen, there exists a sequence $(v_k) \subset \M_\tau \cap C^{0,1}$ such that $v_k \to v$ in $W^{1,1}$. By Lemma \ref{strongcont} there exists $\overline k$ large enough such that
\[
|\F(v_{\overline k}) - \F(u) | \leq \frac{\varepsilon}{2}.
\]
As a consequence, we get that
\begin{equation}\label{infWHLip4}
\inf_{v \in \M_\tau \cap C^{0,1}}\F(v) \leq \F(v_{\overline k}) \leq \F(u) + \frac{\varepsilon}{2}\leq S_F(\tau) + \varepsilon.
\end{equation}
Since this holds for every $\varepsilon >0$, then \eqref{infWHLip4} together with \eqref{infWHLip0} proves \eqref{infWHLip1}. 

To conclude, it suffices to show that
\begin{equation}\label{infWHLip5}
\inf_{v \in \M_\tau \cap C^{0,1}}\F(v) = \inf_{v \in \M_\tau \cap H_F^1}\F(v).
\end{equation}
To this end, we are going to prove that for every $v \in \M_\tau \cap C^{0,1}$ there exists $w \in \M_\tau \cap H_F^1$ such that $\F(v) = \F(w)$. Then \eqref{infWHLip5} follows arguing as in the proof of \eqref{infWHLip4}

Fix $v \in \M_\tau \cap C^{0,1}$. To begin with, we find $u \in \M_\tau \cap C^{0,1}$ such that $u$ has no constancy intervals and $\F(u) = \F(v)$. First of all,  it holds that $v$ possesses at most denumerably constancy intervals. Indeed, recall that $[a,b] \subset [0,1]$ is a constancy interval for $v$ if $a<b$ and $v(t)$ is constant when $t \in [a,b]$. Then, if we denote by $\mathcal{P}$ the collection of constancy intervals of $v$, since for every $I=[a, b] \in \mathcal{P}$ we can find a rational number $q \in (a,b)$ and $\mathbb{Q}$ is countable, $\mathcal{P}$ is at most countable. 

As a consequence, we can write $\mathcal{P} = \{I_n\}$ with $I_n = [a_n, b_n]$, where $n\in \N$ or $n = 1, \ldots, \overline n$ for some $\overline n \in \N$, and $I_n \cap I_m = \emptyset$ when $n \neq m$. If the number of constancy intervals is finite, we still denote them by $I_n$, $n \in \N$, with the convention that $I_n = \emptyset$ when $n>\overline n$. Notice that 
\[
L := \sum_{n=1}^{+\infty} |I_n| < 1,
\]
otherwise $v$ is constant and $v \not \in \M_\tau$. Let then define $\xi: [0, 1-L] \to [0,1]$ as
\[
\xi (s) := s + \sum_{j \in J_s}|I_j|,
\]
where $J_s$ is defined as
\[
J_s := \left\{j \in \N \ \Bigg|\ a_j < s + \sum_{i<j}(b_i-a_i)\right\}.
\]
It is possible to see that $\xi$ is a bijection from $[0, 1-L]$ to $[0,1] \setminus \cup_{n \in \N}(a_n, b_n]$.
Therefore, the function $u: [0,1] \to \R^2$ defined as $u(s) = v(\xi((1-L)s))$ has no constancy intervals. Moreover, some standard computations show that $u \in C^{0,1}$ and that it holds $\F(u) = \F(v)$ and $\A(u) = \A(v)$, as needed.

Let then be $u \in \M_\tau$ with no constancy intervals and such that $\F(v) = \F(u)$.
We define
\[
h:= \int_0^1 F(u, \dot u) \de t \quad \text{ and }\quad \sigma(t): = h^{-1}\int_0^t F(u, \dot u) \de t.
\]
Since $u$ does not have constancy intervals and thanks to \eqref{F:ell}, it turns out that the function $\sigma: [0,1] \to [0, 1]$ is continuous and strictly increasing, then it is a bijection.
Moreover, for every $t_1, t_2 \in [0,1]$ with $t_1 \leq t_2$ it holds that
\[
\sigma(t_2) - \sigma(t_1) = h^{-1}\int^{t_2}_{t_1}F(u,\dot u) \de t\leq h^{-1} m_2 \int_{t_1}^{t_2}|\dot u| \de t \leq h^{-1	}m_2|\dot u|_\infty (t_2- t_1),   
\]
which implies that $\sigma \in C^{0,1}([0,1]; [0,1])$. As an immediate consequence, $\sigma^{-1}$ turns out to be almost everywhere differentiable. Then, let define $w:\R / \Z \to \R^2$ as $w(s) := u\left( \sigma^{-1}\left(s \right) \right)$.
Let $0 \leq s_1 \leq s_2 \leq 1$ and $t_1 \leq t_2$ such that $\sigma(t_i) = s_i$ for $i = 1,2$. We have that 
\[
\begin{aligned}
|w(s_2) - w(s_1)| &= |u(\sigma^{-1}(s_2))- u(\sigma^{-1}(s_1))| = |u(t_2) - u(t_1)| = \left|\int^{t_2}_{t_1}\dot u \de t \right|\\
&\leq \int^{t_2}_{t_1}|\dot u| \de t \leq \frac{1}{m_1}\int_{t_1}^{t_2}F(u, \dot u)\de t = \frac{h}{m_1}|\sigma(t_2) - \sigma(t_1)| = \frac{h}{m_1}|s_2 - s_1|,
\end{aligned}
\]
which implies that $w \in C^{0,1}$.

Again, consider $0 \leq s_1\leq s_2 \leq 1$ and $t_1 \leq t_2$ such that $\sigma(t_i) = s_i$ for $i = 1,2$.
Since from $\sigma^{-1}(\sigma(t)) = t$ we get that 
\[
\frac{\de \sigma^{-1}}{\de s}(\sigma(t))\frac{\de\sigma}{\de t}(t) = 1, \quad \text{a.e. }t \in[0,1],
\]
using also \eqref{Fhyp} we infer that
\begin{equation}\label{infWHLip9}
\begin{aligned}
\int^{s_2}_{s_1} F(w, \dot w) \de s &= \int_{s_1}^{s_2}F\left(u(\sigma^{-1}(s), \frac{\de}{\de t}u(\sigma^{-1}(s))\frac{\de}{\de s}\sigma^{-1}(s) \right)\de s \\
&= \int_{s_1}^{s_2}F\left(u(\sigma^{-1}(s), \frac{\de}{\de t}u(\sigma^{-1}(s))\right)\frac{\de}{\de s}\sigma^{-1}(s) \de s\\
&= \int_{t_1}^{t_2}F\left(u, \dot u\right)\de t = h(\sigma(t_2) - \sigma(t_1)) = h (s_2-s_1).
\end{aligned}
\end{equation}
where we performed the change of variables $s = \sigma(t)$, which is admissible since $\sigma$ is Lipschitz continuous.  
Arguing in the same way, we obtain that $\A(w) = \A(u)$ and $\F(w) = \F(u)$. 

Fix $\varepsilon >0$. As a consequence of the Lebesgue differentiation Theorem and thanks to \eqref{infWHLip9}, for a.e. $s_0 \in [0,1]$ it holds that 
\[
\begin{aligned}
|F(w(s_0), \dot w(s_0)) - h| &= \left|\frac{1}{2\varepsilon}\int_{s_0-\varepsilon}^{s_0+\varepsilon}F(w(s_0), \dot w(s_0)) - F(w, \dot w) \de s\right| \\
&\leq \frac{1}{2\varepsilon}\int_{s_0-\varepsilon}^{s_0+\varepsilon}\left|F(w(s_0), \dot w(s_0)) - F(w, \dot w) \right|\de s \to 0,
\end{aligned}
\]	
as $\varepsilon \to 0$. Then $F(w(s), \dot w(s)) = h$ for a.e. $s \in [0,1]$ and $w \in \M_{\tau}\cap H^1_F$, thus proving the last equality of the Lemma.  
\end{proof}

As a consequence of the previous result, we could take minimizing sequences of $S_F(\tau)$ which are contained in $H^1$, which is a reflexive space. On the other hand, as already pointed out, hypothesis \eqref{Fhyp} are not enough to assure us that sequences minimizing $S_{F}(\tau)$ are bounded in $\M_\tau\cap H^1$.
To overcome this, we are going to study an associated minimization problem where, thanks to \eqref{F:per}, we can recover compactness of the minimizing sequence without the coercivity, and in a second time we will prove that its minimizers coincide with the minimizers of $\F$ in $\M_\tau \cap H^1$. To this end, let us define the variational integral $\Q: H^1 \to \R$ as 
\[
\Q(u) := \int_0^1 Q(u, \dot u) \de t, 
\]
where the Lagrangian $Q$ is defined by
\[
Q(p, q) := F^2(p, q). 
\]
As a consequence of \eqref{Fhyp} and \eqref{F:per} we obtain the following
\begin{subnumcases}{\label{Qprop}}
Q \in C^0(\R^2 \times \R^2; \R); \label{Q:reg}\\
Q(p+n,q) = Q(p,q), \quad \forall n \in \Z^2; \label{Q:per}\\
Q(p,\theta q_1 + (1-\theta)q_2) \leq \theta Q(p, q_1) + (1-\theta)Q(p, q_2) \quad \forall \theta \in [0,1];\label{Q:conv}\\
Q(p, \lambda q) = \lambda^2 Q(p,q) \ \forall \lambda >0,\label{Q:hom}\\
m_1^2 |q|^2 \leq Q(p,q) \leq m_2^2 |q|^2 \label{Q:ell}.
\end{subnumcases}

We have the following existence result.

\begin{lemma}\label{existisoper}
Let $\tau >0$. Then there exists $u_\tau \in \M_\tau\cap H^1$ such that 
\begin{equation}\label{infQ}
\mathcal{Q}(u_\tau) = \inf_{v \in \M_\tau\cap H^1}\mathcal{Q}(v).
\end{equation}
\end{lemma}
\begin{proof}
In the first step of the proof, we show that there always exists a minimizing sequence for $\Q$ in $\M_\tau \cap H^1$ which is bounded with respect to the $H^1$-norm introduced in \eqref{barinorm}.
One one hand, since $\mathcal{Q}(u) \geq 0$ for all $u \in \M_\tau \cap H^1$, the infimum in \eqref{infQ} is finite and there exists a sequence $(u_k) \subset \M_\tau\cap H^1$ such that $\Q(u_k) \to \inf_{\M_\tau\cap H^1}\Q$ as $k \to +\infty$. In addition, as a consequence of \eqref{Q:ell} we have that there exists $C >0$ such that 
\[
\int_0^1 |\dot u_k|^2 \de t \leq C \quad \forall k.
\]

Let now $p_k = [u_k]$. We have only two possibilities: either $(p_k) \subset \R^2$ is a bounded sequence and the claim is proved, or $|p_k| \to +\infty$. If that is the case, we can always choose a sequence of points of the plane $(n_k) \in \Z^2$ being such that $|p_k - n_k| \leq \frac{\sqrt 2}{2}$, and then consider the new sequence of functions $v_k = u_k - n_k$. Using \eqref{F:per} we can see that $(v_k) \subset \M_\tau\cap H^1$ and that it still is a minimizing sequence with the same bound on the seminorm $|\dot v_k|_2$ as the original sequence $(u_k)$. 
In addition, it holds that $|[v_k]| = |[u_k]- n_k|  = |p_k - n_k| \leq \frac{\sqrt 2}{2}$ which proves the claim.  

Let $(u_k) \subset \M_\tau\cap H^1$ be a minimizing sequence bounded with respect to the $H^1$ norm. Up to subsequences there exists $u_\tau \in H^1$ such that $u_k \rightharpoonup u_\tau$ in $H^1$. Since the Lagrangian associated to $\A$ satisfy the hypothesis of Proposition \ref{linearcont} we obtain that $u_\tau \in \M_\tau$. Notice that since $\tau \neq 0$ we also have that $u_\tau \not \equiv const.$ Applying Lemma \ref{lscconvex} to the functional $\Q$ we obtain that 
\[
\inf_{v \in \M_\tau \cap H^1}\mathcal{Q}(v) \leq \mathcal{Q}(u_\tau) \leq \liminf_{k \to +\infty} \mathcal{Q}(u_k) = \inf_{v \in \M_\tau \cap H^1}\mathcal{Q}(v),
\]
which concludes the proof. 
\end{proof}

In the next Lemma we show that $u_\tau$ is indeed the minimizer of $S_{F}(\tau)$ that we were looking for. 

\begin{lemma}
Let $u_\tau \in \M_\tau \cap H^1$ be the minimizer given by Lemma \ref{existisoper}. It holds that $\F(u_\tau) = S_{F}(\tau)$. In addition, $u_\tau \in C^{0,1}$.
\end{lemma}
\begin{proof}

We begin by noticing that the set $\M_\tau$ is invariant under admissible parameter variations as defined in \eqref{admissiblevar}. Moreover, the hypothesis of Lemma \ref{intenergy} are satisfied by $Q = F^2$, so we can conclude that there exists $h^2 \geq 0$ such that $Q(u_\tau(t), \dot u_\tau(t)) = h^2 \geq 0$ for a.e. $t \in [0,1]$. 

Since $u_\tau \in \M_\tau \cap H^1$ is not a constant function, we have that $h \neq 0$, otherwise we would reach a contradiction due to \eqref{Q:ell}. Since $F \geq 0$, this implies that  
\[
F(u_\tau(t), \dot u_\tau(t)) = h >0 \quad a.e.\ t \in [0,1],
\]
i.e. $u_\tau \in \M_\tau \cap H^1_F$, where $H^1_F$ is defined in \eqref{quasinormal}.

Form Cauchy-Schwartz inequality we have that
\begin{equation}\label{CSineq}
(\F(v))^2 \leq \mathcal{Q}(v) \quad \forall \ v \in H^1
\end{equation}
and the equality sign holds if and only if $F(v(t), \dot v (t)) = const.$ a.e. on $[0,1]$, i.e. if $v \in H^1_F$.  
This, together with the previous step, readily implies that
\[
\left(\F(u_\tau)\right)^2 = \Q(u_\tau) = \inf_{\M_\tau \cap H^1} \Q = \inf_{\M_\tau \cap H^1_F} \Q = \left(\inf_{M_\tau \cap H_F^1} \F\right)^2, 
\]
therefore
\[
\F(u_\tau) = \inf_{M_\tau \cap H^1_F}\F.
\]
Hence the first part of the Lemma follows from Lemma \ref{infWHLip}.

As for the regularity of $u_\tau$, thanks to of \eqref{F:ell} we get that $\dot u_\tau \neq 0 $ for a.e. $t \in [0,1]$ and
\[
|\dot u_\tau (t)| \leq \frac{h}{m_1}\ \ a.e.\ t \in [0,1],
\]
hence $\dot u_\tau \in L^{\infty}([0,1])$. This readily implies that $u_\tau \in C^{0,1}$, thus concluding the proof.
\end{proof}

\end{subsection}

\begin{subsection}{Proof of Theorem \ref{mainteo2} and examples}\label{subsectproofth2}
\begin{proof}[Proof of Theorem \ref{mainteo2}]
Suppose that the Lagrangian $F$ satisfies \eqref{Fhyp} and \eqref{Fasper2}. We recall that
\[
\begin{aligned}
S_{F}(\tau) &= \inf_{v \in \M_\tau} \F(v) = \inf_{v \in \M_\tau}\int_0^1 F(v, \dot v)\de t\\
S_{F_{\infty}}(\tau) &= \inf_{v \in \M_\tau} \F_{\infty}(v) = \inf_{v \in \M_\tau} \int_0^1 F_{\infty}(v, \dot v) \de t. 
\end{aligned}
\]

We begin by proving that \eqref{energyorder} holds for every $\tau \in \R$. Thanks to Theorem \ref{mainteo1}, there exists $v_\tau \in \M_\tau$ such that $\F_{\infty}(v_\tau) = S_{F_{\infty}}(\tau)$. Moreover, let $(n_k) \in \Z^2$ be a sequence such that $|n_k| \to +\infty$ as $k \to +\infty$. We easily infer that
\begin{equation}\label{passag2001}
S_{F}(\tau) \leq \F(v_\tau + n_k) = \F(v_\tau + n_k) - \F_\infty(v_\tau + n_k) + S_{F_\infty}(\tau) \quad \forall k.
\end{equation}
Let $u_k := v_\tau + n_k$, and notice that $|[u_k]| \geq n_k -  |[v_\tau]|$, hence $|[u_k]| \to +\infty$ as $k \to + \infty$. To conclude, it suffices to prove that
\begin{equation}\label{exaspe9}
\lim_{k \to + \infty}\left|\F(u_k) - \F_\infty(u_k)\right| = 0.
\end{equation}
Since $|[u_k]| \to + \infty$ as $k \to + \infty$, and since $|\dot u_k|_2= |\dot v_\tau|_2 $ for every $k$, we have that
\[
|u_k(t_2) - u_k(t_1)| = \left| \int^{t_2}_{t_1}\dot u_k \de s \right|\leq C, \quad \forall t_1, t_2 \in [0,1].
\]
Therefore, for every $t \in [0,1]$ it holds that 
\begin{equation}\label{exaspe7}
+\infty \leftarrow |[u_k]| = \left|\int_0^1 u(s) \de s\right| = \left|\int_0^1 (u(s) - u(t))\de s + u(t)\right| \leq C + |u(t)|,
\end{equation}
that is, $\inf_{t \in [0,1]}|u_k| \to +\infty$ as $k \to + \infty$,. 

Since $F$ and $F_\infty$ are both continuous, and in particular uniformly continuous on compact subsets of $\R^2\times \R^2$, and thanks to \eqref{Fper:asint}, we have that for every $\varepsilon >0$ there exists $M>0$ such that
\begin{equation}\label{exaspe8}
\left|F(p, q) - F_\infty(p, q)\right| \leq \varepsilon,  \quad \forall |p|\geq M, \ |q| \leq1,
\end{equation}
where $M$ does not depend on $q$.
Fix $\varepsilon >0$. Thanks to \eqref{exaspe7} and \eqref{exaspe8} we have that there exists $\overline k$ such that for every $k > \overline k$ we get that
\[
\left|F(u_k(t) , \dot u_k(t)) - F_\infty(u_k(t), \dot u_k(t))\right| < \frac{\varepsilon}{1+C},
\]
for every $t \in [0,1]$ such that $|\dot u_k(t)| \leq 1$. 

In particular, for every $k > \overline k$ it holds that 
\[
\begin{aligned}
&|\F(u_k) - \F_\infty(u_k)| \\
&\leq \int_{\{|\dot u_k|\leq 1\}}\left|F(u_k, \dot u_k)- F_\infty(u_k, \dot u_k)\right| \de t + \int_{\{|\dot u_k|>1\}}|\dot u_k|\left|F\left(u_k, \frac{\dot u_k}{|\dot u_k|}\right)- F_\infty\left(u_k, \frac{\dot u_k}{|\dot u_k|}\right)\right|\de t\\
& < \frac{\varepsilon}{1+C}\left( 1 + \int_{0}^1|\dot u_k|\de t\right) \leq \varepsilon,
\end{aligned}
\]
were we used that, thanks to \eqref{F:hom} and \eqref{Fper:asint}, both $F$ and $F_\infty$ turn out to be homogeneous of degree one with respect to the second set of variables. 
Thanks to the arbitrariness of $\varepsilon$, \eqref{exaspe9} is proved.

It remains to prove that for every $\tau \in \R$ there exists a minimizer $u_\tau \in \M_\tau \cap C^{0,1}$ such that $\F(u_\tau) = S_F(\tau)$. 
Since \eqref{Fhyp} and \eqref{Fasper2} hold true for $F$ if and only if they hold true for the Lagrangian $F(p, -q)$, we can restrict to the case $\tau \geq 0$ without loss of generality. Moreover, the case $\tau = 0$ is treated in the same way as in the previous section, i.e. $S_{F}(0) = 0$ and it is achieved exactly by the constant functions. 

Fix $\tau > 0$, and suppose that \eqref{asexcond} holds. The proof follows verbatim the proof of Theorem \ref{mainteo1}, except for one fundamental step. 
Indeed, by Lemma \ref{infWHLip} we can consider a sequence $(u_k) \subset \M_\tau \cap H^1$ such that it minimizes the associated functional $\Q(u) : = \int_0^1 F^2(u, \dot u) \de t$.
Also in this case we readily get a uniform bound on the $L^2$ norms of $\dot u_k$. We claim that there exists $C>0$ such that $|[u_k]| \leq C$. If the claim holds true then the sequence $(u_k)$ turns out to be bounded in $H^1$ and we can conclude exactly as in Theorem \ref{mainteo1}.

Let then prove the claim. Let be $p_k := [u_k]$ and assume by contradiction that, for a subsequence, $|p_k|\to + \infty$ as $k \to +\infty$.  As seen in the proof of Theorem \ref{mainteo1}, there exists a sequence $(n_k) \in \Z^2$ such that $|p_k - n_k| \leq \frac{\sqrt{2}}{2}$ for every $k$. Let us define $v_k := u_k - n_k$. Thanks to the $\Z^2$-invariance of $\A$, we readily get that $(v_k) \subset \M_\tau$. Moreover, since $|\dot v_k|_2 = |\dot u_k|_2 \leq C$ and by construction $|[v_k]| = |p_k - n_k| \leq \frac{\sqrt{2}}{2}$, it is bounded in $H^1$ and there exists $v \in H^1$ such that $v_k \rightharpoonup v$ in $H^1$. In addition, as a consequence of Proposition \ref{linearcont}, we get that $v \in \M_\tau \cap H^1$. 

Taking into account \eqref{quasinormal}, Lemma \ref{infWHLip}, and recalling that the equality in \eqref{CSineq} is attained if and only if $u \in H_F^1$, we infer that
\begin{equation}\label{exaspe1}
\inf_{w \in \M_\tau \cap H^1}\Q(w) \leq \inf_{w \in \M_\tau \cap H^1_F}\Q(w) = \left( \inf_{w \in \M_\tau \cap H^1_F}\F(w)\right)^2 = S_F^2(\tau).
\end{equation}
On the other hand, thanks to H\"older inequality and by the properties of the $\liminf$, we get that
\begin{equation}\label{exaspe2}
\inf_{w \in \M_\tau \cap H^1}\Q(w) = \lim_{k \to + \infty} \Q(u_k) \geq \liminf_{k \to + \infty}\left( \F(u_k)\right)^2 \geq \left( \liminf_{k \to +\infty}\F(u_k)\right)^2.
\end{equation}

Arguing as in the previous part of the proof, we infer that \eqref{exaspe9} holds. This, together with Theorem \ref{lscconvex}, implies that
\begin{equation}\label{exaspe3}
\liminf_{k \to + \infty}\F(u_k) = \lim_{k \to +\infty}\left(\F(u_k) - \F_{\infty}(u_k)\right)+ \liminf_{k \to + \infty}\F_{\infty}(v_k) \geq \F_{\infty}(v) \geq S_{F_{\infty}}(\tau).
\end{equation}
Then, collecting together \eqref{exaspe1}, \eqref{exaspe2} and \eqref{exaspe3} we infer  
\[
S_{F_{\infty}}(\tau) \leq S_{F}(\tau),
\]
which contradicts \eqref{asexcond}, thus proving the claim and completing the proof of Theorem \ref{mainteo2}.
\end{proof}

We want to stress out the fact that condition \eqref{asexcond} is sufficient but not necessary for the existence of minimizer of $S_F(\tau)$, as the next examples show. 
Let $F_{\infty}, F_{0}: \R^2 \times \R^2 \to \R$ satisfy conditions \eqref{Fhyp}. In addition, suppose that $F_\infty$ and $F_{0}$ also satisfies, respectively, \eqref{F:per} and 
\begin{equation}\label{Fexamp}
\begin{cases}
F_{0}(p, q) \geq 0, & \forall (p,q) \in \R^2 \times \R^2,\\
\forall q \in \R^2, \quad F_{0}(p,q) \to 0 & \text{as }|p|\to + \infty. 
\end{cases}
\end{equation}
Then $F:= F_{\infty} + F_{0}$ satisfies conditions \eqref{Fhyp} and \eqref{Fasper2}.
Let $\tau \in \R$ be fixed. For every $u \in \M_\tau$ it holds that 
\[
S_{F_\infty}(\tau) \leq \F_\infty(u) \leq \F(u), 
\]
then by \eqref{energyorder} we have that $S_F(\tau) = S_{F_\infty}(\tau)$ and condition \eqref{asexcond} is not satisfied. 
If in particular $F_{0}(p, q) >0$ for all $(p, q) \in \R^2 \times \R^2$, then no minimizer for $S_F(\tau)$ can exist. Indeed, suppose by contradiction that there exist $u \in \M_\tau$ such that $\F(u) = S_F(\tau)$. Then we easily get
\[
S_{F}(\tau) = S_{F_\infty}(\tau) \leq \F_\infty(u) < \F_\infty(u) + \F_{0}(u) = \F(u) = S_F(\tau), 
\]
which is absurd. 

We are now going to provide an example where, even if condition \eqref{asexcond} is never satisfied, Problem \eqref{minprob} admits a minimizer for every $\tau \in \R$. 
Let $F_\infty$ and $F_0$ be as before. In addition, suppose that there exists an open spherical region $\omega \subset \Sf^1$ such that $F_{0}(\lambda p, q) = 0$ for every $(p,q) \in \omega\times\R^2$ and $\lambda>0$.
By Theorem \ref{mainteo1} we have that for every $\tau \in \R$ there exists $v_\tau \in \M_\tau$ such that $\F_\infty(v_\tau) = S_{F_\infty}(\tau)$.
Moreover, by \eqref{F:ell} we get that 
\[
|v_\tau(t_1) - v_\tau(t_2)| \leq |\dot v_\tau|_1 \leq \frac{S_{F_\infty}(\tau)}{m_1},
\]
hence $\text{diam }v_\tau \leq \frac{S_{F_\infty}(\tau)}{m_1}$. Then, we can always find $n_\tau \in \Z^2$ such that $v_\tau + n_\tau$ completely lies in the region of $\R^2$ where $F_{0} = 0$. As a consequence,
\[
\F(v_\tau + n_\tau) = \F_\infty(v_\tau) = S_{F_\infty}(\tau) = S_F(\tau),  
\]
thus proving that $v_\tau+ n_\tau$ is a minimizer for $F$.

Arguing as in the aforementioned example, we can prove the following sufficient condition for \eqref{asexcond} to hold.
\begin{lemma}\label{suffcondenerlev}
Let $F$ satisfy conditions \eqref{Fhyp} and \eqref{Fasper2}. 
If there exists $\omega \subset \Sf^1$ open spherical region such that is satisfied
\[
F(\lambda p, q)- F_\infty(\lambda p,q) < 0 \quad \forall (p,q) \in {\omega}\times \R^2, \lambda >0,
\]
then \eqref{asexcond} holds and Problem \eqref{minprob} admits a minimizer for every $\tau \in \R$.

\end{lemma}

\end{subsection}

\begin{subsection}{$(H-\lambda)$-loop problem: the periodic and asymptotically constant cases}\label{loopsub}
In this Section we see how previous results apply to Problem \eqref{Kloopprob}. In the first part, we show how to construct a vector field $Q_H: \R^2 \to \R^2$ satisfying \eqref{divergass} and some additional properties, in such a way that the Lagrangian associated to $\E_H(u)$, which is defined in \eqref{energhia}, satisfies either \eqref{Fhyp} and \eqref{F:per} when $H$ is periodic, or \eqref{Fhyp} and \eqref{Fasper2} when $H$ is asymptotically constant. 
As a consequence, we recover existence of extremals for the isoperimetric function
\begin{equation}\label{isoperifunction}
S_H(\tau):=\inf_{v \in \M_\tau}\E_H(v), 
\end{equation}
where $\M_\tau$ is defined \eqref{areaconst}.
Finally, we give the proof of Theorem \ref{mainteo3} and Theorem \ref{mainteo4}.

We begin by treating the case of periodic curvature. Let $H$ satisfy \eqref{k1hp}. Without loss of generality, we can always reduce to the case $[H] = 0$. Indeed, the function $\tilde H = H - [H]$ still satisfies \eqref{k1hp}, and is such that $[\tilde H ]=0$. 
Setting $Q_H(p) = Q_{\tilde H}(p) + \frac{[H]}{2}p$, where $Q_{\tilde H}$ is such that $\text{div }Q_{\tilde H} = \tilde H$, we readily get that $\text{div }Q_H = H$.   
Moreover, it holds that
\[
\E_H (u) = \E_{\tilde H}(u) + \frac{[H]}{2}\int_0^1u \cdot i\dot u \de t = \E_{\tilde H}(u) + [H]\tau, \quad \forall u \in \M_\tau, 
\]
therefore the area constrained minimizers of $\E_H$ and $\E_{\tilde H}$ coincide.

The following holds.
\begin{prop}\label{minexper}
Let $H$ be such that both \eqref{k1hp} and $[H]=0$ are satisfied. Then there exists $Q_H:\R^2 \to \R^2$ such that
\begin{equation}\label{Q1prop}
\begin{cases}
Q_{H} \in C^{1, \alpha}(\R^2; \R^2)\\
|Q_{H}|_\infty < 1\\
Q_{H}(p + n) = Q_{H}(p) & \forall p \in \R^2, \ n \in \Z^2\\
\text{div }Q_{H}= H & \text{ a.e. }p \in \R^2.
\end{cases}
\end{equation}
Moreover, for every $\tau \in \R$ there exists $u_\tau \in \M_\tau \cap C^{0,1}$ such that $\E_H(u_\tau) = S_H(\tau)$ and $|\dot u_\tau| = \text{const.}$.
\end{prop}
\begin{proof}
Consider the Poisson equation
\begin{equation}\label{poissontorus}
\begin{cases}
v \in H^1(\R^2 / \Z^2; \R)\\
-\Delta v = H & \text{ in }\R^2,
\end{cases}
\end{equation}
where $H^1(\R^2 / \Z^2; \R) := \{u \in H_{loc}^1(\R^2; \R) \ |\ u(p + n) = u(p) \ \forall \ p \in \R^2, n \in \Z^2\}$. 
It is possible to see (for instance by means of the Riesz representation Theorem together with the Poincaré-Wirtinger inequality) that, since $H \in L^2(\R^2/\Z^2;\R)$ and $[H] = 0$, there exists $v \in H^1(\R^2 / \Z^2; \R)$ solution of the above equation, which is unique up to transformations in the form $v(p + n) + C$, with $C \in \R$ and $n \in \Z^2$.   
In addition, since $H \in C^{0,\alpha}(\R^2; \R)$ we have that $v \in C^{2,\alpha}(\R^2; \R)$ by elliptic regularity.

Through a standard identification, equation \eqref{poissontorus} can be seen as a Poisson equation on the flat torus $\mathbb{T}^2$, which possesses a structure of Riemannian manifold with the metric induced by $\R^2$. With respect to this structure, it holds that $\mathbb{T}^2$ has zero Ricci curvature and diameter $D= \frac{\sqrt{2}}{2}$. 
As a consequence, thanks to \cite[Theorem 1.1]{WU} we obtain the sharp gradient estimate 
\begin{equation}\label{torusext}
|\nabla v|_\infty \leq \frac{\sqrt{2}}{8}\sup_{p_1 \neq p_2\in \R^2}|H(p_1) - H(p_2)|.
\end{equation}
Then, it is immediate to see that, thanks to \eqref{k1hp}, the vector field defined as $Q_{H} := \nabla v$ satisfies \eqref{Q1prop}.

Let $F: \R^2 \times \R^2 \to \R$ be defined as 
\[
F(p, q) := |q| + Q_H(p) \cdot i q,
\]
in such a way that
\[
\E_H(u) = \int_0^1 F(u, \dot u) \de t.
\]
Some simple computations show that, thanks to \eqref{Q1prop}, $F$ satisfies \eqref{Fhyp} and \eqref{F:per}.
Thus, as a consequence of Theorem \ref{mainteo2} we get that for every $\tau \in \R$ there exists $u_\tau \in \M_\tau \cap C^{0,1}$ such that $S_H(\tau) = \E_H(u_\tau)$ and $F(u_\tau, \dot u_\tau) = \text{const}$. Finally, it suffices to reparametrize $u_\tau$ arguing as in the proof of Lemma \ref{infWHLip} to get that $|\dot u_\tau|= \text{const.}$,
which complete the proof of the Proposition.
\end{proof}

We turn our attention to the asymptotically constant case. Let $H$ satisfy \eqref{k2hp}. Arguing as before, we infer that it suffices to treat the case $H^\infty=0$. Then, the following holds.
\begin{prop}\label{minexac}
Let $H$ be such that both \eqref{k2hp} and $H^\infty=0$ are satisfied. There exists $Q_H:\R^2 \to \R^2$ such that
\begin{equation}\label{Q2prop}
\begin{cases}
Q_H \in C^{1,\alpha} \cap H^1(\R^2; \R^2)\\
|Q_{H}|_\infty < 1\\
|Q_{H}(p)| \to 0  & \text{as } |p| \to +\infty \\
\text{div } Q_{H} = H & \text{a.e.}.
\end{cases}
\end{equation}
Moreover, for every $\tau \in \R$ there exists $u_\tau \in \M_\tau \cap C^{0,1}$ such that $\E_H(u_\tau) = S_H(\tau)$ and $|\dot u_\tau|=\text{const.}$.
\end{prop}
\begin{proof}
We consider the Poisson equation
\begin{equation}\label{poisson}
\begin{cases}
v \in H^1(\R^2; \R)\\
-\Delta v = H & \text{ on }\R^2.
\end{cases}
\end{equation}
It holds that (see e.g. \cite[Lemma 1.8.10]{ziem}),
\[
|H|_2 \leq |H|_{(2,2)} \leq \frac{\sqrt{2}}{2}|H|_{(2,1)},
\]
where $|\cdot|_{(2,1)}$ and $|\cdot|_{(2,2)}$ are the norms of the Lorentz spaces $L(2,1)$ and $L(2,2)$, respectively. 
Then, from \eqref{k2hp} we infer that $H \in C^{0, \alpha} \cap L^2(\R^2; \R)$. Therefore, by standard Poisson equation theory, we get that there exists a unique solution $v: \R^2 \to \R$ of \eqref{poisson} such that $v \in C^{2,\alpha} \cap W^{2,2}(\R^2; \R)$. 
Moreover, as proved in \cite[Theorem 2, Corollary]{cianchi}, it holds that
\[
|\nabla v|_\infty \leq \left(\frac{\pi}{2}\right)^{\frac{3}{2}}|H|_{(2,1)},
\]
where the appearing constant is sharp.

As a consequence, the function defined as $Q_{H} := \nabla v$ satisfies \eqref{Q2prop}. Indeed, most of the properties are obtained by construction, and we limit ourselves to prove that since $Q_{H}$ belongs to $L^2\cap C^{0,1}(\R^2; \R^2)$ then $|Q_{H}(p)| \to 0$ as $|p| \to  + \infty$. Suppose by contradiction that there exist $\delta>0$ and a sequence $(p_k) \subset \R^2$ being such that $|p_k| \to + \infty$ and $|Q_H(p_k)|\geq \delta$. Fix $\varepsilon >0$. For every $\xi \in \R^2$ with $|\xi|\leq \varepsilon$ it holds 
\[
||Q_H(p_k + \xi)| - |Q_H(p_k)|| \leq |Q_H(p_k + \xi) - Q_H(p_k)| \leq C|\xi|\leq C\varepsilon,
\]
where $C$ is the Lipschitz constant of $Q_H$ and does not depend on $\varepsilon$. Therefore, we can take $\varepsilon$ small enough such that
\begin{equation}\label{passag300909}
|Q_H(p)| \geq \delta - C\varepsilon >0 \quad \forall p \in B_{\varepsilon_0}(p_k).
\end{equation}
Let be $R_k:= |p_k|-\varepsilon$. On one hand, since $Q_H \in L^2(\R^2; \R^2)$ we easily infer that
\[
\lim_{k \to +\infty}\int_{\R^2 \setminus B_{R_k}}|Q_H|^2 \de p = 0, 
\]
On the other hand, since $B_{\varepsilon_0}(p_k) \subset (\R^2\setminus B_{R_k})$ by \eqref{passag300909} we infer
\[
\int_{\R^2 \setminus B_{R_k}}|Q_H|^2 \de p \geq \int_{B_{\varepsilon}(p_k)}|Q_K|^2 \de p (\geq \delta - C\varepsilon)^2\pi\varepsilon^2 \quad \forall k, 
\]
thus a contradiction. 

Define $F: \R^2 \times \R^2 \to \R$ as 
\[
F(p, q) := |q| + Q_H(p) \cdot i q,
\]
in such a way that
\[
\E_H(u) = \int_0^1 F(u, \dot u) \de t.
\]
Thanks to \eqref{Q2prop}, $F$ satisfies \eqref{Fhyp} and \eqref{Fasper2} with $F_\infty = |q|$, thus it also holds $\F_\infty(u) = L(u)$.

We claim that \eqref{asexcond} holds for every $\tau \in \R$, then the result follows from Theorem \ref{mainteo2} and a Lipschitz reparametrization.
Notice that the sufficient condition given by Lemma \ref{suffcondenerlev} is never satisfied. Indeed $Q_{H}(p_0)\cdot iq$ can not have the same sign for every $q \in \R^2$ for any $p_0 \in \R^2$. On the other hand, as a consequence of \eqref{classicineq}, we have that $S_{F_\infty}(\tau) = S\sqrt{|\tau|}$ and that it is realized exactly by the functions of the family 
\begin{equation}\label{circfam}
\omega_{p}(t) = p + \sqrt{\frac{|\tau|}{\pi}}e^{-i2\pi\text{sign}(\tau)t},
\end{equation}
whose supports are the circles of radius $\sqrt{\frac{|\tau|}{\pi}}$. In particular, by \eqref{conecond} we infer that there exists $p_0 \in \R^2$ such that $\text{sign}(\tau)H(p) <0$ for every $p \in B_{r}(p_0)$. Since the support of $\omega_{p_0}$ is a Jordan curve, thanks to \eqref{divergencejord} we get that
\[
S_H(\tau) \leq \E_H(\omega_{p_0}) = S\sqrt{|\tau|} + \text{sign}(\tau)\int_{B_{\sqrt{|\tau|/\pi}}(p_0)}H(p) \de p < S_0(\tau).
\]
Thus \eqref{asexcond} is satisfied for every $\tau \in \R$ and the proof is complete.
\end{proof}

The remaining part of the section is dedicated to prove that the minimizers given by Proposition \ref{minexper} and Proposition \ref{minexac} are indeed solutions of Problem \eqref{Kloopprob}.

\begin{lemma}\label{E-L-Eq}
Let $H \in C^0(\R^2; \R)$ be such that there exists $Q_H$ satisfying \eqref{divergass}. Let $\tau \in \R \setminus \{0\}$ and suppose that there exists $u_\tau \in \M_\tau$ which minimizes $S_H(\tau)$. Then there exists $\lambda \in \R$ such that
\begin{equation}\label{weakequation}
\E_H'(u_\tau)[\varphi] - \lambda \A'(u_\tau)[\varphi] = 0 \quad \forall \varphi \in W^{1,1},
\end{equation}
namely, $u_\tau$ is a weak solution of Problem \eqref{Kloopprob}.
\end{lemma}
\begin{proof}
We begin by recalling that the area functional $\A$ can be seen as a particular case of functional $\A_H$, with $H \equiv 1$ and $Q_H(p) := \frac{p}{2}$. Then by Lemma \ref{Qfuncreg} we readily get that $\A \in C^1(W^{1,1}; \R)$ with
\[
\A'(u)[\varphi] = \int_0^1 \varphi \cdot \dot i u, \quad u,\varphi \in W^{1,1}.
\] 
Since $\tau \neq 0$ and $\tilde \M_0 \subset \M_0$, we readily get that both functionals $\A$ and $\E_H$ are differentiable in a neighborhood of $u_\tau$. Using again that $u_\tau \not \in \tilde \M_0$, we infer that $\A'(u_\tau)[W^{1,1}]= \R$. Then we can apply \cite[Theorem 26.1]{Deimling} in order to get that there exists $\lambda \in \R$ being such that
\[
\E'_H(u_\tau)[\varphi] - \lambda \A'(u_\tau)[\varphi] = 0, \quad \forall \varphi \in W^{1,1},
\]
which indeed is the weak formulation of \eqref{Kloopprob}, as desired. 
\end{proof}

\begin{proof}[Proof of Theorem \ref{mainteo3} and Theorem \ref{mainteo4}]
Let $H$ be such that assumptions \eqref{k1hp} are satisfied. 
Define
\[
Q_H(p) = Q_{H-[H]}(p) + \frac{[H]}{2}p, 
\] 
where $Q_{H-[H]}$ is given by Proposition \ref{minexper} and satisfies \eqref{Q1prop}. In particular, $Q_H$ satisfies \eqref{divergass}. Applying Proposition \ref{minexper} we get that there exists $u_{\tau} \in C^{0,1}\cap H^1$ which minimizes $S_H(\tau)$ and satisfies $|\dot u_\tau| = \text{const.}$. Therefore, it suffices to apply Lemma \ref{E-L-Eq} first, then Lemma \ref{solreg} to the curvature $H-\lambda$, in order to readily infer Theorem \ref{mainteo3}.

As for Theorem \ref{mainteo4}, let $H$ be such that \eqref{k2hp} are satisfied. Also in this case it suffices to define 
\[
Q_H(p) = Q_{H-H^\infty}(p) + \frac{H^\infty}{2}p,
\]
where $Q_{H-H^\infty}$ is given by Proposition \ref{minexac} and satisfies \eqref{Q2prop}, and argue as before to get the desired result.

\end{proof}

\end{subsection}

\begin{subsection}{$(H-\lambda)$-loop problem: the mixed case}\label{mixedsub}
This section is devoted to the proof of Theorem \ref{mainteo5}. Let then $H_1$, $H_2$ satisfy \eqref{k1hp} and \eqref{k2hp}, respectively. The main difference between the previous cases and the mixed case is that, in general, it is not assured for every $\tau \in \R$ that the condition
\[
S_{H_1 + H_2}(\tau) < S_{H_1}(\tau)
\]
holds. Indeed, the argument used in Proposition \ref{minexac} does not straightforward applies; by \eqref{winddiv} we have that, given $v_\tau$ a minimizer of $S_{H_1}(\tau)$, it holds
\[
\A_{H_2}(v_\tau) = \int_0^1 Q_{H_2}(v_\tau) \cdot i \dot v_\tau \de t = \int_{\R^{2}}\omega_{v_\tau}(p)H_2(p) \de p,
\] 
where $\omega_{v_\tau}(p)$ is the winding number of $v_\tau$. Since a priori $v_\tau$ is not a simple curve, we don't recover any information on the sign of $\A_{H_2}(v_\tau)$.

Let us define the functional $\D: H^1 \to \R$ as 
\[
\D(u): = \sqrt{\int_0^1|\dot u|^2 \de t}.
\]
By H\"older inequality we have that 
\[
L(u) \leq \D(u) \quad u \in H^1,
\]
and the equality is attained when $|\dot u| = const.$ for a.e. $t \in [0,1]$.
Then we infer that whenever there exists a minimizer $u_\tau \in \M_\tau \cap H^1$ for $S_H(\tau)$ such that $|\dot u_\tau| = \text{const.}$, as in the case of Proposition \ref{minexper}, it holds
\begin{equation}\label{dirichcondconst}
S_H(\tau) = \E_H(u_\tau) = \inf_{v \in \M_\tau \cap H}\D(v) + \A_H(v).
\end{equation}

On the other hand, we notice that for every $\varepsilon \in (0,1]$ the function $\varepsilon H_1$ still satisfies \eqref{k1hp}. Then, as a consequence of Theorem \ref{mainteo3}, there exists $u_{\varepsilon, \tau} \in \M_\tau \cap C^2$ which solves the equation
\[
\ddot u_{\varepsilon, \tau} = i \D(u_{\varepsilon, \tau})(\varepsilon K_1(u_{\varepsilon, \tau}) - \lambda_{\varepsilon, \tau})\dot u_{\varepsilon, \tau},
\]
for some $\lambda_{\varepsilon, \tau} \in \R$ and realizes $S_{\varepsilon H_1}(\tau) = \E_{\varepsilon H_1}(u_\tau)$. 

As the following Proposition shows, if $\varepsilon$ is small enough such minimizers are simple. 

\begin{prop}\label{simplemin}
Let $H_1$ be such that assumptions \eqref{k1hp} are satisfied. For every $0< \tau_0 < \tau_1$, there exists $\varepsilon_0 = \varepsilon_0(\tau_0, \tau_1) \in (0,1)$ such that for every $\tau$ with $\tau_0< |\tau|< \tau_1$ and for every $\varepsilon \in (0, \varepsilon_0)$, the support of every minimizer of $S_{\varepsilon K_1}(\tau)$ is a Jordan curve.
\end{prop}
\begin{proof}
Let $H_1$ be such that \eqref{k1hp} are satisfied, and let $Q_{H_1}$ be the associated vector field, as defined in the proof of Theorem \ref{mainteo3}. First of all, we notice that for every $\varepsilon >0$ we can take $Q_{\varepsilon H_1} = \varepsilon Q_{H_1}$, since it obviously holds $\text{div }Q_{\varepsilon H_1} = \varepsilon \text{ div }Q_{H_1} = \varepsilon H_1$. 
In particular, it holds $|Q_{\varepsilon H_1}|_\infty \leq \varepsilon$. 
Moreover, without loss of generality we can suppose that $[H_1] = 0$, since $S_{\varepsilon H_1}(\tau)$ and $S_{\varepsilon(H_1-[H_1])}(\tau)$ share the same minimizers.

Since every minimizer of $S_{\varepsilon H_1}(\tau)$ is closed and continuous, we only have to prove that it is injective in $(0,1)$. To this end, we are going to argue by contradiction.
Suppose that there exist $0< \tau_0 < \tau_1$ such that there exist three sequences $(\varepsilon_k) \subset (0,1)$, $(\tau_k) \subset (-\tau_1, -\tau_0) \cup (\tau_0, \tau_1)$ and $(u_k) \subset H^1$, being such that $\varepsilon_k \to 0^+$ and $u_k \in \M_{\tau_k}$ is a minimizer of $S_{\varepsilon_k H_1}(\tau_k)$ which is not injective in $(0,1)$.
Since $|\tau_k| \in (\tau_0, \tau_1)$, there exists $\hat \tau \neq 0$, $|\hat \tau| \in [\tau_0, \tau_1]$, and a subsequence being such that $\tau_k \to \hat \tau$. Moreover, extracting if necessary a further subsequence, we have that every $\tau_k$ has the same sign of $\hat \tau$. Then, without loss of generality, we can assume that $\tau_k, \hat \tau >0$. 
 
It holds that
\begin{equation}\label{tecnicality34}
S_{\varepsilon_k H_1}(\tau_k) - S\sqrt{\tau_k} \leq C\varepsilon_k,\quad \forall k,
\end{equation}
where $C$ depends only on $\tau_1$ and $K_1$ but not on $\varepsilon_k$ nor $\tau_k$. 
Indeed, let $\omega_{k} \in \M_{\tau_k}$ be defined as
\[
\omega_k(t) = \sqrt{\frac{\tau_k}{\pi}}e^{-2\pi i t}.
\]
Since it belongs to the family \eqref{circfam}, it realizes \eqref{classicineq} as an equality. Therefore we get that
\[
S_{\varepsilon_k H_1}(\tau_k) - S\sqrt{\tau_k} \leq \D(\omega_k) - S\sqrt{\tau_k} + \A_{\varepsilon_k H_1}(\omega_k) \leq |Q_{\varepsilon_k H_1}|_\infty|\dot \omega_k|_\infty \leq C\varepsilon_k \sqrt{\tau_k} \leq C\varepsilon_k,
\]
where once again we used \eqref{k1hp}.

Arguing as in the proof of Theorem \ref{mainteo1}, we can always assume that there exists $C>0$ such that $|[u_{k}]|\leq C$. On the other hand, by 
\[
S_{\varepsilon_k H_1}(\tau_k) \geq \D(u_{k}) - |\A_{\varepsilon_k H_1}(u_{k})| \geq (1 - C\varepsilon_k) \D(u_{k}),
\]
together with \eqref{tecnicality34} we easily infer that
\begin{equation}\label{tecnicasint}
S\sqrt{\tau_k} \leq \D(u_{k}) \leq (1 + C_0\varepsilon_k)(S\sqrt{\tau_k} + C_1 \varepsilon_k) \leq (1 + C_0)(S \sqrt{\tau_1} + C_1),
\end{equation}
where $C_0$ and $C_1$ do not depend neither on $\tau_k$ nor $\varepsilon_k$. 
Therefore $(u_{k})$ is a bounded sequence in $H^1$ and there exists $u_\infty \in H^1$ such that, up to subsequences, $u_{k} \rightharpoonup u_\infty$ in $H^1$ as $k \to + \infty$. 

On one hand, thanks to Fatou Lemma and taking the limit as $k \to + \infty$ in \eqref{tecnicasint} we infer
\[
\D(u_\infty) \leq \lim_{k \to +\infty} \D(u_{k}) = S\sqrt{\hat\tau}.
\]
On the other hand, since by Proposition \ref{linearcont} we have that 
\[
\hat \tau= \lim_{k \to +\infty} \tau_k =  \lim_{k \to +\infty}\A(u_{k}) = \A(u_\infty),
\]
by the isoperimetric inequality \eqref{classicineq} we obtain that $S \sqrt{\hat \tau}\leq \D(u_\infty)$. Therefore $\D(u_\infty) = S \sqrt{\hat\tau}$, i.e., the function $u_\infty$ realizes the infimum of the isoperimetric inequality, hence it has to be in the form \eqref{circfam}.
Notice that we have actually proved something more, that is 
\begin{equation}\label{tecnic999}
\lim_{k \to + \infty}\D(u_{k}) = \D(u_\infty) = S\sqrt{\hat \tau},
\end{equation}
which, together with Sobolev-Morrey embedding, implies that $u_{k} \to u $ in $H^1$.

Applying Lemma \ref{E-L-Eq} to the functions $u_k$ and testing \eqref{weakequation} with $u_{k}$ we get that
\begin{equation}\label{tecnicality272}
\D(u_{k}) + \int_0^1\varepsilon_k H_1(u_{k}) u_{k} \cdot i \dot u_{k} \de t = 2\lambda_{k} \tau_k, 
\end{equation}
for some $\lambda_k \in \R$. 
Thanks to the Poincaré-Wirtinger inequality, together with the fact that $(u_{k})$ is uniformly bounded in $H^1$, we infer that
\[
\left| \int_0^1\varepsilon_k H_1(u_{k}) u_{k} \cdot i \dot u_{k} \de t\right|\leq \varepsilon_k |H_1|_\infty\sqrt{|[u_{k}]|^2 + (S^2)^{-1} \pi \D^2(u_{k})}\D(u_{k}) \leq C\varepsilon_k,
\]
where the constant $C>0$ does not depends on $\varepsilon_k$ nor $\tau_k$. Hence, taking the limit as $k \to +\infty$ in \eqref{tecnicality272} we obtain 
\begin{equation}\label{lambdalimtec}
\lim_{k \to + \infty}\lambda_{k} = \frac{S \sqrt{\hat \tau}}{2\hat \tau} \quad \text{ and }\quad \left|\lambda_k - \frac{S\sqrt{\hat \tau}}{2\hat \tau}\right| \leq C\varepsilon_k,
\end{equation}
for some $C>0$ which does not depend on $\tau_k$ nor $\varepsilon_k$. 
Recalling that we are assuming that the functions $u_{k}$ are parametrized through arc length, and since it holds that 
\[
|u_{k}|_\infty \leq |[u_{k}]| + \D(u_{k}),
\]
we easily get that there exists $C>0$ such that $|u_{k}|_\infty + |\dot u_{k}|_\infty \leq C$. 
Moreover, given that every $u_{k}$ is a $(\varepsilon_k H_1 - \lambda_k)$-loop and taking \eqref{lambdalimtec} into account, by the fundamental theorem of calculus we get that, for every $t_1, t_2 \in [0,1]$, 
\[
|\dot u_{k} (t_2) - \dot u_{k}(t_1)| \leq \int_{t_1}^{t_2}|\ddot u_{k}|\de t =\D(u_{k}) \int_{t_1}^{t_2}|\varepsilon_k H_1(u_{k}) - \lambda_{k})||\dot u_{k}| \leq C |t_2 - t_1|,
\]
for some $C>0$ which depends on $\tau_0, \tau_1$ but does not depend on $\varepsilon_k$ nor $\tau_k$. This, together with previous remarks, implies that the sequence $(u_{k})$ is uniformly bounded in $C^{1,1}$. Therefore, by standard arguments (see e.g. \cite[Lemma 6.36]{giltru}) we infer that $u_{k} \to u_\infty$ in $C^{1, \gamma}$ for any $\gamma \in (0,1)$.

We are now able to get the needed contradiction. As we have seen, $u_\infty$ belongs to the family \eqref{circfam}. In particular, defined as $p_\infty := [u_\infty]$, we have that $u_\infty(t) = p_\infty + \sqrt{\frac{\hat \tau}{\pi}}e^{-2\pi it}$.
We can explicitly compute
\[
i \dot u_\infty \cdot  (u_\infty - p_\infty) = 2 \hat \tau.
\]
Denoting $p_{k}:= [u_{k}]$, we have that
\[
|i \dot u_\infty \cdot (u_\infty-p_\infty) - i \dot u_{k}\cdot (u_{k}- p_{k})| \leq D(u_{k})|\dot u_{k}- \dot u_\infty|_\infty + 2|\dot u_\infty|_\infty|u_{k} - u_\infty|_\infty,
\]
therefore, since $u_{k} \to u_\infty$ in $C^{1,\gamma}$, we infer that $i \dot u_{k}\cdot (u_{k}- p_{k}) \to i \dot u_\infty \cdot (u_\infty- p_\infty)$ uniformly in $[0,1]$ and in particular there exists $k_0$ such that for every $k\geq k_0$ it holds that
\begin{equation}\label{ortogonality}
i \dot u_k \cdot (u_k - p_k) > 0 \quad \forall \ t \in [0,1].
\end{equation}

By our initial assumption, for every $k \geq k_0$ it holds that $u_k$ is not injective in $(0,1)$. Then there exist two points $0< s_k < t_k \leq 1$ such that $u_k(s_k) = u_k(t_k)$. Up to translation and relabeling the points if necessary, we can always suppose that $0 < s_k = \hat s < t_k \leq 1$. Moreover, taking if necessary a subsequence, there exists $\hat t \in \left[\hat s,1\right]$ such that $t_k \to \hat t$ as $k \to + \infty$.

By the fundamental theorem of calculus we get that
\begin{equation}\label{tecnicthing}
0 = \int_{\hat s}^{t_k} \dot u_k \de t = \int_0^1 \dot u_\infty \chi_{\left[\hat s, t_k\right]}\de t + \int_{\hat s}^{t_k} \dot u_k - \dot u_\infty \de t.
\end{equation}
By the dominated convergence theorem, and since $u_k \to u_\infty$ in $C^{1, \gamma}$, passing to the limit as $k \to + \infty$ in \eqref{tecnicthing} we obtain $0 = \int_{\hat s}^{\hat t}\dot u_\infty \de t$, which readily implies that $\hat t = \hat s$. 
Then, using once again the fundamental theorem of calculus and the convergence $u_k \to u_\infty$ in $C^{1,\gamma}$, we get that 
\begin{equation}\label{shrinking}
\sup_{t \in \left[\hat s, t_k\right]}\left|u_k(t) - u_\infty\left(\hat s\right) \right|\to 0 \quad  \text{ as } k \to + \infty,
\end{equation}
that is, the curves whose supports are $\Gamma_k := u_k\left( \left[\hat s, t_k\right]\right)$ shrink uniformly to the point $u_\infty(\hat s)$. 

Now we see that, taking $k$ large enough, we can find two distinct straight line $l_k$ and $r_k$ passing through $p_k$ and such that at least one of them is tangent to the curve $\Gamma_k$ in a regular point, thus contradicting \eqref{ortogonality}.
Indeed, since $p_k \to p_\infty$, by \eqref{shrinking} we infer that there exists $k_1 \geq k_0$ big enough such that, setting $\delta_k:=\frac{|p_k- u_\infty(\hat s)|}{2}$, for every $k \geq k_1$ the curve $\Gamma_k$ is completely contained in the ball $B_{\delta_k}\left(u_\infty\left(\hat s\right)\right)$ and $p_k \not \in B_{\delta_k}\left(u_\infty\left(\hat s\right)\right)$. As a consequence, there exists at least a straight line passing through $p_k$ such that it does not intersect $\Gamma_k$. Performing a rotation, we find a line $l_k$ and a point $p'_k$ such that $p'_k= u_k(t_k')$ for some $t'_k \in [\hat s, t_k]$, being such that $l_k$ passes through both $p_k$ and $p'_k$, and $\Gamma_k$ completely lies on the right side of $l_k$. With the same argument, we find another line $r_k$ and a point $p''_k$ such that $p''_k = u_k(t''_k)$ for some $t''_k \in [\hat s, t_k]$, being such that $r_k$ passes through $p_k$ and $p''_k$ and $\Gamma_k$ completely lies on the left side of $r_k$.

We claim that $l_k \neq r_k$. Otherwise, since $\Gamma_k$ is a continuous curve, by definition of $l_k$ and $r_k$ it will consists either of a segment or of a point. Both possibilities lead to a contradiction, since $\hat s < t_k$, $|\dot u_k| \neq 0 $ for every $t \in [\hat s, t_k]$ and since the restriction ${u_k}_{|[\hat s, t_k]}$ belongs to $C^2\left(\left(\hat s, t_k\right); \R^2\right)$.

As a consequence, at least one point between $p'_k$ and $p''_k$ is different form $u_k(\hat s)$.
Suppose then that $p'_k \neq u_k(\hat s)$. Using again the regularity of ${u_k}_{|[\hat s, t_k]}$, we infer that the line $l_k$ has to be tangent to $\Gamma_k$ in $p'_k$, which implies that $i \dot u'_k(t'_k)\cdot (u_k(t'_k) - p_k)=0$, thus contradicting \eqref{ortogonality}.
The proof of the Lemma is concluded.
\end{proof}

\begin{proof}[Proof of Theorem \ref{mainteo5}]
Let $H_1$, $H_2$ satisfy \eqref{k1hp} and \eqref{k2hp}, respectively.
Arguing as in Proposition \ref{minexper} and Proposition \ref{minexac} we can construct $Q_{H_1-[H_1]}$ and $Q_{H_2 - H_2^\infty}$ such that they satisfy, respectively, \eqref{Q1prop} and \eqref{Q2prop}. 
Let be $\varepsilon \in (0,1)$ and define
\[
Q_{\varepsilon H_1 + H_2}(p) := Q_{\varepsilon(H_1 - [H_1])}(p) + Q_{H_2-H_2^\infty}(p) + \frac{\varepsilon[H_1]+ H_2^\infty}{2}p
\]
Since \eqref{commonbound} holds, the Lagrangians
\[
\begin{aligned}
&F_\infty(p,q) = |q| + Q_{\varepsilon(H_1 - [H_1])}(p)\\
&F(p,q) = F_\infty(p,q) + Q_{H_2-H_2^\infty}(p)
\end{aligned}
\]
satisfy \eqref{Fhyp} and \eqref{Fasper2}.
In order for \eqref{asexcond} to be verified, it has to hold
\[
S_{\varepsilon(H_1 - [H_1]) + H_2-H_2^\infty}(\tau) < S_{\varepsilon(H_1 - [H_1])}(\tau),
\]
or, equivalently,
\[
S_{\varepsilon H_1 + H_2 }(\tau) < S_{\varepsilon H_1}(\tau) + H^\infty_2 \tau.
\]

Fix $0< \tau_0< \tau_1$, and let $\varepsilon_0$ be the value given by Proposition \ref{simplemin}. For every $\varepsilon \in (0, \varepsilon_0)$ and for every $|\tau| \in (\tau_0, \tau_1)$ there exists a minimizer $v_{\varepsilon, \tau}$ of $S_{\varepsilon H_1}(\tau)$ such that its support is a simple curve. Recall that for every $n \in \Z^2$, $v_{\varepsilon, \tau} + n$ is still a minimizer for $S_{\varepsilon H_1}(\tau)$. Since in addition $\text{diam}(v_{\varepsilon, \tau})<+\infty$ we can always find $n_0 \in \Z^2$ being such that the support of $v_{\varepsilon, \tau}+ n_0$ is completely contained in $\{\lambda p \ |\ \lambda>0, p \in \omega_{\text{sign}(\tau)}\}$ where the set $\omega_{\text{sign}(\tau)}$ is given by \eqref{k2hp}.
Thus, by \eqref{divergencejord} we get that
\[
\begin{aligned}
\A_{H_2}(v_{\varepsilon, \tau}+n_0) &= \int_0^1Q_{H_2-H_2^\infty}(v_{\varepsilon_\tau}+n_0)\cdot i \dot v_{\varepsilon, \tau}\de t + H_2^\infty \tau \\
&= \text{sign}(\tau)\int_{\mathcal{B}_{v_{\varepsilon, \tau}+n_0}}(H_2(p) - H_2^\infty) \de p + H_2^\infty \tau  < H_2^\infty \tau ,
\end{aligned}
\]
where ${\mathcal{B}_{v_{\varepsilon, \tau}+n_0}}$ is the bounded component of $\R^2 \setminus (v_{\varepsilon, \tau}+n_0)([0,1])$.

Therefore, for every $|\tau| \in (\tau_0, \tau_1)$ it holds
\[
S_{\varepsilon H_1 + H_2}(\tau) \leq \D(v_{\varepsilon, \tau}) + \A_{\varepsilon H_1}(v_{\varepsilon, \tau}) + \A_{H_2}(v_{\varepsilon, \tau}+n_0) < S_{\varepsilon H_1}(\tau) + H^\infty_2 \tau.
\]
which implies that we can apply Theorem \ref{mainteo2}, obtaining existence of minimizers of $S_{\varepsilon H_1 + H_2}(\tau)$. To conclude, arguing as in the proof of Theorem \ref{mainteo3} and Theorem \ref{mainteo4}, it suffices to apply Lemma \ref{E-L-Eq} and Lemma \ref{solreg}.
\end{proof}

\end{subsection}

\begin{subsection}{The isoperimetric function}\label{isoperfunctsubsect}

This Subsection is devoted to the study of the properties of the isoperimetric function associated to the prescribed curvature problem, which is defined in \eqref{isoperifunction}. As a byproduct, we derive Theorem \ref{mainteo6}. 

Through all this Subsection, we denote by $H$ a prescribed curvature in the form $H = H_1 + H_2$, where $H_1$, $H_2$ satisfy \eqref{k1hp} and \eqref{k2hp}, respectively, and such that \eqref{commonbound} holds. 
In addition, in order to ease the notation we always assume that, if $H_2 \not \equiv 0$, then
\begin{equation}\label{excondisoper}
S_H(\tau) < S_{H_1}(\tau), \quad \forall \tau \in \R \setminus \{0\}.
\end{equation}
Indeed, some of the results holds true also when this condition is not satisfied.

Since \eqref{excondisoper} holds, arguing as in the proof of Theorem \eqref{mainteo5} we obtain that minimizers $u_\tau \in \M_\tau \cap C^2 $ of $S_H(\tau)$ exist for every $\tau \in \R$. In addition, since such minimizers satisfy $|\dot u_\tau| = \text{const.}$, both \eqref{dirichcondconst} and
\[
\ddot u_\tau = \D(u) (H(u_{\tau}-\lambda))i \dot u_\tau, \quad \text{ for some }\lambda \in \R,
\]
hold.

We recall that
\[
S_H(\tau) = S_{H-[H_1] -H^\infty_2}(\tau) + ([H_1] + H_2^\infty)\tau.
\]
Since the first part of the Subsection is dedicated to the study of the regularity of the function $S_H(\tau)$ and its asymptotic behavior as $\tau \to 0$, without loss of generality we can assume that $[H_1] = H_2^\infty = 0$.  

We begin with the following asymptotic result.
\begin{lemma}\label{bound}
For every $\tau \in \R$ it results 
\begin{equation}\label{isobounds}
(1-|Q_{H}|_\infty)S\sqrt{|\tau|} \leq S_H(\tau) \leq (1 + |Q_{H}|_\infty) S \sqrt{|\tau|}.
\end{equation}
As a consequence, 
\begin{equation}\label{isocont}
S_H(\tau) \to 0 \quad \text{ as }\quad \tau \to 0^+.
\end{equation}
\end{lemma}
\begin{proof}

Arguing as in the beginning of the proof of Theorem \ref{mainteo1} we easily get that $S_H(0) = 0$, which implies that \eqref{isobounds} is trivially verified when $\tau = 0$. Let then be $\tau \neq 0$. 

We recall that the vector field $Q_H$ is such that $Q_H = Q_{H_1} + Q_{H_2}$, where $Q_{H_1}$, $Q_{H_2}$ satisfy \eqref{Q1prop} and \eqref{Q2prop}, respectively. In particular, thanks to \eqref{commonbound} it holds that $|Q_H|_\infty<1$. Hence, a simple computation shows that $\E_H(u) = \D(u) + \A_H(u) \geq \D(u) - |Q_H|_\infty \D(u)$. The bound from below is obtained as a consequence of the isoperimetric inequality \eqref{classicineq}.

On the other hand
\[
S_H(\tau) \leq \D(\omega) + \A_H(\omega) \leq (1+|Q_H|_\infty)\D(\omega) = (1 + |Q_H|_\infty) S \sqrt{|\tau|},
\]
where $\omega$ is any function of the family \eqref{circfam} such that $\A(\omega) = \tau$, which realizes \eqref{classicineq} as an equality.

\end{proof}

In order to study the regularity of $S_H(\tau)$, we introduce, for $\tau \geq 0$, the family of functionals $ \A_{H;\tau}: W^{1,1} \to \R$ defined as
\[
\A_{H; \tau}(u) : = \int_0^1Q_H(\tau u) \cdot i \dot u \de t, 
\]
and the function $\tilde S_H(\tau): [0, + \infty) \to \R$ defined by
\[
\tilde S_H(\tau) := \inf_{u \in \M_1\cap H^1} \E_{H; \tau}(u) \quad\text{where}\quad \E_{H; \tau}(u) := \D(u) + \A_{H;\tau}(u).
\]
We notice that when $\tau=0$ it holds $\A_{H;0}(u) = Q_H(0) \cdot i\int_0^1  \dot u \de t=0$ for every $u \in W^{1,1}$. Therefore, by the isoperimetric inequality we infer that $\tilde S_H(0) = S$.

The functions $S_H(\tau)$ and $\tilde S_H(\tau)$ are related, as stated in the following Lemma. 
\begin{lemma}\label{LemmaTilda}
Let $\tau \in \R\setminus \{0\}$. Then $S_H(\tau) =\sqrt{|\tau|}\tilde S_{\text{sign}(\tau)H}(\sqrt{|\tau|})$.
\end{lemma}
\begin{proof}
Since it holds that 
\[
\text{div }-Q_H = - \text{div }Q_H = -H = \text{div }Q_{-H},
\]
a simple computation (argue as in the beginning of the proof of Theorem \ref{mainteo1}) shows that 
\begin{equation}\label{stildasign}
S_H(\tau) = S_{-H}(-\tau) \quad \forall \tau. 
\end{equation}
Let then be $\tau >0$ and let $u \in \M_\tau$. Define $v := \lambda u$ with $\lambda>0$. It holds that $\A(v) = \lambda^2 \A(u)$, $\D(v) = \lambda \D(u)$ and $\A_K(v) = \lambda \A_{K;\lambda}(u)$. Hence we get that $\E_H(v) = \lambda \E_{H; \lambda}(u)$.
Taking $u \in \M_1$ and $\lambda = \sqrt{\tau}$, by means of a standard argument we obtain that 
\begin{equation}\label{stidlapos}
S_H(\tau) = \sqrt{\tau}\tilde S_H(\sqrt{\tau}).
\end{equation}
When $\tau<0$, we get the desired result combining \eqref{stildasign} and \eqref{stidlapos}.
\end{proof}

As a straightforward consequence of Lemma \ref{LemmaTilda}, since for every $\tau >0$ there exists $u \in \M_{\tau^2}\cap H^1$ which minimizes $S_H(\tau^2)$, then the function $v: = \frac{u}{\tau} \in \M_1 \cap H^1$ realizes the minimum $\tilde S_H(\tau)$. Moreover, since such $u$ is a $(H-\lambda)$-loop, we infer that $v$ is a $(\tilde H - \tau \lambda)$-loop, with $\tilde H(p) = \tau H(\tau p)$.

When $\tau \to 0^+$, we have the following asymptotic result for the function $\tilde S_H(\tau)$. Notice that, thanks to Lemma \ref{LemmaTilda}, this gives a refinement of \eqref{isocont}.
\begin{lemma}\label{stildelimit}
As $\tau \to 0^+$ it holds that $\tilde S_H(\tau) \to S$. 
\end{lemma}
\begin{proof}
We begin by proving the isoperimetric inequality
\begin{equation}\label{easy}
S^2 |\A_{H;\tau}(u)| \leq \tau|H|_{\infty} \D(u)^2,  \quad \forall u \in H^1.
\end{equation}
Let $u\in H^1 \setminus \M_0$, and recall that $u = (u_1, u_2)$ with $u_1, u_2 \in H^1(\R /\Z; \R)$. Since $\A_K(0) = 0$, by \eqref{Qfuncreg3} we get that 
\[
\begin{aligned}
|\A_H(u)| &= |\A_H(u) - \A_H(0)| = \left|\int_0^1 \frac{\de}{\de s}\A_H(su) \de s\right| = \left| \int_0^1 s \left(\int_0^1 H(su)u \cdot i \dot u \de t \right) \de s\right|\\
&\leq \frac{|H|_\infty}{2}\int_0^1|u \cdot i \dot u| \de t \leq \frac{|H|_\infty}{2}\int_0^1|u_1\dot u_2| + |u_2\dot u_1| \de t.
\end{aligned}
\]
Since, as a consequence of the Poincaré-Wirtinger inequality, we have that (see e.g. \cite[Corollary 6.3]{dacoro})
\[
S^2\int_0^1 u \dot v \de t \leq \int_0^1 |\dot u|^2 + |\dot v|^2 \de t, \quad \forall u,v \in H^1(\R / \Z; \R), 
\]
we readily get that
\[
S^2 |\A_H(u)| \leq |H|_{\infty} \D(u)^2.
\]
Therefore, to conclude the proof of \eqref{easy} it suffices to notice that $\A_{H;\tau}(u) = \frac{1}{\tau}\A_H(\tau u)$.

Let now $(\tau_k) \subset \R_+$ be such that $\tau_k \to 0^+$ and let $u_{k} \in \M_1 \cap H^1$ be a sequence of associated minimizers of $\tilde S_H(\tau_k)$. 
Thanks to \eqref{easy} we get that
\begin{equation}\label{easy2}
\E_{H;\tau_k}(u_k) = \D(u_k) + \A_{H;\tau_k}(u_k) \geq \D(u_k) - |\A_{H;\tau_k}(u_k)| \geq \D(u_k) - \frac{\tau_k|H|_\infty}{S^2}\D(u_k)^2. 
\end{equation}
Since $Q_H$ is such that $|Q_H|_\infty <1$, we easily get from Lemma \ref{bound} and Lemma \ref{LemmaTilda} that 
\[
(1 - |Q_H|_\infty)\D(u_k) \leq \E_{H;\tau_k}(u_k) = \tilde S_H(\tau_k) \leq (1 + |Q_H|_\infty)S,
\]
which implies that
\begin{equation}\label{semibound}
\D(u_k) \leq \frac{1+|Q_H|_\infty}{1 - |Q_H|_\infty}S.
\end{equation}
Thanks to \eqref{easy2}, \eqref{semibound} and \eqref{classicineq} we infer that 
\[
\tilde S_H (\tau_k) \geq S- C\tau_k, \quad \forall \ k, 
\]
where $C>0$ is a constant which does not depends on $\tau_k$. 
As a straightforward consequence we obtain
\[
\liminf_{k\to + \infty} \tilde S_H(\tau_k) \geq S.
\]

On the other hand, since by assumption $H_1$ is continuous and $[H_1]=0$, there exists a point $p_0 \in [0,1]\times[0,1]$ and $\delta_1 >0$ small enough such that $H_1 < 0$ on $B_{\delta}(p_0)$ for every $\delta \in (0, \delta_1)$. Moreover, thanks to \eqref{k2hp} we can find $n_0 \in \Z^2$ such that $H_2<0$ on $B_\delta(p_0+n_0)$. Then, taking into account also the periodicity of $H_1$, we infer that $H<0$ in $B_\delta(p_0 + n_0)$ for every $\delta \in (0, \delta_1)$. 
Let then $\omega(t) = \frac{1}{\sqrt{\pi}}e^{-i2\pi t} \in H^1$ be the parametrization of the ball of area $\A(\omega) = 1$ centered at the origin. For every $p \in \R^2$, \eqref{divergencejord} implies that
\[
\tilde S_H (\tau) \leq \D(\omega + p) + \A_{H;\tau}(\omega + p) = S + \frac{1}{\tau} \int_{B_{\tau/\sqrt{\pi}}(\tau p)}H(q) \de q.  
\] 
Then, for every $\tau \in \left(0, \sqrt{\pi}\delta_1\right)$ we can choose $p \in \R^2$ such that $\tau p = p_0+ n_0$. As a consequence we get that 
\[
\int_{B_{\frac{\tau}{\sqrt{\pi}}}(\tau p)}H(q) \de q \leq 0,
\]
which readily implies $\tilde S_H(\tau) \leq S$ for every $\tau \in (0, \sqrt{\pi}\delta_1)$, and in particular 
\[
\limsup_{k \to +\infty} \tilde S_H(\tau_k) \leq S.
\]
The Lemma is proved.
\end{proof}

In the following we study the regularity of $\tilde S_K(\tau)$. 

\begin{lemma}\label{isofunclip}
The mapping $\tau \to \tilde S_H(\tau)$ is locally Lipschitz-continuous in $\R \setminus \{0\}$.
\end{lemma}
\begin{proof}
Let $0 < \tau_-<\tau_+$ be fixed. For every $\tau_-\leq \tau_1 \leq \tau_2 \leq \tau_+$ it holds that 
\begin{equation}\label{lastdays}
\E_{H; \tau_2}(u)-\E_{H; \tau_1}(u) = \A_{H;\tau_2}(u)-\A_{H;\tau_1}(u) = \int^{\tau_2}_{\tau_1} \frac{\de}{\de \rho} \A_{H;\rho} (u) \de \rho.
\end{equation}
When $\rho \neq 0$, we have that
\[
\A_{H;\rho} (u) = \int^1_0 Q_H(\rho u)\cdot i\dot u \de t = \frac{1}{\rho}\A_H(\rho u).
\]
Therefore, using Lemma \ref{Qfuncreg}, we get that
\begin{equation}\label{lipschit1}
\frac{\de}{\de \rho} \A_{H;\rho} (u) = \frac{\rho \frac{\de}{\de \rho}\left(\A_H(\rho u)\right)-\A_H(\rho u)}{\rho^2} = \int_0^1 H(\rho u)u \cdot i \dot u \de t - \frac{1}{\rho}\int_0^1Q_H(\rho u) \cdot i\dot u \de t.
\end{equation}

Let now $\tau \in [\tau_-, \tau_+]$, and let $u_\tau$ be an associated minimizer such that $\tilde S_H(\tau) = \E_{H; \tau}(u_\tau)$. Also in this case \eqref{semibound} holds. Moreover, we claim that there exists a constant $C$, which can depend on $\tau_-$, $\tau_+$, but does not depend on $\tau$, such that for every $\tau \in [\tau_-, \tau_+]$ there exists a minimizer $u_\tau$ of $\tilde S_H(\tau)$ which satisfy
\begin{equation}\label{lipschit10}
|[u_\tau]| \leq C,
\end{equation}
for a constant $C$ which does not depends neither on $u_\tau$ nor $\tau$.
Assuming that the claim holds true, let $\{u_\tau\}$ be the family of minimizers which satisfy \eqref{lipschit10}. Since \eqref{k1hp} and \eqref{k2hp} holds, we infer from \eqref{semibound}, \eqref{lipschit10} and \eqref{PWineq} that
\begin{equation}\label{lipschit2}
\begin{aligned}
\left|\int_0^1 H(\rho u_\tau)u_\tau \cdot \dot u_\tau \de t \right| &\leq |H|_\infty |u_\tau|_2 \D(u_\tau)
\leq |H|_\infty \sqrt{|[u_\tau]|^2 + (S^2\pi)^{-1}\D^2(u_\tau)}\D(u_{\tau}) \leq C_1,
\end{aligned}
\end{equation}
where $C_1$ depends on $H$, $\tau_-$, and $\tau_+$, but not on $\tau$ nor $\rho$. 
Moreover, since we have that $|Q_H|_\infty <1$, from \eqref{semibound} we infer that
\begin{equation}\label{lipschit3}
\left|\int_0^1Q_H(\rho u_\tau) \cdot i\dot u_\tau \de t\right| \leq |Q_H|_\infty \D(u_\tau)\leq \frac{1+|Q_H|_\infty}{1-|Q_H|_\infty}S = C_2,
\end{equation}
where $C_2$ it is independent form $\tau$ and $\rho$. 
Putting together \eqref{lipschit1}, \eqref{lipschit2} and \eqref{lipschit3} we get that
\[
\left| \frac{\de}{\de \rho} \A_{H;\rho} (u_\tau) \right| \leq C_1 + \frac{C_2}{\rho}.
\]
This, together with \eqref{lastdays}, implies that
\begin{equation}\label{lipschit37}
\begin{aligned}
|\E_{H;\tau_2}(u_\tau)-\E_{H;\tau_1}(u_\tau)| &\leq \int^{\tau_2}_{\tau_1} \left|\frac{\de}{\de \rho} \A_{H;\rho} (u_\tau)\right| \de \rho \\
&\leq  C_1|\tau_2-\tau_1| + C_2 \left|\log \tau_2 - \log \tau_1 \right| \leq C_3 |\tau_2-\tau_1|,
\end{aligned}
\end{equation}
where the constant $C_3$ depends only on $H$, $\tau_-$ and $\tau_+$, but not on $\tau$, $\tau_1$ nor $\tau_2$,  since the logarithm is locally Lipschitz-continuous. 

Taking $\tau = \tau_1$ in \eqref{lipschit37} we get that
\[
\tilde S_H(\tau_2) \leq \E_{H; \tau_2}(u_{\tau_1}) \leq \E_{H; \tau_1}(u_{\tau_1}) + C_3 |\tau_2-\tau_1| = \tilde S_K(\tau_1) + C_3 |\tau_2-\tau_1|,
\]
while taking $\tau = \tau_2$ we infer $\tilde S_H(\tau_1) \leq \tilde S_H(\tau_2) + C_3 |\tau_2-\tau_1|$. These imply that the function $\tilde S_H(\tau)$ is locally Lipschitz, as desired.  

To conclude it remains to prove that \eqref{lipschit10} holds true. Let then $\tau \in [\tau_-, \tau_+]$ be fixed. 
In the purely periodic case, i.e. when $H=H_1$ satisfies \eqref{k1hp}, we have that for every minimizers $u_\tau$ of $\tilde S_H(\tau)$ there exists $n_\tau \in \Z^2$ such that $|\tau[u_\tau] - n_\tau|\leq \frac{\sqrt2}{2}$.
Then, if we define $v_\tau := u_\tau - \frac{n_\tau}{\tau}$ it holds that $\A(v_\tau) = \A(u_\tau) = 1$, that $\D(v_\tau) = \D(u_\tau)$ and that
\[
\A_{H;\tau}(v_\tau) = \int_0^1Q_{H}\left(\tau \left(u_\tau - \frac{n_\tau}{\tau}\right)\right)\cdot i \dot u_\tau = \A_{H; \tau}(u_\tau),
\]
therefore $v_\tau$ is still a minimizer for $\tilde S_H(\tau)$. 
Moreover it holds
\[
|[v_\tau]| = \left|[u_\tau] - \frac{n_\tau}{\tau}\right| \leq \frac{\sqrt2}{2\tau}\leq \frac{\sqrt2}{2\tau_-},
\]
hence the claim is proved. 

In order to treat the general case, we need a preliminar result; indeed, we first prove that for every $\tau_0 \in[\tau_-, \tau_+]$ it holds
\begin{equation}\label{lipschit19}
\limsup_{\tau \to \tau_0} \tilde S_H(\tau) \leq \tilde S_H(\tau_0).
\end{equation}
Let $u_{\tau_0}$ be a minimizer for $\tilde S_H(\tau_0)$, and let $\delta \neq 0$. We have that
\begin{equation}\label{lipschit20}
\begin{aligned}
\tilde S_H(\tau_0 + \delta) &\leq \D(u_{\tau_0}) + \A_{H; \tau_0 + \delta}(u_{\tau_0})\\
&=\tilde S_H(\tau_0) + \int_0^1(Q_H(\tau_0 u_{\tau_0} + \delta u_{\tau_0}) - Q_H(\tau_0 u_{\tau_0})) \cdot i \dot u_{\tau_0} \de t\\
&\leq \tilde S_H(\tau_0) + |Q_H(\tau_0 u_{\tau_0} + \delta u_{\tau_0}) - Q_H(\tau_0 u_{\tau_0})|_\infty \int_0^1|\dot u_{\tau_0}| \de t. 
\end{aligned}
\end{equation}

Since $Q_H$ is uniformly continuous in $\R^2$ and $\tau_0 u_{\tau_0} + \delta u_{\tau_0} \to \tau_0 u_{\tau_0}$ uniformly in $[0,1]$ as $\delta \to 0$, passing to the limit in \eqref{lipschit20} we readily obtain \eqref{lipschit19}.

Now, assume by contradiction that there exists $\tau_0 \in [\tau_-, \tau_+]$, a sequence $(\tau_k) \in [\tau_-, \tau_+]$ such that $\tau_k \to \tau_0$ and a sequence $(u_{k}) \subset \M_1\cap H^1$ of associated minimizers of $\tilde S_H(\tau_k)$ such that $|[u_{\tau_k}]| \to + \infty$ as $k \to + \infty$. 
Since $u_k \in H^1$, thanks to the fundamental theorem of calculus we get that
\[
[u_k] = u_k(t_0) + \int_0^1 \int_{t_0}^s \dot u_k \de t \de s,\quad  \forall t_0 \in [0,1].
\]
Then, using \eqref{semibound} we infer that
\[
|[u_k]| \leq |u_k(t_0)| + \D(u_k) \leq |u_k(t_0)| + \frac{1+ |Q_H|_\infty}{1-|Q_H|_\infty}S,
\]
which implies that $\inf_{t \in [0,1]}|u_k| \to +\infty$ as $k \to + \infty$.  
As a consequence, since $Q_{H_2}$ is uniformly continuous and $|Q_{H_2}(p)| \to 0$ as $|p| \to+\infty$, we obtain that
\begin{equation}\label{lipschit40}
\left|\int_0^1 Q_{H_2}(\tau_k u_k) \cdot i u_k \de t \right| \leq \frac{1 + |Q_H|_\infty}{1-|Q_H|_\infty}S |Q_{H_2}(\tau_k u_k)|_\infty \to 0,
\end{equation}
as $k \to +\infty$.
Since by the previous case we have that $\tilde S_{H_1}(\tau)$ is Lipschitz continuous, we infer
\begin{equation}\label{lipschit21}
\liminf_{k \to + \infty}\tilde S_H(\tau_k) \geq \lim_{k\to+\infty}S_{H_1}(\tau_k) + \lim_{k \to + \infty}\int_0^1Q_{H_2}(\tau_k u_k) \cdot i \dot u_k \de t = S_{H_1}(\tau_0).
\end{equation}
Putting together \eqref{lipschit19} and \eqref{lipschit21} we get that $\tilde S_{H_1}(\tau_0) \leq \tilde S_H(\tau_0)$. 
On the other hand, by \eqref{excondisoper} together with Lemma \ref{LemmaTilda} we have that $S_{H}(\tau_0)<S_{H_1}(\tau_0)$ holds, thus a contradiction.
This conclude the proof of \eqref{lipschit10} and consequently the proof of the Lemma. 
\end{proof}

Since Lemma \ref{bound} and Lemma \ref{LemmaTilda} imply that
\[
(1- |Q_H|_\infty)S \leq \tilde S_H(\tau) \leq (1 + |Q_H|_\infty)S
\]
for every $\tau >0$, by Lemma \ref{LemmaTilda}, Lemma \ref{isofunclip} and a simple computation we immediately get that $\tau \to S_K(\tau)$ is locally Lipschitz continuous in $\R_+ \setminus\{0\}$. Moreover, since the curvature $-H$ still satisfy \eqref{k1hp} and \eqref{k2hp}, also $\tilde S_{-H}(\tau)$ turns out to be locally Lipschitz continuous. Then, recalling that \eqref{stildasign} holds, we infer that $S_H(\tau)$ is locally Lipschitz continuous in all of $\R \setminus \{0\}$. 
Taking also \eqref{isocont} into account, we can conclude that $S_K(\tau) \in C^0(\R; \R_+) \cap C^{0,1}_{loc}(\R \setminus \{0\}; \R_+)$.

We remark that, in general, \eqref{excondisoper} is satisfied only in the asymptotically constant case. On the other hand with some minor modification the proof of Lemma \ref{isofunclip} can be adapted to the mixed case, thus obtaining that for every $0 < \tau_0<\tau_1$ and $\varepsilon < \varepsilon_0= \varepsilon_0(\tau_0, \tau_1)$, where $\varepsilon_0$ is given by Theorem \ref{mainteo5}, the function $S_{\varepsilon H_1 + H_2}(\tau)$ is locally Lipschitz continuous in $(- \tau_1, -\tau_0) \cup (\tau_0, \tau_1)$.

We turn our attention to the properties of the set of the Lagrange multipliers, which is defined as
\[
\Lambda(\tau) = \{\lambda \in \R \ |\ \text{exists } u \in \M_\tau \text{ which is a ($H-\lambda$)-loop and  }\E_H(u) = S_H(\tau) \}.
\]
\begin{lemma}\label{lamboundlem}
For every $\tau \in \R \setminus \{0\}$, $\Lambda(\tau)$ is a bounded and closed set. 
In particular, it holds that if $\lambda \in \Lambda(\tau)$, then
\[
\frac{S}{2\sqrt{|\tau|}} - C_2 \leq \text{sign}(\tau)\lambda \leq \frac{C_1}{\sqrt{|\tau|}} + C_2, 
\]
where 
\[
C_1 := \frac{1 + |Q_H|_\infty}{1-|Q_H|_\infty}\frac{S}{2} \quad  \text{and} \quad C_2 := \left(\frac{1 + |Q_H|_\infty}{1-|Q_H|_\infty}\right)^2|H|_\infty.
\]
\end{lemma}
\begin{proof}
Let $\tau \in \R\setminus \{0\}$, and let $u_\tau \in \M_\tau$ be a minimizer for $S_H(\tau)$. By Lemma \ref{E-L-Eq}, there exists $\lambda$ such that $\E_H'(u_\tau)[\varphi]- \lambda \A'(u_\tau)[\varphi] = 0$, for every $\varphi \in H^1$. Let $p \in \R^2$ be such that $[u_\tau+p]=0$ and take as a test function $\varphi = u_\tau +p$. We get that
\[
\D(u_\tau) + \int_0^1 H(u_\tau)(u_\tau+p) \cdot i \dot u_\tau \de t - \lambda 2\tau  = 0,
\]
and thus
\begin{equation}\label{lambound3}
\D(u_\tau) - \left|\int_0^1 H(u_\tau)(u_\tau+p) \cdot i \dot u_\tau \de t \right| \leq 2\lambda\tau \leq \D(u_\tau) + \left|\int_0^1 H(u_\tau)(u_\tau+p) \cdot i \dot u_\tau \de t \right|.
\end{equation}
Arguing as in the proof of \eqref{semibound} we get that
\begin{equation}\label{lambound1}
\D(u_\tau) \leq \frac{1 + |Q_H|_\infty}{1-|Q_H|_\infty}S\sqrt{|\tau|} = 2C_1\sqrt{|\tau|},
\end{equation}
hence we infer that
\begin{equation}\label{lambound2}
\begin{aligned}
\left|\int_0^1 H(u_\tau)(u_\tau+p) \cdot i \dot u_\tau \de t \right| &\leq |H|_\infty |u_\tau+p|_2\D(u_\tau)  \\
&\leq {(S^2\pi)^{-1/2}}|H|_\infty(\D(u_\tau))^2 \leq (S^2\pi)^{-1/2}|H|_\infty 4C_1^2 |\tau| = 2C_2|\tau|, 
\end{aligned}
\end{equation}
where $C_2$ depends only on $H$. 
Putting together \eqref{classicineq}, \eqref{lambound3}, \eqref{lambound1}, and \eqref{lambound2}, we obtain
\[
S\sqrt{|\tau|} - 2C_2 |\tau| \leq 2\lambda\tau \leq 2C_1\sqrt{|\tau|} + 2C_2|\tau|,
\]
hence we infer that $\Lambda(\tau)$ is a bounded set. 

It remains to prove that $\Lambda(\tau)$ is closed. Let $(\lambda_k) \subset \Lambda(\tau)$ be a sequence such that $\lambda_k \to \lambda$ as $k \to + \infty$. If there exists $k_0$ such that $\lambda_k = \lambda$ for every $k \geq k_0$, there is nothing to prove. Suppose then that $\lambda_k \neq \lambda$ definitively. By definition, there exists a sequence $(u_k) \subset \M_\tau\cap H^1$ of $S_H(\tau)$ minimizers such that $u_k$ is a $(H-\lambda_k)$-loop and $S_H(\tau) = \E_H(u_k)$.
 
Notice that $(u_k)$ is a minimizing sequence for $S_H(\tau)$. Also in this case we can argue as in the proof of \eqref{semibound} to get that $\D(u_k) \leq 2C_1 \sqrt{|\tau|}$. Moreover, recalling that \eqref{excondisoper} holds and arguing as in the proofs of Theorem \ref{mainteo1} and Theorem \ref{mainteo2}, we get that, performing if necessary a $\Z^2$-translation of the sequence in the purely periodic case, also the seminorms $|[u_k]|$ are uniformly bounded. It is important to notice that, thanks to \eqref{k1hp}, every $\Z^2$ -translation of a $(H_1-\lambda_k)$-loop still is a $(H_1-\lambda_k)$-loop for the same $\lambda_k$.

Since $(u_k) \subset \M_\tau \cap H^1$ is a bounded sequence in $H^1$, there exists $u \in H^1$ such that $u_k \rightharpoonup u$ in $H^1$.
By the weak lower semicontinuity of the norms and by Lemma \ref{linearcont} we readily get that $u_\tau \in \M_\tau\cap H^1$ and that $u$ is also a minimizer for $S_H(\tau)$. This implies that $\D(u_k) \to \D(u)$, and consequently that $u_k \to u $ in $H^1$ and $|\dot u|= \text{const.}$ Since $\E_H, \A \in C^1(W^{1,1}\setminus \tilde \M_0; \R)$, we infer that
\[
\E_H'(u_k)[\varphi] - \lambda \A'(u_k)[\varphi] \to \E_H'(u)[\varphi]- \lambda \A'(u)[\varphi] \quad \forall \varphi \in W^{1,1}.
\]

On the other hand
\[
|\E_H'(u_k)[\varphi]- \lambda \A'(u_k)[\varphi]| = |\lambda_k - \lambda||\A'(u_k)[\varphi]| \leq |\lambda_k -\lambda| \D(u_k) |\varphi|_2 \quad \forall \varphi \in H^1,
\]
which implies
\[
|\E_H'(u_k)[\varphi]- \lambda \A'(u_k)[\varphi]| \leq  C |\lambda_k - \lambda| \to 0 \quad \forall \varphi \in H^1,
\]
where the constant $C$ depends on $H$, $\varphi$ and $\tau$ but not on $k$.
Then we can conclude that 
\[
\E_H'(u)[\varphi]- \lambda \A'(u)[\varphi] = 0, \quad \forall \varphi \in H^1.
\]
Thus, thanks to Lemma \ref{solreg} we get that $u$ is a $(H-\lambda)$-loop. This implies that $\lambda \in \Lambda(\tau)$ and concludes the proof of the Lemma.
\end{proof}

We notice that Lemma \ref{lamboundlem} readily implies that there exists $\tau_0$ small enough such that for every $\tau \in [-\tau_0, 0)\cup(0, \tau_0]$ it holds $0 \not \in \Lambda(\tau)$. Therefore, no minimizer $u_\tau$ of $S_H(\tau)$ could be a proper $H$-loop. 
In the next Lemma we show explicitly the connection between the set of the Lagrange multipliers $\Lambda(\tau)$ and the isoperimetric function $S_H(\tau)$. 

\begin{lemma} \label{Lagrinf}
For every $\tau \in \R \setminus \{0\}$, it holds 
\begin{equation}\label{isolambdaineq}
\limsup_{\varepsilon \to 0^+} \frac{S_H(\tau + \varepsilon)-S_H(\tau)}{\varepsilon}\leq \min \Lambda (\tau) \leq \max \Lambda (\tau) \leq \liminf_{\varepsilon \to 0^-} \frac{S_H(\tau+ \varepsilon)-S_H(\tau)}{\varepsilon}.
\end{equation}
Therefore, for almost every $\tau \in \R\setminus \{0\}$ it holds that $\Lambda(\tau) = \{S_H'(\tau)\}$.

Moreover, if $S_H(\tau)$ reaches a local minimum in $\tau \in \R\setminus \{0\}$, then $\Lambda(\tau) = \{0\}$. 
\end{lemma}
\begin{proof}
Let $\lambda \in \Lambda(\tau)$ and let $u \in \M_\tau$ be an associated $(H-\lambda)$-loop which minimize $S_H(\tau)$.
Fix $\varepsilon_0>0$ such that $1 + \frac{\varepsilon}{\tau}>0$ for every $\varepsilon \in (0, \varepsilon_0)$, and define for every $\varepsilon \in (0, \varepsilon_0)$ the function $u_\varepsilon := \sqrt{1 + \frac{\varepsilon}{\tau}}u$. A simple computation shows that $\A(u_\varepsilon) = \tau + \varepsilon$. Since $S_H(\tau + \varepsilon) \leq \E_H(u_\varepsilon)$ we have that
\[
\frac{S_H(\tau+ \varepsilon)-S_H(\tau)}{\varepsilon} \leq \frac{\E_H(u_\varepsilon) - \E_H(u)}{\varepsilon} = \frac{\sqrt{1+\frac{\varepsilon}{\tau}}-1}{\varepsilon}\frac{\E_H(u_\varepsilon)-\E_H(u)}{\sqrt{1+\frac{\varepsilon}{\tau}}-1},
\]
which implies
\[
\limsup_{\varepsilon \to 0^+} \frac{S_H(\tau+ \varepsilon)-S_H(\tau)}{\varepsilon} \leq \lim_{\varepsilon \to 0}\frac{\sqrt{1+\frac{\varepsilon}{\tau}}-1}{\varepsilon} \lim_{s\to1}\frac{\E_H(su)-\E_H(u)}{s-1} = \frac{\E_H'(u)[u]}{2\tau}.
\]
Since $u$ is a $(H-\lambda)$-loop, it holds that $\E_H'(u)[u] = 2\lambda \tau$. Hence we get
\[
\limsup_{\varepsilon \to 0^+} \frac{S_H(\tau+ \varepsilon)-S_H(\tau)}{\varepsilon}  \leq \lambda. 
\]
which proves the first inequality in \eqref{isolambdaineq}. As for the third one, in can be proved in the same way with some minor changes in the notation. 
Since, as a consequence of Lemma \ref{isofunclip}, the mapping $\tau \to S_H(\tau)$ is differentiable for a.e. $\tau \in \R\setminus\{0\}$, we readily get from \eqref{isolambdaineq} that for a.e. $\tau \in \R \setminus \{0\}$ the set $\Lambda(\tau)$ reduces to a singleton and in particular $\Lambda(\tau)  = \{S'_H(\tau)\}$. 

As for the last point of the Lemma, suppose that $S_H(\tau)$ reaches a local minimum in a point $\tau \in \R \setminus \{0\}$. Then there exists $\varepsilon_0$ small enough such that $S_H(\tau + \varepsilon) \geq S_H(\tau)$ for every $\varepsilon \in [-\varepsilon_0, \varepsilon_0]$. Therefore, again by \eqref{isolambdaineq} we infer that 
\[
0 \leq \min \Lambda(\tau) \leq \max \Lambda(\tau) \leq 0,
\]
hence the desired result.
\end{proof}

As a byproduct of Lemma \ref{Lagrinf}, we obtain the proof of Theorem \ref{mainteo6}.
\begin{proof}[Proof of Theorem \ref{mainteo6}] 
Let $H$ be such that \eqref{k1hp} and $[H]=0$ are satisfied. Since Theorem \ref{mainteo4} holds, for every $\tau \in \R \setminus \{0\}$ there exists a solution of Problem \eqref{Kloopprob} for some $\lambda \in \R$.

First of all, we prove that for every $\delta >0$ it holds that
\begin{equation}\label{essup}
\esssup_{\tau \in (0,\delta)}|S'_{\tilde H}(\tau)| = +\infty.
\end{equation}
Suppose by contradiction that there exists $\delta_0>0$ and $M>0$ such that $\esssup_{\tau \in (0,\delta)}|S'_{\tilde H}(\tau)|\leq M$. Since by Lemma \ref{isofunclip} the function $S_{\tilde H}(\tau)$ is differentiable almost everywhere, there exists a sequence $(\tau_k) \subset (-\delta_0, \delta_0) \setminus \{0\}$ being such that $\tau_k \to 0$ and $S'_{\tilde H}(\tau_k)$ is well defined for every $k$. Up to subsequences, we can always suppose that every $\tau_k$ has the same sign. In particular, without loss of generality we assume that $\tau_k \geq 0$. Moreover, taking a further subsequence if necessary, we can assume that $\tau_k \geq 2 \tau_{k+1}$. Exploiting again Lemma \ref{isofunclip}, we can apply the fundamental Theorem of calculus and get that
\[
S_{\tilde H}(\tau_k) \leq S_{\tilde H}(\tau_{k+1}) + \int_{\tau_{k+1}}^{\tau_k}|S'_{\tilde H}(s)| \de s \leq  S_{\tilde H}(\tau_{k+1}) + M(\tau_k - \tau_{k+1}).
\] 
Therefore, using that $\sqrt{\frac{\tau_{k+1}}{\tau_k}}\leq \frac{\sqrt{2}}{2}$ we infer
\[
\frac{S_{\tilde H}(\tau_k)}{\sqrt{\tau_k}}\leq \frac{\sqrt{2}}{2}\frac{S_{\tilde H}(\tau_{k+1})}{\sqrt{\tau_{k+1}}} + M \sqrt{\tau_k}
\]
Taking the limit as $k \to +\infty$, thanks to Lemma \ref{stildelimit} we obtain $S \leq \frac{\sqrt{2}}{2}S$, thus a contradiction. 
Since \eqref{essup} implies that there exists a sequence $\tau_k \to 0$ being such that $|S'_{H}(\tau_k)| \to +\infty$, the first result of Theorem \ref{mainteo6} follows form Lemma \ref{Lagrinf}.

As for the second part of the Theorem, two situations can occur. First of all, suppose that there exists a sequence $(\tau_k) \in \R$ such that $|\tau_k|\to +\infty$ and $S'_H(\tau_k) \leq 0$.
Arguing as in the previous part, up to subsequences we can suppose that every $\tau_k $ has the same sign, and in particular whitout loss of generality we can suppose $\tau_k \geq 0$. Morever, taking a further subseauence if necessary, suppose that it is satisfied
\begin{equation}\label{techincal6876}
\frac{1+ |Q_H|_\infty}{1 - |Q_H|_\infty}\sqrt{|\tau_k|}< \sqrt{|\tau_{k+1}|}.
\end{equation}
We claim that for every $k$ there exists $\tau'_k \in (\tau_k, \tau_{k+1})$ such that $S'_H(\tau'_k)>0$. Indeed, if that is not the case by the fundamental theorem of calculus we easily get that $S_H(\tau_{k+1})\leq S_H(\tau_{k})$. Then, thanks to Lemma \ref{bound} we infer that
\[
(1 - |Q_H|_\infty)\sqrt{|\tau_{k+1}|}\leq (1 + |Q_H|_\infty)\sqrt{\tau_k},
\]
which contradicts \eqref{techincal6876}.
As a consequence of the claim, there exists another sequence $(\tau''_k)$ with $\tau''_k \in[\tau_k, \tau'_k)$, being such that $S_H(\tau''_k)$ is a local minimum of $S_H(\tau)$ for every $k$. Therefore, we get the desired result thanks to Lemma \ref{Lagrinf}.

The other possibility is that there exists $\tau_0$ being such that $S'_H(\tau) \geq 0$ for a.e. $\tau \geq \tau_0$. 
In this case, for every $M >0$ it holds that
\begin{equation}\label{essinf}
\essinf_{\tau \in (M, +\infty)}S'_H(\tau) = 0.
\end{equation}
Suppose by contradiction that there exists $M_0 >0$ such that $\essinf_{\tau \in (M_0, +\infty)}S'_H(\tau) \geq \delta >0$. 
Again by the fundamental theorem of calculus together with Lemma \ref{bound} we get that
\[
(1 + |Q_H|_\infty)S\sqrt{\tau}\geq S_H(\tau) = S_H(M_0) + \int^\tau_{M_0}S'_H(s)\de s \geq \delta \tau + (S_H(M_0) - \delta M_0),
\]
which readily gives a contradiction.
By \eqref{essinf} we get that there exists a sequence $\tau_k$ such that $\tau_k \to +\infty$ and $S'_H(\tau_k) \to 0$. Then also in this case we recover the desired result thanks to Lemma \ref{Lagrinf}.

Since the case of $H$ satisfying \eqref{k2hp} and $H^{\infty} = 0$ can be treated exactly in the same way, the proof of the Theorem is complete.
\end{proof}
\end{subsection}

\end{section}
\clearpage
\begin{section}{Immersed curves of prescribed radial curvature}\label{immersedloops}
This section is devoted to the proof of Theorem \ref{mainteo0}. In particular, we present the proof of the case $A > 0$, where $H$-loops will rise as small (in the $C^{2,\alpha}$-sense) perturbations of the family
\begin{equation}\label{basecur}
U_{n, R}(t):= Re^{i \frac{t}{n}} + e^{it}, \quad n \in \N, R > 0. 
\end{equation}
When $A<0$ it is sufficient to consider the family 
\[
V_{n,R}(t):= Re^{-i \frac{t}{n}} + e^{it}, \quad n \in \N, R > 0,
\]
and proceed as in the previous case, with some minor changes.

In order to precisely state the result, we also introduce the unit vector
\begin{equation}\label{normdir}
\Ne_{n,R} : = \frac{i \dot U_{n,R}}{|\dot U_{n,R}|}.
\end{equation}
Then the following holds.
\begin{teo}\label{mainteo}
Let $H \in C^2(\R^2; \R)$ be a radially symmetric function in the form
\begin{equation}\label{H0}
H(p) = h(|p|) \quad \text{ with }\quad  
h(s) = 1 + \frac{A}{s^{\gamma}} + \frac{\tilde h(s)}{s^{\gamma + \beta}} \quad \text{ when $s$ is large}, 
\end{equation}
with $A >0$, $\gamma>1$, $\beta \geq0$ and $\tilde h \in C^2((0, +\infty);\R)$ being such that the following are satisfied:
\begin{align}
&\begin{cases}\label{H1}
\tilde h (s) \text{ is bounded} &  \text{if }\beta >1,\\
\tilde h (s) = B + o\left(s^{\beta-1} \right) \text{ with }B \in \R &\text{if }\beta \in (0, 1], \\
\tilde h(s) = B + o(s^{-1}) \text{ with } B > -A & \text{if } \beta = 0,
\end{cases}\\\label{H2}
&|h''(s)| \leq  \frac{C}{s^{\gamma +1 + \min\{1, \beta\}}} \quad\text{ for some }C>0,  \text{ when $s$ is large}.
\end{align}
There exist $\overline n \geq 2$, $\overline M>0$ and two sequences $(R_n) \subset \R$ and $(\phi_n) \subset C^{2,\alpha}\left(\R / 2\pi \frac{n}{n-1}; \R\right)$ such that for every $n\geq \overline n$ are satisfied
\[
\begin{aligned}
&R_n = (r_n n)^{\frac{1}{\gamma+2}} \quad \text{ with } r_n \in (r_0, r_1) \quad \text{ for some } 0<r_0 < r_1,\\
&\|\phi_n\|_{C^{2,\alpha}\left(\R / 2\pi \frac{n}{n-1}; \R\right)} \leq \overline M n^{-\frac{\gamma}{\gamma+2}} ,\\
&\K(U_{n, R_n} + \phi_n \Ne_{n, R_n}) = H(U_{n, R_n} + \phi_n \Ne_{n, R_n}).
\end{aligned}
\]
\end{teo}

This Section is organized as follows. In Subsection \ref{prel} we introduce an equivalent formulation of the problem involving a second order linear operator. In Subsection \ref{lin} some properties of such operator are stated. Subsection \ref{redux} is devoted to the finite-reduction of the problem, through an application of the contraction principle. In Subsection \ref{var} and Subsection \ref{rota} we conclude the proof of Theorem \ref{mainteo} carrying out a variational argument. Finally, all technical computations and estimates are left to the Appendix.

\begin{subsection}{Preliminary results}\label{prel}

We are interested in finding solutions of 
\begin{equation}\label{problem1}
\begin{cases}
\phi \in C^{2,\alpha}_{2\pi \frac{n}{n-1}},\\
\K(U_{n, R} + \phi \Ne_{n, R}) = H(U_{n, R} + \phi \Ne_{n, R}) & \text{in }\R,
\end{cases}
\end{equation}
where $U_{n,R}$ and $N_{n,R}$ are defined in \eqref{basecur} and \eqref{normdir}, respectively.
 
Let $0 < a\leq b$ and $\delta \in (0,1)$ be fixed. For every $n \in \N$ we define
\begin{equation}\label{Rset}
S_n:= [a n^\delta, b n^\delta],
\end{equation}
and we will always assume that $R \in S_n$. 
Moreover, by \eqref{H0} we have that there exists $s_0>0$ being such that
\[
h(s) = 
1 + \frac{A}{s^{\gamma}} + \frac{\tilde h(s)}{s^{\gamma + \beta}} \quad \forall s \geq s_0.
\]
Therefore, we set
\begin{equation}\label{nsotto}
\underline n : = \min \{n\geq 2 \ |\ 2+s_0<an^\delta \text{ and } bn^{-(1-\delta)} <1 \}, 
\end{equation}
so that $S_n \subset (2 + s_0, n)$ for every $n \geq \underline n$.
In particular, we recover the following inequality, that we will use extensively through all the Section: for $n\geq \underline n$, $R \in S_n$, it holds
\begin{equation}\label{passag17}
1 + s_0 < C_1 n^\delta \leq |U_{n,R}| \leq C_2n^\delta \quad \forall t \in [0,2\pi n].
\end{equation}

In order to get rid of the parameter dependence in the functional space, we define 
\begin{equation}\label{littleun}
u_{n,R}(t) := U_{n,R}\left(\frac{n}{n-1}t\right), \quad \n_{n,R}(t) := \Ne_{n,R}\left(\frac{n}{n-1}t\right).
\end{equation}
Then we can equivalently solve  
\begin{equation}\label{problem2}
\begin{cases}
\varphi \in C^{2,\alpha}_{2\pi},\\
\K(u_{n, R} + \varphi \n_{n, R}) = H(u_{n, R} + \varphi \n_{n, R}) & \text{in }\R,
\end{cases}
\end{equation}
thus obtaining a solution for \eqref{problem1} given by $\phi = \varphi\left( \frac{n-1}{n}t\right)$, which also satisfies \[
\|\phi\|_{2,\alpha; 2\pi n} \leq \|\varphi\|_{2,\alpha; 2\pi}.
\]
We point out that, if $\varphi \in C^{2,\alpha}_{2\pi}$, both $\K(u_{n, R} + \varphi \n_{n, R})$ and $H(u_{n, R} + \varphi \n_{n, R})$ belong to $C^{0,\alpha}_{2\pi}$.

In the next part of the Subsection, we present an equivalent formulation of \eqref{problem2}.
To this end, let us introduce the linear operators 
\begin{equation}\label{linfop}
\Le_\infty\varphi := \varphi'' + \varphi,
\end{equation}
and
\[
\Le_{n, R} \varphi := a_{n,R}\varphi'' + b_{n, R}\varphi' + c_{n,R}\varphi,
\]
where
\begin{equation}\label{lincoeff}
\begin{aligned}
&a_{n,R} := \frac{1}{|\dot u_{n,R}|^2} \quad b_{n,R} := -\frac{\dot u_{n,R}\cdot \ddot u_{n,R}}{|\dot u_{n,R}|^4}\\
&c_{n,R} := \frac{2(\dot u_{n,R} \cdot \ddot u_{n,R})^2 - 2|\ddot u_{n,R}|^2 |\dot u_{n,R}|^2  + 3(i \dot u_{n,R} \cdot \ddot u_{n,R})^2}{|\dot u_{n,R}|^6} 
\end{aligned}
\end{equation}
Some computations show that $\Le_{n,R}, \Le_\infty : C^{2,\alpha}_{2\pi }\to C^{0,\alpha}_{2\pi }$ and that 
\[
\Le_{n,R} \to \Le_\infty \quad \text{ in }\mathcal{L}(C^{2,\alpha}_{2\pi}, C^{0,\alpha}_2\pi), \quad \forall R \in S_n, \text{ as }n \to \infty.
\]
Moreover, it holds that
\[
\frac{\de }{\de \varphi }\K(u_{n,R} + \varphi \n_{n,R})_{|\varphi = 0} = \Le_{n,R}\varphi,
\]
that is, $\Le_{n,R}$ is the linearized operator of $\varphi \mapsto \K(u_{n,R}+\varphi n_{n,R})$ around $\varphi = 0$. We refer to the Appendix for the formal justification of such claims (see Lemma \ref{lineopcoverg} and Lemma \ref{appresuno}, respectively).

Defining $\Res^1_{n,R}: C^{2,\alpha}_{2\pi} \to C^{0, \alpha}_{2\pi}$ as
\begin{equation}\label{resunodefini}
\Res^1_{n,R} (\varphi) := \K(u_{n,R} + \varphi \n_{n,R}) - \K(u_{n,R}) - \Le_{n,R}\varphi
\end{equation}
we obtain the following identity
\begin{equation}\label{Kdec}
\Le_{\infty}\varphi = \K(u_{n,R} + \varphi \n_{n,R}) - \K(u_{n,R}) - (\Le_{n,R}-\Le_\infty)\varphi - \Res^1_{n,R}(\varphi).
\end{equation}

On the other hand, by the mean value theorem we get the decomposition
\begin{equation}\label{Hdec}
H(u_{n,R} + \varphi \n_{n,R}) = H(u_{n,R}) + \varphi \nabla H (u_{n,R}) \cdot \n_{n,R} + \Res^2_{n,R}(\varphi).
\end{equation}
Since $H$ is radial, exploiting the definition of $u_{n,R}$ and $\n_{n,R}$ we get that  $\Res^2_{n,R}: C^{2,\alpha}_{2\pi} \to C^{0,\alpha}_{2\pi}$

In conclusion, we can define the functional 
\begin{equation}\label{funfF}
\F_{n,R}(\varphi) :=  - \K(u_{n,R}) + H(u_{n,R}) + \varphi \nabla H(u_{n,R})\cdot \n_{n,R} - (\Le_{n,R} - \Le_\infty) \varphi + \Res^2_{n,R}(\varphi) - \Res^1_{n,R}(\varphi),
\end{equation}
and as a consequence of the previous discussion it holds that $\F_{n,R}: C^{2,\alpha}_{2\pi} \to C^{0,\alpha}_{2\pi}$.

Thus, putting together \eqref{Kdec}, \eqref{Hdec} and \eqref{funfF}, we infer that \eqref{problem2} is equivalent to 
\[
\begin{cases}
\varphi \in C^{2, \alpha}_{2\pi},\\
\Le_\infty \varphi = \F_{n,R}(\varphi) & \text{in }\R,
\end{cases}
\]
with $\F_{n,R}(\varphi) \in C^{0,\alpha}_{2\pi}$. 
\end{subsection}

\begin{subsection}{The linearized problem}\label{lin}
This Subsection is devoted to the study of the properties of the linearized problem
\begin{equation}\label{linearform}
\begin{cases}
\varphi \in C^{2,\alpha}_{2\pi},\\
\Le_\infty \varphi =  f & \text{in }\R,
\end{cases}
\end{equation}
where $\Le_\infty$ is defined in \eqref{linfop} and $f \in C^{0,\alpha}_{2\pi}$. This is a very classical problem; we briefly present here some basic facts which are needed for the present work, and fix some notation.   

As is known, every solution of
\[
\begin{cases}
\omega \in C^{2,\alpha}_{2\pi},\\
\Le_\infty \omega =  0 & \text{in }\R,
\end{cases}
\]
is in the form 
\begin{equation}\label{homsol}
\omega(t) = c_1 \omega_1( t) + c_2 \omega_2( t), \quad c_1, c_2 \in \R,  
\end{equation}
where
\begin{equation}\label{omegaker}
\omega_1(t) := \cos t, \quad \omega_2(t) := \sin t.
\end{equation}

Also the inhomogeneous equation 
\begin{equation}\label{inomeq}
\Le_\infty \varphi = f \quad \text{in }\R,
\end{equation}
possesses an explicit formulation for its solutions, that is
\begin{equation}\label{inomsol}
\varphi(t) = c_1 \omega_1(t) + c_2\omega_2(t) + \eta_f(t), \quad c_1, c_2 \in \R, 
\end{equation}
where 
\[
\eta_f(t) := \int_0^t F(s) \cos (t-s) \de s \quad \text{ with } \quad F(s) := \int_0^sf(\tau) \de \tau.
\] 
Some standard computations show that $\varphi$ is $2\pi$-periodic (and thus a solution of \eqref{linearform}) if and only if
\begin{equation}\label{perpendicul}
\int_0^{2\pi}f \omega_i \de t = 0 \quad \text{for }i=1,2.
\end{equation}
Justified by that, we introduce the decomposition
\[
\begin{aligned}
&C^{2, \alpha}_{2\pi} = X^{\parallel} + X^{\perp}, \\
&C^{0,\alpha}_{2\pi} = Y^{\parallel} + Y^{\perp},
\end{aligned}
\]
where
\[
\begin{aligned}
&X^{\parallel} = Y^{\parallel} := \ker(\Le_\infty) = \text{span}\{\omega_1, \omega_2\},\\
&X^{\perp} := \{f \in C^{2,\alpha}_{2\pi} \ |\ \eqref{perpendicul} \text{ is satisfied}\},\\
&Y^{\perp} := \{f \in C^{0,\alpha}_{2\pi} \ |\ \eqref{perpendicul} \text{ is satisfied}\},
\end{aligned}
\]

The following holds. 
\begin{lemma}\label{linearprobest}
The operator $\Le_\infty$ is a bijection from $X^\perp$ to $Y^\perp$, that is, for every $f \in Y^\perp$ there exists a unique $\varphi_f \in X^\perp$ such that $\Le_\infty \varphi_f = f$ and vice-versa.  
Moreover it holds
\begin{equation}\label{linearestimate}
\|\varphi_f\|_{2, \alpha} \leq C\|f\|_{0,\alpha} \quad \forall f \in Y^\perp,
\end{equation}
where the constant $C$ only depends on $\alpha$. 
\end{lemma}
\begin{proof}
The result is standard, nevertheless we give a short sketch of the proof. Let $f \in Y^\perp$ be fixed. As we have seen, there exists $\varphi \in C^{2, \alpha}_{2\pi}$ in the form \eqref{inomsol} such that $\Le_\infty \varphi = f$. We define
\[
\varphi_f(t) := \eta_f(t) - \sum_{i=1,2} \frac{1}{\pi}\left(\int_0^{2\pi}\omega_i(t) \eta_f(t) \de t\right)\omega_i(t).
\]
It is immediate to see that $\varphi_f \in X^\perp$. Also the uniqueness can be readily verified, since the difference of two solutions of equation \eqref{inomeq} has to be in the form \eqref{homsol}.

A simple computation shows that
\[
\|\varphi_f\|_{2, \alpha} \leq \|\eta_f\|_{2,\alpha} + C|\eta_f|_\infty \leq C\|\eta_f\|_{2,\alpha}.
\]
Therefore, to conclude it suffices to see that $\|\eta_f\|_{2,\alpha}\leq C\|f\|_{0,\alpha}$. Since this involves only basic estimates, we limit ourselves to show that
\begin{equation}\label{passag987}
[\eta''_f]_\alpha \leq C\|f\|_{0,\alpha}.
\end{equation}
Being $\eta_f$ a solution of \eqref{inomeq}, it holds that $\eta''_f = f - \eta_f$. Then, taking $0 \leq t_1 \leq t_2 \leq 2\pi$, we have that
\[
\begin{aligned}
&|\eta''_f(t_2) -\eta''_f(t_1)|\\
 &\leq |f(t_2) - f(t_1)| + \int^{t_2}_{t_1}|F(s)||\cos (t_2-s)|\de s + \int^{t_1}_0|F(s)||\cos(t_2 -s ) - \cos (t_1-s)|\de s\\
&\leq [f]_\alpha|t_2-t_1|^\alpha+ |f|_\infty \left( \int^{t_2}_{t_1}s \de s + |t_2-t_1|\int^{t_1}_0 s \de s\right) \leq [f]_\alpha|t_2-t_1|^\alpha+C|f|_\infty |t_2-t_1|^\alpha.
\end{aligned}
\]
Therefore we infer \eqref{passag987} and conclude the proof of the Lemma. 
\end{proof}

\end{subsection}

\begin{subsection}{The finite-dimensional reduction}\label{redux}

In this Subsection we prove the following Proposition. 
\begin{prop}\label{reduced}
Let $\alpha \in (0,1)$. For every $0<a\leq b$ and $\delta \in (0,1)$, there exist $\overline n \geq \underline n$, where $\underline n$ is as in \eqref{nsotto}, and $\overline M>0$, being such that for every $n\geq \overline n$ and $R \in S_n$ there exists $\varphi_{n,R} \in C^{2,\alpha}_{2\pi}$ which satisfies
\begin{equation}\label{eqconmolt}
\K(u_{n,R} + \varphi_{n,R} \n_{n,R}) - H(u_{n,R} + \varphi_{n,R} \n_{n,R}) = \sum_{i=1,2} \lambda_i \omega_i \quad \text{in }\R,
\end{equation}
where $\omega_i$ are as in \eqref{omegaker} and
\begin{equation}\label{lambrappr}
\lambda_i = \frac{1}{\pi}\int_0^{2\pi}(\K(u_{n,R} + \varphi_{n,R} \n_{n,R})- H(u_{n,R} + \varphi_{n,R} \n_{n,R}))\omega_i \de t.
\end{equation}
Moreover, it holds
\begin{equation}\label{estimphi}
\|\varphi_{n,R}\|_{2,\alpha}\leq \overline M n^{-\tilde \gamma}, \quad \text{ with }\quad \tilde \gamma := \{\delta \gamma, 1-\delta\}.
\end{equation}
\end{prop}
\begin{proof}
In Subsection \ref{prel} we proved that \eqref{problem1} is equivalent to
\begin{equation}\label{problem3}
\begin{cases}
\varphi \in C^{2,\alpha}_{2\pi},\\
\Le_\infty \varphi = \F_{n,R}(\varphi) & \text{in }\R,
\end{cases}
\end{equation}
where $\Le_\infty$ and $\F_{n,R}(\varphi)\in C^{0,\alpha}_{2\pi}$ are defined in \eqref{linfop} and \eqref{funfF}, respectively.

The aim is to rewrite \eqref{problem3} as a fixed point problem. In general, this is not possible: as seen in Section \ref{lin}, in order for $\Le_\infty$ to be invertible it would have to hold $\F_{n,R}(\varphi) \in Y^\perp$.
Consider then the projection of $\F_{n,R}(\varphi)$ on $Y^\perp$, given by 
\[
\hat \F_{n,R}(\varphi) := \F_{n,R}(\varphi) + \sum_{i=1,2} \lambda_i \omega_i
\]
where
\[
\lambda_i := - \frac{1}{\pi}\int_0^{2\pi}\F_{n,R}(\varphi)\omega_i \de t,
\]
and the fixed point problem
\begin{equation}\label{fixpopoint}
\begin{cases}
\varphi \in C^{2,\alpha}_{2\pi}\\
\varphi = \Q_{n,R}(\varphi)
\end{cases}
\end{equation}
where $\Q_{n,R}:= \Le_\infty^{-1}\circ \hat \F_{n,R}$. 
Some elementary computation shows that $\varphi \in C^{2,\alpha}_{2\pi}$ satisfies \eqref{eqconmolt} if and only if it solves \eqref{fixpopoint}. 
We stress that in general \eqref{problem3} and \eqref{fixpopoint} are not equivalent, unless $\lambda_1, \lambda_2=0$, and indeed this issue is addressed in Subsection \ref{var} and Subsection \ref{rota}.  

For $n \geq \underline n$, and $M>0$ we define the set 
\begin{equation}\label{mballdef}
B_{n; M} = \left\{\varphi \in X^\perp \ \Bigg|\ \|\varphi\|_{2,\alpha}\leq M n^{-\tilde \gamma}\right\}.
\end{equation}
The following Lemma holds.
\begin{lemma}\label{contraction}
For every $0<a\leq b$ and $\delta \in (0,1)$, there exist $\overline n \geq \underline n$ and $\overline M>0$ being such that for every $n \geq \overline n$ and $R \in S_n$ it holds that
\begin{align}
&\Q_{n,R}(\varphi) \in B_{n;\overline M} &  &\forall \varphi \in B_{n, \overline M}\label{internfin}\\
&\|\Q_{n,R}(\varphi_1) - \Q_{n,R}(\varphi_2)\|_{2,\alpha} \leq L \|\varphi_2 - \varphi_1\|_{2,\alpha}, & &\forall \varphi_1, \varphi_2 \in B_{n;\overline M}, \label{Lipfin}
\end{align}
with $L<1$. 
\end{lemma}

Let us conclude the proof of the Preposition. Let $\overline n$ and $\overline M$ be as in Lemma \ref{contraction}. Then for every $n \geq \overline n$ and $R \in S_n$, the functional $\Q_{n,R}$ satisfy the assumption of the contraction principle and \eqref{fixpopoint} admits a solution $\varphi_{n,R}\in B_{n; \overline M}$. Thus \eqref{eqconmolt} and \eqref{estimphi} are proved. 
A simple computation shows that \eqref{lambrappr} plainly follows from \eqref{eqconmolt}, keeping in mind the definitions of $\F_{n,R}$ and $\lambda_i$. 
Thus the proof is complete.
\end{proof}

\begin{proof}[Proof of Lemma \ref{contraction}]
By construction, we have that $\hat \F_{n,R}(\varphi) \in Y^\perp$. Therefore, the functional $\Q_{n,R}: X^\perp \to X^\perp$ is well defined, and thanks to inequality \eqref{linearestimate} and some standard computations, we get that for all $\varphi, \varphi_1, \varphi_2 \in X^\perp$ it holds
\begin{align}\label{M1}
&\|\Q_{n,R}(\varphi)\|_{2,\alpha} \leq M_1\|\F_{n,R}(\varphi)\|_{0,\alpha}, \\
&\|\Q_{n,R}(\varphi_2)- \Q_{n,R}(\varphi_1)\|_{2,\alpha}\leq C\|\F_{n,R}(\varphi_1) - \F_{n,R}(\varphi_2)\|_{0,\alpha},\label{Lip}
\end{align}
where $M_1$ is a constant depending only on $\alpha$.

The remaining part of the proof is divided into several Steps. We recall here the elementary inequality
\begin{equation}\label{holderineq}
\|fg\|_{0,\alpha} \leq \|f\|_{0,\alpha}\|g\|_{0,\alpha}, \quad \forall f,g \in C^{0,\alpha}_{2\pi},
\end{equation}
which will be extensively used through all the proof.
\vspace{5pt}
\begin{cla}
There exists $M_2>0$ such that for every $n \geq \underline n$ and $R \in S_n$ it holds
\begin{equation}\label{M2}
\|H(u_{n,R}) - \K(u_{n,R})\|_{0,\alpha}\leq M_2 \frac{1}{n^{\tilde \gamma}}.
\end{equation}
\end{cla}
\vspace{5pt}
We begin by noticing that
\[
\|H(u_{n,R}) - \K(u_{n,R})\|_{0,\alpha} \leq \|H(u_{n,R}) - 1\|_{0,\alpha} + \|\K(u_{n,R})-1\|_{0,\alpha}.
\]
On one hand, by \eqref{holderineq} we infer that
\[
\|\K(u_{n,R})-1\|_{0,\alpha} \leq \||\dot u_{n,R}|^{-3}\|_{0,\alpha}\left( \left\|i \dot u_{n,R}\cdot \ddot u_{n,R} - \frac{n^3}{(n-1)^3}\right\|_{0,\alpha} + \left\||\dot u_{n,R}|^3- \frac{n^3}{(n-1)^3}\right\|_{0,\alpha}\right),
\]
hence by \eqref{stimeupic} we obtain
\[
\|\K(u_{n,R})-1\|_{0,\alpha} \leq C\frac{1}{n^{1-\delta}}.
\]

On the other hand, by \eqref{passag17} and we infer
\begin{equation}\label{passag200}
|H(u_{n,R}) -1| \leq \frac{|A|}{|u_{n,R}|^\gamma} + \frac{C}{|u_{n,R}|^{\gamma + \beta}} \leq C \frac{1}{n^{\delta\gamma}}, \quad \forall t \in [0, 2\pi].
\end{equation}
Moreover, \eqref{H1}-\eqref{H2} imply that
\begin{equation}\label{hprimeest}
|\nabla H (p)| \leq \frac{1}{|p|^{\gamma+\min\{1,\beta\}}}, \quad \text{ when }|p|>s_0. 
\end{equation}
therefore, using also \eqref{stimeupic}, we get that for every $0 \leq t_1 \leq t_2 \leq 2\pi$ it holds 
\[
\begin{aligned}
&|H(u_{n,R}(t_2)) - H(u_{n,R}(t_1))|\\
 &\leq C |\nabla H(u(\sigma t_2 + (1-\sigma)t_1)||t_2-t_1|^{\alpha} \leq C\frac{|t_2-t_1|^\alpha}{|u_{n,R}(\sigma t_2 + (1-\sigma) t_1)|^{\gamma+\min\{1,\beta\}}}
\end{aligned}
\]
for some $\sigma \in (0,1)$. Thus, from \eqref{passag17} we obtain
\[
[H(u_{n,R})-1]_{\alpha} \leq C \frac{1}{n^{\delta (\gamma+ \min\{1, \beta\})}},
\]
which together with \eqref{passag200} implies $\|H(u_{n,R})-1\|_{0,\alpha}\leq C n^{-\delta \gamma}$. The final estimate readily follows. 
\begin{cla}
For every $n \geq \underline n$, $R \in S_n$ and $\varphi \in C^{2,\alpha}_{2\pi}$ it holds
\begin{equation}\label{Lnest}
\|(\Le_{n,R} - \Le_\infty)\varphi\|_{0,\alpha}\leq C\frac{R}{n} \|\varphi\|_{2,\alpha}
\end{equation}
\end{cla}
By definition of $\Le_{n,R}, \Le_\infty$ we get that
\[
\begin{aligned}
\|(\Le_{n,R} - \Le_\infty)\varphi\|_{0,\alpha} &\leq \|a_{n,R}-1\|_{0,\alpha}\|\varphi''\|_{0,\alpha} + \|b_{n,R}\|_{0,\alpha}\|\varphi'\|_{0,\alpha} + \|c_{n,R}-1\|_{0,\alpha}\|\varphi\|_{0,\alpha}\\
 &\leq (\|a_{n,R}-1\|_{0,\alpha} + \|b_{n,R}\|_{0,\alpha} +\|c_{n,R}-1\|_{0,\alpha})\|\varphi\|_{2,\alpha}.
\end{aligned}
\]
where $a_{n,R}, b_{n,R}, c_{n,R}$ are defined in \eqref{lincoeff}. Then is sufficient to apply estimates \eqref{coeffestlin} to get the desired result.

\begin{cla}\label{nablaHest}
For every $n \geq \underline n$, $R \in S_n$ and $\varphi \in C^{2,\alpha}_{2\pi}$ it holds
\[
\begin{aligned}
\|\varphi  \nabla H(u_{n,R}) \cdot \n_{n,R}\|_{0,\alpha} \leq \frac{C}{n^{\delta(\gamma+\min\{1, \beta\})}} \|\varphi\|_{2,\alpha}.
\end{aligned}
\]
\end{cla}
Thanks to \eqref{holderineq} it is sufficient to prove that
\begin{equation}\label{passagg21}
\|\nabla H(u_{n,R})	\cdot\n_{n,R}\|_{0,\alpha} \leq \frac{C}{n^{\delta(\gamma +\min\{1,\beta\})}}.
\end{equation}
On one hand, by \eqref{passag17} and \eqref{hprimeest} we readily get that
\begin{equation}\label{passag666}
|\nabla H(u_{n,R})\cdot \n_{n,R}|\leq |\nabla H(u_{n,R})| \leq \frac{C}{|u_{n,R}|^{\gamma+\min\{1, \beta\}}}\leq \frac{C}{n^{\delta(\gamma+\min\{1,\beta\})}}, \quad \forall t \in [0,2\pi].
\end{equation}
On the other hand, for every $0 \leq t_1 \leq t_2\leq 2\pi$ we have that
\begin{equation}\label{passag672}
\begin{aligned}
|\nabla H(u_{n,R}(t_2))\cdot \n_{n,R}(t_2) - &\nabla H(u_{n,R}(t_1))\cdot \n_{n,R}(t_1)| \\
&\leq C \left|\frac{\de}{\de t}\left(\nabla H(u_{n,R})\cdot \n_{n,R}\right)_{|\sigma t_1+ (1-\sigma)t_1}\right||t_2-t_1|^\alpha
\end{aligned}
\end{equation}
for some $\sigma \in (0,1)$. Moreover, it holds (see also \eqref{formulozze})
\[
\begin{aligned}
&\frac{\de }{\de t}\left(\nabla H(u_{n,R})\cdot \n_{n,R}\right)\\
 &= \left(h''(|u_{n,R}|) - \frac{h'(|u_{n,R}|)}{|u_{n,R}|}\right)\frac{(u_{n,R} \cdot \dot u_{n,R})(u_{n,R}\cdot \n_{n,R})}{|u_{n,R}|^2} + \frac{h'(|u_{n,R}|)}{|u_{n,R}|}(u_{n,R}\cdot \dot \n_{n,R}).
\end{aligned}
\]
Therefore, by \eqref{passag17}, \eqref{H2} and \eqref{hprimeest} we get that
\[
\left|\frac{\de }{\de t}(\nabla H(u_{n,R})\cdot \n_{n,R})\right| \leq C\frac{1}{n^{\delta(\gamma+1+\min\{1, \beta\})}}(|\dot u_{n,R}| + |u_{n,R}\cdot \dot \n_{n,R}|), \quad \forall t \in [0,2\pi].
\]
Since \eqref{passag17} and \eqref{stimeupic} imply
\begin{equation}\label{passag6000}
|\dot u_{n,R}| + |u_{n,R}\cdot \dot \n_{n,R}| \leq C(1 + n^\delta), \quad \forall t \in [0,2\pi],
\end{equation}
we obtain
\begin{equation}\label{passag673}
\left|\frac{\de }{\de t}(\nabla H(u_{n,R})\cdot \n_{n,R})\right| \leq C \frac{1}{n^{\delta(\gamma +\min\{1,\beta\})}}, \quad \forall t \in [0,2\pi].
\end{equation}
Putting together \eqref{passag666}, \eqref{passag672} and \eqref{passag673} we recover \eqref{passagg21}, thus concluding the proof of the Step.

\begin{cla}
For every $M>0$ there exists $ n_1(M)\geq \underline n$ such that for every $n\geq n_1(M)$, $R \in S_n$ and $\varphi, \varphi_1, \varphi_2 \in B_{n;M}$ it holds 
\begin{align}
&\|\Res^2_{n,R}(\varphi)\|_{0,\alpha} \leq C \frac{1}{n^{\delta (\gamma+\min\{1, \beta\})}}\|\varphi\|_{2,\alpha}, \label{res2est1}\\
&\|\Res^2_{n,R}(\varphi_1) - \Res^2_{n,R}(\varphi_2)\|_{0,\alpha} \leq C \frac{1}{n^{\delta(\gamma + \min\{1, \beta\})}}\|\varphi_1-\varphi_2\|_{2,\alpha}. \label{res2est2}
\end{align}
\end{cla}
By \eqref{Hdec} together with the fundamental theorem of calculus we get that 
\[
\Res^2_{n,R}(\varphi) = \varphi \int_0^1 (\nabla H(u_{n,R} + r\varphi\n_{n,R}) - \nabla H(u_{n,R}))\cdot \n_{n,R}\de r.
\]
We claim that it holds
\begin{equation}\label{clam1}
\left\|\int_0^1 (\nabla H(u_{n,R} + r\varphi\n_{n,R}) - \nabla H(u_{n,R}))\cdot \n_{n,R}\de r\right\|_{0,\alpha}\leq C\frac{1}{n^{\delta(\gamma+\min\{1,\beta\})}}.
\end{equation}
Then the \eqref{res2est1} is a straightforward consequence of \eqref{holderineq}.

In order to prove \eqref{clam1}, we start by noticing that $|u_{n,R} + r \varphi\n_{n,R}| \geq ||u_{n,R}| - |\varphi||$. Then, defining
\begin{equation}\label{overbarn}
 n_1(M) := \min \left \{n \geq \underline n \ \Bigg| \ an^\delta  - \frac{ M}{n^{\tilde \gamma}}>2+ s_0\right\},
\end{equation}
we obtain that, when $n\geq n_1(M)$, $R \in S_n$ and $\varphi \in B_{n; M}$ it holds 
\begin{equation}\label{bigu}
|u_{n,R} + r \varphi\n_{n,R}| \geq n^\delta \left(a - \frac{1}{n^\delta} - \frac{M}{n^{\delta + \tilde \gamma}}\right) \geq C n^{\delta}>1+ s_0 \quad \forall t \in [0,2\pi].
\end{equation}

As a consequence, by \eqref{hprimeest} we get that 
\[
|\nabla H(u_{n,R} + r \varphi \n_{n,R})|\leq \frac{C}{|u_{n,R}+ r\varphi \n_{n,R}|^{\gamma+\min\{1,\beta\}}} \leq C \frac{1}{n^{\delta(\gamma+\min\{1,\beta\})}}, \quad \forall t \in [0, 2\pi],
\]
which together with \eqref{passag666} implies 
\begin{equation}\label{pass13}
\left| \int_0^1 (\nabla H(u_{n,R} + r\varphi\n_{n,R}) - \nabla H(u_{n,R}))\cdot \n_{n,R}\de r \right| \leq C \frac{1}{n^{\delta(\gamma +\min\{1, \beta\})}}, \quad \forall t \in [0,2\pi].
\end{equation}

We claim that, for every $r \in (0,1)$ and for every $0 \leq t_1 \leq t_2 \leq 2\pi$ it holds
\begin{equation}\label{cla25}
|\nabla H(u_{n,R} + r\varphi \n_{n,R})\cdot \n_{n,R}(t_2) - \nabla H(u_{n,R} + r\varphi \n_{n,R})\cdot \n_{n,R}(t_1)|\leq C\frac{|t_2-t_1|^\alpha}{n^{\delta(\gamma +\min\{1,\beta\})}}.
\end{equation}
Therefore, thanks to \eqref{passagg21}, we infer that
\begin{equation}\label{passag3000}
\left[\int_0^1 (\nabla H(u_{n,R} + r\varphi\n_{n,R}) - \nabla H(u_{n,R}))\cdot \n_{n,R}\de r\right]_{\alpha} \leq C\frac{1}{n^{\delta(\gamma + \min\{1, \beta\})}}.
\end{equation}
Hence \eqref{clam1} follows from \eqref{pass13} and \eqref{passag3000}

In order to prove \eqref{cla25} we notice that, arguing as in the proof of Step \ref{nablaHest}, we get that
\begin{equation}\label{passag700}
\begin{aligned}
|\nabla H(u_{n,R} + r\varphi \n_{n,R})\cdot \n_{n,R}(t_2) &- \nabla H(u_{n,R} + r\varphi \n_{n,R})\cdot \n_{n,R}(t_1)|\\
 &\leq \left| \frac{\de}{\de t}\left( \nabla H(u_{n,R} + r \varphi \n_{n,R})\cdot \n_{n,R}\right)_{|\sigma t_2 + (1-\sigma)t_2} \right||t_2-t_1|^\alpha,
\end{aligned}
\end{equation}
for some $\sigma \in (0,1)$. Moreover, it holds
\[
\begin{aligned}
&\frac{\de}{\de t}\left( \nabla H(u_{n,R} + r \varphi \n_{n,R})\cdot \n_{n,R}\right) \\
&= \left( h''(|u_{n,R} + r \varphi \n_{n,R}|) - \frac{h'(|u_{n,R} + r \varphi \n_{n,R}|)}{|u_{n,R} + r \varphi \n_{n,R}|}\right)\frac{(u_{n,R}\cdot \n_{n,R} + r \varphi )}{|u_{n,R} + r \varphi \n_{n,R}|} \frac{\de}{\de t}|u_{n,R}+ r \varphi \n_{n,R}|\\
&+  \frac{h'(|u_{n,R} + r \varphi \n_{n,R}|)}{|u_{n,R} + r \varphi \n_{n,R}|} (u_{n,R}\cdot \dot \n_{n,R} + r \varphi'),
\end{aligned}
\]
which, together with \eqref{H2}, \eqref{hprimeest} and \eqref{bigu}, implies that
\[
\begin{aligned}
&\left|\frac{\de}{\de t}\left( \nabla H(u_{n,R} + r \varphi \n_{n,R})\cdot \n_{n,R}\right) \right|\\
 &\leq \frac{C}{n^{\delta(\gamma+1 + \min\{1,\beta\})}}\left( |\dot u_{n,R}| +  |\varphi ||\dot \n_{n,R}| + |\varphi'|+ |u_{n,R}\cdot \dot \n_{n,R}|\right), \quad \forall t\in [0,2\pi].
\end{aligned}
\]
Arguing as in \eqref{passag6000} we get that
\[
|\dot u_{n,R}| +  |\varphi ||\dot \n_{n,R}| + |\varphi'|+ |u_{n,R}\cdot \dot \n_{n,R}| \leq C\left(1 + \frac{1}{n^{\tilde \gamma}} + n^\delta\right), \quad \forall t \in [0,2\pi], 
\]
which implies
\[
\left|\frac{\de}{\de t}\left( \nabla H(u_{n,R} + r \varphi \n_{n,R})\cdot \n_{n,R}\right) \right| \leq C\frac{1+n^\delta}{n^{\delta(\gamma+1 + \min\{1, \beta\})}}\leq C\frac{1}{n^{\delta(\gamma+\min\{1, \beta\})}}, \quad \forall t \in [0,2\pi].
\]
This, together with \eqref{passag700}, gives \eqref{cla25}. Thus the proof of \eqref{res2est1} is complete.

As for \eqref{res2est2}, we notice that
\[
\Res^2_{n,R}(\varphi_2) - \Res^2_{n,R}(\varphi_1) = (\varphi_2-\varphi_1)\int_0^1 (\nabla H( u_{n,R} + (r\varphi_2 + (1-r)\varphi_1)\n_{n,R})- \nabla (u_{n,R}))\cdot \n_{n,R}\de r.
\]
Hence the conclusion follows arguing exactly as in the previous part of the proof. We limit ourselves to notice that, with the same choice of $n_1(M)$ as in \eqref{overbarn}, for $n\geq n_1(M)$, $R \in S_n$ and $\varphi_1, \varphi_2 \in B_{n;M}$ it holds that 
\[
|u_{n,R} + (r\varphi_2 + (1-r)\varphi_1)\n_{n,R}| \geq n^\delta\left(a - \frac{1}{n^\delta} - \frac{M}{n^{\delta + \tilde \gamma}}\right) \geq C n^\delta >1 + s_0.
\]

\begin{cla}
For every $M >0$ there exists $n_2(M) \geq \underline n$ such for every $n \geq n_2(M)$, $R \in S_n$ and $\varphi, \varphi_1, \varphi_2 \in B_{n;M}$ hold
\begin{align}
&\|\Res^1_{n,R}(\varphi)\|_{0,\alpha}\leq C\frac{1}{n^{\tilde \gamma}}\|\varphi\|_{2,\alpha}, \label{res1est1} \\
&\|\Res^1_{n,R}(\varphi_1) - \Res^1_{n,R}(\varphi_2)\|_{0,\alpha} \leq C \frac{1}{n^{\tilde \gamma}}\|\varphi_1- \varphi_2\|_{2,\alpha}. \label{res1est2}
\end{align}
\end{cla}
Let $n_2(M)$ be the same as in Lemma \ref{appresuno}. Then for every $n \geq n_2(M)$, $R \in {S_n}$ and $\varphi \in B_{n;M}$ we have an explicit formulation of $\Res^1_{n,R}(\varphi)$ in terms of Taylor expansions. As one can see, $\Res^1_{n,R}(\varphi)$ contains $\varphi$ and its derivatives only up to the second order, and does not contains constant terms, or terms which are linear with respect to $\varphi$ and its derivatives. Through explicit computations (see \eqref{stimeupic}), we infer that the coefficients, which depend on the derivatives of $u_{n,R}$ up to the third order, are bounded in the $C^{0,\alpha}_{2\pi}$ norm by a constant. Then we get that
\[
\|\Res^1_{n,R}(\varphi)\|_{0,\alpha} \leq C\|\varphi\|_{2,\alpha}^2.
\]
Since $\varphi \in B_{n; M}$, we readily obtain \eqref{res1est1}.

As for the second estimate, arguing in a similar way we infer that, whenever $n\geq n_2(M)$, $R \in S_n$ and $\varphi_1, \varphi_2 \in B_{n;M}$, it holds
\[
\|\Res^1_{n,R}(\varphi_1) - \Res^1_{n,R}(\varphi_2)\|_{0,\alpha} \leq C (\|\varphi_1\|_{2,\alpha} + \|\varphi_2\|_{2,\alpha})\|\varphi_1- \varphi_2\|_{2,\alpha}.
\]
thus getting \eqref{res1est2}.
\vspace{3mm}
\begin{cla}
Conclusion.
\end{cla}
Let be 
\[
\overline M := 2M_1 M_2,
\]
where $M_1, M_2$ are as in \eqref{M1} and \eqref{M2}, respectively. Moreover, let $\overline n$ being such that
\[
\overline n \geq \max\{\underline n, n_1(\overline M), n_2(\overline M)\},
\]
where $n_1(\overline M), n_2(\overline M)$ are defined in Step 4 and Step 5. With such choices, Step 2-5 hold true at the same time for every $n\geq \overline n$, $R \in S_n$ and $\varphi \in B_{n,\overline M}$. Then, using also \eqref{M1} and Step 1, we infer
\[
\|\Q_{n,R}(\varphi)\|_{2,\alpha} \leq \frac{\overline M}{2}\frac{1}{n^{\tilde \gamma}} + C\left(\frac{1}{n^{1-\delta}} + \frac{1}{n^{\delta(\gamma +\min\{1, \beta\})}}+ \frac{1}{n^{\tilde \gamma}} \right)\|\varphi\|_{2,\alpha}.
\]
Taking if necessary a bigger $\overline n$ we obtain \eqref{internfin}.

In a similar way, thanks to \eqref{Lip} and Step 2-5 we get that
\[
\|\Q_{n,R}(\varphi_1) - \Q_{n,R}(\varphi_1)\|_{2,\alpha} \leq C\left(\frac{1}{n^{1-\delta}} + \frac{1}{n^{\delta(\gamma +\min\{1, \beta\})}} + \frac{1}{n^{\tilde \gamma}}\right)\|\varphi_1 - \varphi_2\|_{2,\alpha}. 
\]
Therefore, increasing once again $\overline n$ if necessary, also \eqref{Lipfin} is satisfied and the proof of the Lemma is complete.

\end{proof}

\end{subsection}
\begin{subsection}{The variational argument}\label{var}

As seen in Section \ref{redux}, for every choice of $a, b, \delta$ in \eqref{Rset} there exists $\overline n \in \N$ and $\overline M >0 $ such that for every $n \geq \overline n$ and $R \in S_n$ there exists $\varphi_{n, R} \in X^\perp$ such that
\[
\begin{cases}
\|\varphi_{n,R}\|_{2,\alpha} \leq \overline M \frac{1}{n^{\tilde \gamma}}, \\
\K(u_{n,R} + \varphi_{n,R}\n_{n,R}) - H(u_{n,R} + \varphi_{n,R}\n_{n,R}) = \sum_{i=1,2}\lambda_i \omega_i & \text{in }\R, 
\end{cases}
\]
where $\omega_i$ are as in \eqref{omegaker}, and the Lagrange multiplier $\lambda_i = \lambda_i(\varphi_{n,R})$ are given by
\[
\lambda_i = \frac{1}{\pi}\int_0^{2\pi}(\K(u_{n,R} + \varphi_{n,R} \n_{n,R})- H(u_{n,R} + \varphi_{n,R} \n_{n,R}))\omega_i \de t. 
\]

The aim in this Section is to prove that, choosing suitably $a, b$ and $\delta$, when $n$ is large enough we can find $R_n \in S_n$ such that $\lambda_1(\varphi_{n,R_n}) = 0$. To be more precise, the following Proposition holds.
 
\begin{prop}\label{firstgone}
Let be $\delta = \frac{1}{\gamma +2}$, $a = r_0^{\frac{1}{\gamma +2}}$, $b = r_1^{\frac{1}{\gamma+2}}$, so that $R \in S_n$ can be written as $R = (rn)^{\frac{1}{\gamma +2}}$ with $r \in [r_0, r_1]$.
Set
\[
\tilde A := 
\begin{cases}
A & \beta >0\\
A+B & \beta =0,
\end{cases}
\]
and let $r_0, r_1$ be such that
\[
2^{\frac{\gamma+2}{2}}r_0 < \frac{\tilde A\gamma}{2} \quad \text{and} \quad \frac{\tilde A\gamma}{2}< r_1,
\]
holds. 
Then there exist $\tilde n \geq \overline n$ and $\overline M>0$ such that for every $n\geq \tilde n$ there exists $R_n= (r_n n)^{\frac{1}{\gamma+2}} \in S_n$ with $r_n \in [r_0, r_1]$, and $\varphi_{n} \in C^{2,\alpha}_{2\pi}$ which satisfies
\begin{equation}\label{firstlmgone}
\begin{aligned}
&\|\varphi_{n}\|_{2,\alpha}\leq \overline M n^{-\frac{\gamma}{\gamma+2}},\\
&\K(u_{n} + \varphi_{n} \n_{n}) - H(u_{n} + \varphi_{n} \n_{n}) = \lambda_2 \sin t,
\end{aligned}
\end{equation}
where $u_n := u_{n, R_n}$ and $\n_n := \n_{n, R_n}$.
\end{prop}

Proposition \ref{firstgone} is a consequence of a variational argument. 
We recall that the energy associated to the prescribed curvature problem $\E_H: W^{1,1}_{2\pi n}\to \R$ is defined in \eqref{energhia} as
\[
\E_H(U) = L(U) + \A_H(U) = \int_0^{2\pi n} |\dot U|\de t + \int_0^{2\pi n}Q_H(U) \cdot i \dot U\de t,
\]
while the length functional $L$ and the anisotropic area functional $\A_H$ are defined in \eqref{functionals}.

In this case we can explicitly provide the vector field $Q_H$: since assumption \eqref{H0} holds, we can define 
\begin{equation}\label{explicitfield}
Q_H(p) := \left(\frac{1}{|p|}\int_0^{|p|} h( s ) s \de s\right) \frac{p}{|p|}.
\end{equation}
A simple computations shows that \eqref{explicitfield} satisfies conditions \eqref{divergass}. Then Lemma \ref{Qfuncreg} applies and we infer
\begin{equation}\label{energydifferent}
\E'_H(U)[V] = \int_0^{2\pi n}\frac{\dot U}{|\dot U|}\cdot \dot V + H(U)i \dot U \cdot V \de t, \quad U \in (W^{1,1}_{2\pi n}\setminus \R^2),\ V \in W^{1,1}_{2\pi n}.
\end{equation}

Let now be $U, V \in C^{2, \alpha}_{2\pi n} \setminus \R^2$. Integrating by parts in \eqref{energydifferent} we readily get that
\[
\E'_H(U)[V] = \int_0^{2\pi n}\left( -\frac{\ddot U}{|\dot U|} + \frac{(\dot U \cdot \ddot U)}{|\dot U|^3}\dot U + H(U)i \dot U\right) \cdot V \de t. 
\]
Notice that it holds
\[
\E'_H(U)[\dot U] = 0 \quad \text{ and }\quad \E_H'(U) [i \dot U] = \int_0^{2 \pi n} (H(U) - \K(U))|\dot U|^2 \de t 
\]
Since $\dot U \cdot i \dot U = 0$ for every $t \in [0, 2\pi n]$, given $V \in C^{2,\alpha}_{2\pi n}$ it is always possible to decompose it as 
\[
V = \frac{V \cdot \dot U}{|\dot U|^2} \dot U + \frac{V \cdot i \dot U}{|\dot U|^2} i \dot U.
\]
Then we obtain that
\begin{equation}\label{energyformula}
\E'_H(U)[V] = \int_0^{2\pi n}(H(U) - \K(U))(V \cdot i \dot U) \de t, \quad  U \in (C^{2,\alpha}_{2\pi n}\setminus \R^2), V \in C^{2,\alpha}_{2\pi n}.
\end{equation}

\vspace{5mm}

Let $a, b, \delta$ be fixed and let $\overline n, \overline M$ and the family of fucntions $\{\varphi_{n,R}\}$ be as in Proposition \ref{reduced}. Thanks to \eqref{lambrappr} we have that $\lambda_1=0$ if and only if 
\begin{equation}\label{Kvarest}
\int_0^{2\pi}(\K(u_{n,R} + \varphi_{n,R} \n_{n,R})- 1)\cos t \de t = \int_0^{2\pi}(H(u_{n,R} + \varphi_{n,R} \n_{n,R})- 1)\cos t \de t
\end{equation}

\begin{proof}[Estimate of $\K$]
We decompose the left-hand side of \eqref{Kvarest} in the following way:
\begin{equation}\label{asterisco}
\begin{aligned}
\int_0^{2\pi}&(\K(u_{n,R} + \varphi_{n,R} \n_{n,R})- 1)\cos t \de t\\
&=\int_0^{2\pi}(\K(u_{n,R} + \varphi_{n,R} \n_{n,R})- \K(u_{n,R}))\cos t \de t \\
&+\int_0^{2\pi}(\K(u_{n,R}- 1)\left( \frac{R}{n} + \cos t\right) \de t\\
&- \frac{R}{n}\int_0^{2\pi}(\K(u_{n,R})- 1)\de t.
\end{aligned}
\end{equation}
As for the first term, recalling \eqref{resunodefini} we get that
\[
\begin{aligned}
\int_0^{2\pi}&(\K(u_{n,R} + \varphi_{n,R} \n_{n,R})- \K(u_{n,R}))\cos t \de t\\
 &= \int_0^{2\pi}\Le_{\infty}\varphi_{n,R} \cos t \de t +  \int_0^{2\pi}(\Le_{n,R} - \Le_\infty)\varphi_{n,R}\cos t \de t + \int_0^{2\pi}\Res^1_{n,R}(\varphi_{n,R})\cos t \de t.
\end{aligned}
\]
Integrating by parts twice we readily get that
\[
\int_0^{2\pi} \Le_\infty \varphi_{n,R} \cos t \de t = \int_0^{2\pi}\varphi_{n,R}\Le_\infty \cos t = 0, 
\]
since $\cos t \in \ker(\Le_\infty)$. Moreover, by estimate \eqref{Lnest} and \eqref{res1est1} and since $\varphi_{n,R}$ satisfy \eqref{estimphi}, we get that for every $n\geq \overline n$ and $R \in S_n$ it holds 
\[
\left|\int_0^{2\pi}(\K(u_{n,R} + \varphi_{n,R} \n_{n,R})- \K(u_{n,R}))\cos t \de t\right| \leq C\left(\frac{R}{n^{1+\tilde \gamma}} + \frac{1}{n^{2\tilde \gamma}}\right) \int_0^{2 \pi }|\cos t| \de t \leq C \frac{1}{n^{2\tilde \gamma}},
\]
by definition of $\tilde \gamma$. Then there exists a sequence of continuous functions $K^1_n(R): S_n \to \R$ uniformly bounded with respect to both $n$ and $R$ such that
\begin{equation}\label{est1}
\int_0^{2\pi}(\K(u_{n,R} + \varphi_{n,R} \n_{n,R})- \K(u_{n,R}))\cos t \de t =\frac{1}{n^{2\tilde \gamma}} K^1_n(R) .
\end{equation}

In a similar way, exploiting estimate \eqref{passag200} we get that
\[
\left|\int_0^{2\pi}(\K(u_{n,R})- 1)\de t\right| \leq C \frac{R}{n}\leq C\frac{1}{n^{1-\delta}},
\]
which implies that there exists another sequence $K^2_n(R): S_n \to \R$ being such that
\begin{equation}\label{est2}
\frac{R}{n}\int_0^{2\pi}(\K(u_{n,R})- 1)\de t = \frac{1}{n^{2(1-\delta)}}K^2_n(R).
\end{equation}

In order to estimate the remaining term we begin to notice that, thanks to \eqref{energyformula}, it holds that
\[
\frac{\partial }{\partial R}\E_H(U_{n, R}) = \E'_H(U_{n,R})\left[e^{i\frac{t}{n}}\right] = \int_0^{2\pi n}(\K(U_{n,R})- H(U_{n,R}))\left(\frac{R}{n} + \cos\left( \left(\frac{n-1}{n}\right)t\right)\right) \de t.
\]
When in particular $H \equiv 1$, and performing the change of variables $\frac{n-1}{n}t = s$, we obtain 
\[
\frac{\partial }{\partial R}\E_1(U_{n, R}) = \frac{n}{n-1}\int_0^{2\pi (n-1)}( \K(u_{n,R})- 1)\left(\frac{R}{n} + \cos s\right) \de s.
\]
Since, as seen in Subsection \ref{prel}, $ \K(u_{n,R})$ is a $2\pi$-periodic function, we infer that
\begin{equation}\label{radiusderiv}
\frac{1}{n}\frac{\partial }{\partial R}\E_1(U_{n, R}) = \int_0^{2\pi}(\K(u_{n,R})- 1)\left(\frac{R}{n} + \cos t\right) \de t.
\end{equation}

We claim that there exists a sequence of continuous function $K^3_n:S_n \to \R$, uniformly bounded with respect to $n$ and $R$, being such that
\begin{equation}\label{est3}
\frac{1}{n}\frac{\partial }{\partial R}\E_1(U_{n, R}) = - \frac{2\pi R}{n} + \frac{1}{n^{2-\delta}}K^3_n(R).
\end{equation}
Therefore, putting together \eqref{asterisco}, \eqref{est1}, \eqref{est2}, \eqref{radiusderiv}, and \eqref{est3} we get that
\begin{equation}\label{Kfinest}
\int_0^{2\pi}(\K(u_{n,R} + \varphi_{n,R} \n_{n,R})- 1)\cos t \de t = -\frac{2\pi R}{n} +  \frac{1}{n^{2\tilde \gamma}}K^1_n(R)+ \frac{1}{n^{2(1-\delta)}}K^2_n(R)+ \frac{1}{n^{2-\delta}}K^3_n(R).
\end{equation}

The following Lemma is devoted to the proof of \eqref{est3}. 

\begin{lemma}
Let be $n \geq \underline n$, where $\underline n$ is as in \eqref{nsotto}, and let be $R \in S_n$. There exists three sequences of continuous functions, $L^1_n, L^2_n, L^3_n : S_n \to \R$ which are uniformly bounded with respect to $n$ and $R$, such that
\[
\frac{1}{n}\frac{\partial }{\partial R}\E_1(U_{n,R}) = -\frac{2\pi R}{n} + \frac{2\pi R}{n^2}\left(1 + L^1_n(R) + \frac{R^2}{n^2}L^2_n(R) + \frac{R^4}{n^4}L^3_n(R)\right).
\]
\end{lemma}
\begin{proof}
Recall that, by definition, it holds $\E_1(U_{n,R}) = L(U_{n,R}) + \A(U_{n,R})$.
A simple computation shows that
\[
\frac{1}{n}\frac{\partial }{\partial R}\A(U_{n,R}) = - \frac{2\pi R}{n}.
\]
It remains to estimate the part of the energy associated to the length functional $L$. 
Since when $n \geq \underline n$ and $R \in S_n$ it holds that $R<n$, we have that $\left|\frac{2Rn}{R^2 + n^2}\cos s\right|<1$ for all $s \in [0,2\pi]$. As a consequence, the following Taylor expansion holds:
\[
\begin{aligned}
L(U_{n,R}) &=\int_0^{2\pi n}|\dot U|\de t =  n \int_0^{2\pi} \sqrt{1 + \frac{R^2}{n^2} + \frac{2R}{n}\cos s}\de s\\
&=n \left(1 + \frac{1}{2}\frac{R^2}{n^2} + \sum_{j=2}^\infty c_j \left(\frac{R^2}{n^2}\right)^j\right)\int_0^{2\pi}1 + \frac{Rn}{n^2 + R^2}\cos s + \sum_{k=2}^\infty c_k \left(\frac{2Rn}{n^2 + R^2} \right)^k (\cos s)^k \de s,
\end{aligned}
\]
where the second equality comes form the standard changes of variables $\frac{n-1}{n}t = s$, and $c_j, c_k$ are the coefficients of the expansion of $\sqrt{1 + x}$.
For the same reason, we can integrate the series term by term. Since it holds that
\begin{equation}\label{cosintcoeff}
\int_0^{2\pi} (\cos s)^k \de s = 
\begin{cases}
\frac{2 \pi }{2^k}\binom{k}{k/2} & k\text{ even}\\
0& k \text{ odd}
\end{cases}
\end{equation}
renaming $c'_k = \frac{c_k}{2\pi} \int_0^{2\pi}(\cos s)^k \de s$ we have the expression 
\begin{equation}\label{conc}
\begin{aligned}
L(U_{n,R}) &= 2\pi n  \left(1 +\frac{R^2}{2n^2} + \sum_{j=2}^\infty c_j \left(\frac{R^2}{n^2}\right)^j\right)\left(1 + \sum_{k=1}^\infty c'_{2k}  \left(\frac{2Rn}{n^2 + R^2} \right)^{2k} \right)\\
&=2\pi n \left( 1 + \frac{R^2}{2n^2}\right) + 2\pi n \left(\sum_{j=2}^\infty c_j \left(\frac{R^2}{n^2}\right)^j \right)+2\pi n\left( \left(1 + \frac{R^2}{2n^2}\right)\sum_{k=1}^\infty c'_{2k}  \left(\frac{2Rn}{n^2 + R^2} \right)^{2k}\right)\\
&+2\pi n \left( \sum_{k=1}^\infty c'_{2k}  \left(\frac{2Rn}{n^2 + R^2} \right)^{2k}\sum_{j=2}^\infty c_j \left(\frac{R^2}{n^2}\right)^j \right)
\end{aligned}
\end{equation}
We claim that there exist three sequences $L^1_n, L^2_n, L^3_n : S_n \to \R$ of continuous functions, uniformly bounded with respect to $n$ and $R$, being such that
\[
\frac{\partial }{\partial R}L(U_{n,R}) = 2\pi n \left( \frac{R}{n^2} + \frac{R}{n^2} L^1_n(R) + \frac{R^3}{n^4}L^2_n(R) + \frac{R^5}{n^6}L^3_n(R)\right),
\]
thus completing the proof of the Lemma. 
 
Indeed, this can be seen by standard computation. For the sake of completeness we show how to manage the second term in the right-hand side of \eqref{conc}.
Differentiating it term by term we get that
\[
\begin{aligned}
\frac{\partial }{\partial R} \left( \sum_{j=2}^\infty c_j \left(\frac{R^2}{n^2}\right)^j \right) = \frac{R^3}{n^4}\sum_{j=2}^\infty 2c_j j \left(\frac{R^2}{n^2}\right)^{j-2}, 
\end{aligned}
\]
and since 
\[
\left|\sum_{j=2}^\infty 2jc_j \left(\frac{R^2}{n^2}\right)^{j-2}\right| \leq C, 
\]
we get that there exists a sequence of continuous functions $L^2_n(R): S_n \to \R$, bounded uniformly with respect to $n$ and $R$ and such that
\[
\frac{\partial }{\partial R} \left( \sum_{j=2}^\infty c_j \left(\frac{R^2}{n^2}\right)^j \right) = \frac{R^3}{n^4}L^2_n(R).
\] 
The proof is complete. 
\end{proof}
\phantom \qedhere
\end{proof}

\begin{proof}[Estimate of $H$]
We turn our attention to the right-hand side of \eqref{Kvarest}. Thanks to \eqref{H0}, it can be decomposed as
\begin{equation}\label{Hdeco}
\begin{aligned}
\int_0^{2\pi}&(H(u_{n,R} + \varphi_{n,R} \n_{n,R})- 1)\cos t \de t \\
& = A\int_0^{2\pi}\frac{\cos t}{|u_{n,R}|^\gamma}\de t\\
& + A\int_0^{2\pi}\cos t \left(\frac{1}{|u_{n,R} + \varphi_{n,R}\n_{n,R}|^\gamma} - \frac{1}{|u_{n,R}|^\gamma}\right)\de t\\
& + \int_0^{2\pi}  \tilde h(|u_{n,R} + \varphi_{n,R}\n_{n,R}|) \left(\frac{1}{|u_{n,R} + \varphi_{n,R}\n_{n,R}|^{\gamma+\beta}} - \frac{1}{|u_{n,R}|^{\gamma+\beta}}\right)\cos t \de t\\
& + \int_0^{2\pi}  \frac{(\tilde h(|u_{n,R} + \varphi_{n,R}\n_{n,R}|)- \tilde h(|u_{n,R}|))}{|u_{n,R}|^{\gamma+\beta}} \cos t\de t\\
& + \int_0^{2\pi}  \frac{\tilde h(|u_{n,R}|)}{|u_{n,R}|^{\gamma+\beta}}\cos t \de t
\end{aligned}
\end{equation}

Since $n \geq \overline n$, $R \in S_n$ and $\varphi_{n,R} \in B_{n; \overline M}$, we have that \eqref{bigu} holds. As a consequence, by means of the mean value theorem we infer that for every $\tau >0$ it holds
\[
\begin{aligned}
||u_{n,R} + \varphi_{n,R}\n_{n,R}|^{-\tau} - |u_{n,R}|^{-\tau}|=\left| \frac{\de}{\de s}|u_{n,R} + s \varphi_{n,R}\n_{n,R}|^{-\tau}\right|\\
\leq \tau |u_{n,R} + s \varphi_{n,R}\n_{n,R}|^{-\tau-1}\|\varphi\|_{2,\alpha} \leq C\frac{1}{n^{\delta(\tau +1)+ \tilde \gamma}}.
\end{aligned}
\]
Therefore, since by \eqref{H2} we have that, independently from $\beta$, $\tilde h$ is bounded, we get that there exists a sequence of functions $\hat H^1_n:S_n \to \R$ continuous and bounded uniformly with respect to $n$ and $R$ being such that
\begin{equation}\label{H1dec}
\begin{aligned}
&A\int_0^{2\pi}\cos t \left(\frac{1}{|u_{n,R} + \varphi_{n,R}\n_{n,R}|^\gamma} - \frac{1}{|u_{n,R}|^\gamma}\right)\de t\\
& + \int_0^{2\pi}  \tilde h(|u_{n,R} + \varphi_{n,R}\n_{n,R}|)\cos t \left(\frac{1}{|u_{n,R} + \varphi_{n,R}\n_{n,R}|^{\gamma+\beta}} - \frac{1}{|u_{n,R}|^{\gamma+\beta}}\right) \de t = \frac{1}{n^{\delta(\gamma+1) + \tilde \gamma}}\hat H^1_n(R).
\end{aligned}
\end{equation}

Since for $n\geq \overline n$ and $R \in S_n$ it holds that $R >1$, we have the following equality
\[
|u_{n,R}|^{-\gamma} = \frac{1}{(R^2+1)^{\frac{\gamma}{2}}}\left( 1 - \frac{\gamma R}{R^2+1}\cos t + \sum_{k=2}^\infty c_k\left(\frac{2R}{R^2+1}\right)^k(\cos t)^k\right),
\] 
where the terms $c_k$ are the coefficients of the Taylor expansion of $(1 + x)^{-\frac{\gamma}{2}}$.
Then, integrating term by term and recalling \eqref{cosintcoeff}, we obtain that there exists a sequence of continuous functions $\hat H^2_n:S_n \to \R$, bounded uniformly with respect to $n$ and $R$, being such that
\begin{equation}\label{H2dec}
\begin{aligned}
A&\int_0^{2\pi}\frac{\cos t}{|u_{n,R}|^\gamma}\de t \\
&=\frac{A}{(R^2+1)^{\frac{\gamma}{2}}}\left(- \frac{\gamma\pi R}{R^2+1} + \sum_{k=2}^\infty c_k\left(\frac{2R}{R^2+1}\right)^k\int_0^{2\pi}(\cos t)^{k+1}\de t\right)\\
&=-\frac{A\gamma\pi R}{(R^2+1)^{\frac{\gamma}{2}+1}} + \frac{16\pi AR^3}{(R^2+1)^{\frac{\gamma}{2}+3}}\sum_{k=2}^\infty \frac{c_{2k-1}}{2^{k}}\binom{2k}{k}\left(\frac{2R}{R^2+1}\right)^{2k-4}\\
&= -\frac{A\gamma\pi R}{(R^2+1)^{\frac{\gamma}{2}+1}} + \frac{1}{n^{\delta(\gamma+3)}}\hat H_n^2(R).
\end{aligned}
\end{equation}

Let us turn our attention to the fourth term in the right-hand side of \eqref{Hdeco}: exploiting the definition of $H$, we get
\[
\nabla H(p) = \left( \frac{-A\gamma}{|p|^{\gamma+1}} + \frac{\tilde h'(|p|)}{|p|^{\gamma + \beta}}- (\gamma + \beta) \frac{\tilde h(|p|)}{|p|^{\gamma + \beta +1}}\right) \frac{p}{|p|};
\]
then, by \eqref{H2}together with \eqref{hprimeest} we obtain that
\[
|\tilde h'(|p|)| \leq \frac{C}{|p|^{-\beta + \min \{1,\beta\}}}, \quad \text{ when }|p|>1.
\]
As a consequence we have that
\[
\begin{aligned}
|\tilde h(|u_{n,R} +& \varphi_{n,R}\n_{n,R}|) - \tilde h(|u_{n,R}|)|= \left| \frac{\de}{\de \sigma} \tilde h(|u_{n,R} + \sigma \varphi_{n,R}\n_{n,R}|)\right|\\
 &\leq |\tilde h'(|u_{n,R} + \sigma \varphi_{n,R}\n_{n,R}|)|\frac{\|\varphi_{n,R}\|_{2,\alpha}}{|u_{n,R} + \sigma\varphi_{n,R}\n_{n,R}|} \leq C \frac{\|\varphi_{n,R}\|_{2,\alpha}}{|u_{n,R} + \sigma\varphi_{n,R}\n_{n,R}|^{1-\beta + \min \{1, \beta\}}},
\end{aligned}
\]
for any $\sigma \in [0,1]$. 
Therefore, using \eqref{passag17}, \eqref{bigu} and since $\varphi_{n,R}\in B_{n; \overline M}$, we get that
\[
\left|\frac{(\tilde h(|u_{n,R} + \varphi_{n,R}\n_{n,R}|- \tilde h(|u_{n,R}|)}{|u_{n,R}|^{\gamma + \beta}}\right| \leq  C \frac{1}{n^{\delta(\gamma + 1 + \min\{1, \beta\}) + \tilde \gamma}}.
\]
Then there exists a further sequence of continuous functions $\hat H^3_n:S_n \to \R$, bounded uniformly with respect to $n$ and $R$, being such that
\begin{equation}\label{H4dec}
\int_0^{2\pi}\frac{(\tilde h(|u_{n,R} + \varphi_{n,R}\n_{n,R}|)- \tilde h(|u_{n,R}|))\cos t}{|u_{n,R}|^{\gamma + \beta}} \de t = \frac{1}{n^{\delta(\gamma+1 + \min\{1, \beta\}) + \tilde \gamma}}\hat H^3_n(R).
\end{equation}

Let define $\hat \gamma := \min \{ \tilde \gamma, 2 \delta\}$.
From \eqref{Hdeco}, \eqref{H1dec}, \eqref{H2dec} and \eqref{H4dec} we can conclude that there exists another sequence of continuous functions $\hat H_n:S_n \to \R$, bounded uniformly with respect to $n$ and $R$, being such that
\begin{equation}\label{Hfinest}
\begin{aligned}
\int_0^{2\pi}&(H(u_{n,R} + \varphi_{n,R} \n_{n,R})- 1)\cos t \de t \\
&=-\frac{A\gamma\pi R}{(R^2+1)^{\frac{\gamma}{2}+1}} + \int_0^{2\pi}\frac{\tilde h(|u_{n,R}|)}{|u_{n,R}|^{\gamma + \beta}}\cos t \de t  + \frac{1}{n^{\delta(\gamma+1)+\hat \gamma}}\hat H_n(R).
\end{aligned}
\end{equation}
\phantom \qedhere
\end{proof}

\begin{proof}[Proof of Proposition \ref{firstlmgone}]
Set $\delta = \frac{1}{\gamma +2}$, so that $\tilde \gamma = \frac{\gamma}{\gamma+2}$. Multiplying both sides of \eqref{Kvarest} by $\frac{n}{R}$ and using \eqref{Kfinest} and \eqref{Hfinest} we get that $\lambda_1=0$ if and only if
\begin{equation}\label{semifinal}
\begin{aligned}
-2\pi &+  \frac{1}{R n^{\frac{\gamma-2}{\gamma+2}}}K^1_n(R)+ \frac{1}{Rn^{\frac{\gamma}{\gamma+2}}}K^2_n(R)+ \frac{1}{Rn^\frac{\gamma+1}{\gamma+2}}K^3_n(R) \\
&= -\frac{A\gamma\pi n}{(R^2+1)^{\frac{\gamma}{2}+1}} +  \frac{n}{R} \int_0^{2\pi}\frac{\tilde h(|u_{n,R}|)}{|u_{n,R}|^{\gamma + \beta}}\cos t \de t  + \frac{1}{Rn^{\frac{-1}{\gamma+2}+\hat \gamma}}\hat H_n(R).
\end{aligned}
\end{equation}

Notice that, since $R \in S_n$, the first terms in the right-hand side of \eqref{semifinal} is not infinitesimal with respect to $n$.
On the other hand, using again that $R\in S_n$ and that $\gamma >1$, we get that there exists $\varepsilon >0$ and a sequence of continuous functions $F^1_n:S_n \to \R$ uniformly bounded with respect to $n$ and $R$ being such that
\begin{equation}\label{F1semif}
 \frac{1}{R n^{\frac{\gamma-2}{\gamma+2}}}K^1_n(R)+ \frac{1}{Rn^{\frac{\gamma}{\gamma+2}}}K^2_n(R)+ \frac{1}{Rn^\frac{\gamma+1}{\gamma+2}}K^3_n(R) - \frac{1}{Rn^{\frac{-1}{\gamma+2}+\hat \gamma}}\hat H_n(R)= \frac{1}{n^\varepsilon}F^1_n(R).
\end{equation}

As for the second term in the right-hand side of \eqref{semifinal}, it has a different behavior depending on the value of $\beta$.
When $\beta >1$, since by \eqref{H1} $\tilde h$ is bounded, we readily get that
\[
\left| \frac{n}{R}\int_0^{2\pi}\frac{\tilde h(|u_{n,R}|)}{|u_{n,R}|^{\gamma+\beta}}\cos t \de t\right| \leq \frac{C}{n^{\frac{\beta-1}{\gamma+2}}},
\]
then it converges uniformly to zero as $n \to + \infty$.
Let now be $0 \leq \beta \leq 1$. Again by \eqref{H1} we have that $\tilde h$ can be written in the form 
\[
\tilde h (s) = B + \hat h(s) \quad \text{ with }\quad \lim_{s \to +\infty}\frac{\hat h(s)}{ s^{\beta-1}} = 0.
\] 
Then, using that $R \in S_n$ we get that
\[
\begin{aligned}
\left| \frac{n}{R}\int_0^{2\pi}\frac{\tilde h(|u_{n,R}|)}{|u_{n,R}|^{\gamma+\beta}}\cos t \de t\right| \leq \left|\frac{n}{R}\int_0^{2\pi}\frac{ B}{|u_{n,R}|^{\gamma +\beta}}\cos t \de t \right|+ C\int_0^{2\pi}\frac{|\hat h(|u_{n,R}|)|}{|u_{n,R}|^{\beta-1}}\de t.
\end{aligned}
\]
Thanks to the continuity of $u_{n,R}$ and $\tilde h$, we can find a sequence $(t_n) \subset [0, 2\pi)$ such that $|\hat h(|u_{n,R}(t_n)|)| = |\hat h(|u_{n,R}|)|_\infty$. Using that $R \in S_n$ this implies that
\[
\int_0^{2\pi}\frac{|\hat h(|u_{n,R}|)|}{|u_{n,R}|^{\beta-1}}\de t \leq C \frac{|\hat h(|u_{n,R}(t_n)|)|}{|u_{n,R}(t_n)|^{\beta-1}}\frac{|u_{n,R}(t_n)|^{\beta-1}}{|u_{n,R}|_\infty^{\beta-1}}\leq C \frac{|\hat h(|u_{n,R}(t_n)|)|}{|u_{n,R}(t_n)|} \to 0 \quad \text{ as }n \to +\infty.
\]
As for the remaining term, arguing as in \eqref{H2dec} we get that
\[
\frac{n}{R}\int_0^{2\pi}\frac{ B}{|u_{n,R}|^{\gamma +\beta}}\cos t \de t = 
-\frac{B\gamma\pi n}{(R^2+1)^{\frac{\gamma + \beta}{2}+1}} + \frac{16\pi B nR^2}{(R^2+1)^{\frac{\gamma+\beta}{2}+3}}\sum_{k=2}^\infty \frac{c_{2k-1}}{2^{k}}\binom{2k}{k}\left(\frac{2R}{R^2+1}\right)^{2k-4}.
\]
where $c_k$ are the coefficients of the Taylor expansion of $(1+x)^{-\frac{\gamma+\beta}{2}}$.
While the second term in the right-hand side converges uniformly to zero for any $\beta \in [0,1]$, the first term is negligible only if $\beta \neq 0$. 
As a consequence of the previous discussion we infer that there exist two sequences of continuous functions, which we both denote as $F_n^2$ with a slight abuse of notation, such that they converge uniformly to zero as $n \to +\infty$ and
\begin{equation}\label{F2semif}
 \frac{n}{R}\int_0^{2\pi}\frac{\tilde h(|u_{n,R}|)}{|u_{n,R}|^{\gamma+\beta}}\cos t \de t = 
\begin{cases}
-F_n^2(R) & \text{ if }\beta >0\\
-\frac{B\gamma\pi n}{(R^2+1)^{\frac{\gamma}{2}+1}} - F_n^2(R) & \text{ if }\beta =0.
\end{cases}
\end{equation}

Setting 
\[
\tilde A = 
\begin{cases}
A & \text{ if }\beta>0\\
A+B & \text{ if }\beta = 0,
\end{cases}
\]
and putting together \eqref{semifinal}, \eqref{F1semif} and \eqref{F2semif} we get that $\lambda_1 =0$ if and only if there exists $R \in S_n$ such that
\begin{equation}\label{final}
-\frac{\tilde A\gamma\pi n}{(R^2+1)^{\frac{\gamma}{2}+1}} +2\pi =\frac{1}{n^{\varepsilon}}F^1_n(R) +F^2_n(R).
\end{equation}
Let be $R = (r n)^\delta$ with $r \in [r_0, r_1]$, where $r_0$ and $r_1$ are chosen in such a way that
\[
2^{\frac{\gamma+2}{2}}r_0 < \frac{\tilde A\gamma}{2} \quad \text{and} \quad \frac{\tilde A\gamma}{2}< r_1
\]
are satisfied, and define 
\[
f_n(r) := -\frac{\tilde A\gamma\pi n}{((rn)^{\frac{2}{\gamma+2}}+1)^{\frac{\gamma+2}{2}}} + 2\pi.
\]

Let $\hat n$ be such that $\hat n \geq \overline n$, where $\overline n$ is the value given by Proposition \ref{reduced}, and such that
$\hat n r_0 >1$. 
Then it holds that
\[
((r_0\hat n)^{\frac{2}{\gamma+2}} +1)^{\frac{\gamma+2}{2}} = r_0\hat n\left( 1 + \frac{1}{(r_0\hat n)^{\frac{2}{\gamma+2}}}\right)^{\frac{\gamma+2}{2}}< \hat n 2^{\frac{\gamma+2}{2}}r_0 < \frac{\tilde A \gamma \hat n}{2}.
\]
Since a simple computation shows that for every $\hat r \in [r_0, r_1]$ fixed, the sequence $f_n(\hat r)$ is decreasing, we infer that $f_n(r_0) \leq f_{\hat n}( r_0) =: m_- <0$ for every $n\geq \hat n$.  

On the other hand we have that 
\[
f_n (r_1) \geq -\frac{\tilde A\gamma \pi}{r_1} + 2\pi =: m_+ >0, \quad \forall n \geq \hat n.
\]
We stress that neither $m_-$ nor $m_+$ depend on $n$. Therefore, since the sequences of functions $(\frac{1}{n^{\varepsilon}}F_n^1)$ and $(F^2_n)$ both converge to zero uniformly with respect to $n$, taking a bigger $\hat n$ if necessary we obtain
\[
m_- < \frac{1}{n^{\varepsilon}}F_n^1((rn)^{\frac{2}{\gamma+2}}) + F^2_n((rn)^{\frac{2}{\gamma+2}})) < m_+, \quad \forall r \in [r_0, r_1] , \forall n\geq \hat n.
\]
As a consequence, by the intermediate value theorem we get that for every $n \geq \hat n$ there exists $r_n \in (r_0, r_1)$ such that \eqref{final} is satisfied, thus concluding the proof of the Proposition.

\end{proof}
\end{subsection}
\begin{subsection}{The rotational invariance}\label{rota}
Let $\hat n$, $\overline M$, $(R_n)$ and $(\varphi_n)$ be as in Proposition \ref{firstgone}.
In this section we prove that there exists $\tilde n \geq \hat n$ such that also the second Lagrange multiplier $\lambda_2$ associated to $\varphi_{n}$ is null, thus concluding the proof of Theorem \ref{mainteo}.

In order to do that we will exploit the rotational invariance of the energy $\E_H$. Given $\theta \in [0, 2\pi)$ we denote a counterclockwise rotation of a vector $v \in \R^2$ as $e^{i \theta}v = (\cos \theta, \sin \theta) \cdot v$.
Recall also that $e^{i\theta }u \cdot e^{i\theta}v = u \cdot v$.

It is immediate to see that the function $Q_H$ defined in \eqref{explicitfield} satisfy $e^{-i\theta}Q_H(e^{i\theta}p) = Q_H(p)$.
Then a simple computation shows that for every function $Y \in C^{2,\alpha}(\R / 2\pi n; \R^2)$ the energy associated to the prescribed curvature problem, defined in \eqref{energhia}, satisfies
\[
\E_H(Y) = \E_H(e^{i\theta}Y) \quad  \text{ for every } \theta \in [0, 2\pi),
\]
i.e. it is rotationally invariant. 

As a consequence, taking \eqref{energyformula} into account, we obtain that
\[
\begin{aligned}
0 &= \frac{\de}{\de \theta}\E_H(e^{i\theta}Y)\\ 
&= \E'_H(e^{i\theta}Y)\left[\frac{\de}{\de \theta}e^{i\theta}Y\right] = \int_0^{2 \pi n}(H(e^{i\theta}Y) - \K(e^{i\theta}Y))\frac{\de}{\de \theta}\left(e^{i\theta}Y\right) \cdot i \frac{\de}{\de t}\left(e^{i\theta}Y\right) \de t.
\end{aligned}
\]
Since both $H$ and $\K$ are rotationally invariant, we infer that for every $Y \in C^{2,\alpha}(\R /2\pi n; \R^2)$ it holds that
\[
0 = \int_0^{2\pi n}(H(Y) - \K(Y)) i e^{i\theta}Y \cdot i e^{i\theta}\dot Y \de t = 
\int_0^{2\pi n}(H(Y) - \K(Y)) Y \cdot \dot Y \de t. 
\]

Taking $Y = U_{n,R_n} + \varphi_n\left(\frac{n-1}{n}t\right) \Ne_{n,R_n}$, and performing the usual change of variables we obtain 
\begin{equation}\label{pass96}
0 = n \int_0^{2\pi}(H(u_{n} + \varphi_n \n_{n}) - \K(u_{n} + \varphi_n \n_{n})) (u_{n} + \varphi_n \n_{n}) \cdot (\dot u_{n} + \varphi'_n \n_{n} + \varphi_n \dot \n_{n}) \de t,
\end{equation}
where $u_n := u_{n, R_n}$ and $\n_n := \n_{n, R_n}$.
On one hand, by \eqref{firstlmgone} we have that
\begin{equation}\label{pass97}
H(u_{n} + \varphi_n \n_{n}) - \K(u_{n} + \varphi_n \n_{n}) = \lambda_2 \sin t,
\end{equation}
while by direct computation we get that
\begin{equation}\label{pass98}
\begin{aligned}
(u_{n} + \varphi_n \n_{n}) \cdot (\dot u_{n} + \varphi'_n \n_{n} + \varphi_n \dot \n_{n}) = - (r_nn)^{\frac{1}{\gamma+2}} \sin t + \varphi'_n u_{n} \cdot \n_{n} + \varphi_n u_{n} \cdot \dot \n_{n} + \varphi_n \varphi'_n.
\end{aligned}
\end{equation}
Putting together \eqref{pass96}, \eqref{pass97} and \eqref{pass98} we obtain
\[
\lambda_2 \int_0^{2\pi}  (r_nn)^{\frac{1}{\gamma+2}} (\sin t)^2 \de t = \lambda_2 \left( \int_0^{2\pi}(\varphi_n' u_{n} \cdot \n_{n} + \varphi_n u_{n} \cdot \dot \n_{n})\sin t \de t + \int_0^{2\pi}\varphi_n \varphi_n' \sin t \de t\right)
\]
Through an integration by parts we infer that
\[
\int_0^{2\pi}(\varphi_n' u_{n} \cdot \n_{n} + \varphi_n u_{n} \cdot \dot \n_{n})\sin t \de t =\int_0^{2\pi } (u_{n} \sin t) \cdot \frac{\de}{\de t}(\varphi_n \n_{n})\de t = \int_0^{2 \pi }\varphi_n (u_{n} \cdot \n_{n}) \cos t\de t,
\]
hence we get that 
\[
\lambda_2 \pi = \lambda_2 \frac{1}{(r_nn)^{\frac{1}{\gamma+2}}}\left(\int_0^{2 \pi }\varphi_n (u_{n} \cdot \n_{n}) \cos t\de t + \int_0^{2\pi} \varphi_n \varphi'_n \sin t \de t\right)
\]
Since $n\geq \hat n$, $R_n \in S_n$ and $\varphi_n \in B_{n; \overline M}$, it holds that
\[
\left|\frac{1}{(r_nn)^{\frac{1}{\gamma+2}}}\left(\int_0^{2 \pi }\varphi_n (u_{n} \cdot \n_{n}) \cos t\de t + \int_0^{2\pi} \varphi_n \varphi'_n \sin t \de t\right)\right| \leq C\left(\frac{ 1}{n^{\frac{\gamma}{\gamma+2}}} + \frac{1}{n^{\frac{2\gamma+1}{\gamma+2}}}\right)\leq C\frac{1}{n^{\frac{\gamma}{\gamma+2}}},
\]
then there exists a bounded sequence $(T_n) \subset \R$ such that
\[
\pi \lambda_2 = \lambda_2 \frac{T_n}{n^{\frac{\gamma}{\gamma+2}}}.
\]
Taking $\tilde n \geq \hat n$ such that it is satisfied 
\[
\left| \frac{T_n}{n^{\frac{\gamma}{\gamma+2}}}\right|< \pi \quad \forall n\geq \tilde n,
\]
we readily get that $\lambda_2 = 0$, otherwise we will reach a contradiction. This concludes the proof Theorem \ref{mainteo}. 
\end{subsection}

\section*{Appendix: Estimates and useful formulae}
\renewcommand\thesection{\Alph{section}}
\renewcommand\theequation{\thesection.\arabic{equation}}

\setcounter{section}{1}
\setcounter{subsection}{0}
\setcounter{equation}{0}

\addcontentsline{toc}{subsection}{Appendix: Estimates and useful formulae}
Here we collect useful estimates and equalities, involving the function $u_{n,R}$ and its derivatives, that are frequently used through all the present work. Since all of them are derived through basic computations, they are presented here without proof. As a byproduct, we derive Lemma \ref{lineopcoverg} and Lemma \ref{appresuno}, which show, respectively, the relation between the operators $\Le_{n,R}$ and $\Le_\infty$, and the relation between $\varphi \mapsto \K(u_{n,R} + \varphi \n_{n,R})$ and $\Le_{n,R}$. 

\setcounter{teo}{0}

We recall the definitions
\[
u_{n,R}(t) = R e^{i\frac{t}{n-1}} +e^{i \frac{nt}{n-1}} \quad \text{ and }\quad \n_{n,R} = \frac{i \dot u_{n,R}}{|u_{n,R}|}.
\]
The following holds:
\[
\begin{aligned}
\dot \n_{n,R} &= \frac{i \ddot u_{n,R} }{|\dot u_{n,R}|} -  \frac{( \dot u_{n,R} \cdot \ddot u_{n,R})  i \dot u_{n,R}}{|\dot u_{n,R}|^3}\\
\ddot \n_{n,R} &= \frac{i \dddot u_{n,R}}{|\dot u_{n,R}|} -  \frac{( \dot u_{n,R} \cdot \ddot u_{n,R})  i \ddot u_{n,R}}{|\dot u_{n,R}|^3}\\
 &-\frac{|\ddot u_{n,R}|^2 i \dot u_{n,R} + (\dot u_{n,R} \cdot \dddot u_{n,R})i \dot u_{n,R} + (\dot u_{n,R} \cdot \ddot u_{n,R})i \ddot u_{n,R}}{|\dot u_{n,R}|^3} + 3 \frac{(\dot u_{n,R} \cdot \ddot u_{n,R})^2 i \dot u_{n,R}}{|\dot u_{n,R}|^5} 
\end{aligned}
\]
and
\begin{equation}\label{formulozze}
\begin{aligned}
&i \dot u_{n,R} \cdot \n_{n,R} = |\dot u_{n,R}| & &i \dot u_{n,R} \cdot \dot \n_{n,R} = 0\\
&i \n_{n,R} \cdot \ddot u_{n,R} =- \frac{\dot u_{n,R} \cdot \ddot u_{n,R}}{|\dot u_{n,R}|} & &\dot u_{n,R} \cdot \n_{n,R} = 0\\
&i \dot u_{n,R} \cdot \ddot \n_{n,R} = \frac{(\dot u_{n,R} \cdot \ddot u_{n,R})^2}{|\dot u_{n,R}|^3} - \frac{|\ddot u_{n,R}|^2}{|\dot u_{n,R}|} &
&i \dot \n_{n,R} \cdot \ddot u_{n,R} = - \frac{|\ddot u_{n,R}|^2}{|\dot u_{n,R}|} + \frac{(\dot u_{n,R} \cdot \ddot u_{n,R})^2}{|\dot u_{n,R}|^3}\\
&\dot u_{n,R} \cdot \dot \n_{n,R} = -\frac{i \dot u_{n,R} \cdot \ddot u_{n,R}}{|\dot u_{n,R}|} & & \n_{n,R}\cdot \dot \n_{n,R} = 0\\
& i \n_{n,R}\cdot \dot \n_{n,R} = \frac{i \dot u_{n,R}\cdot \ddot u_{n,R}}{|\dot u_{n,R}|^2}
\end{aligned}
\end{equation}

It is also useful to recall the following basic equalities, which hold for every $\theta_1, \theta_2 \in \R$:
\[
\begin{aligned}
&e^{i\theta_1}\cdot i e^{i\theta_2} = \sin (\theta_2 - \theta_1)\\
&e^{i \theta_1} \cdot e^{i \theta_2} = \cos (\theta_2 - \theta_1).
\end{aligned}
\]

\begin{lemma}\label{lineopcoverg}
It holds 
\[
\Le_{n,R} \to \Le_\infty \quad \text{ in }\mathcal{L}(C^{2,\alpha}_{2\pi}; C^{0,\alpha}_{2\pi}) \quad \text{for all }R \in S_n \text{ as }n\to +\infty.
\] 
\end{lemma}
\begin{proof}
Let $n \geq \underline n$, where $\underline n$ is defined in \eqref{nsotto}. Then for every $n \geq \underline n$ and $R \in S_{n}$ the following estimates hold:  
\begin{equation}\label{stimeupic}
\begin{aligned}
&\left \||\dot u_{n,R}|^s - \frac{n^s}{(n-1)^s}\right\|_{0,\alpha} \leq C\frac{R}{n}, && \forall s >0,\\
&\left \||\ddot u_{n,R}|^s - \frac{n^{2s}}{(n-1)^{2s}}\right\|_{0,\alpha} \leq C\frac{R}{n^2} && \forall s>0, \\
&\|\dot u_{n,R}\cdot \ddot u_{n,R}\|_{0,\alpha}\leq C\frac{R}{n}, \\
&\left\|i \dot u_{n,R}\cdot \ddot u_{n,R} - \frac{n^3}{(n-1)^3}\right\|_{0,\alpha} \leq C\frac{R}{n}
\end{aligned}
\end{equation}
The appearing constants depend only on $a, b, \delta$ and $\alpha$. By means of \eqref{stimeupic} and \eqref{holderineq} we get that
\begin{equation}\label{coeffestlin}
\begin{aligned}
\|a_{n,R} - 1\|_{0,\alpha} \leq C \frac{R}{n}, \quad
\|b_{n,R}\|_{0,\alpha} \leq C \frac{R}{n}, \quad
\|c_{n,R}-1\|_{0,\alpha} \leq C\frac{R}{n},
\end{aligned}
\end{equation}
where $a_{n,R}, b_{n,R}$ and $c_{n,R}$ are defined in \eqref{lincoeff}.
The result easily follows.
\end{proof}

\begin{lemma}\label{appresuno}
Let be $\gamma >0$, $0 <a \leq b$, $\delta \in (0,1)$ and $M>0$. There exists $n_2(M) \geq \underline n$ such that for every $n \geq n_2(M)$, $R \in S_n$ and $\varphi \in B_{n;M}$, where $B_{n;M}$ is defined in \eqref{mballdef}, the function $\Res^1_{n,R}(\varphi)$ defined in \eqref{resunodefini} can be explicitly expressed by means of Taylor expansions.  

On the other hand, let $n\geq \underline n$, $R \in S_n$ be fixed. 
It holds that
\[
\frac{\de}{\de \varphi}\K(u_{n,R} + \varphi \n_{n,R})_{|\varphi=0}= \Le_{n,R},
\]
that is, $\Le_{n,R}$ is the linearized operator of the map $\varphi \mapsto \K(u_{n,R} + \varphi \n_{n,R})$ around $\varphi = 0$. 
\end{lemma}
\begin{proof}
First of all we notice that
\[
\begin{aligned}
|\dot u_{n,R} + \varphi' \n_{n,R} + \varphi \dot \n_{n,R}|^2 &= |\dot u_{n,R}|^2 + |\varphi'|^2 + |\varphi|^2|\dot \n_{n,R}|^2 + 2\varphi \dot u_{n,R} \cdot \dot\n_{n,R}\\
& = |\dot u_{n,R}|^2 \left( 1 + \frac{|\varphi'|^2}{|\dot u_{n,R}|^2} +\frac{|\varphi|^2|\dot \n_{n,R}|^2}{|\dot u_{n,R}|^2} + 2\varphi \frac{\dot u_{n,R} \cdot \dot\n_{n,R}}{|\dot u_{n,R}|^2} \right)
\end{aligned}
\]
Thanks to \eqref{formulozze} and \eqref{stimeupic} we obtain that
\[
\left|\frac{|\varphi'|^2}{|\dot u_{n,R}|^2} +\frac{|\varphi|^2|\dot \n_{n,R}|^2}{|\dot u_{n,R}|^2} + 2\varphi \frac{\dot u_{n,R} \cdot \dot\n_{n,R}}{|\dot u_{n,R}|^2}\right| \leq C_1(\|\varphi\|_{2,\alpha}^2 + \|\varphi\|_{2,\alpha})
\]
where $C_1$ depends only on $a,b,\alpha$. 
Let $n_2(M)$ be such that $ n_2 (M)\geq \underline n$ and for every $n \geq n_2(M)$ and $\varphi \in B_n$ it holds 
\[
C_1(\|\varphi\|_{2,\alpha}^2 + \|\varphi\|_{2,\alpha}) < 1.
\]
Then, for every $n\geq n_2(M)$ we can write the Taylor expansion
\[
\begin{aligned}
&|\dot u_{n,R} + \varphi' \n_{n,R} + \varphi \dot \n_{n,R}|^{-3}\\
&= \frac{1}{|\dot u_{n,R}|^3}\bigg( 1 - \frac{3}{2}\left( \frac{|\varphi'|^2}{|\dot u_{n,R}|^2} +\frac{|\varphi|^2|\dot \n_{n,R}|^2}{|\dot u_{n,R}|^2} + 2\varphi \frac{\dot u_{n,R} \cdot \dot\n_{n,R}}{|\dot u_{n,R}|^2}\right)\\
&+ \sum_{k=2}^\infty c_k \left(\frac{|\varphi'|^2}{|\dot u_{n,R}|^2} +\frac{|\varphi|^2|\dot \n_{n,R}|^2}{|\dot u_{n,R}|^2} + 2\varphi \frac{\dot u_{n,R} \cdot \dot\n_{n,R}}{|\dot u_{n,R}|^2} \right)^k\bigg).
\end{aligned}
\]
On the other hand we have that
\[
\begin{aligned}
&i(\dot u_{n,R} + \varphi \dot \n_{n,R} + \varphi' \n_{n,R}) \cdot (\ddot u_{n,R} + \varphi'' \n_{n,R} + 2 \varphi' \dot \n_{n,R} + \varphi \ddot \n_{n,R}) \\
&= i \dot u_{n,R} \cdot \ddot u_{n,R} + \varphi'' i \dot u_{n,R} \cdot \n_{n,R} + \varphi i\dot u_{n,R} \cdot \ddot \n_{n,R} + \varphi i \dot \n_{n,R} \cdot \ddot u_{n,R} + \varphi' i\n_{n,R} \cdot \ddot u_{n,R}\\
 &+ i(\varphi \dot \n_{n,R}+ \varphi' \n_{n,R} ) \cdot (\varphi'' \n_{n,R} + 2 \varphi'\dot \n_{n,R} + \varphi \ddot \n_{n,R})
\end{aligned}
\]
Then we get that
\begin{equation}\label{svilupposerie}
\begin{aligned}
&\Res^1_{n,R}(\varphi) = \K(u_{n,R} + \varphi \n_{n,R}) - \K(u_{n,R}) - \Le_{n,R}\varphi\\
&=\frac{1}{|\dot u_{n,R}|^3}(i \dot u_{n,R} \cdot \ddot u_{n,R})\bigg( - \frac{3}{2}\left( \frac{|\varphi'|^2}{|\dot u_{n,R}|^2} +\frac{|\varphi|^2|\dot \n_{n,R}|^2}{|\dot u_{n,R}|^2}\right)\\
&+ \sum_{k=2}^\infty c_k \left(\frac{|\varphi'|^2}{|\dot u_{n,R}|^2} +\frac{|\varphi|^2|\dot \n_{n,R}|^2}{|\dot u_{n,R}|^2} + 2\varphi \frac{\dot u_{n,R} \cdot \dot\n_{n,R}}{|\dot u_{n,R}|^2} \right)^k\bigg)\\
&+ ( i \dot u_{n,R} \cdot( \varphi''\n_{n,R}+\varphi \ddot\n_{n,R}) + (\varphi i \dot \n_{n,R} + \varphi' i\n_{n,R}) \cdot \ddot u_{n,R}) \frac{1}{|\dot u_{n,R}|^3} \\
&\cdot\sum_{k=1}^\infty c_k \left(\frac{|\varphi'|^2}{|\dot u_{n,R}|^2} +\frac{|\varphi|^2|\dot \n_{n,R}|^2}{|\dot u_{n,R}|^2} + 2\varphi \frac{\dot u_{n,R} \cdot \dot\n_{n,R}}{|\dot u_{n,R}|^2} \right)^k\\
&+ i(\varphi \dot \n_{n,R}+ \varphi' \n_{n,R} ) \cdot (\varphi'' \n_{n,R} + 2 \varphi'\dot \n_{n,R} + \varphi \ddot \n_{n,R})\frac{1}{|\dot u_{n,R}|^3}\\
\cdot&\sum_{k=0}^\infty c_k \left(\frac{|\varphi'|^2}{|\dot u_{n,R}|^2} +\frac{|\varphi|^2|\dot \n_{n,R}|^2}{|\dot u_{n,R}|^2} + 2\varphi \frac{\dot u_{n,R} \cdot \dot\n_{n,R}}{|\dot u_{n,R}|^2} \right)^k,
\end{aligned}
\end{equation}
thus proving the first part of the Lemma.

As for the second part, let $v \in C^{2,\alpha}_{2\pi}$ being such that $\|v\|_{2,\alpha} = 1$. We want to prove that
\[
\lim_{t \to 0}\left|\frac{\K(u_{n,R}+ tv \n_{n,R})- \K(u_{n,R})}{|t|}  - \Le_{n,R}v\right| = 0.
\] 
Since $t \to 0$, in order to evaluate $\K(u_{n,R}+ tv \n_{n,R})- \K(u_{n,R})$ we can argue exactly as in the previous part of the proof, thus obtaining \eqref{svilupposerie} with $\varphi = tv$.
Then
\[
\lim_{t \to 0}\left|\frac{\K(u_{n,R}+ tv \n_{n,R})- \K(u_{n,R})}{|t|}  - \Le_{n,R}v\right| = \lim_{t\to 0}\frac{|\Res^1_{n,R}(tv)|}{|t|},
\]
and since $\Res^1_{n,R}(tv) = o(t)$ as $t \to 0$, the conlusion readily follows.
\end{proof}
\end{section}
\end{chapter}

\begin{chapter}{The fractional Brezis-Nirenberg problem}\label{FBNchapter}
\begin{section}{Introduction}\label{introdaction}

Let $N \in \mathbb{N}$ and $s \in (0,1)$ be such that $N>2s$, let $\lambda>0$, and let $\Omega \subset \R^N$ be a bounded domain with smooth boundary. Consider the following non local semilinear elliptic problem:
\begin{equation}
\label{fracBrezis}
\begin{cases}
(-\Delta)^s u = \lambda u + |u|^{2_s^*-2}u &\hbox{in}\ \Omega\ , \\
u = 0 &\hbox{in}\ \R^N \setminus \Omega\ ,
\end{cases}
\end{equation}
where $2_s^* := \frac{2N}{N-2s}$ is the critical fractional Sobolev exponent for the embedding of $\mathcal{D}^s(\R^N)$ into $L^{2_s^*}(\R^N)$, and $(-\Delta)^s$ is the s-Laplacian operator, which is defined as
\[
(-\Delta)^s u(x) := C_{N,s} P.V. \int_{\R^N}\frac{u(x)-u(y)}{|x-y|^{N+2s}}\de y= C_{N,s} \lim_{\varepsilon \to 0^+} \int_{\R^N \setminus B_\varepsilon(x)}\frac{u(x)-u(y)}{|x-y|^{N+2s}}\de y,
\]
where the constant $C_{N,s}$ is given by

\[
C_{N,s} := \frac{2^{2s} \Gamma\left(\frac{N}{2}+s\right)}{\pi^{\frac{N}{2}}|\Gamma(-s)|}.
\]

Problem \eqref{fracBrezis} is known as the fractional Brezis-Nirenberg problem, since in the local case the first existence result for positive solutions was given in the celebrated paper \cite{BN}. In \cite{BN}, Brezis and Nirenberg overcame the difficulties due to the lack of compactness of the embedding $H_0^1 \hookrightarrow L^{2^*}$, and showed that the dimension plays a crucial role in the problem. In fact, they proved that when $N\geq4$ there exist positive solutions for every $\lambda \in (0,\lambda_1(\Omega))$, where $\lambda_1(\Omega)$ denotes the first eigenvalue of the classical Dirichlet-Laplacian on $\Omega$. The case $N=3$ is more delicate. Brezis and Nirenberg proved that there exists $\lambda^*(\Omega)>0$ such that  positive solutions exist for every $\lambda \in (\lambda^*(\Omega),\lambda_1(\Omega))$.  When $\Omega=B_R$ is a ball, they also proved that $\lambda^*(B_R)=\frac{\lambda_1(B_R)}{4}$ and 
a positive solution exists if and only if $\lambda \in \left(\frac{\lambda_1(B_R)}{4}, \lambda_1(B_R)\right)$.
\vspace{5pt} 

After the pioneering paper \cite{BN} many results have been obtained, concerning asymptotic analysis, multiplicity, existence and nonexistence, both of positive and of sign-changing solutions (see \cite{ ABP,  Pacella,CFP, CSS,CW, Demmak, HAN, IacVair, IacVair2, Rey, SZ}). 
We point out that in the sign-changing case the dimension $N=3$ exhibits additional difficulties: it is not known yet if there exist non radial sign-changing solutions for $\lambda \in (0,\frac{\lambda_1(B_R)}{4})$. A partial answer to this question was given by Ben Ayed, El Mehdi and Pacella in \cite{Pacella2}. Nevertheless, even in the other dimensions  several interesting phenomena are observed. In fact, Atkinson, Brezis and Peletier in \cite{ABP}, and Adimurthi and Yadava in \cite{AY2}, showed with different proofs that for $N=4,5,6$ there exists $\lambda^{**}(N)>0$ such that, for $\lambda \in (0,\lambda^{**}(N))$, there is no radial sign-changing solution of the Brezis-Nirenberg problem in the ball. Instead, they do exist for any $\lambda \in (0,\lambda_1(B_R))$ if $N\geq 7$, as proved by Cerami, Solimini and Struwe in \cite{CSS}.
\vspace{5pt}

In recent years, a great attention has been devoted to studying non local equations and a natural question is if it is possible to extend the known results about semilinear elliptic problems in the fractional framework. In the case of positive solutions of the fractional Brezis--Nirenberg problem, the picture is quite clear. Servadei and Valdinoci in \cite{ValSerLD}, \cite{ValSer}, proved existence of positive solutions for Problem \eqref{fracBrezis} and their results perfectly agree with the classical ones: if $\lambda_{1,s} = \lambda_{1,s}(\Omega)$ is the first eigenvalue of the fractional Laplacian with homogeneous Dirichelet boundary condition, then Problem \eqref{fracBrezis} admits a nontrivial solution whenever $N \geq 4s$ and $\lambda \in (0, \lambda_{1,s})$. When $2s<N<4s$ there exists $\lambda^*_s=\lambda^*_s(\Omega)$ such that a solution of Problem \eqref{fracBrezis} exists for $\lambda>\lambda^*_s$, and $\lambda$ different from the eigenvalues of the fractional Laplacian. Other interesting results have been obtained by Musina and Nazarov in \cite{Musina2}, for the fractional Dirichlet-Laplace operator $(-\Delta)^m$, $0<m<\frac{n}{2}$.

The asymptotic behavior of least energy positive solutions of Problem \eqref{fracBrezis} (in the case of the spectral fractional Laplacian), as $\lambda \to 0^+$, has been studied by Choi, Kim and Lee in \cite{CHoiKimLee}. Even in this case the results perfectly fit with the classical ones of Han and Rey (see \cite{HAN}, \cite{Rey}).

On the contrary, there is not much literature for sign-changing solutions (see \cite{BFS}) and very few is known about their qualitative properties. In fact, even in the radial case, due to the non local interactions, there are serious difficulties when trying to determine the number of sign changes of the solutions, and this number does not correspond, in general, to the number of connected components of the complement of the nodal set (minus one). Moreover, since we deal with sign-changing solutions, no information is available about their monotonicity via the fractional moving plane method (see \cite{chen}). In addition, as pointed out in the seminal paper of  Frank, Lenzmann and Silvestre \cite{FrLe2}, we cannot apply standard ODE techniques for the fractional Laplacian and this technical gap causes serious troubles.
\vspace{5pt} 

We denote by $X^s_0(\Omega)$ the Sobolev space of the functions $u \in H^s(\R^N)$ such that $u = 0 $ in $\R^N\setminus \Omega$, endowed with the norm
\[
\|u\|^2_s := {\frac{C_{N,s}}{2}\int_{\R^{2N}} \frac{|u(x)-u(y)|^2}{|x-y|^{N+2s}}\de x \de y},
\] 
whose associated scalar product is
\[
(u, v)_s := \frac{C_{N,s}}{2}\int_{\R^{2N}} \frac{(u(x)-u(y))(v(x)-v(y))}{|x-y|^{N+2s}}\de x \de y.
\]

Weak solutions of Problem \eqref{fracBrezis} correspond to critical points of the energy functional 
\[
I(u) := \frac{1}{2}(\|u\|_s^2-\lambda|u|^2_2)- \frac{1}{2_s^*}|u|_{2_s^*}^{2_s^*},\quad u \in X_0^s(\Omega),
\]
 where $|\cdot|_p$ is the standard $L^p$-norm for $p\geq 1$. Existence of radial sign-changing solutions in the ball for Problem \eqref{fracBrezis} is granted for any $s \in (0,1)$, $N>6s$, $\lambda \in (0,\lambda_{1,s}(B_R))$. The proof, which is essentially the same of Cerami, Solimini and Struwe, \cite{CSS}, is given in Section \ref{fracbnex}. In view of known regularity results for the fractional Laplacian, weak solutions $u \in X_0^s(\Omega)$ of \eqref{fracBrezis} turn out to be of class $C^{0,s}(\R^N)$ (as it follows by combining \cite[Theorem 1.1]{HolderReg} and \cite[Theorem 3.2]{IanMosSqua}), and this regularity is optimal. For the interior regularity in $\Omega$ we have better results (see \cite{HolderReg}).

In this Chapter we face with the following problems.
\vspace{7pt}

\textbf{Problem a):} Let $B_R \subset \R^N$ be the ball of radius $R$ centered at the origin. Consider the following simple property:
\[
(\mathcal{P})\ \ \ \ \ \ \ \ \ \hbox{if}\ u \ \hbox{is a radial solution of Problem \eqref{fracBrezis} in $B_R$ and}\ u(0)=0\ \hbox{then}\ u\equiv0.\\[6pt]
\]
It is well known that in the local case $(\mathcal{P})$ holds, but in the fractional framework it is basically unknown when dealing with nodal solutions. The only result in this direction is due to Frank, Lenzmann and Silvestre, who, in \cite{FrLe2}, by using a monotonicity formula argument, showed that $(\mathcal{P})$ holds for radial solutions vanishing at infinity of fractional linear equations of the kind $(-\Delta)^s u + Vu=0$ in $\R^N$, where $V=V(r)$ is radial and non-decreasing, $r=|x|$. 

Unfortunately, in the case of bounded domains this argument does not work properly. In fact, let $u$ be a radial solution of Problem \eqref{fracBrezis} in $B_R$ and let $W:\overline{\R^{N+1}_+} \to \R$ be the extension of $u$ to the upper half space $\R^{N+1}_+ = \R^N \times \R_+$ (see Subection \ref{extensionsubsect} for the definition). The function $W$ is also known as the Caffarelli-Silvestre extension in view of their celebrated paper \cite{CS}. Recalling that $W=W(x,y)$ is cylindrically symmetric with respect to $x\in \R^N$, let us formally write the expression
\[
H(r)= d_s \int_0^{+\infty} \frac{t^a}{2}[W_r^2(r,t)-W^2_y(r,t)] \ dt - \frac{\lambda}{2}|u(r)|^2-\frac{1}{2_s^*}|u(r)|^{2_s^*},\quad r \geq 0,
\]
where $d_s = \frac{1}{2^{{1-2s}}}\frac{\Gamma\left(s\right)}{\Gamma\left(1-s\right)}$. Then, when  trying to repeat the proof of the monotonicity formula, as in the remarkable paper of Cabr\'e and Sire (see \cite[Lemma 5.4]{Yannick}), we cannot deduce that $H$ is decreasing for all $r>0$ because $-d_s\lim_{y\to 0^+}\ y^{1-2s}W_y(r,y)=\lambda u + |u|^{2_s^*-s}u$ just on $(0,R)$.

Now, since $W$ is cylindrically symmetric, we have that $W_r(0,y)\equiv 0$ for any $y>0$, and assuming that $u(0)=0$ we deduce that $H(0)\leq 0$. But, even if $H$ is decreasing in $(0,R)$, we have no information on the value $H(R)=d_s \int_0^{+\infty} \frac{t^a}{2}[W_r^2(R,t)-W^2_y(R,t)] \ dt$, while, in \cite{FrLe2}, by proving that $H$ is decreasing in $(0,+\infty)$, and since $\lim_{r \to +\infty} H(r)=0$, $H(0)\leq 0$, they deduce that $H\equiv 0$ and $u\equiv 0$.

We stress that even other approaches fail in the nodal case. For example, if we try to apply the strong maximum principle, as in the version stated by Cabr\'e and Sire in \cite[Remark 4.2]{Yannick}, assuming that $u\geq 0=u(0)$ in a neighborhood of the origin we must find a small positive $\varepsilon>0$ such that the extension $W$ is non negative in $\Gamma^+_\varepsilon=  \{ (x,y) \in \overline{\R^{N+1}_+} \ | \ y\geq0, \sqrt{|x|^2+y^2} = \varepsilon\}$. Unfortunately, if $u$ changes sign, then also $W$ changes sign (see Section \ref{nodalboundextsect}) and it can happen that for any small $\varepsilon>0$ the set $\Gamma^+_\varepsilon$ intersects $\{W<0\}$, and thus we cannot exclude that $u(0)=0$. This is not surprising because, due to the non local interaction terms, we have that $u^+$, $u^-$ are not weak super, sub solutions of Problem \eqref{fracBrezis} in $\{u>0\}$, $\{u<0\}$, respectively.

Also with the recent version of the fractional strong maximum principle stated by Musina and Nazarov in \cite[Corollary 4.2]{Musina}, considering any subdomain of $B_R \cap \{u\geq0\}$, we deduce only that $u> \inf_{\R^N} u$. Clearly if $u$ is a nodal solution of Problem \eqref{fracBrezis} we cannot exclude that $u(0)=0$ and $u\not\equiv 0$.
\vspace{7pt}

\textbf{Problem b):} Determine the number of connected components of the complement of the nodal set and the number of sign changes of least energy nodal solutions of \eqref{fracBrezis}, when $\lambda$ is close to zero.
\vspace{5pt}

We say that $u_\lambda$ is a least energy sign-changing solution of \eqref{fracBrezis} if $I(u_\lambda)= \inf_{\mathcal{M}} I$, where $\mathcal{M}$ is the nodal Nehari set, i.e.  
\[
 \mathcal{M} := \{u \in X_0^s(\Omega) \ |\ u^\pm \not \equiv 0, I'(u)[u^\pm] = 0 \}.
\]

In view of the previous discussion we remark again that the number of connected components of 
\[
\{ u_\lambda \neq 0\}
\]
does not correspond, in general, to the number of sign changes (plus one). 
Despite that, even assuming that these numbers coincide, in view of the non local interactions between the nodal components it is not possible, via standard energy arguments, to determine them. In fact, let $u_\lambda$ be a least energy solution of Problem \eqref{fracBrezis} and let $K_{i,\lambda}$ be a connected component of $\{u_\lambda \neq 0\}$. Setting  
\[
u_{i,\lambda}:=u_{\lambda} \ \mathbbm{1}_{K_{i,\lambda}},
\]
where $\mathbbm{1}_{K_{i,\lambda}}$ is the characteristic function of ${K_{i,\lambda}}$, then, from $I^\prime (u_\lambda) [u_{i,\lambda}] =0$ we have
\[
\|u_{i,\lambda}\|^2_s + (u_{i,\lambda},u_\lambda-u_{i,\lambda})_s = \lambda |u_{i,\lambda}|^2_{2,K_{i,\lambda}} +  |u_{i,\lambda}|^{2^{*}_s}_{{2^{*}_s},K_{i,\lambda}},
\]
and by a simple computation we see that

\begin{equation}\label{interactionterm}
(u_{i,\lambda},u_\lambda-u_{i,\lambda})_s = -C_{n,s} \int_{\mathbb{R}^{2N}} \frac{u_{i,\lambda}(x) (u_{\lambda}(y)-u_{i,\lambda}(y))}{|x-y|^{N+2s}} \ dx dy.
\end{equation}

Even if it is not difficult to show that $\|u_\lambda\|_s^2 \to 2 S_s^{N/2s}$, as $\lambda \to 0^+$, where $S_s$ is the best fractional Sobolev constant (see \eqref{nonlocSob}), we do not have any information about the limit value of \eqref{interactionterm} nor on its sign. In particular, the presence of this interaction term between $u_{i,\lambda}$ and the other nodal components does not allow us to replicate the proof of Ben Ayed, El Mehdi and Pacella (see \cite[Proof of Theorem 1.1]{Pacella}). In fact, in the local case, by using Poincar\'e and Sobolev inequalities, one can deduce that
\[
\int_{K_{i,\lambda}} |\nabla u_{i,\lambda}|^2 \geq \left(1+o(1)\right)  S_1^{N/2},
\]
and being $\|u_\lambda\|_1^2 \to 2 S_1^{N/2}$ it follows that $u_\lambda$ cannot have more than two nodal components.
\vspace{7pt}

\textbf{Problem c):} Determine the asymptotic profile of least energy nodal solutions of \eqref{fracBrezis} as $\lambda \to 0^+$.
\vspace{5pt}

The aim of this Chapter is to contribute to Problem a), Problem b) and Problem c) in the case of least energy nodal radial solutions of the fractional Brezis--Nirenberg problem in the ball. 

Our results are the following:

\begin{teo}\label{mainteoproba}
Let $N>6s$, $s \in (\frac{1}{2},1)$ and let $R>0$. There exists $\bar\lambda>0$ such that for any $\lambda \in (0,\bar \lambda)$, any least energy sign-changing radial solution $u_{\lambda}$ to \eqref{fracBrezis} in $B_R$ does not vanish at zero.
\end{teo}

\begin{teo}\label{mainteocc}
Let $N > 6s$, $s \in \left(0, 1\right)$ and let $R>0$. There exists $\hat \lambda_s>0$ such that, for any $\lambda \in (0, \hat \lambda_s)$, any least energy sign-changing radial solution $u_{s,\lambda}$ to \eqref{fracBrezis} in $B_R$ changes sign at most twice. Moreover, the zeros of $u_{s,\lambda}= u_{s,\lambda}(r)$ in $(0,R)$ coincide with its nodes, i.e. with the sign-changes of $u_{s,\lambda}$. More precisely, one and only one of the following hold:
\begin{enumerate}[(a)]
\item if $u_{s, \lambda}$ changes sign twice then it vanishes in $[0,R)$ only at the nodes,
\item if $u_{s, \lambda}$ changes sign once then it vanishes in $(0,R)$ only at the node and it can vanish also at the origin.
\end{enumerate}

\end{teo}

\begin{teo}\label{mainteorem}
Let $N\geq 7$ and let $R>0$. There exist $\tilde\lambda>0$ such that for any $\lambda \in (0,\tilde\lambda)$ there exists $\bar s \in (0,1)$ such that for any $s \in (\bar s, 1)$, any least energy sign-changing radial solution $u_{\lambda}$ to \eqref{fracBrezis} in $B_R$ changes sign exactly once. 
\end{teo}

\begin{teo}\label{mainteorem2}
Let $N>6s$, $s \in (\frac{1}{2},1)$ and let $R>0$. Let $(u_\lambda)$ be a family of least energy sign-changing radial solutions to \eqref{fracBrezis} in $B_R$, such that $u_\lambda(r)$ changes sign exactly once in $(0,R)$ for all sufficiently small $\lambda>0$. Assume, without loss of generality, that $u_\lambda\geq0$ in a neighborhood of the origin, and set $M_{\lambda,\pm}:=\|u_\lambda^\pm\|_\infty$. Then:
\begin{itemize}
\item[i)] $M_{\lambda,\pm}\to + \infty $ as $\lambda \to 0^+$,
\item[ii)] denoting by $r_\lambda \in (0,R)$ the node of $u_\lambda$ and by $s_\lambda \in (r_\lambda, R)$ any point where $u_\lambda=u_\lambda(r)$ achieves $-M_{\lambda,-}$ we have $r_\lambda, s_\lambda \to 0$ as $\lambda \to 0^+$,
\item[iii)] $\frac{M_{\lambda,+}}{M_{\lambda,-}} \to + \infty $ as $\lambda \to 0^+$,
\item[iv)] setting $\beta:= \frac{2}{N-2s}$, then the rescaled function
\[
\tilde u^+_\lambda(x) := \frac{1}{M_{\lambda,+}}u^+_\lambda\left( \frac{x}{M_{\lambda,+}^\beta}\right), \quad x \in \R^N,
\]
converges in $C^{0,\alpha}_{loc}(\R^N)$, as $\lambda \to 0^+$, for some $\alpha=\alpha(s) \in (0,1)$, to the fractional standard bubble $U_s$ in $\R^N$ centered at $0$ and such that $U_s(0)=1$.
\end{itemize}
\end{teo}

Theorem \ref{mainteoproba} is a consequence of a more general result, which ensures that $u_\lambda(0)$ is bounded away from zero, by a constant which is uniform with respect to $\lambda$. The idea is to argue by contradiction and to construct a family of rescaled functions $\tilde u_\lambda$ such that $\tilde u_\lambda(0) \to 0$ as $\lambda \to 0^+$. By a standard argument $\tilde u_\lambda$ converges, in compact subsets of $\R^N$, to a solution $\tilde u$ of the fractional critical problem  $(-\Delta)^s U=|U|^{2_s^*-2}U$ in $\R^n$. Then, by energy considerations and the fractional strong maximum principle, we deduce that $\tilde u$ has to be positive in $\R^N$, contradicting that $\tilde u(0)=0$.
\vspace{5pt}

The proofs of Theorem \ref{mainteocc}, Theorem \ref{mainteorem} rely on the combination of several tools. The first step is to prove that the number of the nodal components of the extension $W$ is two. This is done by arguing as in the papers \cite{FrLe1}, \cite{FrLe2}. Then, exploiting the radiality of the solutions, we prove that our solutions change sign at most twice. In view of this information, Theorem \ref{mainteocc} follows from a topological argument based on the Jordan's curve theorem, the fractional strong maximum principle and on a nice result of Fall and Felli (see \cite[Theorem 1.4]{fellifall}) which ensures that our solutions cannot vanish in a set of positive measure.
 
 For Theorem \ref{mainteorem}, the fundamental step is to argue by contradiction and to prove that if two nodes exist for $s$ close to $1$ then they persist for the limit profile. This is done by performing an asymptotic analysis of the nodes of the solutions when $s \to 1^-$, fine energy estimates, a quite complex technical result (see the Appendix, Theorem \ref{subsol}) and the strong maximum principle for the standard Laplacian. At the end, it is not difficult to prove that the limit function is a nodal solution of the classical Brezis--Nirenberg problem and it is of least energy, and thus we get a contradiction since such solutions change sign exactly once.

We point out that the restriction to $N\geq 7$ is essential for the result because existence of sign-changing radial solutions in the ball for the classical Brezis--Nirenberg problem, when $\lambda$ is close to $0$, holds only for $n\geq 7$ (see \cite{ABP}, \cite{AY2}, \cite{CSS}).
\vspace{5pt}

The proof of Theorem \ref{mainteorem2} is based on the analysis of rescaled functions. We observe that statement iii) strongly relies on the fact that $u_\lambda$ possesses exactly one node. In fact, assuming that $u_\lambda$ has at least two nodes, even if we still have that $u_\lambda^+$ and $u_\lambda^-$ carry the same energy as $\lambda \to 0^+$, the energy spreading between the components of $\{u_\lambda>0\}$, which are at least two, does not allow us to establish the leading term between $M_{\lambda,+}$ and $M_{\lambda,-}$.


We point out that no information about the limit profile of suitable rescalings of $u_\lambda^-$ is provided. The reason is that, differently from the results of \cite{Iacopetti}, $u_\lambda^-$ is not a solution of Problem \eqref{fracBrezis} in $\{u_\lambda <0\}$, and we cannot apply ODE techniques. Finally, the restriction $s>\frac{1}{2}$ is technical because we make use intensively of the fractional Strauss inequality, as in the version stated in \cite[Proposition 1]{choozawa}, and it is known that such inequality fails for the values $0<s \leq \frac{1}{2}$ (see \cite[Remark 2, Remark 4]{choozawa}).
\vspace{5pt}


As a final remark, we point out that our proofs work, with slightly adjustments, for fractional semilinear problems with subcritical nonlinearities.
\vspace{5pt}

This Chapter is organized as follows: in Section \ref{SectionIntro} we recall some known results about the fractional framework and in Section \ref{fracbnex} we prove the existence of radial solutions of Problem \eqref{fracBrezis} in the ball. In Section \ref{asintoticsect} we prove some preliminary results about the asymptotic analysis of the energy as $\lambda\to0^+$, and in Section \ref{nodalboundextsect} we study the nodal set of the extension. In Section \ref{uniformboundsect} we provide uniform bounds, with respect to the parameter $s$, for the $L^\infty$-norm and the energy of the solutions. Finally in Sections \ref{7}, \ref{8}, \ref{9}, \ref{10} we prove, respectively, Theorem \ref{mainteoproba}, Theorem \ref{mainteocc}, Theorem \ref{mainteorem}, and Theorem \ref{mainteorem2}. In the Appendix we prove some technical results and Theorem \ref{subsol}.

\vspace{10pt}
\textbf{Notation.} 
The constant $\omega_N$ denotes the $N$-dimensional measure of the unit sphere $\Sf^N$. We denote by $B_R(x_0)$ the ball centered in $x_0 \in \R^N$ with radius $R>0$. If $x_0 = 0$ we simply write $B_R$. 
Let $\Omega$ be a domain in $\R^N$, we denote by
\[
|u|_p = \left(\int_\Omega |u|^p \de x\right)^{\frac{1}{p}}
\]
the usual $L^p(\Omega)$ norm, for $p \in [1, \infty)$, and by $|u|_\infty$ the standard norm in $L^\infty(\Omega)$.  
Moreover, for $k \in \N$, $\alpha \in (0,1)$ we set
\[
\begin{aligned}
&|D^k u|_{\infty; \Omega} := \sup_{\underset{|\gamma| = k}{\gamma \in \N^N}}\sup_{x \in \Omega} |D^\gamma u(x)|,\\
&[u]_{k,\alpha; \Omega} := \sup_{\underset{|\gamma| = k}{\gamma \in \N^N}}\sup_{x,y \in \Omega} \frac{|D^\gamma u(x) - D^\gamma u (y)|}{|x-y|^\alpha}
\end{aligned}
\]
so that
\[
\|u\|_{k, \alpha; \Omega} := \sum_{j=0}^k |D^k u|_{\infty; \Omega} + [u]_{k,\alpha; \Omega}
\]
denotes the standard norm in $C^{k,\alpha}(\overline{\Omega})$. If $\Omega = \R^N$ we omit the subscript in the above norms.
\end{section}

\begin{section}{Preliminary results}\label{SectionIntro}
In this section we present some well known results about fractional Sobolev spaces (Subsection \ref{funcframe} and \ref{fracsobconstsubsect}), the regularity of the solutions of fractional problems (Subsection \ref{regularitysubsect}) and the extension properties for the fractional Laplacian (Subsection \ref{extensionsubsect}), which are extensively used through all the Chapter. 

\begin{subsection}{Functional framework}\label{funcframe}

Let $\Omega$ be a bounded domain. In Section \ref{introdaction} we have introduced the Sobolev spaces $X^s_0(\Omega)$, for $s \in (0,1)$.  A weak solution for \eqref{fracBrezis} is defined as a function $u \in X^s_0(\Omega)$ such that
\[
(u,\varphi)_s = \lambda \int_\Omega u \varphi \de x + \int_\Omega |u|^{2^*_s-2}u\varphi \de x
\]
holds for every $\varphi \in X^s_0(\Omega)$. Recalling that $C_c^\infty(\Omega)$ is a dense subspace of $X^s_0(\Omega)$, it suffices to verify the previous relation for all $\varphi \in C^\infty_c(\Omega)$.

We introduce also the homogeneous Sobolev spaces $\mathcal{D}^s(\R^N)$, defined as as the completion of $C^\infty_c(\R^N)$ with respect to the norm $\|\cdot \|_s$. When $N > 2s$, it holds that $\mathcal{D}^s(\R^N) \hookrightarrow L^{2^*_s}(\R^N)$ and also the usual Sobolev and Rellich-Kondrakov embeddings hold true (see \cite[Theorem 6.7, Corollary 7.2]{Hitch}): for all $p\in [1, 2^*_s]$ we have
\[
\begin{aligned}
&X^s_0(\Omega) \hookrightarrow L^p(\Omega),\\
&\mathcal{D}^s(\R^N) \hookrightarrow L^p_{loc}(\R^N),
\end{aligned}
\]
and the embeddings are compact for $1 \leq p< 2^*_s$. 
To clarify the connection between the fractional Laplacian and the classic Laplacian we mention the following asymptotic results.
\begin{prop}\label{lapcosas}
For any $N > 1$, the following statements hold:
\[
\begin{aligned}
&\lim_{s \to 1^-}\frac{C_{N, s}}{s(1-s)} = \frac{4N}{\omega_{N-1}},\\
&\lim_{s \to 0^+}\frac{C_{N, s}}{s(1-s)} = \frac{2}{\omega_{N-1}}.
\end{aligned}
\]
Moreover, we have that
\[
\begin{aligned}
&\lim_{s \to 1^-}\|u\|_s^2 = |\nabla u|_2^2 \quad & &\forall u \in H^1(\R^N),\\
&\lim_{s \to 0^+}\|u\|^2_s = |u|^2_2 \quad & &\forall u \in \bigcup_{0<s<1}H^s(\R^N).
\end{aligned}
\]

\end{prop}

In order to simplify the presentation of some statements, with a slight abuse of notation, we will denote by $(-\Delta)^1$ the usual Laplace operator $-\Delta$ and with $\|u\|_1^2 = |\nabla u|_2^2$ the usual $H^1$-seminorm. 
We also recall that the fractional Laplacian and the fractional Sobolev spaces $H^s(\R^N)$ can be equivalently defined via the Fourier transform. In particular, given $s >-\frac{N}{2}$ and $u: \R^N \to \R$ we can formally define 
\[
(-\Delta)^s u := \mathcal{F}^{-1}(|\xi|^{2s}\mathcal{F}u)
\] 
and 
\[
\|u\|_{s}^2 := \int_{\R^N}|\xi|^{2s}|\mathcal{F}u|^2 \de \xi.
\]
When $s \in (0,1)$ these definition are equivalent to the standard ones (see \cite[Proposition 3.3, Proposition 3.4]{Hitch}).

Next result is known (see \cite[Lemma 3.5]{Bogdan}), nevertheless we give here the proof in order to underline the independence of the constant appearing in the statement from the parameter $s$.
\begin{lemma}\label{bogdy}
Let $\varphi \in C^\infty_c(\R^N)$, and let $\Omega_\varphi = \text{supp }\varphi$. There exists $C= C(n, \Omega_\varphi)>0$ such that for every $s\in (0,1)$ it holds 
\[
|(-\Delta)^s \varphi (x) | \leq C(|\varphi|_\infty + |D^2\varphi|_\infty) \frac{1}{(1 + |x|)^{n+2s}}\quad \forall x \in \R^N. 
\]
\end{lemma}
\begin{proof}
Since $\Omega_\varphi$ is bounded, there exists $R>0$ such that $\Omega_\varphi \subset B_R$ and $\text{dist}(\Omega_\varphi, R) = 1$. 
When $x \not \in B_R$, we have that
\[
|(-\Delta)^s \varphi(x)| \leq C_{N,s}P.V.\int_{\R^N}\frac{|\varphi(y)|}{|x-y|^{N+2s}}\de y = C_{N,s}\int_{\Omega_\varphi}\frac{|\varphi(y)|}{|x-y|^{N+2s}}\de y.
\]
where in the last equality we have taken rid of the principal value since for every $y \in \Omega_\varphi$ it holds that $|x-y| \geq 1$. Moreover, here $1+|x| \leq (R+1) |x-y|$, and hence
\[
|(-\Delta)^s \varphi(x)| \leq C_{N,s}\left(\frac{R+1}{1+|x|}\right)^{N+2s}\int_{\Omega_\varphi}|\varphi(y)|\de y \leq C(N, \Omega_\varphi)\frac{|\varphi|_\infty}{(1+|x|)^{N+2s}},
\]
where we used the fact that $C_{N,s}$ is uniformly bounded for $s \in [0,1]$ thanks to Proposition \ref{lapcosas}.

Now, let us fix $x \in B_R$. Writing the fractional Laplacian in the alternative form (see \cite[Lemma 3.2]{Hitch}) we have that
\[
(-\Delta)^s \varphi (x) = -\frac{C_{N,s}}{2}\int_{\R^N}\frac{\varphi(x+y)+\varphi(x-y)-2\varphi(x)}{|y|^{N+2s}}\de y.
\]
In addition, since $\varphi \in C^\infty_c(\R^N)$ we can write $|\varphi(x+y)+\varphi(x-y)-2\varphi(x)|\leq 2|D^2\varphi|_\infty |y|^2$, and thus we get that
\[
\begin{aligned}
|(-\Delta)^s \varphi (x)| \leq& \frac{C_{N,s}\omega_{N}}{2}\left(4 |\varphi|_\infty\int_1^\infty\rho^{-1-2s}\de \rho + |D^2\varphi|_\infty \int_0^1 \rho^{-1+2(1-s)}\de \rho\right)\\
=&\ C_1(N)\left(\frac{C_{N,s}}{s} + \frac{C_{N,s}}{1-s}\right)(|\varphi|_\infty + |D^2\varphi|_\infty) \leq C_2(N)(|\varphi|_\infty + |D^2\varphi|_\infty),
\end{aligned}
\]
where the last inequality follows from Proposition \ref{lapcosas}.
Hence, recalling that $|x|\leq R$, we have
\[
\begin{aligned}
|(-\Delta)^s \varphi (x)| &\leq (1+|x|)^{N+2s}\frac{ C_2(N)}{(1+|x|)^{N+2s}}(|\varphi|_\infty + |D^2\varphi|_\infty) \\
&\leq \frac{C_3(N, \Omega_\varphi)}{(1+|x|)^{N+2s}}(|\varphi|_\infty + |D^2\varphi|_\infty),&
\end{aligned}
\]
as desired. The proof is complete.
\end{proof}

A consequence of Lemma \ref{bogdy} is the following useful result. 
\begin{lemma}\label{veryweak}
Let $u \in \mathcal{D}^s(\R^N)$ and $\varphi \in C^\infty_c(\Omega)$. Then 
\[
(u, \varphi)_s = \int_{\R^N}u (-\Delta)^s\varphi \de x. 
\]
\end{lemma}
\begin{proof}
Let us set $(-\Delta)_\varepsilon^s \varphi(x) := C_{N,s} \int_{\R^N} \frac{\varphi(x)-\varphi(y)}{|x-y|^{N+2s}}\mathbbm{1}_{\R^N \setminus B_\varepsilon(x)}(y) \de y$. By definition and since $\varphi$ is smooth we have 
\[
\lim_{\varepsilon \to 0^+}(-\Delta)_\varepsilon^s \varphi(x) = (-\Delta)^s \varphi(x).
\]
Notice that also in this case it holds that
\[
(-\Delta)_\varepsilon^s \varphi (x) = -\frac{C_{N,s}}{2}\int_{\R^N}\frac{\varphi(x+y)+\varphi(x-y)-2\varphi(x)}{|y|^{N+2s}}\mathbbm{1}_{\R^N \setminus B_\varepsilon}(y)\de y,
\]
and arguing as in Lemma \ref{bogdy} we get that
\[
|(-\Delta)_\varepsilon^s \varphi(x) u(x)| \leq C(N, \Omega_\varphi, \varphi)\frac{|u(x)|}{(1+ |x|)^{N+2s}},
\]
where $C$ does not depend on $\varepsilon$.

Since $u \in \mathcal{D}^s(\R^N)$, we have
\[
\int_{\R^N}\frac{|u(x)|}{(1+ |x|)^{N+2s}}\de x \leq |u|_{2^*_s}\left(\int_{\R^N}\frac{1}{(1+|x|)^{2N}}\de x\right)^{\frac{N+2s}{2N}}< +\infty,
\]
and applying Lebesgue's dominated convergence theorem we get 
\[
\int_{\R^N}u (-\Delta)^s \varphi \de x = \lim_{\varepsilon \to 0^+}\int_{\R^N} u (-\Delta)^s_\varepsilon \varphi \de x.
\]
Now, since the singular kernel that appears in the definition of the fractional Laplacian is symmetric, we can easily check that
\begin{equation}\label{fubinigagliardo}
\begin{aligned}
&C_{N,s}\int_{\{|x-y|>\varepsilon\}}\frac{\varphi(x)-\varphi(y)}{|x-y|^{N+2s}}u(x) \de x \de y\\
=&\ \frac{C_{N,s}}{2}\int_{\{|x-y|>\varepsilon\}}\frac{(u(x)-u(y))(\varphi(x)-\varphi(y))}{|x-y|^{N+2s}}\de x \de y.
\end{aligned}
\end{equation}
Since the Gagliardo seminorm of $u$ and $\varphi$ are finite, the right-hand side of \eqref{fubinigagliardo} is finite and we can apply Fubini's theorem on the left-hand side. Moreover, the right-hand side of \eqref{fubinigagliardo} goes to $(u, \varphi)_s$ as $\varepsilon \to 0^+$ and thus we obtain that
\[
\begin{aligned}
\int_{\R^N} u (-\Delta)^s\varphi \de x &= \lim_{\varepsilon \to 0^+}\int_{\R^N} u (-\Delta)^s_\varepsilon\varphi \de x  \\
=& \lim_{\varepsilon \to 0^+} C_{N,s}\int_{\{|x-y|>\varepsilon\}}\frac{\varphi(x)-\varphi(y)}{|x-y|^{N+2s}}u(x) \de x \de y  \\
=& \lim_{\varepsilon \to 0^+}\frac{C_{N,s}}{2}\int_{\{|x-y|>\varepsilon\}}\frac{(u(x)-u(y))(\varphi(x)-\varphi(y))}{|x-y|^{N+2s}} \de x \de y = (u,\varphi)_s,
\end{aligned}
\] 
and the proof is complete.
\end{proof}

\end{subsection}

\begin{subsection}{Fractional Sobolev constant and Dirichelet eigenvalues}\label{fracsobconstsubsect}
Let us recall the definition of the best Sobolev constant for the embedding $\mathcal{D}^s(\R^N) \hookrightarrow L^{2^*_s}(\R^N)$,
\begin{equation}
\label{nonlocSob}
S_s := \inf_{u \in \mathcal{D}^s(\R^N)\setminus\{0\}} \frac{\|u\|^2_s}{|u|_{2_s^*}^{2}}.
\end{equation}
The value of $S_s$ is explicitly known, and we have also an explicit expression for the minimizers. 

\begin{teo}[{\cite[Theorem 1.1]{CotTav}}]\label{SobolevEmb}
Let $N>2s$. Then for every $u \in \mathcal{D}^s(\R^N)$ it holds 
\begin{equation}\label{Sobineq}
S_s|u|_{2^*_s}^{2} \leq \|u\|_s^2,
\end{equation}
where
\[
S_s = 2^{2s}\pi^s \frac{\Gamma\left( \frac{N+2s}{2}\right)}{\Gamma\left( \frac{N-2s}{2}\right)}\left[ \frac{\Gamma\left( \frac{N}{2}\right)}{\Gamma\left( N\right)} \right]^{\frac{2s}{N}}.
\]
The equality in \eqref{Sobineq} is achieved only by functions of the family 
\[
k \frac{1}{(\mu^2 + |x-x_0|^2)^{\frac{N-2s}{2}}},
\]
where $k \in \R$, $\mu>0$ and $x_0 \in \R^N$. 
\end{teo}

\begin{rem}\label{sobolevcostbound}
The function $f:[0,1]\to \R^+$, $s \mapsto S_s$, is continuous. In particular, there exist two positive constants $\underline S$, $\overline S>0$ such that  $\underline S \leq S_s \leq \overline S \quad \forall s \in [0,1].$
\end{rem}

\begin{rem}
If we take 
\[
k = k_\mu := \left[S_s^{\frac{N}{2s}}\mu^{N}\left(\int_{\R^N}\frac{1}{(1+|x|^2)^N}\de x\right)^{-1}\right]^{\frac{1}{2^*_s}}
\]
then the functions
\begin{equation}\label{eq:bubble}
U^s_{x_0, \mu}(x) : = k_\mu \frac{1}{(\mu^2 + |x-x_0|^2)^{\frac{N-2s}{2}}},
\end{equation}
also known as ``standard bubbles'', satisfy the equation
\[
(-\Delta)^s U^s_{x_0, \mu} = {U^s_{x_0, \mu}}^{2^*_s-1} \quad \text{in } \R^N
\]
for all $\mu >0$, $x_0 \in \R^N$ and 
\[
\|U^s_{x_0, \mu}\|^2_s = |U^s_{x_0, \mu}|^{2^*_s}_{2^*_s} = S_s^{\frac{N}{2s}}.
\]
\end{rem}
   
\begin{prop}\label{stimebubble}
Let  $s\in(0,1)$, $N>2s$, let $\Omega \subset \R^N$ be a domain, and let $x_0 \in \Omega$ and $\rho>0$ be such that $B_{4\rho}(x_0) \subset \Omega$. Let $\varphi \in C^\infty_c(B_{2\rho}(x_0); [0,1])$ be such that $\varphi \equiv 1$ in $B_\rho(x_0)$.
For $\varepsilon >0$, define 
\[u^s_\varepsilon (x) := \varphi(x)\varepsilon^{-(N-2s)/2}U_{x_0, \mu}^s\left(\frac{x-x_0}{\varepsilon}+x_0\right),
\]
where $U^s_{x_0, \mu}$ is as in \eqref{eq:bubble}.  
Then the following estimates hold:
\begin{equation}\label{stime}
\begin{aligned}
&\|u^s_\varepsilon\|_s^2 \leq S_s^{\frac{N}{2s}} + C \varepsilon^{N-2s} \\
&S_s^{\frac{N}{2s}}-C\varepsilon^N\leq |u^s_\varepsilon|_{2^*_s}^{2^*_s} \leq S_s^{\frac{N}{2s}}\\
&0\leq|u^s_\varepsilon|_{2^*_s-1}^{2^*_s-1} \leq C \varepsilon^{\frac{N-2s}{2}}\\
&0\leq|u^s_\varepsilon|_1 \leq C\varepsilon^{\frac{N-2s}{2}}\\
&|u^s_\varepsilon|_2^2 \geq
\begin{cases}
C\varepsilon^{2s} - C\varepsilon^{N-2s} & \text{if } N>4s \\
C\varepsilon^{2s}|\ln \varepsilon| + C\varepsilon^{2s} & \text{if } N=4s \\
C\varepsilon^{N-2s} - C\varepsilon^{2s} & \text{if } N<4s 
\end{cases}
\end{aligned}
\end{equation}
where all the constants are positive and depend on $N$, $\mu$, $x_0$, $\rho$ and $s$. In addition, if we fix $0<s_0<s_1\leq1$ then when $N>4s_1$ the appearing constants are uniformly bounded with respect to $s \in (s_0, s_1)$. 
\end{prop}
\begin{proof}
The proof of the estimates is contained in \cite[Proposition 12]{Ser} and \cite[Proposition 21, 22]{ValSer}. For the second part we observe that for fixed $0<s_0<s_1 \leq 1$ and $N>4s_1$, being $S_s$ uniformly bounded and since we use the normalizing constant $C_{N,s}$ in the definition of $\|\cdot\|_s$ then, thanks to Proposition \ref{lapcosas} and by elementary computations, we see that the constant appearing in the proof of the aforementioned propositions are uniformly bounded with respect to $s \in (s_0, s_1)$. 
\end{proof}

Another quantity which plays a central role is the first eigenvalue of the $s$-Laplacian under homogeneous Dirichelet conditions, whose variational characterization is given by
\[
\lambda_{1,s} := \inf_{u \in X^s_0(\Omega)\setminus\{0\}} \frac{\|u\|_s^2}{|u|^2_2}.
\]

We also recall the fractional Poincar\'e inequality (see e.g., \cite[Proposition 2.7]{braparsqu})
\begin{prop}
Let $\Omega \subset \R^N$ be an open bounded set and $s \in (0,1)$. Then for every $u \in X^s_0(\Omega)$ it holds that 
\[
C |u|^2_2 \leq \|u\|^2_s
\]
where $C = \tilde C\frac{C_{N,s}}{s(1-s)}\frac{s}{(\text{diam}(\Omega))^{2s}}>0$ and $\tilde C>0$ depends only on $N$. 
\end{prop}
\begin{rem}\label{eigenbounds}
As a consequence of Proposition \ref{lapcosas} we deduce that the constant of the Poincar\'e inequality is uniformly bounded when $s$ is close to one. This implies that for every $s_0 \in (0,1)$ it holds that
\[
\underline \lambda(s_0) := \inf_{s \in [s_0,1)} \lambda_{1,s} >0. 
\]
Moreover, thanks to Lemma \ref{bogdy} and Lemma \ref{veryweak}, we have
\[
\overline \lambda := \sup_{s \in (0,1)} \lambda_{1,s} < \infty. 
\]
\end{rem} 
\end{subsection}
\begin{subsection}{Regularity of solutions}\label{regularitysubsect}
We collect here some regularity results that will be used through the Chapter. 
\begin{teo}[{\cite[Theorem 3.2]{IanMosSqua}}]
\label{reginf}
Let $\Omega \subset \R^N$ be a bounded $C^{1,1}$ domain. Let $f:\Omega \times \R \to \R$ be a Carathéodory function such that 
\[
|f(x, t)| \leq \alpha(1 + |t|^{q-1}) \ \text{ a.e. in }\Omega \text{ and for all } t \in \R
\]
with $\alpha>0$ and $1 \leq q \leq 2^*_s$.
Then, for every weak solution $ u \in X^s_0(\Omega)$ of 
\[
\begin{cases}
(-\Delta)^s u = f(x, u) & \text{in }\Omega, \\
u=0 & \text{in }\R^N \setminus \Omega,
\end{cases}
\]
we have $u \in L^\infty(\Omega)$.
\end{teo}

\begin{prop}[{\cite[Corollary 2.4, Corollary 2.5]{HolderReg}}]\label{holderrreg}
Assume that $u \in C^\infty(\R^N)$ is a solution of $(-\Delta)^su = g$ in $B_2$. Let $\beta  >0$ with $\beta = k + \alpha$ with $k \in \N$ and $\alpha \in (0, 1)$.
\begin{enumerate}[(i)]
\item Assume that $\beta \in (0, 2s)$. Then we get that
\[
\|u\|_{k, \alpha; \overline{B_{1/2}}} \leq C\left( |g|_{\infty; {B_2}} + |u|_{\infty; {B_2}}+ \int_{\R^N}\frac{u(x)}{(1+|x|)^{N+2s}}\right)
\]
where the constant $C >0$ depends only on $N$, $s$ and $\beta$.

\item  Assume that neither $\beta$ nor $\beta + 2s$ is an integer. Then, denoting $\beta + 2s = k' + \alpha'$, with $k^\prime \in \N$ and $\alpha^\prime \in (0,1)$, we get that
\[
\|u\|_{k', \alpha'; \overline{B_{1/2}}} \leq C\left( \|g\|_{k , \alpha; \overline{B_2}} + \|u\|_{k, \alpha; \overline{B_2}}+ \int_{\R^N}\frac{u(x)}{(1+|x|)^{N+2s}}\right)
\]
where the constant $C >0$ depends only on $N$, $s$ and $\beta$. 

\end{enumerate}
\end{prop}

\begin{rem}\label{soave}
Let $s_0 \in (0,1)$ and $s \in (s_0, 1)$. As pointed out in \cite[Lemma 2.2]{uniquenondeg}, \cite[Lemma 4.4]{Yannick} and recalling that Proposition \ref{holderrreg} is a consequence of {\cite[Proposition 2.8, Proposition 2.9]{sylv}}, we obtain that for suitable choices of $s_0$ and $\beta$, the appearing constants in Proposition \ref{holderrreg} are uniformly bounded in $s$. In particular, this is true if we take $s_0 >0$, $\beta = s$ in $(i)$, and $s_0 > \frac{2}{3}$, $\beta = s$ in $(ii)$. 
Moreover, a careful analysis of the proof shows that the hypothesis $u \in C^\infty(\R^N)$ is purely technical, and both results hold true as soon as the right hand side is finite.  
More precisely, after a standard covering argument, Proposition \ref{holderrreg} can be restated in the following way:
\vspace{10pt}

Let $\Omega\subset \R^N$ be a domain. Let $u \in \mathcal{D}^s(\R^N) \cap L^\infty(\R^N)$ be a weak solution of $(-\Delta)^s u = g$ in $\Omega$. Then for every $K' \subset \subset K \subset \subset \Omega$ the following hold:
\begin{enumerate}[(a)]
\item Let be $s_0 \in (0,1)$ and $s \in [s_0, 1)$. Assume that $g \in L^\infty(\Omega)$. Then $u \in C^{0,s}(K')$ and it holds that
\[
\|u\|_{0, s;K'}\leq C(|u|_\infty + |g|_{\infty;K}),
\]
for a constant $C>0$ depending on $n$, $K$, $K'$ and $s_0$.
\item Let be $s_0 \in \left(\frac{2}{3}, 1\right)$ and $s \in [s_0, 1)$. Assume that $g \in C^{0,s}(\overline\Omega)$. Then $u \in C^{2,3s-2}(K')$. Moreover,
\[
\|u\|_{2, 3s-2; K'}\leq C(|u|_\infty + \|g\|_{0,s ;K}),
\]
for a constant $C>0$ depending only on $n$, $K$, $K'$ and $s_0$.
\end{enumerate}
\end{rem}

\begin{rem}\label{gilbcont}
Let $s_0 \in (0,1)$, let $(s_j) \subset [s_0, 1)$ and let $(\Omega_j)$ be a family of domains such that $\Omega_j \subset \Omega_{j+1}$, which invades $\R^N$ as $j \to +\infty$. Assume now that $(u_j)$ and $(g_j)$ are two families such that $u_j \in H^{s_j}(\R^N) \cap L^\infty(\R^N)$ and $g_j \in L^\infty(\R^N)$, which satisfy in the weak sense $(-\Delta)^{s_j} u_j = g_j$ in $\Omega_j$. 
Then fixing two compact sets $K_1 \subset \subset K_2 \subset \R^N$, we have that $u_j$ satisfy $(-\Delta)^{s_j}u_j = g_j$ definitely in $K_2$. Thus, by point $(a)$ of Remark \ref{soave}, we get that 
\[
\|u_j \|_{0, s_j; K_1} \leq C (|u_j|_\infty + |g_j|_{\infty; K_2})
\]
where $C>0$ depends only on $s_0$, $K_1$ and $K_2$. 

If $(u_j)$ and $(g_j)$ are uniformly bounded in $L^\infty(\R^N)$, this implies that $\|u_j\|_{0, s_0; K_1} \leq C$ where $C$ does not depend on $j$. Hence, thanks to \cite[Lemma 6.36]{giltru} we have that 
\[
u_j \to u \quad \text{in }C^{0, \alpha}(K_2)
\]
for any fixed $\alpha < s_0$. If in addition $s_0 > \frac{2}{3}$ and there exists $C$ such that $\|g_j\|_{0, s_j; K_2} \leq C$, with the same argument as before and using $(b)$ of Remark \ref{soave} we can prove that,

\[
\|u_j \|_{2, 3s_0-2; K_1} \leq C  \quad \text{ and } \quad u_j \to u \quad \text{in }C^{2, \alpha}(K_2)
\]
for any fixed $\alpha < 3s_0 -2$.

\end{rem}

\begin{teo}[{\cite[Proposition 1.1, Theorem 1.2]{HolderReg}}]\label{boundregRO}
Let $\Omega$ be a bounded $C^{1,1}$ domain, $g \in L^\infty(\Omega)$, let $u$ be a solution of
\[
\begin{cases}
(-\Delta)^s u = g & \text{in }\Omega, \\
u = 0 & \text{in }\R^N \setminus \Omega,
\end{cases}
\]
and $\delta(x) := \text{dist}(x, \partial \Omega)$. Then the following holds.
\begin{enumerate}[i)]
\item $u \in C^s(\R^N)$, 
\item the function $\frac{u}{\delta^s}_{|\Omega}$ can be continuously extended to $\overline \Omega$. Moreover, we have $\frac{u}{\delta^s} \in C^\alpha(\overline \Omega)$ and 
\[
\left\| \frac{u}{\delta^s}\right\|_{0, \alpha; \overline \Omega} \leq C |g|_{\infty; \Omega}
\]
for some $\alpha>0$ satisfying $\alpha < \min\{s, 1-s\}$. The constant $\alpha$ and $C$ depend only on $\Omega$ and $s$.
\end{enumerate}
\end{teo}
\begin{rem}
The constant $C$ appearing in Thereom \ref{boundregRO} is not, in general, bounded as $s \to 1^-$.
\end{rem}
\end{subsection}

\begin{subsection}{Extension properties for the fractional Laplacian}\label{extensionsubsect}

We introduce now the extension properties of $\mathcal{D}^s(\R^N)$ functions. All results are well known, in particular we follow the approach of \cite{FrLe1} and \cite{FrLe2}, were all the following definitions and theorems can be found. 

Let $ s \in (0,1)$ and $N > 2s$. We set $\R^{N+1}_+ := \R^N \times \R_+$, we write $z \in \R^{N+1}_+$ as $z= (x, y)$ where $x \in \R^N$ and $y >0$, and we set $|z|=|(x,y)|:=\sqrt{x^2+y^2}$. We define $\mathcal{D}^{1,s}(\R^{N+1}_+)$ as the completion of $C^\infty_c(\overbar {\R^{N+1}_+})$ with respect to the quadratic form 
\[
D^2_s(u) := d_s \int \int_{\R^{N+1}_+} y^{1-2s} |\nabla u|^2 \de x \de y, 
\]
where 
\[
d_s := \frac{2^{2s}}{2}\frac{\Gamma\left(s\right)}{\Gamma\left(1-s\right)}.
\]

Let $P_{N,s}: \R^{N+1}_+ \to \R$ be the function defined by
\[
P_{N,s}(x, y) := p_{N,s} \frac{y^{2s}}{(y^2+|x|^2)^{\frac{N+2s}{2}}}, 
\] 
where 
\[
p_{N,s} = \frac{\Gamma\left( \frac{N+2s}{2}\right)}{\pi^{\frac{N}{2}}\Gamma(s)}
\]
is such that $p_{N,s}\int_{\R^N}\frac{y^{2s}}{(y^2 + |x|^2)^{\frac{N+2s}{2}}} \de x = 1$ for every $y>0$. 

Given $u \in \mathcal{D}^s(\R^N)$, we define the extension $E_s u: \R^{N+1}_+ \to \R$ of $u$ as the function
\[
E_s u(x, y) = \int_{\R^N}P_{N,s}(x-\xi, y) u(\xi) \de \xi.
\]

The function $E_s u$ satisfy the following properties. The proofs are contained in \cite{FrLe1} and, as pointed out in \cite{FrLe2}, they can be extended to the case $N \geq 1 $. 
\begin{prop}[{\cite[Proposition 3.1]{FrLe1}}]\label{extension}
Let $0<s<1$. If $u \in \mathcal{D}^s(\R^N)$, then $E_s u \in \mathcal{D}^{1,s}(\R^{N+1}_+)$ and satisfies
\begin{equation}\label{energyext}
D^2_s(E_s u) = \| u\|_s^2.
\end{equation}

Moreover $E_s u $ is a weak solution to the problem 
\[
-\text{div }\left( y^{1-2s} \nabla U\right) = 0 \quad \text{in }\R^{N+1}_+,
\]
and satisfies 
\[
\lim_{\varepsilon \to 0^+}\|E_s u(\cdot , \varepsilon) - u \|_s =  0.
\]
In addition it holds that
\begin{equation}\label{dualconv}
\lim_{\varepsilon \to 0^+} \left \| \left(-d_s \varepsilon^{1-2s}\frac{\partial E_s u}{\partial y}(\cdot, \varepsilon)\right) - (-\Delta)^s u\right \|_{-s} = 0.
\end{equation}

Such extension is also unique: if a fuction $U$ is such that $U(x, 0) = u(x)$ in the trace sense and it satisfies the above properties, then $U = E_s u$. 
\end{prop}

\begin{prop}[{\cite[Proposition 3.2]{FrLe1}}]\label{traceteo}
Let $s \in (0,1)$. There exists a unique linear bounded operator $T$, such that $T: \mathcal{D}^{1,s}(\R^{N+1}_+) \to \mathcal{D}^s(\R^N)$ and $Tu(x) = u(x, 0)$ whenever $u \in C^\infty_c(\overbar{\R^{n+1}_+})$.
Moreover, the following inequality holds for all $u \in \mathcal{D}^{1,s}(\R^{N+1}_+)$,
\begin{equation}\label{traceinequality}
D^2_s(u) \geq \| Tu\|^2_s. 
\end{equation}
The equality in \eqref{traceinequality} is attained if and only if $ u = E_s f$ for some $f \in \mathcal{D}^s(\R^N)$. 
\end{prop}

Exploiting the aforementioned results, we can prove the following.
\begin{lemma}\label{extreg}
\begin{enumerate}[(i)]
\item If $u \in C^{0, s}(\R^N)$, then $E_s u \in C^2(\R^{N+1}_+) \cap C^{0, s}(\overline{\R^{N+1}_+})$;
\item If $u \in \mathcal{D}^s(\R^N)$, then 
\begin{equation}\label{delhopital}
\lim_{\varepsilon \to 0^+}\int_{\R^N}\left(-2sd_s \frac{E_s u (x, \varepsilon) - E_s u(x, 0)}{\varepsilon^{2s}}\right)\varphi(x) \de x = (u, \varphi)_s \quad \forall \varphi \in C^\infty_c(\R^N),
\end{equation}
\item Moreover, if $u \in H^s(\R^N)$, then
\begin{equation}\label{extlimder}
\lim_{\varepsilon \to 0^+}\int_{\R^N}\left(-d_s \varepsilon^{1-2s} \frac{\partial E_s u}{\partial y}(x,\varepsilon)\right)\varphi(x) \de x = (u, \varphi)_s \quad \forall \varphi \in C^\infty_c(\R^N). 
\end{equation}
\item For every $u \in H^s(\R^N)$ and $\varphi \in \mathcal{D}^{1,s}(\R^{N+1}_+)$ it holds
\[
d_s\int_{\R^{N+1}_+} y^{1-2s} \nabla E_s u \cdot \nabla \varphi \de x \de y = (u, T\varphi)_s.
\]
\end{enumerate}
\end{lemma}
\begin{proof}
For $(i)$, the interior regularity is a consequence of standard regularity theory for elliptic operators. The H\"older continuity up to the boundary can be shown explicitly via the definition of $E_su$. In fact, performing the change of variables $\frac{x-\xi}{y} = \eta$ we obtain
\begin{equation}\label{tecnic90}
E_su (x,y) = p_{N,s}\int_{\R^N}y^{2s}\frac{u(\xi)}{(y^2 + |x-\xi|^2)^{\frac{N+2s}{2}}}\de \xi = p_{N,s} \int_{\R^N} \frac{u(x+\eta y)}{(1+|\eta|^2)^{\frac{N+2s}{2}}}\de \eta,
\end{equation}
from which we infer that $|E_s u|_\infty \leq |u|_\infty$. We also observe that for any $x \in \R^N$ it holds that $\lim_{y\to0^+}E_s(x,y)=u(x)$ . Moreover
\[
\begin{aligned}
|E_su (x_1,y_1) - E_su (x_2,y_2)| \leq&\ p_{n,s} \int_{\R^N} \frac{|u(x_1+\eta y_1)- u(x_2+ \eta y_2)|}{(1+|\eta|^2)^{\frac{N+2s}{2}}}\de \eta \\
\leq&\ [u]_{0,s; \R^N}\left(|x_1-x_2|^s + C_1 |y_1-y_2|^s\right)\\
 \leq&\ C_2 | (x_1, y_1) - (x_2, y_2)|^s,
\end{aligned}
\]
where $C_1 = p_{N,s}(|B_1| + \frac{1}{s})$ and $C_2$ comes from $C_1$ and the elementary inequality
\[
(1 + t^s) \leq C(1 + t^2)^{\frac{s}{2}}, \ \ \ t>0,
\]
where $C>0$ is a constant which depends only on $s$. In particular $C_2$ depends only on $N$ and $s$.
At the end we deduce that
\[
\|E_s u\|_{0,s; \R^{N+1}_+} \leq C \|u\|_{0,s; \R^N},
\]
which completes the proof of $(i)$.
\vspace{10pt}

For $(ii)$, using \eqref{tecnic90} after a change of variables we get that 
\begin{equation}\label{tecnic88}
- \frac{E_s u (x, \varepsilon) - E_s u(x, 0)}{\varepsilon^{2s}} = p_{N,s}\int_{\R^N}\frac{u(x) - u(y)}{(\varepsilon^2 + |x-y|^2)^{\frac{N+2s}{2}}}\de y
\end{equation}
On the other hand, since $u$, $\varphi \in \mathcal{D}^s(\R^N)$ we have that 
\begin{equation}\label{tecnic89}
\lim_{\varepsilon \to 0^+}\frac{C_{n,s}}{2}\int_{\R^{2N}}\frac{(u(x) - u(y))(\varphi(x) - \varphi(y))}{(\varepsilon^2 + |x-y|^2)^{\frac{n+2s}{2}}}\de x \de y = (u, \varphi)_s.
\end{equation}
Moreover, applying Fubini-Tonelli's Theorem to the left-hand side of \eqref{tecnic89} we get that 
\[
\frac{C_{N,s}}{2}\int_{\R^{2N}}\frac{(u(x) - u(y))(\varphi(x) - \varphi(y))}{(\varepsilon^2 + |x-y|^2)^{\frac{n+2s}{2}}}\de x \de y = C_{N,s}\int_{\R^N} \left(\int_{\R^N}\frac{u(x) - u(y)}{(\varepsilon^2 + |x-y|^2)^{\frac{N+2s}{2}}}\de y \right) \varphi(x) \de x.
\]
The conclusion follows from \eqref{tecnic88} and \eqref{tecnic89}, noticing that $\frac{C_{N,s}}{p_{N,s}} = 2sd_s$.

In order to prove $(iii)$ we first observe that if $u \in \mathcal{D}^s(\R^N)$, then $y^{1-2s}\frac{\partial E_s u}{\partial y}(\cdot, y) \in L^2(\R^N)$ for a.e. $y \geq 0$. Indeed, since $E_su \in \mathcal{D}^{1,s}(\R^{N+1}_+)$ we have that 
\[
+ \infty > \int_{\R^{N+1}_+}y^{1-2s}|\nabla E_s u|^2 \de x \de y \geq \int_{\R^{N+1}_+} y^{1-2s}\left|\frac{\partial E_s u}{\partial y}(x, y)\right|^2 \de x \de y,
\] 
thus the claim is a consequence of the Fubini-Tonelli's theorem (see e.g. \cite[Theorem 8.8, Theorem 8.12]{rudin}).
Let $\varphi \in C^\infty_c(\R^N)$. From the previous observation, thanks to \eqref{dualconv} and since $u \in L^2(\R^N)$ we get that
\[
\begin{aligned}
&\lim_{\varepsilon \to 0^+}\int_{\R^N}\left(-d_s \varepsilon^{1-2s} \frac{\partial E_s u}{\partial y}(x,\varepsilon)\right)\varphi(x) \de x\\
 =&\  \lim_{\varepsilon \to 0^+}\int_{\R^N}|\xi|^{-s}\mathcal{F}\left(-d_s \varepsilon^{1-2s} \frac{\partial E_s u}{\partial y}(\cdot,\varepsilon)\right)|\xi|^s\overline{\mathcal{F}\varphi} \de \xi \\
=&\ \int_{\R^N}|\xi|^{-s}\mathcal{F}((-\Delta)^s u)|\xi|^s\overline{\mathcal{F}\varphi} \de \xi = \int_{\R^N} \mathcal{F}u |\xi|^{2s}\overline{\mathcal{F}\varphi} \de \xi = \int_{\R^N} u (-\Delta)^s \varphi \de x.
\end{aligned}
\] 
Then \eqref{extlimder} is a consequence of Lemma \ref{veryweak}, and the proof $(iii)$ is complete.

For $(iv)$, let us fix $u \in C^\infty_c(\R^N)$, $\varphi \in C^\infty_c\left(\overline{\R^{N+1}_+}\right)$, then it holds that
\[
\begin{aligned}
&d_s\int_{\R^{N+1}_+} y^{1-2s} \nabla E_s u(x,y) \cdot \nabla \varphi (x,y) \de x \de y  \\
=&\ d_s\int_{\{0<y<\varepsilon\}} y^{1-2s} \nabla E_s u(x,y) \cdot \nabla \varphi (x,y) \de x \de y  \\
+&\ d_s\int_{\{y>\varepsilon\}} y^{1-2s} \nabla E_s u(x,y) \cdot \nabla \varphi (x,y) \de x \de y = (I) + (II). 
\end{aligned}
\]
Since $E_su,\varphi \in \mathcal{D}^{1,s}(\R^{N+1}_+)$, by H\"older inequality and Lebesgue's convergence theorem, $(I) \to 0$ as $\varepsilon \to 0^+$. 
On the other hand, integrating by parts, since $\varphi \in C^{\infty}_c(\R^{N+1}_+)$ and $-y$ is the outward normal to $\partial (\R^{N+1}_+ \cap\{y>\varepsilon\})$, we get
\[
\begin{aligned}
(II) =& \int_{\{y=\varepsilon\}} \left(-d_sy^{1-2s} \frac{\partial E_s u(x,y)}{\partial y}\right) \varphi (x,y) \de x \\
-& d_s\int_{\{y>\varepsilon\}} \text{div}\left( y^{1-2s} \nabla E_s u(x,y)\right)\varphi (x,y) \de x \de y  \\
= & \int_{\{y=\varepsilon\}} \left(-d_sy^{1-2s} \frac{\partial E_s u(x,y)}{\partial y}\right) \varphi (x,y) \de x,
\end{aligned}
\]
where we used that in the interior of the domain the equation $-\text{div}(y^{1-2s}\nabla E_su)=0$ is satisfied pointwise.

As a consequence we get that
\[
(II) = \int_{\R^N} \left(-d_s\varepsilon^{1-2s} \frac{\partial E_s u}{\partial y}(x,\varepsilon)\right) (\varphi (x,\varepsilon) - \varphi(x, 0)) \de x + \int_{\R^N} \left(-d_s\varepsilon^{1-2s} \frac{\partial E_s u}{\partial y}(x, \varepsilon)\right) \varphi (x,0) \de x.
\]
Arguing as in the proof of $(ii)$ we get that 
\[
\begin{aligned}
&\bigg|\int_{\R^N} \left(-d_s\varepsilon^{1-2s} \frac{\partial E_s u}{\partial y}(x,\varepsilon)\right) (\varphi (x,\varepsilon) - \varphi(x, 0)) \de x\bigg| \\
\leq&\ \int_{\R^N} |\xi|^{-s}\left|\mathcal{F}\left(-d_s\varepsilon^{1-2s} \frac{\partial E_s u}{\partial y}(\cdot,\varepsilon)(\xi)\right)\right| |\xi|^s |\overline{\mathcal{F}(\varphi (\cdot,\varepsilon) - \varphi(\cdot , 0))}| \de \xi\\
\leq&\ \left\| -d_s\varepsilon^{1-2s} \frac{\partial E_s u}{\partial y}(\cdot,\varepsilon)\right\|_{-s}\|\varphi(\cdot, \varepsilon) - \varphi(\cdot, 0)\|_s. 
\end{aligned}
\]
Since $\varphi \in C^\infty_c \subset \mathcal{S}$, where $\mathcal{S}$ denote the Schwartz space, thanks to the properties of Fourier transform we readily see that $\|\varphi(\cdot, \varepsilon) - \varphi(\cdot, 0)\|_s \to 0$. Moreover, thanks to \eqref{dualconv} we have that $\left\| -d_s\varepsilon^{1-2s} \frac{\partial E_s u}{\partial y}(\cdot,\varepsilon)\right\|_{-s} \leq C$ for some costant $C>0$ which does not depends on $\varepsilon$. 
From this and \eqref{extlimder} follows that $\lim_{\varepsilon \to 0^+}(II) = (u, T\varphi)_s$, which concludes the proof when $u \in C^\infty_c(\R^N)$ and $\varphi \in C^\infty_c(\overline{\R^{N+1}_+})$. 

Arguing by density we can extend the result to every $\varphi \in \mathcal{D}^{1,s}(\R^{N+1}_+)$ thanks to \eqref{traceinequality}. To conclude we use again a density argument noticing that, due to \eqref{energyext}, if $u_j \to u$ in $\mathcal{D}^s(\R^N)$ then $E_s u_j \to E_s u$ in $\mathcal{D}^{1,s}(\R^{N+1}_+)$.
\end{proof}

To conclude, we recall the following version of the strong maximum principle.

\begin{prop}[{\cite[Remark 4.2]{Yannick}}]\label{yannick}
Let $u: \R^{N+1}_+\to \R$ be a weak solution of 
\[
\begin{cases}
-div (y^{1-2s}\nabla u) \geq 0 & B_R^+,\\
-y^{1-2s} \frac{\partial u}{\partial y} \geq 0 & \Gamma^0_R,\\
u \geq 0 & \Gamma^+_R,
\end{cases}
\]
where
\[
\begin{aligned}
&B^+_R = \{ (x,y) \in \R^{N+1}_+ \ | \ y>0, |(x,y)| < R\}\\
&\Gamma^+_R =  \{ (x,y) \in \R^{N+1}_+ \ | \ y\geq0, |(x,y)| = R\}\\
&\Gamma^0_R =  \{ (x, y) \in \R^{N+1}_+ \ | \ y=0, |x| < R\}
\end{aligned}
\]

Then either $u >0$ or $u \equiv 0$ on $B_R^+ \cup \Gamma^0_R$.
\end{prop}

\end{subsection}
\begin{subsection}{Miscellanea}
We recall here some useful properties of rescaled functions.
Let be $u \in X^s_0(D)$ where $D$ is a radially symmetric domain and $M>0$. We define
\[
\tilde u(x) = \frac{1}{M}u \left( \frac{x}{M^{\beta_s}}\right),
\]
where $\beta_s = \frac{2}{N-2s}$. The following result is elementary so we omit the proof, we just point out that all the proofs are based on the definition of $\tilde u$ and a change of variables in the integrals.
\begin{lemma}\label{As:rescaling}
We have:
\begin{enumerate}
\item[i)] $\|u\|^2_s = \| \tilde u\|^2_s$;
\item[ii)] $|u|^{2^*_s}_{2^*_s, D} = |\tilde u|^{2^*_s}_{2^*_s, M^{\beta_s} D}$;
\item[iii)] $|u|^2_{2, D} =  \frac{1}{M^{2s\beta_s}}|\tilde u|^2_{2, M^{\beta_s} D}$, 
\end{enumerate}
where $M^{\beta_s} D = \{M^{\beta_s} x  \ |\ x \in D\}$.
\end{lemma}

Next result is the fractional Strauss lemma for radial functions.
\begin{prop}[{\cite[Proposition 1]{choozawa}}]\label{strauss}
Let $N \geq 2$ and $s \in \left(\frac{1}{2}, 1\right)$. Then for all $u \in \mathcal{D}^s(\R^N)$ such that $u = u(|x|)$ it holds
\begin{equation}\label{straussineq}
\sup_{x \in \R^N \setminus\{0\}}|x|^{\frac{N-2s}{2}}|u(x)| \leq K_{N,s} \|u\|^2_s
\end{equation}
where
\[
K_{N,s} = \left(\frac{\Gamma(2s-1)\Gamma\left( \frac{N-2s}{2}\right)\Gamma \left(\frac{N}{2} \right)}{2^{2s}\pi^{\frac{N}{2}}\Gamma(s)^2\Gamma \left( \frac{N-2(1-s)}{2}\right)}\right).
\]
\end{prop}

Another fundamental result is the following topological lemma.

\begin{lemma}[{\cite[Lemma D.1]{FrLe1}}]\label{topolenz}
Let $x_1<x_2<x_3<x_4$ be real numbers. Suppose that $\gamma$, $\tilde \gamma:[0, 1] \to \overline{\R^2_+}$ are simple (i.e. injective) continuous curves such that
\[
\begin{aligned}
\gamma(0) = (x_1, 0), \quad \gamma(1) = (x_3, 0), \quad \gamma(t) \in \R^2_+\text{ for }t \in (0, 1),\\
\tilde \gamma(0) = (x_2, 0), \quad \tilde \gamma(1) = (x_4, 0), \quad \tilde \gamma(t) \in \R^2_+ \text{ for } t \in (0, 1).
\end{aligned}
\]
Then $\gamma$ and $\tilde \gamma$ intersect in $\R^2_+$, i.e., we have $\gamma(t) = \tilde \gamma(t_*)$ for some $t$, $t_* \in (0, 1)$. 
\end{lemma}

\end{subsection}
\end{section}

\begin{section}{Existence of sign changing solutions}\label{fracbnex}
In this Section we prove the existence of sign changing solutions for Problem \eqref{fracBrezis}. Since through all this section the parameters $s \in(0,1)$ and $\lambda \in (0, \lambda_{1,s})$ will be fixed, we will often omit them in the subscripts in order to simplify the notation. 

Let us define the functional $I = I_{s, \lambda}: X_0^s(\Omega)\to \R$ as
\[
I(u) := \frac{1}{2}(\|u\|_s^2-\lambda|u|^2_2)- \frac{1}{2^*_s}|u|_{2^*_s}^{2^*_s}
\]
Every critical point of $I$ is a solution of Problem \eqref{fracBrezis}, in fact we have that 
\[
I'(u) [\varphi] = (u, \varphi)_s - \lambda \int_\Omega u \varphi \de x - \int_\Omega |u|^{2^*_s-2}u\varphi \de x.
\]
Let us consider the Nehari manifold
\[
\mathcal{N}_{s, \lambda} = \mathcal{N} := \left\{u \in X^s_0(\Omega) \ |\ u \not \equiv 0, I'(u)[u]=0 \right\},
\]
and define $c_\mathcal{N}(s, \lambda) = c_\mathcal{N} := \inf_\mathcal{N} I(u)$. 
In the case of $\Omega = B_R$ we define also the radial Nehari manifold as
\[
\mathcal{N}_{s, \lambda;rad} = \mathcal{N}_{rad} := \{ u \in \mathcal{N} \ | \ u \text{ is radial}\}
\]
and we set $c_{\mathcal{N}_{rad}}(s, \lambda) = c_{\mathcal{N}_{rad}} := \inf_{\mathcal{N}_{rad}} I(u)$.

\begin{rem}
The functional $I$ is even, i.e. $I(-u) = I(u)$ for any $u \in X^s_0(\Omega)$, and hence, without loss of generality, if $u$ is a critical point of $I$, we can always assume that $u(0) \geq 0$.
\end{rem}

Let us consider also the functional $J_{s, \lambda}: X^s_0(\Omega)\setminus \{0\}\to \R$ defined by
\[
J_{s, \lambda}(u) = J(u) := \frac{\|u\|_s^2-\lambda|u|^2_2}{|u|^2_{2^*_s}},
\]
and set 
\[
S_{s, \lambda} := \inf_{X^s_0(\Omega)\setminus\{0\}}J(u). 
\]
\begin{rem}\label{valdinergia}
As proved in \cite[Section 4]{ValSer} if $N\geq 4s$, $s \in (0, 1)$, then $S_{s, \lambda} < S_s$, for every $\lambda \in (0, \lambda_{1,s})$. 
\end{rem}

\begin{prop}\label{posSol}
Let $N \geq 4s$. Then there exists $u^0 \in \mathcal{N}$ such that $I(u^0) = c_\mathcal{N}$ and $u^0 >0$ in $\Omega$. Furthermore, it holds that
\[
c_\mathcal{N} = \frac{s}{N}S_{s, \lambda}^{\frac{N}{2s}}.
\] 
If $\Omega = B_R$ then $u^0$ is also radially symmetric and decreasing as a function of the radius and $c_{\mathcal{N}} = c_{\mathcal{N}_{rad}}$.
\end{prop}
\begin{proof}

From the results of \cite[Proposition 20, Chapter 4]{ValSer}, we know that there exists a minimizer $\overbar u \in X^s_0(\Omega)\setminus \{0\}$ for the functional $J$. Moreover, since in general $J(|u|) \leq J(u)$, such a minimizer has to be non negative. After a rescaling (notice that $J(u) = J(Ku)$ for every $K>0$), we have that there exists $\hat K$ such that $u^0 := \hat K\overline u \in \mathcal{N}$ and $I'(u^0) = 0$, which implies that $u^0$ is a solution of Problem \eqref{fracBrezis}. Then we can apply the fractional strong maximum principle (see e.g., \cite[Corollary 4.2]{Musina}) and infer that $u^0 > 0$ in $\Omega$. We observe that if $u \in \mathcal{N}$ it holds
\[
I(u) = \frac{s}{N}(J(u))^{\frac{N}{2s}}.
\]
Therefore $u^0$ is also a minimizer of $I$ in $\mathcal{N}$, and we obtain that $c_\mathcal{N} = \frac{s}{N}S_{s, \lambda}^{\frac{N}{2s}}$. Finally, thanks to Theorem \ref{reginf}, $u^0 \in L^\infty(\R^N)$ and thus we can apply \cite[Theorem 4.1]{BirkWako} and obtain that $u^0$ is radially symmetric and decreasing. The proof is complete.
\end{proof}

If $u \in X^s_0(\Omega)$ we denote as usual by $u^+$, $u^-$, respectively, the positive and the negative parts of $u$, i.e. the functions defined by
\[
\begin{aligned}
&u^+ (x):= \max(u(x), 0) &&  x \in \Omega, \\
&u^-(x) := \max(-u(x), 0) &&  x \in \Omega,
\end{aligned}
\]
so that $u = u^+ - u^-$ and $|u| = u^+ + u^-$. We define the nodal Nehari set as
\[
\mathcal{M}_{s, \lambda} = \mathcal{M} := \{u \in X_0^s(\Omega) \ |\ u^\pm \not \equiv 0, I'(u)[u^\pm] = 0 \},
\]
and when $\Omega = B_R$ we define also the radial nodal Nehari set as
\[
\mathcal{M}_{s, \lambda; rad} = \mathcal{M}_{rad} := \{ u \in \mathcal{M} \ | \ u \text{ is radial}\}.
\]

Let $u \in \mathcal{M}$. By definition we have 
\[
0 = I'(u)[u^+] = (u,u^+)_s - \lambda|u^+|^2_2 - |u^+|^{2^*_s}_{2^*_s} = \|u^+\|_s^2 - (u^-, u^+)_s  - \lambda|u^+|^2_2 - |u^+|^{2^*_s}_{2^*_s},
\]
where, if $\Omega^+ = \{ u>0\}$ and $\Omega^- = \{ u < 0\}$, we have
\[
(u^-, u^+)_s = - \frac{C_{N,s}}{2}\int_{\Omega^+\times \Omega^-} \frac{u^+(x)u^-(y)}{|x-y|^{N+2s}}\de x \de y - \frac{C_{N,s}}{2}\int_{\Omega^-\times \Omega^+} \frac{u^+(y)u^-(x)}{|x-y|^{N+2s}}\de x \de y.
\]
Let us define the function $\eta_s : X^s_0(\Omega) \to [0, +\infty)$ as 
\begin{equation}\label{etadefinition}
\eta_s(u) = \eta(u) := \frac{C_{N,s}}{2} \int_{\R^{2N}}\frac{u^+(x)u^-(y)}{|x-y|^{N+2s}}\de x \de y.
\end{equation}
Notice that if $u\in \mathcal{M}$ then $\eta_s(u)>0$. Therefore if 
$u \in \mathcal{M}$ it holds 
\begin{equation}\label{eqcarattnehari}
\|u^\pm\|_s^2- \lambda|u^\pm|^2_2 = |u^\pm|^{2^*_s}_{2^*_s} - 2 \eta (u).
\end{equation}
Motivated by that, we define the functionals $f_{s, \lambda}^\pm: X^s_0(\Omega) \to \R$ as
\begin{equation}\label{caratterizzazione}
f^\pm_{s, \lambda}(u) = f^\pm(u) = :
\begin{cases}
0 & \text{if } u^\pm  = 0,\\
\frac{|u^\pm|^{2^*_s}_{2^*_s} - 2 \eta (u)}{\|u^\pm\|_s^2- \lambda|u^\pm|^2_2} & \text{if } u^\pm \neq 0,
\end{cases}
\end{equation}
and we can give a charachterisation of the nodal Nehari set as
\[
\mathcal{M} = \{ u \in X^s_0(\Omega) \ |\ f^+(u) = 1 = f^-(u) \}.
\]
\begin{rem}
We observe that $\mathcal{M} \subset \mathcal{N}$ and $\mathcal{M} \neq \emptyset$. The first fact is obvious, for the second we observe that for every sign-changing function $u \in X^s_0(\Omega)$ we can always find $\alpha, \beta >0$ such that $\alpha u^+-\beta u^- \in \mathcal{M}$  by solving the system
\[
\begin{cases}
\alpha^{2^*_s-2}|u^+|^{2^*_s}_{2^*_s} - \frac{\beta}{\alpha}2\eta(u) = \|u^+\|_s^2- \lambda|u^+|^2_2,\\
\beta^{2^*_s-2}|u^-|^{2^*_s}_{2^*_s} - \frac{\alpha}{\beta}2\eta(u) = \|u^-\|_s^2- \lambda|u^-|^2_2.
\end{cases}
\]
\end{rem}
Let us define
\[
c_\mathcal{M}(s,\lambda) = c_\mathcal{M} = \inf_{u \in \mathcal{M}}I(u),
\]
and similarly
\[
c_{\mathcal{M}_{rad}}(s,\lambda) = c_{\mathcal{M}_{rad}} = \inf_{u \in \mathcal{M}_{rad}}I(u).
\]

\setcounter{cla}{0}

\begin{teo}\label{conditionedexistence}
Let $N>2s$ and $\lambda \in (0, \lambda_{1,s})$. If 
\begin{equation}\label{condition}
c_{\mathcal{M}} < c_\mathcal{N} + S_s^{\frac{N}{2s}},
\end{equation}
there exists a sign-changing solution $u \in \mathcal{M}$ of Problem \eqref{fracBrezis} such that $I(u) = c_{\mathcal{M}}$.
If $\Omega = B_R$ and 
\begin{equation}\label{conditionrad}
c_{\mathcal{M}_{rad}} < c_\mathcal{N} + S_s^{\frac{N}{2s}},
\end{equation}
there exists a radial sign-changing solution $u \in \mathcal{M}_{rad}$ of Problem \eqref{fracBrezis} such that $I(u) = c_{\mathcal{M}_{rad}}$.
\end{teo}
\begin{proof}

We divide the proof in several steps. Let us set
\[
\mathcal{V}_{s, \lambda} = \mathcal{V} := \left\{ u \in X^s_0(\Omega) \ |\ |f^\pm(u) - 1 | < \frac{1}{2}\right\},
\]
where $f^\pm$ is defined in \eqref{caratterizzazione}.
Since $\lambda \in (0, \lambda_{1_s})$, if $u \in \mathcal{V}$ then $u^\pm \not \equiv 0$, and 
\begin{equation}\label{boundV}
|u^\pm|^{2^*-2}_{2^*} \geq \frac{S_s}{2}\left(1 - \frac{\lambda}{\lambda_{1,s}}\right)> 0. 
\end{equation}

\begin{cla}\label{Compactness}
If $(u_j) \subset \mathcal{V}$ is a sequence such that
\[
I(u_j) \to c \quad \text{and} \quad I'(u_j) \to 0\quad \text{in }X^{-s}_0(\Omega) \quad \text{as }j \to +\infty, 
\]
and if we assume that $c$ satisfies
\[
c < c_\mathcal{N} + \frac{s}{N}S_s^{\frac{N}{2s}}
\]
then $(u_j)$ is strongly relatively compact in $X^s_0(\Omega)$. 
\end{cla}

Indeed, arguing as in \cite[Theorem 1, Claim 2-3]{ValSer}, we have that $(u_j)$ is bounded in $X^s_0(\Omega)$ and $u_j \rightharpoonup u$ in $X^s_0(\Omega)$, where $u$ is a solution of Problem \eqref{fracBrezis}. By Sobolev embedding we have that $u_j \to u$ a.e., thus also $u_j^\pm \to u^\pm$ a.e. Moreover, since it holds that $\|u_j^\pm\|^2 \leq \|u_j\|^2$, we get that $u^+_j \rightharpoonup w_1$ and $u^-_j \rightharpoonup w_2$ in $X^s_0(\Omega)$. Also in this case this implies $u^+_j \to w_1$, $u^-_j \to w_2$ a.e. and, by uniqueness of the limit, we infer that $u^+ = w_1$ and $u^- = w_2$, i.e., $u_j^\pm \rightharpoonup u^\pm$ in $X_0^s(\Omega)$. 
In addition,
\begin{equation}
\label{passaggio1}
\begin{aligned}
o(1) =&\ I'(u_j)[u^\pm_j-u^\pm] -  I'(u)[u^\pm_j-u^\pm]  \\
=&\ \|u^\pm_j - u^\pm\|_s^2 - |u^\pm_j - u^\pm|^{2^*_s}_{2^*_s} + 2( \eta(u_j)- \eta(u)) + o(1).
\end{aligned}
\end{equation}
Since $u^\pm_j \to u^\pm$ a.e., by Fatou's lemma we have that 
\[
\eta(u) \leq \liminf_{j \to +\infty} \eta (u_j),
\]
while by \eqref{nonlocSob}
\[
S_s^{\frac{2^*_s}{2}}|u^\pm_j - u^\pm|^{2^*_s}_{2^*_s} \leq  \|u^\pm_j - u^\pm\|_s^{2^*_s},
\]
thus we obtain
\[
o(1) \geq \|u^\pm_j - u^\pm\|_s^2 (1 - S_s^{-\frac{2_s^*}{2}}\|u^\pm_j - u^\pm\|_s^{2_s^*-2}) + o(1).
\]
This implies that, setting $L^\pm = \lim_{j\to +\infty} \|u^\pm_j - u^\pm\|_s$, either $L^\pm = 0$ and $u^\pm_j \to u^\pm$ strongly in $X^s_0(\Omega)$, or $(L^{\pm})^2 \geq S_s^{\frac{N}{2s}}$. 
On the other hand, by the Brezis-Lieb Lemma (see e.g., \cite[Theorem 1]{brezislieb}) and \eqref{passaggio1}, we have
\[
I(u^\pm_j) = I(u^\pm_j-u^\pm) + I(u^\pm) + o(1) = \frac{s}{N}\|u^\pm_j - u^\pm\|_s^2 - ( \eta(u_j)- \eta(u)) + I(u^{\pm}) + o(1).
\]
Then we obtain
\[
\begin{aligned}
I(u_j) =&\ I(u_j^+) + I(u_j^-) + 2\eta(u_j) \\
=&\ I(u^+) + I(u^-) + 2\eta(u) + \frac{s}{N}(\|u^+_j - u^+\|_s^2 + \|u^-_j - u^-\|_s^2) + o(1)  \\
=&\ I(u) + \frac{s}{N}(\|u^+_j - u^+\|_s^2 + \|u^-_j - u^-\|_s^2) + o(1).
\end{aligned}
\]
If either $u^+_j \not \to u^+$ or $u^-_j \not \to u^-$ we hence obtain
\[
c = \lim_{j\to +\infty}I(u_j) \geq I(u) + \frac{s}{N}S_s^{\frac{N}{2s}} \geq c_\mathcal{N} + \frac{s}{N}S_s^{\frac{N}{2s}},
\]
which is a contradiction. The proof of Step \ref{Compactness} is complete.

Let us denote by $\mathcal{C}_P$ the cone of non-negative functions in $X^s_0(\Omega)$, and let $\Sigma$ be the set of maps $\sigma$ such that
\[
\begin{cases}
\sigma \in C(Q, X^s_0(\Omega)) & \text{where }Q = [0,1]\times[0,1] \\
\sigma(s, 0) = 0 & \forall s \in [0,1]\\
\sigma(0,t) \in \mathcal{C}_P & \forall t \in [0,1]\\
\sigma(1, t) \in -\mathcal{C}_P & \forall t \in [0,1]\\
f^+(\sigma(s,1))+f^-(\sigma(s,1)) \geq 2 & \forall s \in [0,1]\\
I(\sigma(s,1)) < 0 & \forall s \in [0,1]
\end{cases}
\]
We have that $\Sigma \neq \emptyset$. For istance, take $u \in \mathcal{M}$ and consider
\[
\sigma(s,t) = t((1-s)\alpha u^+ - s\alpha u^-);
\]
if $\alpha>0$ is large enough then $\sigma \in \Sigma$.

\begin{cla} \label{Miranda}
We claim that
\[
\inf_{\sigma \in \Sigma}\sup_{u \in \sigma(Q)}I(u) = \inf_{u \in \mathcal{M}}I(u).
\]
\end{cla}

Indeed, let be $\sigma \in \Sigma$. We have that
\[
f^+(\sigma(x))-f^-(\sigma(x))
\begin{cases}
\geq 0 & \forall x \in \{(0,t) \ | \ t \in [0,1]\}\\
\leq 0 & \forall x \in \{(1,t) \ | \ t \in [0,1]\}
\end{cases}
\]
and 
\[
f^+(\sigma(x))+f^-(\sigma(x))
\begin{cases}
\geq 2 & \forall x \in \{(s,1) \ | \ s \in [0,1]\}\\
< 2 & \forall x \in \{(s,0) \ | \ s \in [0,1]\}
\end{cases}
\]
so from Miranda's Theorem (see e.g., \cite{Mirranda}) we deduce that exists $x_0 \in Q$ such that
\[
\begin{aligned}
&f^+(\sigma(x_0))+f^-(\sigma(x_0)) = 2, \\
&f^+(\sigma(x_0)) = f^-(\sigma(x_0))
\end{aligned}
\]
hence $u_0 = \sigma(x_0) \in \mathcal{M}$.

On the other hand, let $\overbar u \in \mathcal{M}$. There exists a map $\overbar \sigma \in \Sigma$ such that
\begin{equation}
\label{maps}
\begin{cases}
&\overbar \sigma (Q) \subset A= \{\alpha \overbar u^+ - \beta \overbar u^- \ | \ \alpha, \beta \geq 0\} \\
&\exists x_0  \ | \ \overbar \sigma (x_0) = \overbar u.
\end{cases}
\end{equation}
We have already seen a map of that kind. It is obvious that 
\[
I(\overbar u) \leq \sup_{\overbar\sigma(Q)}I(u) \leq \sup_A I(u).
\]
On the other hand, since $\overbar u \in \mathcal{M}$ we have that
\[
I(\alpha \overbar u^+ - \beta \overbar u^-) = \left(\frac{\alpha^2}{2}- \frac{\alpha^{2^*_s}}{2^*_s}\right)|\overbar u^+|_{2^*_s}^{2^*_s} + \left(\frac{\beta^2}{2}- \frac{\beta^{2^*_s}}{2^*_s}\right)|\overbar u^-|_{2^*_s}^{2^*_s} - (\alpha-\beta)^2 \eta (\overbar u).
\]
Therefore
\[
I(\alpha \overbar u^+ - \beta \overbar u^-) \leq \frac{s}{N}|\overbar u|_{2^*_s}^{2^*_s} = I(\overbar u) \quad \forall \alpha, \beta,
\]
which implies that
\[
\max_{\overbar \sigma (Q)} I(u) = I(\overbar u)
\]
and this concludes the proof of Step \ref{Miranda}.

Consider a minimizing sequence $(\overbar u_j) \subset \mathcal{M}$ and denote by $\overbar \sigma_j$ the corresponding sequence of maps in the class $\Sigma$ satisfying \eqref{maps}. Then by Step \ref{Miranda} it holds that
\[
\lim_{j\to \infty} \max_{\overbar \sigma_j(Q)}I(u) = \lim_{j\to \infty} I(\overbar u_j) = c_\mathcal{M}.
\]

\begin{cla}\label{PSstrLemma}
There exists $(u_j) \subset X^s_0(\Omega)$ such that
\begin{equation}
\label{PSstr}
\begin{aligned}
&\lim_{j\to +\infty} \text{d}(u_j, \overbar \sigma_j(Q))=0,\\
&\lim_{j\to +\infty} I'(u_j) = 0 && \text{in }X^{-s}_0(\Omega),\\
&\lim_{j \to +\infty}I(u_j) = c_\mathcal{M}.
\end{aligned}
\end{equation}
\end{cla}
The proof is essentially the one contained in \cite[Theorem A]{CSS} and is based on a standard deformation lemma argument (see \cite[Lemma 1]{Hofer} and \cite[Theorem 3.4]{Struwe}), Step \ref{Compactness} and Step \ref{Miranda}.

\begin{cla}
Proof of the existence.
\end{cla}

Let $(\overbar u_j) \subset \mathcal{M}$ be a minimizing sequence for $c_\mathcal{M}$ and let $(u_j) \subset X^s_0(\Omega)$ be the associated sequence built in Step \ref{PSstrLemma}. By \eqref{PSstr} and recalling \eqref{maps} we know that there exists a sequence $(v_j)$ which can be wreitten in the form 
\begin{equation}\label{misc2}
v_j = \alpha_j \overbar u^+_j - \beta_j \overbar u^-_j \in \overbar \sigma_j(Q)
\end{equation}
where $\alpha_j$, $\beta_j \geq 0$, such that
\[
\text{dist}(u_j, v_j) \to 0.
\]
By the continuity of $I$ and \eqref{PSstr} we deduce that for every $\varepsilon>0$ there exists $\nu \in \N$ such that for every $j\geq \nu$ 
\[
I(\overbar u_j) < c_\mathcal{M}+\varepsilon \quad \text{and}\quad I(v_j) > c_\mathcal{M}-\varepsilon.
\]
This implies that 
\begin{equation}\label{misc1}
I(v_j) > I(\overbar u_j) - 2\varepsilon. 
\end{equation}
Using twice \eqref{misc1} with \eqref{misc2}, by an elementary computation, we get that
\begin{equation}\label{misc3}
\begin{aligned}
&I(\alpha_j \overbar u^+_j) > \frac{s}{N}|\overbar u^+_j|^{2^*_s}_{2^*_s} + (\beta_j^2-2\alpha_j\beta_j) \eta(\overbar u_j) - 2 \varepsilon,
\\
&I(\beta_j \overbar u^-_j) > \frac{s}{N}|\overbar u^-_j|^{2^*_s}_{2^*_s} + (\alpha_j^2-2\alpha_j\beta_j) \eta(\overbar u_j) - 2 \varepsilon.
\end{aligned}
\end{equation}
Notice that since $(\overbar u_j) \subset \mathcal{M} \subset \mathcal{V}$, from \eqref{boundV} we get that both $|\overbar u^+_j|^{2^*_s}_{2^*_s}$ and $|\overbar u^-_j|^{2^*_s}_{2^*_s}$ are bounded from below by a constant $C$. Moreover, arguing as in \cite[Claim 2]{ValSer}) we get that $(\overline u_j)$ is bounded in $X_0^s(\Omega)$. Then by definition and Cauchy's inequality we get that 
\[
\eta(\overline{u}_j) = -(\overline{u}_j^+, \overline{u}_j^-) \leq \|\overline{u}_j^+\|_s\|\overline{u}_j^-\|_s \leq \frac{1}{2}(\|\overline{u}^+_j\|_s^2 + \|\overline{u}_j^-\|_s^2) \leq \frac{\|\overline{u}_j\|_s^2}{2} \leq \infty.
\]
This implies that $\alpha_j$ and $\beta_j$ cannot vanish as $j \to +\infty$. In fact, assume that $\beta_j \to 0$ (the case $a_j \to 0$ can be treated in the same way), then from \eqref{misc3} it follows that 
\[
o(1) > \frac{s}{N}|\overbar u^-_j|^{2^*_s}_{2^*_s} + \alpha_j^2 \eta(\overbar u_j) - 2 \varepsilon +o(1) \geq C-2\varepsilon+o(1) 
\]
which is absurd if $\varepsilon$ is small enough.
This implies that $v_j^\pm \not \equiv 0$ and by \eqref{PSstr} this is true also for $u_j^\pm$. Since $I'(u_j) \to 0$, we get
\[
I'(u_j)[u^{\pm}_j] = \pm \|u^\pm_j\|_s^2 \mp\lambda |u^\pm_j|^2_2 \mp |\overbar u^\pm_j|^{2^*_s}_{2^*_s} \pm 2\eta(u_j) \to 0,
\]
thus $u_j \in \mathcal{V}$ for $j$ large enough. 
Thanks to hypothesis \eqref{condition} we can apply Step \ref{Compactness}, hence $u_j \to u \in X^s_0(\Omega)$, where $u$ is such that $I(u) = c_\mathcal{M}$ and $I'(u) = 0$. Then $u$ is a critical point for $I$ and is a solution of Problem \eqref{fracBrezis}. Since also $u_j^{\pm}\to u^\pm$ strongly in $X^s_0(\Omega)$ and $u_j \in \mathcal{V}$, we deduce that $u^\pm \not \equiv 0$. In particular, using $u^\pm$ as test functions we obtain
\[
0= I'(u)[u^\pm] = \pm\|u^\pm\|^2 \mp \lambda |u^\pm|^2_2 \mp |\overbar u^\pm|^{2^*}_{2^*} \pm 2\eta(u),
\]
that is, $u \in \mathcal{M}$. Then proof of the first part is then complete. Since the proof of the radial case is identical to the previous one, we omit it.
\end{proof}

In the next Lemma we show that condition \eqref{condition} and \eqref{conditionrad} are satisfied, respectively, when $N \geq 6s$ and $N>6s$.
\begin{lemma}\label{EnergyBound}
Let $s \in (0,1)$, $\lambda \in (0, \lambda_{1,s})$. If $N\geq 6s$ then
\[
c_{\mathcal{M}}< c_{\mathcal{N}} + \frac{s}{N}S_s^{\frac{N}{2s}}.
\]

If $\Omega = B_R$ and $N > 6s$, then
\[
c_{\mathcal{M}_{rad}}< c_{\mathcal{N}} + \frac{s}{N}S_s^{\frac{N}{2s}}.
\] 

\end{lemma}
\begin{proof}
Thanks to Step \ref{Miranda} of Theorem \ref{conditionedexistence} it suffices to show that
\[
\sup_{\alpha, \beta\geq 0} I(\alpha u^0 - \beta u_\varepsilon) < c_\mathcal{N} + \frac{s}{N}S_s^{\frac{N}{2s}},
\]
where $u^0$ is as in Proposition \ref{posSol} and $u_\varepsilon$ is as in Proposition \ref{stimebubble}.

First of all we notice that
\[
\frac{1}{2}\|\alpha u^0 - \beta u_\varepsilon\|_s^2 \leq \alpha^2 \|u^0\|_s^2 + \beta^2 \|u_\varepsilon\|_s^2.
\]
Thanks to the properties of $u^0$ and by \eqref{stime} we get that, if we take $\varepsilon < 1$, 
\begin{equation}\label{boundcrit}
\frac{1}{2}\|\alpha u^0 - \beta u_\varepsilon\|_s^2 \leq \left( 1-\frac{\lambda}{\lambda_{1,s}}\right)^{-1}S_{s,\lambda}^{\frac{N}{2s}}\alpha^2 + (S_s^{\frac{N}{2s}} + C_1\varepsilon^N)\beta^2 \leq C_2(\alpha+\beta)^2,
\end{equation}
where $C_1>0$ is as in \eqref{stime} and $C_2 >0$ depends on $N$, $s$, and $\lambda$, but not on $\varepsilon$.  
Let us focus now on the $L^{2^*_s}$-norm. By mean value theorem and the fundamental theorem of calculus we have
\[
\begin{aligned}
&\ |\alpha u^0 -\beta u_\varepsilon|^{2^*_s}_{2^*_s} - |\alpha u^0|^{2^*_s}_{2^*_s} - |\beta u_\varepsilon|^{2^*_s}_{2^*_s}  \\
= &\ -2^*_s(2^*_s-1)\int_0^1 \left[\int_\Omega |\tau \alpha u^0 - \mu \beta u_\varepsilon|^{2^*_s-2}\beta u_\varepsilon\alpha u^0 \de x \right]\de \tau,
\end{aligned}
\]
where $\mu = \mu(x)$ is a measurable function such that $0<\mu<1$. 
Hence 
\begin{equation}
\label{stimacrit}
\begin{aligned}
&\left||\alpha u^0 -\beta u_\varepsilon|^{2^*_s}_{2^*_s} -  |\alpha u^0|^{2^*_s}_{2^*_s} -  |\beta u_\varepsilon|^{2^*_s}_{2^*_s}\right|\\
\leq &\ C_* \left( \int_\Omega |\alpha u^0|^{2^*_s-1}|\beta u_\varepsilon| \de x + \int_\Omega |\beta u_\varepsilon|^{2^*_s-1}|\alpha u^0| \de x\right).
\end{aligned}
\end{equation}
where $C_* = 2^*_s(2^*_s-1)\max\{1, 2^{2^*_s-3}\}$. 
Since $u^0 \in L^\infty(\R^N)$, by Young's inequality and \eqref{stime} we get that
\[
\begin{aligned}
&\left||\alpha u^0 -\beta u_\varepsilon|^{2^*_s}_{2^*_s} -  |\alpha u^0|^{2^*_s}_{2^*_s} -  |\beta u_\varepsilon|^{2^*_s}_{2^*_s}\right| \\
\leq&\ C_* |\alpha u^0|^{2^*_s-1}_\infty | \beta u_\varepsilon|_1 + C_* |\alpha u^0|_\infty |\beta u_\varepsilon|_{2^*_s-1}^{2^*_s-1}  \\
\leq&\ \frac{\theta}{2} |\alpha u^0|^{2^*_s}_{\infty} +  C_\theta \beta^{2^*_s} \varepsilon^{N} + C_* |\alpha u^0|_\infty \beta^{2^*_s-1}\varepsilon^{\frac{N-2s}{2}},
\end{aligned}
\]
where $\theta \in (0,1)$ will be chosen later and $C_\theta>0$ is a function depending on $N$, $s$ and $\theta$, such that $C_\theta \to +\infty$ as $\theta \to 0^+$.
Applying Young's inequality again we obtain that for any sufficiently small $\varepsilon>0$  it holds
\[
\begin{aligned}
&\left||\alpha u^0 -\beta u_\varepsilon|^{2^*_s}_{2^*_s} -  |\alpha u^0|^{2^*_s}_{2^*_s} -  |\beta u_\varepsilon|^{2^*_s}_{2^*_s}\right| \\
\leq &\ \theta |\alpha u^0|^{2^*_s}_{\infty} + C_\theta\beta^{2^*_s} \varepsilon^{N} + C'_\theta\beta^{2^*_s}\varepsilon^{\frac{N(N-2s)}{N+2s}}\ \leq \theta |\alpha u^0|^{2^*_s}_{\infty} + C''_\theta\beta^{2^*_s} \varepsilon^{\frac{N(N-2s)}{N+2s}},
\end{aligned}
\]
therefore we have that
\[
\begin{aligned}
|\alpha u^0 -\beta u_\varepsilon|^{2^*_s}_{2^*_s} \geq &\ |\alpha u^0|^{2^*_s}_{2^*_s} - \theta|\alpha u^0|^{2^*_s}_\infty + |\beta u_\varepsilon|^{2^*_s}_{2^*_s} - C_\theta\beta^{2^*_s} \varepsilon^{\frac{N(N-2s)}{N+2s}} \\
=&\ \alpha^{2^*_s}\left(|u^0|^{2^*_s}_{2^*_s} - \theta|u^0|^{2^*_s}_\infty\right) + \beta^{2^*_s}\left(|u_\varepsilon|^{2^*_s}_{2^*_s} - C''_\theta\varepsilon^{\frac{N(N-2s)}{N+2s}}\right).
\end{aligned}
\]
Taking $\theta\in(0,1)$ such that $|u^0|^{2^*_s}_{2^*_s} - \theta|u^0|^{2^*_s}_\infty >0$ and $\tilde C>0$ such that
$|u^0|^{2^*_s}_{2^*_s} - \theta|u^0|^{2^*_s}_\infty \geq \tilde C>0$, with $\tilde C$ and $\theta$ depending on $N$, $\lambda$, $s$ and $u^0$ but not on $\varepsilon$, and then using again \eqref{stime}, we infer that
\[
\begin{aligned}
|\alpha u^0 -\beta u_\varepsilon|^{2^*_s}_{2^*_s} \geq&\ \tilde  C\alpha^{2^*} + \beta^{2^*}\left(S^{\frac{N}{2s}}_s - C_1\varepsilon^N - C''_\theta\varepsilon^{\frac{N(N-2s)}{N+2s}}\right)\\
\geq&\ \hat C(\alpha^{2^*} + \beta^{2^*})  \geq \hat C2^{1-2^*_s}(\alpha + \beta)^{2^*},
\end{aligned}
\]
for $\varepsilon$ small enough so that 
\[
S^{\frac{N}{2s}}_s - C_1\varepsilon^N - C''_\theta\varepsilon^{\frac{N(N-2s)}{N+2s}} \geq \hat C := \min\left\{\tilde C, \frac{S_s^{\frac{N}{2s}}}{2}\right\}.
\]
This implies, together with \eqref{boundcrit}, that there exists $C_3$, $C_4 >0$ which depends only on $N$, $s$ and $\lambda$ such that 
\[
I(\alpha u^0 - \beta u_\varepsilon) \leq C_3(\alpha + \beta)^2 (C_4 - (\alpha + \beta)^{2_s^*-2}).
\]
Therefore, if $(\alpha + \beta)^{2_s^*-2} \geq C_4$ we get that $I(\alpha u^0 - \beta u_\varepsilon) \leq 0$. Hence we can restrict to $\alpha$ and $\beta$ such that $\alpha + \beta \leq C_4^{\frac{1}{2_s^*-2}}$. Using again \eqref{stimacrit} we get that
\[
\begin{aligned}
I(\alpha u^0- \beta u_\varepsilon)\leq &\  \frac{\alpha^2}{2}(\|u^0\|_s^2-\lambda |u^0|^2_2) + \frac{\beta^2}{2}\|u_\varepsilon\|_s^2 -\alpha\beta\left[(u^0, u_\varepsilon)_s-\lambda \int_\Omega u^0u_\varepsilon \de x\right]  \\
 -&\lambda \frac{\beta^2}{2} |u_\varepsilon|_2^2 - \frac{\alpha^{2^*_s}}{2^*_s}|u^0|^{2^*_s}_{2^*_s}- \frac{\beta^{2^*_s}}{2^*_s}|u_\varepsilon|^{2^*_s}_{2^*_s}+C_5\int_\Omega |u_\varepsilon|^{2^*_s-1}u^0 \de x + C_5\int_\Omega |u^0|^{2^*_s-1}u_\varepsilon \de x,
\end{aligned}
\]
where $C_5$ depends on $C_*$ and $C_4$. 
Since $u^0$ is a solution of Problem \eqref{fracBrezis} and $u_0 \in L^\infty(\R^N)$, we obtain that
\[
\begin{aligned}
I(\alpha u^0- \beta u_\varepsilon) \leq& \left(\frac{\alpha^2}{2}- \frac{\alpha^{2^*_s}}{2^*_s}\right)|u^0|^{2^*_s}_{2^*_s} + \frac{\beta^2}{2}\|u_\varepsilon\|_s^2 +\alpha\beta |u^0|_{\infty;B_\rho(x_0)}^{2^*_s-1}|u_\varepsilon|_1 \\
 -&\ \lambda \frac{\beta^2}{2} |u_\varepsilon|_2^2 - \frac{\beta^{2^*_s}}{2^*_s}|u_\varepsilon|^{2^*_s}_{2^*_s}+C_5|u_\varepsilon|_{2^*_s-1}^{2^*_s-1}|u^0|_{\infty;B_\rho(x_0)} + C_5|u^0|_{\infty;B_\rho(x_0)}^{2^*_s-1}|u_\varepsilon|_1 
\end{aligned}
\]
where $x_0 \in \Omega$ and $\rho>0$ are as in the definition of $u_\varepsilon$.  
Since the maximum of the function $f(t) = \frac{t^2}{2}- \frac{t^{2^*_s}}{2^*_s}$ for $t \geq 0$ is attained for $t = 1$ we have $\left(\frac{\alpha^2}{2}- \frac{\alpha^{2^*_s}}{2^*_s}\right) \leq \frac{s}{N}$. Moreover, being $u^0 \in \mathcal{N}$ it holds $\frac{s}{N}|u^0|^{2^*_s}_{2^*_s} = I(u^0) = c_\mathcal{N}$ and thus we deduce that
\[
\begin{aligned}
I(\alpha u^0- \beta u_\varepsilon) \leq &\  c_\mathcal{N} + \frac{\beta^2}{2}\|u_\varepsilon\|_s^2 - \frac{\beta^{2^*_s}}{2^*_s}|u_\varepsilon|^{2^*_s}_{2^*_s} \\
 -&\ \lambda \frac{\beta^2}{2} |u_\varepsilon|_2^2 +C_5|u^0|_{\infty;B_\rho(x_0)}|u_\varepsilon|_{2^*_s-1}^{2^*_s-1} + C_6|u^0|^{2^*_s-1}_{\infty;B_\rho(x_0)}|u_\varepsilon|_1, 
\end{aligned}
\]
where $C_6$ comes from $C_5$ and $C_4$.
Now, using \eqref{stime} and since $\sup_{\beta \geq 0}\left(\frac{\beta^2}{2}- \frac{\beta^{2^*_s}}{2^*_s}\right) \leq \frac{s}{N}$ we get that
\[
\begin{aligned}
I(\alpha u^0- \beta u_\varepsilon) \leq &\ c_\mathcal{N}+ \frac{s}{N}S_s^{\frac{N}{2s}}+  C_7\varepsilon^{N} \\
+&\ C_7\lambda (\varepsilon^{N-2s} - C_8\varepsilon^{2s}) + C_7(|u^0|_{\infty;B_\rho(x_0)} + |u^0|_{\infty;B_\rho(x_0)}^{2^*_s-1})\varepsilon^{\frac{N-2s}{2}}. 
\end{aligned}
\]
Once again $C_7$ and $C_8$ depend only on $N$, $s$ and $\lambda$. 
Since $\varepsilon <<1$ and $\lambda \leq \lambda_{1,s}$, this leads to 

\begin{equation}\label{technicality2.0}
I(\alpha u^0-\beta u_\varepsilon) \leq  c_\mathcal{N} + \frac{s}{N}S_s^{\frac{N}{2s}}+ C_7(|u^0|_{\infty;B_\rho(x_0)} + |u^0|_{\infty;B_\rho(x_0)}^{2^*_s-1})\varepsilon^{\frac{N-2s}{2}}- C_9 \lambda \varepsilon^{2s}.
\end{equation}
Since $C_7$ and $C_8$ do not depend on $\varepsilon$ we can always take $\varepsilon$ such that, when $N > 6s$, 
\begin{equation}\label{anothertechnicality}
C_7(|u^0|_{\infty;B_\rho(x_0)} + |u^0|_{\infty;B_\rho(x_0)}^{2^*_s-1})\varepsilon^{\frac{N-2s}{2}}- C_9 \lambda \varepsilon^{2s} <0
\end{equation}
and thus we get the thesis.

If $N = 6s$ the sign of the left-hand side in \eqref{anothertechnicality} does not depend on $\varepsilon$ anymore.
Nevertheless, a careful analysis of the proof of Proposition \ref{stimebubble} (in particular of the estimates in \cite[Proposition 12]{Ser}, \cite[Proposition 21, 22]{ValSer}), and of the previous passages, shows that there exists $\overline \tau \in (0,1)$ which depends only on $N$ and $s$ but not on $\rho$ nor $x_0$ such that, taking $N = 6s$, $0<\rho <1$, $\mu = \rho$ and $\varepsilon = \tau \rho$ with $\tau \in (0, \overline \tau)$, inequality \eqref{technicality2.0} can be written as
\[
I(\alpha u^0-\beta u_\varepsilon) \leq  c_\mathcal{N} + \frac{1}{6}S_s^{3}+ (\tilde C_1(|u^0|_{\infty;B_\rho(x_0)} + |u^0|_{\infty;B_\rho(x_0)}^{2^*_s-1})- \tilde C_2 \lambda) \tau^{2s},
\]
where $\tilde C_1$ and $\tilde C_2$ depend only on $N$ and $s$ but not on $\rho$ nor $x_0$. At the end we obtain the desired result observing that, since $u^0$ decreases along the radii, the point $x_0$ and the ball $B_\rho(x_0)$ can be chosen near the boundary of $\Omega$ in such the way that $|u^0|_{\infty;B_\rho(x_0)}$ is so small that $$\tilde C_1(|u^0|_{\infty;B_\rho(x_0)} + |u^0|_{\infty; B_\rho(x_0)}^{2^*-1})- \tilde C_2\lambda<0.$$ 
In the case of radial functions, if $N=6s$ this last argument fails since we are forced to choose $x_0 = 0$ in the definition of the function $u_\varepsilon(x)$, while the rest of the proof applies verbatim and thus get existence of radial solutions just for $N>6s$. The proof is complete.
\end{proof}

From Theorem \ref{conditionedexistence} and Lemma \ref{EnergyBound} we obtain the following. 

\begin{teo}\label{exsimm}
Let $s \in (0,1)$, and let $N \geq 6s$, $\lambda \in (0, \lambda_{1,s})$. Then there exists a sign-changing solution $u \in \mathcal{M}$ of Problem \eqref{fracBrezis} such that $I(u) = c_{\mathcal{M}}$.
If $\Omega = B_R$, $N > 6s$ and $\lambda \in (0, \lambda_{1,s})$, there exists radial sign-changing solution $u \in \mathcal{M}_{rad}$ of Problem \eqref{fracBrezis} such that $I(u) = c_{\mathcal{M}_{rad}}$.
\end{teo}

\end{section}

\begin{section}{Asymptotic analysis of the energy as $\lambda\to 0^+$}\label{asintoticsect}

In this section we study the asymptotic behavior of the energy of least energy solutions of Problem \eqref{fracBrezis}, as $\lambda \to 0^+$. 

\begin{rem}\label{ss:as}
We observe that, as a straightforward consequence of the definitions of $S_s$, $S_{s, \lambda}$ and $\lambda_{1,s}$ we get that
\[
\left(1-\frac{\lambda}{\lambda_{1,s}}\right)S_{s} \leq S_{s,\lambda}\leq S_s,
\]
thus $S_{s, \lambda}\to S_s$ as $\lambda \to 0^+$.
Moreover, as a consequence of Remark \ref{sobolevcostbound} and Remark \ref{eigenbounds}, we get that for every $s_0 \in (0,1)$ it holds
\[
\lim_{\lambda \to 0^+ }\sup_{s \in [s_0, 1)} |S_s - S_{s, \lambda}|\leq \lim_{\lambda \to 0^+} \frac{\overline S}{\underline \lambda(s_0)}\lambda= 0,
\]
where $\overline S$, $\underline \lambda(s_0)$ are the two positive constants appearing in Remark \ref{sobolevcostbound} and Remark \ref{eigenbounds}.
\end{rem}

We have the following asymptotic result for $c_\mathcal{N}(s, \lambda)$, $c_\mathcal{M}(s, \lambda)$ and $c_{\mathcal{M}_{rad}}(s, \lambda)$.

\begin{lemma}\label{eneras}
Let $s \in (0,1)$ and let $N \geq 6s$. As $\lambda \to 0^+$ it holds 
\[
c_\mathcal{N}(s, \lambda) \to  \frac{s}{N}S_s^{\frac{N}{2s}} \quad \text{ and } \quad c_\mathcal{M}(s,\lambda) \to 2 \frac{s}{N}S_s^{\frac{N}{2s}}.
\]
Moreover, for every $s_0 \in (0, 1)$ it holds that 
\[
\hbox{(i)}\ \ \lim_{\lambda \to 0^+}\sup_{s \in [s_0, 1)}\left|\frac{s}{N}S_s^{\frac{N}{2s}} - c_{\mathcal{N}}(s, \lambda)\right| = 0 \quad \text{and}\quad \hbox{(ii)} \ \
\lim_{\lambda \to 0^+}\sup_{s \in [s_0, 1)}\left|2\frac{s}{N}S_s^{\frac{N}{2s}} - c_{\mathcal{M}}(s, \lambda)\right| = 0
\]
If $\Omega = B_R$ and $N>6s$ the same results hold for $c_{\mathcal{M}_{rad}}(s, \lambda)$.
\end{lemma}
\begin{proof} 
By Proposition \ref{posSol} we have $c_\mathcal{N}(s, \lambda) = \frac{s}{N}S_{s, \lambda}^{\frac{N}{2s}}$, and thus $(i)$ is a consequence of Remark \ref{ss:as}. 
In fact, thanks to Remark \ref{sobolevcostbound} and Remark \ref{eigenbounds}, we get that 
\begin{equation}\label{uniformasintenergy}
0 \leq \frac{s}{N}S_s^{\frac{N}{2s}} - c_{\mathcal{N}}(s, \lambda) \leq \frac{s}{N}\left(1 - \left(1- \frac{\lambda}{\lambda_{1,s}}\right)^{\frac{N}{2s}}\right)S_s^{\frac{N}{2s}} \leq C(s_0, \lambda)
\end{equation}
where $C(s_0, \lambda) >0$ is such that $C(s_0,\lambda) \to 0$ as $\lambda \to 0^+$. 
 
For $(ii)$, let us recall that by Lemma \ref{EnergyBound} it holds $c_\mathcal{M}(s, \lambda) < c_\mathcal{N}(s, \lambda) + \frac{s}{N}S_s^{\frac{N}{2s}}$.
Let $u_{s,\lambda}$ be a minimizer of $c_\mathcal{M}(s, \lambda)$. As seen in the proof of Step \ref{Miranda} of Theorem \ref{conditionedexistence}, we have that, for every $\alpha$, $\beta \in \R_+$ it holds
\[
I_{s,\lambda}(\alpha u^+_{s,\lambda} - \beta u^-_{s,\lambda}) \leq I_{s,\lambda}(u_{s,\lambda}) = c_\mathcal{M}(s,\lambda).
\]
On the other hand, we can always choose $\alpha$ and $\beta$ such that $\alpha u^+_{s,\lambda}$, $\beta u^-_{s,\lambda} \in \mathcal{N}_{s,\lambda}$, and since $\eta_s(u_{s, \lambda}) >0$ we get that
\[
I_{s,\lambda}(\alpha u^+_{s,\lambda} - \beta u^-_{s,\lambda}) > I_{s,\lambda}(\alpha u^+_{s,\lambda}) + I_{s,\lambda}(\beta u^-_{s,\lambda}) \geq 2 c_\mathcal{N}(s, \lambda).
\]
At the end we obtain
\begin{equation}\label{energyrelation}
2c_\mathcal{N}(s, \lambda) < c_\mathcal{M}(s, \lambda)< c_\mathcal{N}(s, \lambda) + \frac{s}{N}S_s^{\frac{N}{2s}},
\end{equation}
and the result easily follows. Indeed, since \eqref{energyrelation} can be rewritten as 
\[
0 < 2\frac{s}{N}S_s^{\frac{N}{2s}} - c_{\mathcal{M}}(s, \lambda) < 2 \left(\frac{s}{N}S_s^{\frac{N}{2s}} - c_{\mathcal{N}}(s, \lambda) \right),
\]
the limit is uniform with respect to $s \in [s_0, 1)$ thanks to \eqref{uniformasintenergy}.

The proof for the radial case is identical and we omit it. 
\end{proof}

\begin{lemma}\label{As:energas}
Let $s \in (0,1)$ and let $N \geq 6s$. Let $(u_{s, \lambda}) \subset \mathcal{M}_{s, \lambda}$ be a family of solutions of Problem \eqref{fracBrezis} such that $I_{s,\lambda}(u_{s,\lambda}) = c_\mathcal{M}(s, \lambda)$ and set $M_{s, \lambda, \pm} := |u^\pm_{s, \lambda}|_\infty$. As $\lambda \to 0^+$ we have:
\begin{enumerate}[(i)]
\item $\|u_{s,\lambda}^\pm\|_s^2 \to S_s^{\frac{N}{2s}} $;
\item $|u_{s,\lambda}^\pm|_{2_s^*}^{2_s^*} \to S_s^{\frac{N}{2s}} $;
\item $\lambda|u_{s, \lambda}^\pm|_2^2 \to 0$;
\item $\eta_s(u_{s,\lambda}) \to 0$ ;
\item $u_{s,\lambda} \rightharpoonup 0$ in $X^s_0(\Omega)$;
\item $M_{s, \lambda,\pm} \to + \infty$. 
\end{enumerate}
where $\eta_s$ is as in \eqref{etadefinition}.
When $N>6s$ the same results hold for a family $(u_{s, \lambda}) \subset \mathcal{M}_{s, \lambda; rad}$ of radial solutions of Problem \eqref{fracBrezis} such that $I_{s, \lambda}(u_{s,\lambda}) = c_{\mathcal{M}_{rad}}(s, \lambda)$.
Moreover, for every $s_0 \in (0,1)$ the limits $(i)-(iv)$ are uniform with respect to $s \in [s_0,1)$. 
\end{lemma}
\begin{proof}
Let $u_{s,\lambda} \in \mathcal{M}_{s, \lambda}$. From the definition of $\mathcal{M}_{s, \lambda}$ (see also \eqref{eqcarattnehari}), Theorem \ref{SobolevEmb} and the variational characterization of the eigenvalues, we get that
\[
\begin{aligned}
0 =&\ \|u^\pm_{s,\lambda}\|_s^2+2\eta_s (u_{s,\lambda}) - \lambda |u^\pm_{s,\lambda}|^2_2-|u^\pm_{s,\lambda}|_{2_s^*}^{2_s^*}\\
\geq&\ \|u^\pm_{s,\lambda}\|_s^2 \left(\left(1-\frac{\lambda}{\lambda_{1,s}}\right)- S_s^{-\frac{2_s^*}{2}} \|u^\pm_{s,\lambda}\|_s^{2_s^*-2}\right),
\end{aligned}
\]
which implies that
\begin{equation}\label{energyext6}
\liminf_{\lambda \to 0^+} \|u^\pm_{s,\lambda}\|_s^2 \geq S_s^{\frac{N}{2s}}.
\end{equation}
Since
\begin{equation}\label{energysplit}
\|u_{s,\lambda}\|^2_s = \|u^+_{s,\lambda}\|^2_s + \|u^-_{s,\lambda}\|^2_s + 4\eta_s(u_{s,\lambda}) \geq \|u^+_{s,\lambda}\|^2_s + \|u^-_{s,\lambda}\|^2_s
\end{equation}
it follows that
\begin{equation}\label{energyext1}
\liminf_{\lambda \to 0^+} \|u_{s,\lambda}\|_s^2 \geq 2S_s^{\frac{N}{2s}}.
\end{equation}
On the other hand, since $u_{s,\lambda} \in \mathcal{N}_{s, \lambda}$ and $I_{s, \lambda}(u_{s, \lambda}) = c_\mathcal{M}(s, \lambda)$, thanks to Lemma \ref{eneras} we have
\[
\lim_{\lambda \to 0^+}\frac{s}{N}|u_{s,\lambda}|_{2_s^*}^{2_s^*} = \lim_{\lambda \to 0^+}I_{s, \lambda}(u_{s,\lambda}) =\lim_{\lambda \to 0^+} c_\mathcal{M}(s,\lambda) = 2 \frac{s}{N}  S_s^{\frac{N}{2s}}, 
\]
namely
\begin{equation}\label{energyext2}
\lim_{\lambda \to 0^+}|u_{s,\lambda}|_{2_s^*}^{2_s^*} = 2S_s^{\frac{N}{2s}}.
\end{equation}
Using again that $u_{s, \lambda} \in \mathcal{N}_{s, \lambda}$ and the characterization of the eigenvalues we get that
\[
\left(1-\frac{\lambda}{\lambda_{1,s}}\right) \|u_{s,\lambda}\|_s^2 \leq |u_{s,\lambda}|_{2_s^*}^{2_s^*}.
\]
From the previous inequality, \eqref{energyext1} and \eqref{energyext2} it follows that
\begin{equation}\label{energyext3}
\lim_{\lambda \to 0^+}\|u_{s,\lambda} \|_s^2 = 2S_s^{\frac{N}{2s}}.
\end{equation}
Therefore, from \eqref{energyext2}, \eqref{energyext3} and since $u_{s,\lambda} \in \mathcal{N}_{s,\lambda}$ we deduce $(iii)$.

Now observe that, in view of \eqref{energyext6} and \eqref{energysplit}, we have
\[
\begin{aligned}
2S_s^{\frac{N}{2s}} &=& \lim_{\lambda \to 0^+} \|u_{s,\lambda} \|_s^2 \geq \limsup_{\lambda \to 0^+} \left(\|u^+_{s,\lambda}\|_s^2 + \|u^-_{s,\lambda}\|_s^2\right)\\
&\geq & \liminf_{\lambda \to 0^+} \|u^+_{s,\lambda}\|_s^2+ \liminf_{\lambda \to 0^+} \|u^-_{s,\lambda}\|_s^2 \geq 2S_s^{\frac{N}{2s}}.
\end{aligned}
\]
Hence we obtain that
\[
\lim_{\lambda \to 0^+} \left(\|u^+_{s,\lambda}\|_s^2 + \|u^-_{s,\lambda}\|_s^2 \right) = 2S_s^{\frac{N}{2s}}, 
\]
and, in view of \eqref{energyext6}, we deduce that
\[
\lim_{\lambda \to 0^+} \|u^\pm_{s,\lambda}\|_s^2 = S_s^{\frac{N}{2s}},
\]
which proves $(i)$, and $(iv)$ follows from \eqref{energysplit} and \eqref{energyext3}. Then, the relation $(ii)$, is a consequence of $(i)$, $(iii)$, $(iv)$, and the definition of $\mathcal{M}_{s,\lambda}$. 

For $(v)$, from \eqref{energyext3} we get that, up to a subsequence, there exists $u_s \in X^s_0(\Omega)$ such that $u_{s, \lambda} \rightharpoonup u_s$ in $X^s_0(\Omega)$ and $u_{s,\lambda} \to u_s$ a.e. as $\lambda \to 0^+$. Moreover, $u_s$ is a weak solution of the equation
\begin{equation}\label{eqsob}
(-\Delta)^s u_s = |u_s|^{2^*_s-2}u_s \quad \text{in }\Omega.
\end{equation}
In addition, by $(ii)$ and Fatou's Lemma we have
\begin{equation}\label{fatouext}
|u_s|^{2^*_s}_{2^*_s} \leq \liminf_{\lambda \to 0^+}|u_{s,\lambda}|^{2^*_s}_{2^*_s} = 2 S_s^{\frac{N}{2s}}.
\end{equation}
Suppose that both $u^+_s$ and $u^-_s$ are not trivial. Then, using $u^\pm_s$ as test functions in \eqref{eqsob} we get
\[
\|u_s^\pm\|^2_s + 2 \eta_s (u_s) = |u^\pm_s|^{2^*_s}_{2^*_s}.
\]
Therefore, by definition of $S_s$, we deduce that
\[
S_s \leq \frac{\|u^{\pm}_s\|^2_s}{|u^\pm_{s}|^{2}_{2^*_s}} = |u^{\pm}_s|^{2^*_s-2}_{2^*_s}-2\frac{\eta_s(u_s)}{|u^\pm_s|^2_{2^*_s}} < |u^{\pm}_s|^{2^*_s-2}_{2^*_s}.
\]
This implies that
\[
2S_s^{\frac{N}{2s}}< |u_s|^{2^*_s}_{2^*_s},
\]
which contradicts \eqref{fatouext}. As a consequence, either $u^+_{s} \equiv 0$ or $u^-_s \equiv 0$ i.e. $u_s$ is of constant sign. Assume for instance that $u_s \geq 0$. Hence, being $u_s$ a $L^\infty$ (see Lemma \ref{reginf}) non-negative solution of \eqref{eqsob}, then $(v)$ is a consequence of the fractional Pohozaev identity (see \cite[Corollary 1.3]{Pohozaev}). 

To prove the last point of the Lemma, we argue again by contradiction. Let $C>0$ be such that $M_{s,\lambda,+} \leq C$ for all $\lambda$. Then $|u_{s,\lambda}^+|^{2_s^*} \leq (M_{s,\lambda,+})^{2_s^*} \leq C^{2_s^*}$. Since by the previous point we have also that $u_{s,\lambda}^+ \to 0$ a.e, we can apply Lebesgue's convergence theorem to obtain that, as $\lambda \to 0^+$, 
\[
|u_{s,\lambda}^+|_{2_s^*}^{2_s^*} = \int_\Omega |u^+_{s,\lambda}|^{2_s^*}\de x \to 0,
\]  
which contradicts $(ii)$. The same proof holds for $M_{s,\lambda,-}$. 

As for the radial case, the proof is identical. Moreover, in view of Remark \ref{sobolevcostbound}, Remark \ref{eigenbounds} and Lemma \ref{eneras}, the limits $(i)-(iv)$ are uniform with respect to $s \in [s_0, 1)$.
\end{proof}
\end{section}

\begin{section}{Nodal components of the extension and nodal bounds}\label{nodalboundextsect}
In this section we study the nodal set of the extension of least energy sign-changing solutions of Problem \eqref{fracBrezis}. Let $u_{s, \lambda}$ be such a solution and let $W_{s, \lambda} = E_s u_{s,\lambda}$ be the extension of $u_{s, \lambda}$ (see Section \ref{SectionIntro} for the definition). Since $W_{s, \lambda}$ is continuous up to the boundary (see Lemma \ref{extreg}) and its restriction to $\R^N$ is $u_{s,\lambda}$, then also $W_{s, \lambda}$ changes sign. Next result states the number of nodal regions of $W_{s, \lambda}$, i.e. the number of the connected components of $\left\{x \in \R^{N+1}_+ \ | \ W_{s, \lambda}(x) \neq 0\right\}$, is two.

\begin{teo}\label{3nodal}
Let $s \in (0,1)$, $N\geq 6s$ and let $\Omega \subset \R^N$ be a smooth bounded domain. Then, there exists $\hat \lambda_s \leq \lambda_{1,s}$ such that for every $\lambda \in (0, \hat \lambda_s)$ the function $W_{s,\lambda}$ has exactly two nodal regions. Moreover, for every $s_0 \in (0,1)$ there exists $\hat \lambda(s_0)$ which depends on $N$ and $s_0$ but not on $s$ such that for every $\lambda \in (0, \hat \lambda(s_0))$ and $s \in [s_0, 1)$, previous result holds.  

\end{teo}
\begin{proof}
Let $\{\Omega_i\}$ be the set of the nodal regions of $W_{s,\lambda}$ in $\R^{N+1}_+$ and for each of them let us set $W^i_{s, \lambda} := W_{s, \lambda} \mathbbm{1}_{\overline{\Omega_i}}$, where $\mathbbm{1}_{\overline{\Omega_i}}$ is the characteristic function of $\overline{\Omega_i}$. 
First of all we notice that it cannot happen that $W^i_{s,\lambda}(x, 0 ) = 0$ for all $x \in \R^N$. Indeed, by \eqref{traceinequality} we have
\[
\begin{aligned}
D^2_s(W_{s,\lambda}) =&\ d_s \int_{\R^{N+1}_+}y^{1-2s} |\nabla W^i_{s,\lambda}|^2 \de x \de y + d_s \int_{\R^{N+1}_+}y^{1-2s} |\nabla (W_{s,\lambda}-W_{s, \lambda}^i)|^2 \de x \de y  \\
 \geq&\ d_s \int_{\R^{N+1}_+}y^{1-2s} |\nabla (W_{s,\lambda}-W_{s,\lambda}^i)|^2 \de x \de y \geq \|(W_{s,\lambda}-W^i_{s,\lambda})(x, 0)\|_s^2\\
= &\ \|W_{s,\lambda}(x, 0)\|_s^2. 
\end{aligned}
\]
Therefore, thanks to \eqref{energyext} we infer that $\|u_{s,\lambda}\|_s^2 = d_s \int_{\R^{N+1}_+}y^{1-2s} |\nabla (W_{s,\lambda}-W^i_{s,\lambda})|^2 \de x \de y$. Since the extension is unique, this implies that $W_{s,\lambda} = W_{s,\lambda} - W_{s,\lambda}^i$, that is, $W_{s, \lambda}^i \equiv 0$ in $\R^{N+1}_+$, which contradicts the definition of $\Omega_i$ and proves the claim.

As a consequence, we have that there is no nodal region such that $\overline{\Omega_i} \cap \Omega = \emptyset$. Using also \eqref{traceinequality}, we get that $W^i_{s, \lambda}(x, 0)$ is a non trivial function in $X^s_0(\Omega)$. Moreover, thanks the continuity of $W_{s, \lambda}$, the support of $W^i_{s, \lambda}(x, 0)$ turns out to be a non empty union of subsets of $\Omega$ where $u_{s, \lambda}$ has the same sign. In addition, for every $i, j$ the intersection between the supports of $W_{s, \lambda}^i(x, 0)$ and $W^j_{s, \lambda}(x, 0)$ consists of a set of null measure.
 
Since $u_{s,\lambda}$ is a solution of Problem \eqref{fracBrezis}, from $(iv)$ of Lemma \ref{extreg}, we obtain that for every $\phi \in \mathcal{D}^{1,s}(\R^{N+1}_+)$ such that $\phi(x, 0) \in X^s_0(\Omega)$ it holds
\begin{equation}\label{weakextension}
\begin{aligned}
&d_s\int_{\R^{N+1}_+} y^{1-2s} \nabla W_{s,\lambda}(x,y) \cdot \nabla \phi (x,y) \de x \de y = (u_{s,\lambda} (x), \phi(x, 0))_s \\
=&\ \lambda \int_\Omega u_{s,\lambda}(x) \phi(x, 0) \de x + \int_\Omega |u_{s,\lambda}(x)|^{2_s^*-2}u_{s,\lambda}(x) \phi (x, 0) \de x.
\end{aligned}
\end{equation}

Then, using $W^i_{s,\lambda}$ as a test function in \eqref{weakextension}, we have
\[
d_s\int_{\R^{N+1}_+} y^{1-2s} |\nabla W^i_{s,\lambda}|^2 \de x \de y = \lambda \int_\Omega |W^i_{s,\lambda}(x,0)|^2 \de x + \int_\Omega |W^i_{s,\lambda}(x,0)|^{2_s^*}\de x.
\]
Therefore, by \eqref{traceinequality}, the Sobolev inequality and the variational characterization of $\lambda_{1,s}$ we obtain 
\[
0 \leq D^2_s(W^i_{s,\lambda}) \left[ - \left(1 - \frac{\lambda}{\lambda_{1,s}}\right) + S_s^{-\frac{2_s^*}{2}}D^{\frac{2s}{N-2s}}_s (W^i_{s,\lambda})\right],
\]
and, as $\lambda \to 0^+$, we get that
\[
D^2_s(W^i_{s,\lambda}) \geq S_s^{\frac{N}{2s}}(1 + o(1)).
\]

At the end, let $K$ be the numer of nodal regions of $W_{s,\lambda}$, and assume that $K>2$. Thus by Lemma \ref{As:energas} and Proposition \ref{extension} we obtain that 
\begin{equation}\label{3nodaltecnic}
2 S_s^{\frac{N}{2s}} + o(1) = \|u_{s,\lambda}\|^2_s = D^2_s(W_{s,\lambda}) = \sum_{i=1}^K D^2_s (W_{s,\lambda}^i) \geq K S_s^{\frac{N}{2s}}(1 + o(1)),
\end{equation}
which gives a contradiction.

For the last point of the Theorem, let us fix $s_0 \in (0, 1)$. As pointed out in Remark \ref{eigenbounds} there exists $\underline \lambda(s_0)$ such that $\underline \lambda(s_0) \leq \lambda_{1,s}$ for every $s \in [s_0, 1)$. Then, when $\lambda \in \left( 0, \underline \lambda(s_0)\right)$, existence of solutions is ensured by Theorem \ref{exsimm}.
Moreover, as stated in Lemma \ref{As:energas} we have that 
\[
\sup_{s \in [s_0, 1)}\left|S_s^{\frac{N}{2s}}- \|u^\pm_{s, \lambda}\|_s^2\right| \leq C_1(\lambda) \quad \text{ and } \sup_{s \in [s_0, 1)}\eta_s(u_{s, \lambda}) \leq C_2(\lambda),
\]
where the functions $C_1,$ $C_2$ depend on $N$ and $s_0$ but not on $s$, and are such that $C_1(\lambda)$, $C_2(\lambda) \to 0$ as $\lambda \to 0^+$.  
Then when $\lambda < \underline \lambda(s_0)$, from \eqref{3nodaltecnic}, we deduce that   
\[
2 S_s^{\frac{N}{2s}} + C_3(\lambda) > K \left(1- \frac{\lambda}{\lambda_{1,s}}\right) S_s^{\frac{N}{2s}} \geq K \left(1- \frac{\lambda}{\underline \lambda(s_0)}\right) S_s^{\frac{N}{2s}}
\]
where $ C_3(\lambda)$ still depends only on $N$, $s_0$ and $\lambda$. Then, recalling Remark \ref{sobolevcostbound}, we obtain that
\[
2 +  2 o(\lambda) > K \left(1- \frac{\lambda}{\underline \lambda(s_0)}\right)
\]
where $o(\lambda)$ does not depend on $s$. Clearly, if $K >2$, there exists a sufficiently small $\tilde \lambda(s_0)$ such that a contradiction holds. Therefore the only possibility is that $K=2$ for all $\lambda \in (0,\hat\lambda(s_0))$, where $0<\hat\lambda(s_0)<\min\{\tilde \lambda(s_0), \underline \lambda(s_0)\}$. The proof is complete.

\end{proof}

The previous result holds true for least energy sign-changing solutions of Problem \eqref{fracBrezis} in general domains, but gives information just for the nodal set of their extensions. For radial solutions we can say more. 

\begin{teo}\label{twicechange}
Let $N>6s$, $s \in (0,1)$, and $R>0$. Let $u_{s,\lambda}$ be a least energy radial sign-changing solution for Problem \eqref{fracBrezis} in $B_R$. If $\lambda \in (0, \hat \lambda_s)$ where $\hat \lambda_s$ is the number given by Theorem \ref{3nodal}, then $u_{s, \lambda} = u_{s,\lambda}(r)$ changes sign at most twice.
Let $s_0 \in (0,1)$. Then the same result holds for every $s \in [s_0, 1)$ and $\lambda \in (0, \hat \lambda(s_0))$ where  $\hat \lambda(s_0)$ is the number given by Theorem \ref{3nodal}.
\end{teo}
\begin{proof}
The proof is the same as in \cite[Proposition 5.3]{FrLe2}, we recall it here for the sake of completeness. 
Suppose then that $u_{s,\lambda}$ changes sign more than twice, i.e. there exist
\[
0<r_1<r_2<r_3< r_4 < R
\]
such that $u_{s,\lambda}(r_i)u_{s,\lambda}(r_{i+1}) < 0$ for $i=1,2,3$. Taking if necessary $-u_{s,\lambda}$, we can always assume that $u_{s,\lambda}(r_1)>0$. 

Let now $W_{s, \lambda} = E_s u_{s,\lambda}$ be the extension of $u_{s,\lambda}$ to $\R^{N+1}_+$. Since $u_{s,\lambda} \in C^{0,s}(\R^N)$ (see Theorem \ref{boundregRO}) then $W_{s,\lambda} \in C^{0,s}(\overline{\R^{N+1}_+})$ (see Lemma \ref{extreg}). For our purposes it suffices that $W_{s,\lambda} \in C^{0}(\overline{\R^{N+1}_+})$.  
Since $u_{s,\lambda}$ is radially symmetric, then $W_{s,\lambda}$ is cylindrical symmetric in $\R^{N+1}_+$. Indeed, let $T \in SO(N)$ be a rotation around the origin, then
\[
\begin{aligned}
W_{s,\lambda}(Tx, y) &= p_{N,s}\int_{\R^N}\frac{u_{s,\lambda}(Tx + y\eta)}{(1+ |\eta|^2)^{\frac{N+2s}{2}}}\de \eta =\\
&= p_{N,s}\int_{\R^N}\frac{u_{s,\lambda}(Tx + yT\xi)}{(1+ |T\xi|^2)^{\frac{N+2s}{2}}}\de \xi = p_{N,s}\int_{\R^N}\frac{u_{s,\lambda}(T(x + y\xi))}{(1+ |T\xi|^2)^{\frac{N+2s}{2}}}\de \xi \\
&= p_{N,s}\int_{\R^N}\frac{u_{s,\lambda}(x + y\xi)}{(1+ |\xi|^2)^{\frac{N+2s}{2}}}\de \xi = W_{s,\lambda}(x,y).
\end{aligned}
\]

Let $\Omega_+,\Omega_- \subset {\R^{N+1}_+}$ be the nodal domains of $W_{s,\lambda}$, which are exactly two in view of Theorem \ref{3nodal}, and cylindrical symmetric with respect to the $y$-axis. Let us set
\[
\begin{aligned}
P &:= \{(|x|, y) \in \{r\geq 0\} \times \{y > 0\} \ | \ (x, y) \in \Omega_+ \} \\
G &:= \{(|x|, y) \in \{r\geq 0\} \times \{y > 0\} \ | \ (x, y) \in \Omega_- \}. 
\end{aligned}
\]
Thanks to the continuity of $W_{s, \lambda}$ we get that 
\[ 
\begin{aligned}
(r_i, \varepsilon) \in P \text{ for } i = 1, 3, \ \ \ (r_i, \varepsilon) \in G \text{ for } i = 2, 4,
\end{aligned}
\]
for all sufficiently small $\varepsilon>0$. 
Fixing $\varepsilon>0$, since $P, G$ are arcwise connected, this implies that there exist two continuous curves $\gamma^\varepsilon_\pm \in C^0([0,1]; \{r \geq 0\} \times \{y > 0\})$ such that 
\[
\begin{aligned}
&\gamma^\varepsilon_+(0) = (r_1, \varepsilon), \quad \gamma_+(1) = (r_3, \varepsilon), \quad \gamma_+(t) \in P \ \forall t \in [0,1],\\
&\gamma^\varepsilon_-(0) = (r_2, \varepsilon), \quad \gamma_-(1) = (r_4, \varepsilon), \quad \gamma_-(t) \in G \ \forall t \in [0,1].
\end{aligned}
\]

Moreover, since $W_{s, \lambda}$ is continuous up to the boundary, and since $W_{s, \lambda}(r_i) \neq 0$ for $i = 1, \ldots, 4$, we can always modify $\gamma^\varepsilon_\pm$ in order to obtain two curves $\gamma_\pm \in C^0([0, 1]; \{r \geq 0\} \times \{y\geq0\})$ such that they are injective and satisfy
\[
\begin{aligned}
&\gamma_+(0) = (r_1, 0), \quad \gamma_+(1) = (r_3, 0), \quad \gamma_+(t) \in P \ \forall t \in (0,1),\\
&\gamma_-(0) = (r_2, 0), \quad \gamma_-(1) = (r_4, 0), \quad \gamma_-(t) \in G \ \forall t \in (0,1).
\end{aligned}
\]

Now we can apply the topological Lemma \ref{topolenz} to conclude that $\gamma_+$ and $\gamma_-$ intersect in $\{r \geq 0\}\times \{y>0\}$, i.e. there exist $t_1, t_2 \in (0,1)$ such that $\gamma^+(t_1) = \gamma^-(t_2)$, which is absurd since by definition $\Omega^+ \cap \Omega^- = \emptyset$. The proof is complete.
\end{proof}

We state now another crucial preliminary result.

\begin{lemma}\label{lemmatecnico}
Let $s_0 \in (0,1)$ and let $\hat\lambda(s_0)$ be the number given by Theorem \ref{twicechange}. Let $s \in (s_0,1)$, $N>6s$, $R>0$ and let $u_{s,\lambda}$ be a least energy radial sign-changing solution of Problem \eqref{fracBrezis} in $B_R$ such that $u_{s,\lambda}$ changes sign exactly twice and $u_{s,\lambda}(0)\geq 0$. Let us denote by $0<r_{s}^1<r_s^2<R$ the nodes of $u_{s,\lambda}$. Let $W_{s, \lambda} = E_s u_{s,\lambda}$ be extension of $u_{s,\lambda}$. Then, for every $\overline \rho \in (r^2_s, R)$ such that $u_{s, \lambda}(\overline \rho)>0$, there exists $\overline \delta = \overline \delta(\overline \rho)>0$ such that 
\begin{equation}\label{asintotre1}
W_{s, \lambda} (x, y) \geq 0 \quad \forall |x| > \overline \rho, \quad \forall y \in (0, \overline \delta).
\end{equation}
\end{lemma}
\begin{proof}
 Let $W_{s, \lambda} = E_s u_{s,\lambda}$ be the extension of $u_{s,\lambda}$, then, thanks to Theorem \ref{3nodal} the function $W_{s, \lambda}$ has exactly two nodal regions. Let us denote them by
\[
\begin{aligned}
&\Omega ^+ = \{ (x, y) \in {\R^{N+1}_+} \ | \ W_{s, \lambda}(x, y) >0 \},\\
&\Omega ^- = \{ (x, y) \in {\R^{N+1}_+} \ | \ W_{s, \lambda}(x ,y) <0 \}.
\end{aligned}
\]
Moreover, since $u_{s, \lambda}$ is radially symmetric, then $W_{s, \lambda}$ is cylindrically symmetric and we set
\[
\begin{aligned}
&P = \{ (r, y) \in \{r\geq0\} \times \{y > 0\} \ | \ W_{s, \lambda}(r, y) >0 \},\\
&G = \{ (r, y) \in \{r\geq0\} \times \{y > 0\} \ | \ W_{s, \lambda}(r, y) <0 \}.
\end{aligned}
\]

Since we are assuming that $u_{s,\lambda} = u_{s,\lambda}(r)$ changes sign twice, there exist $\rho_1$, $\rho_2 >0$ such that 
\[
0 < \rho_1 < r^1_s<\rho_2 < r^2_s <\overline \rho <  R
\]
and $u_{s, \lambda}(\rho_1)>0$ while $u_{s,\lambda}(\rho_2)<0$.  
Recalling that $u_{s, \lambda}(\overline \rho) >0$, then, arguing as in Theorem \ref{twicechange}, there exists a simple curve $\gamma^+ \in C([0,1];\{r\geq0\} \times \{y \geq 0\})$ such that 
\[
\gamma_+(0) = (\rho_1, 0), \quad \gamma_+(1) = (\overline \rho, 0), \quad \gamma_+(t) \in P \ \ \ \forall t \in (0,1).
\]
In addition, without loss of generality, we can assume that $\gamma_+([0,1]) \cap \{r=0\}=\emptyset$.
We notice that since $W_{s, \lambda}$ is continuous up to the boundary and  $\gamma_+([0,1])$ is a compact subset of $ \{r\geq0\} \times \{y \geq 0\}$ there exists $\overline\delta>0$ such that $ \text{dist}(\gamma^+([0,1]),G)>\overline\delta>0$.  

Now, by Jordan's curve theorem the closed and simple curve whose support is $\gamma^+([0,1]) \cup ([\rho_1, \overline \rho]\times \{0\})$ divides the set $\{r \geq 0\} \times \{y \geq 0\}$ in two regions, a bounded one which we call $A_b$, and unbounded one $A_u$. Since $u_{s, \lambda}(\rho_2) < 0$ and $\rho_2 \in (\rho_1, \overline \rho)$, by continuity and since $W_{s, \lambda}$ possesses exactly two nodal regions, we deduce that $G \cap A_b \neq \emptyset$. This, together with Jordan's curve theorem implies that $G \subset A_b$. 

Let $(r, y) \in [\overline \rho, +\infty) \times (0, \overline \delta)$, we claim that $W_{s, \lambda}(r, y) \geq 0$. Indeed suppose that there exits a point $(r_0, y_0) \in [\overline \rho, +\infty) \times (0, \overline \delta)$ such that $W_{s, \lambda}(r_0, y_0) <0$. This implies that $(r_0, y_0) \in G \subset A_b$. On the other hand, since $\gamma^+(t) \not \in \{r \geq 0\} \times \{0\}$ when $t \neq 0, 1$, we have that $(r_0, 0) \in A_u$, and thus, as a further consequence of the Jordan  curve theorem, $\gamma^+$ intersects any curve $\gamma_*$ connecting $(r_0, y_0)$ and $(r_0, 0)$, whose support $\gamma_*([0,1])$ intersects $\{y=0\}$ just in $(r_0, 0)$. In particular, choosing as $\gamma_*$ the segment joining $(r_0, y_0)$ and $(r_0, 0)$, there exists $t_0$ such that $\gamma^+(t_0)$  lies in the interior of that segment. But this implies that $\text{dist}(\gamma^+(t_0), (r_0, y_0)) < \overline \delta$, and by the definition of $\overline\delta$ we deduce that $(r_0, y_0)$ cannot belong to $G$, which gives a contradiction. The proof is complete. 
\end{proof}
 \end{section}

\begin{section}{Uniform bounds with respect to $s$ and pre-compactness}\label{uniformboundsect}

We begin this section by recalling a general result of approximation for the fractional Laplacian that will be useful in the sequel. For the statement to be meaningful, we remark that the space $H^s(\R^N)$ and the operator $(-\Delta)^s$ can be defined via the Fourier transform also for $s \geq 1$, and such definitions are equivalent to the usual ones when $s \in (0,1)$ (see \cite[Proposition 3.4]{Hitch}).

\begin{lemma}[{\cite[Lemma 2.4]{uniquenondeg}}]
Let $s, \sigma \in (0,1]$ and $\delta > 2|\sigma -s|$. Then, for any $\varphi \in H^{2(\sigma + \delta)}(\R^N)$, it holds that
\[
|(-\Delta)^\sigma \varphi - (-\Delta)^s \varphi|_2 \leq C |\sigma -s| \|\varphi\|_{2(\sigma+\delta)},
\] 
for some $C = C(\sigma, \delta) >0$. 
\end{lemma}

\begin{rem}\label{valdi}
Let $\varphi \in C^\infty_c(\R^N)$. Since $C^\infty_c(\R^N) \subset H^s (\R^N)$ for all $s \geq 0$ as a consequence of previous Lemma we obtain that for all $\sigma \in (0,1]$,
\[
|(-\Delta)^\sigma \varphi - (-\Delta)^s \varphi|_2  \to 0 \quad \text{when } s \to \sigma.
\]
\end{rem}

In the following lemma we refine the estimate stated in Remark \ref{valdinergia}.

\begin{lemma}\label{asSobcost}
Let $0 < s_0 < s_1 \leq 1$. Let $N>4s_1$ and $\lambda \in (0, \underline \lambda(s_0))$, where $\underline \lambda(s_0)$ is the number given by Remark \ref{eigenbounds}. Then, for every $s \in (s_0, s_1)$, it holds
\[
S_{s,\lambda} \leq S_s - q(\lambda)
\]
where $q(\lambda) = q(\lambda, s_0, s_1, N, \Omega) >0 $ for $\lambda \in (0, \underline \lambda(s_0))$ and $q(\lambda) \to 0$ as $\lambda \to 0^+$. 
\end{lemma} 

\begin{proof}
Let $u^s_\varepsilon$ be as in Proposition \ref{stimebubble}. For every $s \in (s_0, s_1)$, by definition of $S_{s, \lambda}$ and \eqref{stime}, for $\varepsilon <1$, we have that 
\[
\begin{aligned}
S_{s,\lambda} \leq & \ \frac{\|u^s_\varepsilon\|^2_s - \lambda|u^s_\varepsilon|^2_2}{|u^2_\varepsilon|^{2}_{2^*_s}} \leq \frac{S_s^{\frac{N}{2s}} + C_1\varepsilon^{N-2s} - \lambda C_2 \varepsilon^{2s}}{\left(S_s^{\frac{N}{2s}}-C_3\varepsilon^{N}\right)^{\frac{2}{2^*_s}}} \\
\leq&\ S_s + C_4\varepsilon^{N-2s}-\lambda C_5 \varepsilon^{2s} \leq S_s + C_5\varepsilon^{2s_1}(C_6\varepsilon^{N-4s_1} - \lambda),
\end{aligned}
\]
where the constants do not depend neither on $s$ nor on $\varepsilon$. 
Then taking a fixed $\varepsilon_0$ small enough so that $C_6\varepsilon_0^{N-4s_1} - \lambda < 0$, we obtain the desired result with $q(\lambda) = C_5\varepsilon_0^{2s_1}(\lambda- C_6\varepsilon_0^{N-4s_1})$.
\end{proof}

In view of the previous results we obtain a uniform $L^\infty$ bound for the least energy positive solutions of Problem \eqref{fracBrezis}.

\begin{prop}\label{u0bound}
Let $0<s_0<s_1 \leq 1$ and $N>4s_1$. For every $s \in [s_0, s_1)$ and for any fixed $\lambda \in (0, \underline \lambda(s_0))$, where $\underline \lambda(s_0))$ is the number given by Remark \ref{eigenbounds}, let $u^0_{s,\lambda} \in \mathcal{N}_{s,\lambda}$ be such that $I_{s, \lambda}(u^0_{s, \lambda}) = c_\mathcal{N}(s, \lambda)$. It holds that 
\[
0 < \sup_{s \in [s_0,s_1)}|u^0_{s, \lambda}|_\infty< + \infty.
\]
\end{prop}
\begin{proof}
The first inequality is trivial. For the other inequality we argue by contradiction. Let us set $\delta_s := |u^0_{s, \lambda}|_\infty$ and assume that there exists $\sigma \in [s_0, s_1]$ and a sequence $(s_k) \subset (s_0, s_1)$ such that $\delta_{s_k} \to +\infty$ when $s_k \to \sigma$. From now on, in order to simplify the notation, we omit the subscript $k$.

Let consider the rescaled function
\[
v_{s, \lambda}(x) := \frac{1}{\delta_s}u^0_{s, \lambda}\left( \frac{x}{\delta_s^{\beta_s}}\right)
\]
where $\beta_s := \frac{2}{N-2s}$. Notice that $v_{s, \lambda} \in X_0^s\left(B_{\delta_s^{\beta_s} R}\right)$.

Since $u^0_{s, \lambda} \in \mathcal{N}_{s, \lambda}$ we have that 
\[
\begin{aligned}
c_\mathcal{N}(s, \lambda) =&\ \frac{1}{2}(\|u^0_{s,\lambda}\|_s^2 - \lambda|u^0_{s,\lambda}|^2_2) - \frac {1}{2^*_s}|u^0_{s,\lambda}|^{2^*_s}_{2^*_s} = \frac{s}{N} (\|u^0_{s,\lambda}\|_s^2 - \lambda|u^0_{s,\lambda}|^2_2)\\
\geq &\ \frac{s}{N}\left (1- \frac{\lambda}{\lambda_{1,s}}\right) \|u^0_{s,\lambda}\|^2_s \geq \frac{s_0}{N}\left (1- \frac{\lambda}{\underline \lambda(s_0)}\right) \|u^0_{s,\lambda}\|^2_s,
\end{aligned}
\] 
and thanks to Lemma \ref{As:rescaling}, Lemma \ref{eneras} and Remark \ref{sobolevcostbound} there exists $\tilde C >0$ such that 
\begin{equation}\label{Hsbound}
0 < \sup_{s \in [s_0, s_1)} \|v_{s, \lambda}\|^2_s = \sup_{s \in [s_0, s_1)} \|u^0_{s, \lambda}\|^2_s \leq \tilde C. 
\end{equation}

An easy computation shows that $v_{s, \lambda}$ is a weak solution of
\begin{equation}\label{ministry}
(-\Delta)^s v_{s, \lambda} = \frac{\lambda}{\delta_s^{2s \beta_s}}v_{s, \lambda} + |v_{s, \lambda}|^{2^*_s-2}v_{s, \lambda} \qquad \text{in }B_{\delta_s^{\beta_s} R}.
\end{equation}

As a consequence of that and since $|v_{s,\lambda}|_\infty = 1$, thanks to Remark \ref{gilbcont}, there exists $v_\lambda$ such that $v_{s, \lambda} \to v_\lambda$ in $C^{0, \alpha}_{loc}(\R^N)$ for any fixed $\alpha < s_0$ as $s \to \sigma$. Moreover, the convergence on compact subsets of $\R^N$ implies that $v_\lambda \not \equiv 0$. Indeed, recall that, as seen in Proposition \ref{posSol}, $u^0_{s,\lambda}$ is radial and achieves its maximum at the origin, hence $v_{s,\lambda}(0) = 1$ and thus $v_\lambda(0)=1$. 

Coming back to the original sequence $u^0_{s, \lambda}$, thanks to Lemma \ref{asSobcost} and being $ u^0_{s, \lambda} \in \mathcal{N}_{s, \lambda}$, we have
\[
\frac{s}{N}|u^0_{s, \lambda}|_{2^*_s}^{2^*_s} = I_{s, \lambda}(u^0_{s, \lambda}) = c_{\mathcal{N}}(s, \lambda) \leq \frac{s}{N} S_s^{\frac{N}{2s}}- q(\lambda).
\]
Therefore, by Fatou's Lemma and Lemma \ref{As:rescaling} we obtain

\begin{equation}\label{upextlinfproof}
|v_\lambda|_{2^*_{\sigma}}^{2^*_{\sigma}} \leq \liminf_{s \to \sigma}|v_{s, \lambda}|_{2^*_{s}}^{2^*_{s}} = \liminf_{s \to \sigma}|u^0_{s, \lambda}|_{2^*_{s}}^{2^*_{s}} \leq S_{\sigma}^{\frac{N}{2\sigma}} - \frac{N}{\sigma}q(\lambda).
\end{equation}

To reach a contradiction we need to obtain also a lower bound for the energy $|v_\lambda|_{2^*_{\sigma}}^{2^*_{\sigma}}$. 
To this end, let us fix $\varphi \in C^\infty_c(\R^N)$. We claim that, as $s \to \sigma$,
\begin{equation}\label{equazionelimite}
\begin{aligned}
(v_{s,\lambda}, \varphi)_s &= \int_{\R^N} v_{s, \lambda} (-\Delta)^s \varphi \de x \\
&= \int_{\R^N} v_{s, \lambda} ((-\Delta)^s \varphi- (-\Delta)^{\sigma} \varphi) \de x + \int_{\R^N} (v_{s, \lambda}- v_\lambda) (-\Delta)^{\sigma} \varphi \de x \\
&+\int_{\R^N} v_{\lambda} (-\Delta)^{\sigma} \varphi \de x = \int_{\R^N}  v_{\lambda} (-\Delta)^{\sigma} \varphi \de x + o(1).
\end{aligned}
\end{equation}

First of all, we point out that since $v_{s, \lambda} \in X^s_0(B_{\delta_s^{\beta_s}R}) \subset \mathcal{D}^s(\R^N)$, the first equality follows from Lemma \ref{veryweak}. Moreover, as a consequence of Remark \ref{valdi}, we have that $(-\Delta)^s \varphi - (-\Delta)^\sigma \varphi \to 0$ a.e. in $\R^N$ as $s \to \sigma$. Furthermore, thanks to Lemma \ref{bogdy} we have that, since $s \in [s_0, 1)$ and $\sigma \in [s_0,1]$, 
\[
\begin{aligned}
\left| (-\Delta)^s \varphi - (-\Delta)^\sigma \varphi  \right| \leq \frac{C}{(1 + |x|)^{N+2s}} +&  \frac{C}{(1 + |x|)^{N+2\sigma}} \leq 2C \frac{1}{(1 + |x|)^{N+2s_0}} \in L^1(\R^N),
\end{aligned}
\]
where $C>0$ depends on $N$ and $\varphi$ but not on $s$. 
Applying Lebesgue's dominated convergence theorem we get that
\[
\left|\int_{\R^N} v_{s, \lambda} ((-\Delta)^s \varphi- (-\Delta)^{\sigma} \varphi) \de x\right| \leq \int_{\R^N} \left|(-\Delta)^s \varphi- (-\Delta)^{\sigma} \varphi\right| \de x \to 0. 
\]
In a similar way, considering that $v_{s, \lambda} \to v_\lambda$ a.e., we prove that
\[
\left|\int_{\R^N} (v_{s, \lambda}- v_\lambda) (-\Delta)^{\sigma} \varphi \de x \right| \to 0
\]
and the claim is proved. In view of \eqref{ministry} and \eqref{equazionelimite} we obtain the relation
\begin{equation}\label{lbound:equ}
\int_{\R^N}  v_\lambda (-\Delta)^{\sigma} \varphi \de x = \int_{\R^N} |v_{\lambda}|^{2^*_{\sigma}-2}v_{\lambda}\varphi \de x \quad \forall \varphi \in C^\infty_c(\R^N).
\end{equation}

Now we have to consider two different cases: when $\sigma <1$, we easily deduce that $v_\lambda \in \D^{\sigma}(\R^N)$, as a straightforward consequence of Fatou's Lemma (recall that $C_{N,s}$ is continuos on $s \in [0,1]$). Indeed it holds
\[
\begin{aligned}
\|v_\lambda\|_\sigma^2 =&\ \frac{C_{N,\sigma}}{2}\int_{\R^{2N}}\frac{|v_{\lambda}(x)-v_{\lambda}(y)|^2}{|x-y|^{N+2\sigma}}\de x \de y  \\
\leq &\ \liminf_{s \to \sigma} \frac{C_{N,s}}{2}\int_{\R^{2N}}\frac{|v_{s,\lambda}(x)-v_{s,\lambda}(y)|^2}{|x-y|^{N+2s}}\de x   = \liminf_{s \to \sigma} \|v_{s, \lambda}\|_s^2 \leq C.
\end{aligned}
\]
Then we can apply Lemma \ref{veryweak} again, and by density, we obtain that $v_\lambda$ weakly satisfies the equation
\[
(-\Delta)^{\sigma} v_\lambda = |v_\lambda|^{2^*_{\sigma}-2}v_\lambda \quad \text{in}\ \R^N.
\]

Therefore, using $v_\lambda \in D^\sigma(\R^N)$ as a test function and since $v_\lambda \not \equiv 0$, we obtain
\[
S_\sigma \leq \frac{\|v_\lambda\|_\sigma^2}{|v_\lambda|_{2^*_{\sigma}}^2} = |v_\lambda|_{2^*_\sigma}^{2^*_\sigma-2},
\]
i.e., $S_\sigma^{\frac{N}{2\sigma}} \leq |v_\lambda|_{2^*_\sigma}^{2^*_\sigma}$, which, in view of \eqref{upextlinfproof}, readily gives a contradiction. 

When $\sigma =1$ the argument via Fatou's Lemma fails since $C_{N,s} \to 0$ as $s \to 1^-$, and a more careful approach is needed.  
First of all notice that since in this case $s \to 1^-$, then, passing if necessary to a subsequence, we can assume that $\frac{2}{3}<  s_0  < s <1$. Then by Remark \ref{gilbcont} we get that $v_{s, \lambda}\to v_\lambda$ in $C^{2,\gamma}_{loc}(\R^N)$ for $\gamma < 3 s_0-2$. In particular $v_\lambda \in C^2(\R^N)$. This allows us to integrate by parts in \eqref{lbound:equ}, thus obtaining that $v_\lambda$ weakly satisfies the equation
\[
-\Delta v_\lambda = |v_\lambda|^{2^*_{1}-2}v_\lambda \quad \text{in }\ \R^N.
\]
Since $v_\lambda \in C^2(\R^N)$, this actually implies that $v_{\lambda}$ satisfies $-\Delta v_{\lambda} = |v_{\lambda}|^{2^*_1-2}$ in a classical sense. 
Moreover, by \eqref{upextlinfproof} we have that $v_\lambda \in L^{2^*_1}(\R^N)$. Therefore we can apply \cite[Theorem 2, Corollary 3]{farina}, obtaining that $v_\lambda \in \D^1(\R^N)$ and $\|v_\lambda\|^2_1 = |v_\lambda|^{2^*_1}_{2^*_1}$. Hence, also in this case, we recover the estimate $S_1^{\frac{N}{2}} \leq |v_\lambda|_{2^*_1}^{2^*_1}$ and as before we get a contradiction.
\end{proof}

Thanks to Proposition \ref{u0bound} we can improve the inequality obtained in Lemma \ref{EnergyBound}. More precisely, the following result holds. 
\begin{cor}\label{ImprEnerBound}
Let $0 < s_0< s_1 \leq 1$, $N> 6s_1$ and $\lambda \in (0, \underline \lambda(s_0))$, where $\underline \lambda(s_0)$ is given by Remark \ref{eigenbounds}. Then there exists $Q(\lambda) >0$ such that for every $s \in [s_0, s_1)$ it holds that
\[
c_\mathcal{M}(s, \lambda) \leq c_\mathcal{N}(s, \lambda) + \frac{s}{N}S_s^{\frac{N}{2s}} - Q(\lambda).
\]
\end{cor}
\begin{proof}
At the end of the proof of Lemma \ref{EnergyBound} we have obtained, for every $s \in (0,1)$, the inequality
\[
c_\mathcal{M}(s, \lambda) \leq  c_\mathcal{N}(s, \lambda) + \frac{s}{N}S_s^{\frac{N}{2s}}+ C_1(s, \lambda)|u^0_{s, \lambda}|_{L^\infty(B_\rho(x_0))}\varepsilon^{\frac{N-2s}{2}}- \lambda C_2(s)\varepsilon^{2s},
\]
for any sufficiently small $\varepsilon>0$, and where $x_0 \in \Omega$ and $\rho>0$ are fixed.
As a matter of fact, by the previous proposition we have that $\sup_{s \in [s_0,1)}|u^0_{s, \lambda}|_\infty< C$, and, as seen in the proof of  Lemma \ref{EnergyBound}, the choice of the constants $C_1$ and $C_2$ depend on $s$ just for the value of $|u^0_{s, \lambda}|_\infty$. Therefore, from the proof of  Lemma \ref{EnergyBound}, we deduce that $C_1$, $C_2$ can be chosen in uniform way with respect to $s \in [s_0,1)$ and hence
\[
c_\mathcal{M}(s, \lambda) \leq  c_\mathcal{N}(s, \lambda) + \frac{s}{N}S_s^{\frac{N}{2s}}+ C_1(\lambda)\varepsilon^{\frac{N-2s_1}{2}}- \lambda C_2\varepsilon^{2s_1}.
\]
Taking $\varepsilon$ sufficiently small (depending on $\lambda$), the claim holds true with 
\[
Q(\lambda) =  C_2\varepsilon^{2s_1}\left(\lambda - C_1(\lambda) \varepsilon^{\frac{N-6s_1}{2}}\right).
\]

\end{proof}

We can now obtain an $L^\infty$ bound for the sequence of radial sign-changing solutions. 

\begin{lemma}\label{Linftyboundsc}
Let $\frac{1}{2}<s_0<s_1 \leq 1$. Let $N > 6s_1$, $R>0$ and let us fix $\lambda \in \left(0, \underline \lambda(s_0)\right)$, where $\underline\lambda(s_0)$ is the number given by Remark \ref{eigenbounds}.
For every $s \in (s_0,s_1)$ let $u_{s,\lambda} \in \mathcal{M}_{s,\lambda; rad}$ be a radial solution of Problem \eqref{fracBrezis} in $B_R$ such that $I_{s, \lambda}(u_{s, \lambda}) = c_{\mathcal{M}_{rad}}(s, \lambda)$. It holds that 
\[
0 < \sup_{s \in [s_0,s_1)}|u_{s, \lambda}|_\infty< + \infty.
\]
\end{lemma}
\begin{proof}
The first inequality is trivial. For the other inequality, as in Proposition \ref{u0bound}, we argue by contradiction. Let us set $\delta_s := |u_{s, \lambda}|_\infty$ and suppose that there exists $\sigma \in [s_0, s_1]$ and a sequence $s \to \sigma$ such that $\delta_s \to +\infty$.  Consider the rescaled function
\[
v_{s, \lambda}(x) = \frac{1}{\delta_s}u_{s, \lambda}\left( \frac{x}{\delta_s^{\beta_s}}\right), \quad x \in \R^N, 
\]
where $\beta_s = \frac{2}{N-2s}$. Clearly $v_s \in X_0^s\left(B_{\delta_s^{\beta_s} R}\right)$.
Since $u_{s, \lambda} \in \mathcal{N}_{s, \lambda}$ we have that 
\[
\begin{aligned}
c_\mathcal{M}(s, \lambda) =& \frac{1}{2}(\|u_{s,\lambda}\|_s^2 - \lambda|u_{s,\lambda}|^2_2) - \frac {1}{2^*_s}|u_{s,\lambda}|^{2^*_s}_{2^*_s} \\
=& \frac{s}{N} (\|u_{s,\lambda}\|_s^2 - \lambda|u_{s,\lambda}|^2_2) \geq \frac{s}{N}\left (1- \frac{\lambda}{\lambda_{1,s}}\right) \|u_{s,\lambda}\|^2_s.
\end{aligned}
\]
Hence, thanks to Lemma \ref{As:rescaling}, Lemma \ref{eneras} and Remark \ref{sobolevcostbound}, there exists $C>0$ such that  
\begin{equation}\label{energybound17}
0 < \sup_{s \in [s_0, s_1)} \|v_{s, \lambda}\|^2_s = \sup_{s \in [s_0, s_1)} \|u_{s, \lambda}\|^2_s \leq C. 
\end{equation}

An easy computation shows that $v_{s, \lambda}$ weakly satisfies  the equation
\[
(-\Delta)^s v_{s, \lambda} = \frac{\lambda}{\delta_s^{2s\beta_s}}v_{s, \lambda} + |v_{s, \lambda}|^{2^*_s-2}v_{s, \lambda} \qquad \text{in }B_{\delta_s^\beta R}.
\]

As a consequence of that, since $|v_{s_\lambda}|_\infty = 1$, by Remark \ref{gilbcont} it follows that as $s \to \sigma$ we have that $v_{s, \lambda} \to v_\lambda$ in $C^{0, \alpha}_{loc}(\R^N)$ for any fixed $\alpha < s_0$. 
Moreover, let us observe that thanks to Proposition \ref{strauss} and \eqref{energybound17}, for any $x_s$ such that $|u(x_s)| = \delta_s$ it holds 
\[
(\delta_s^{\beta_s} |x_s|)^{\frac{N-2s}{2}} = |x_s|^{\frac{N-2s}{2}} \leq |u_{s,\lambda}(x_s)| \leq K_{N,s}\|u_{s, \lambda}\|^2_s \leq \hat C,  
\] 
where $\hat C$ depends only on $s_0$. 
This implies that there exists a compact set $K \subset \subset \R^N$ such that $\delta^{\beta_s}_s x_s \in K$ for all $s$ sufficiently close to $\sigma$. Then, by the $C^{0,\alpha}_{loc}$-convergence, we get that there exists $\hat x \in K$ such that $v_\lambda(\hat x) = 1$, and thus $v_\lambda \not \equiv 0$. 

As in proof of Proposition \ref{u0bound} we obtain that $v_\lambda$ is a weak solution of the equation
\begin{equation}\label{tecn:sobolev}
(-\Delta)^{\sigma} v_\lambda = |v_\lambda|^{2^*_{\sigma}-2}v_\lambda \qquad \text{in }\ \R^N.
\end{equation}
Being $ u_{s, \lambda} \in \mathcal{N}_{s, \lambda}$, thanks to Corollary \ref{ImprEnerBound} we get that
\[
\frac{s}{N}|u_{s, \lambda}|_{2^*_s}^{2^*_s} = I_{s, \lambda}(u_{s, \lambda}) = c_{\mathcal{M}_{rad}}(s, \lambda) \leq \frac{s}{N}S_{s, \lambda}^{\frac{N}{2s}}+\frac{s}{N} S_s^{\frac{N}{2s}} - Q(\lambda).
\]
Thus, by Fatou's Lemma, Lemma \ref{As:rescaling} and Remark \ref{valdinergia} we obtain
\[
|v_\lambda|_{2^*_{\sigma}}^{2^*_{\sigma}} \leq \liminf_{s \to \sigma}|v_{s, \lambda}|_{2^*_{s}}^{2^*_{s}} = \liminf_{s \to \sigma}|u_{s, \lambda}|_{2^*_{s}}^{2^*_{s}} \leq \liminf_{s \to \sigma}S_{s, \lambda}^{\frac{N}{2s}} +S_{\sigma}^{\frac{N}{2\sigma}} - \frac{N}{\sigma} Q(\lambda)< {2} S_{\sigma}^{\frac{N}{2\sigma}}.
\]

On the other hand, for every $\sigma \in (0,1]$ if $u$ is a non trivial sign-changing solution of \eqref{tecn:sobolev} then $|u|^{2^*_\sigma}_{2^*_\sigma} > 2S_\sigma$. This is known when $\sigma =1$. When $\sigma < 1$, by definition of the Sobolev constant and testing \eqref{tecn:sobolev} with $u^\pm$ we obtain
\begin{equation}\label{nodalenerbound}
S_\sigma \leq \frac{\|u^\pm \|^2_\sigma}{|u^\pm|^2_{2^*_\sigma}} \leq |u^\pm|^{2^*_\sigma-2}_{2^*_\sigma} - \frac{\eta_\sigma(u)}{|u^\pm|^{2}_{2^*_\sigma}} < |u^\pm|^{2^*_\sigma-2}_{2^*_\sigma}.
\end{equation}
Therefore since $v_\lambda \not \equiv 0$, the only possibility is that $v_\lambda$ is of constant sign. 
Assume for istance that $v_\lambda \geq 0$. Then $v_{s, \lambda}^+ \to v_\lambda$ a.e. and by Fatou's Lemma we get that

\begin{equation}\label{fatoulemma12}
|v_\lambda|^{2^*_\sigma}_{2^*_s} \leq \liminf_{s \to \sigma}|v_{s, \lambda}^+|^{2^*_s}_{2^*_s}.
\end{equation}

Since $u_{s, \lambda} \in \mathcal{M}_{s, \lambda}$ and by definition of $S_{s, \lambda}$ we have
 
\begin{equation}\label{fatoulemma22}
S_{s, \lambda} \leq \frac{\|u^-_{s,\lambda}\|_s^2 - \lambda |u^-_{s, \lambda}|^2_2}{|u^-_{s, \lambda}|_{2^*_s}^2} = |u^-_{s,\lambda}|^{2^*_s-2}_{2^*_s} - \frac{\eta_s (u_{s,\lambda})}{|u^-_{s, \lambda}|_{2^*_s}^2} < |u^-_{s,\lambda}|^{2^*_s-2}_{2^*_s} = |v^-_{s, \lambda}|^{2^*_s-2}_{2^*_s}.
\end{equation}

which together with \eqref{fatoulemma12} implies

\begin{equation}\label{fatoulemma15}
|v_\lambda|^{2^*_\sigma}_{2^*_\sigma} + \liminf_{s \to \sigma}S_{s, \lambda}^{\frac{N}{2s}}  \leq \liminf_{s \to \sigma}\left( |v_{s,\lambda}^+|^{2^*_s}_{2^*_s} + |v_{s,\lambda}^-|^{2^*_s}_{2^*_s}\right) = \liminf_{s \to \sigma}|v_{s,\lambda}|^{2^*_s}_{2^*_s}.
\end{equation}
On the other hand by Corollary \ref{ImprEnerBound} we have that
\[
|v_{s, \lambda}|^{2^*_s}_{2^*_s} = |u_{s,\lambda}|^{2^*_s}_{2^*_s}\leq S_{s, \lambda}^{\frac{N}{2s}} +S_{s}^{\frac{N}{2\sigma}} - Q(\lambda),
\]
and recalling that $Q$ does not depend on $s$ and $S_s$ is continuous with respect to $s$, we deduce that
\begin{equation}\label{fatoulemma16}
\liminf_{s \to \sigma}|v_{s, \lambda}|^{2^*_s}_{2^*_s} < S_{\sigma}^{\frac{N}{2\sigma}}+ \liminf_{s \to \sigma}S_{s,\lambda}^{\frac{N}{2s}}.
\end{equation}

Joining \eqref{fatoulemma15} and \eqref{fatoulemma16} we obtain that $|v_\lambda|^{2^*_\sigma}_{2^*_\sigma} < S_{\sigma}^{\frac{N}{2\sigma}}$, and this is a contradiction to the fact that every non trivial solution $u$ of \eqref{tecn:sobolev} must satisfy $|u|^{2^*_\sigma}_{2^*_\sigma} \geq S^\frac{N}{2\sigma}_\sigma $.
\end{proof}

Thanks to this uniform $L^\infty$-bound on sign-changing solutions of Problem \eqref{fracBrezis}, we have the following result.

\begin{teo}\label{sconverg}
Let $\frac{1}{2}<s_0<s_1 \leq 1$. Let $N > 6s_1$, $R>0$ and let $\hat \lambda(s_0) $ is the number given by Theorem \ref{twicechange}. For any fixed $\lambda \in (0, \hat \lambda(s_0))$, let $(u_{s,\lambda})_s$ be a family, $s \in [s_0, s_1)$, of radial sign-changing solutions of Problem \eqref{fracBrezis} with $I_{s, \lambda}(u_{s, \lambda}) = c_{\mathcal{M}_{rad}}(s, \lambda)$. 
Assume that $s \to \sigma$, for some $\sigma \in [s_0, s_1]$. Then, for any fixed $\alpha < s_0$, up to a subsequence,  as $s \to \sigma$, we have $u_{s, \lambda} \to u_{\sigma, \lambda}$ in $C^{0,\alpha}_{loc}(B_R)$. Moreover $u_{\sigma, \lambda} \in X^\sigma_0(B_R)$ and is a weak non trivial solution of
\[
\begin{cases}
(-\Delta)^\sigma u_{\sigma,\lambda} = \lambda u_{\sigma, \lambda} + |u_{\sigma, \lambda}|^{2^*_\sigma-2}u_{\sigma, \lambda} & \text{in }B_R\\
u_{\sigma, \lambda} = 0 & \text{in }\R^N \setminus B_R
\end{cases}
\]
In addition
\[
\lim_{s \to \sigma} I_{s, \lambda}(u_{s, \lambda}) = I_{\sigma, \lambda}(u_{\sigma, \lambda}).
\]
\end{teo}
\begin{proof}
Let us fix $\frac{1}{2}<s_0<s_1 \leq 1$ and $\lambda \in (0, \hat \lambda(s_0))$. Let $(u_{s, \lambda})_s$ be a family of least energy radial sign-changing solutions of Problem \eqref{fracBrezis}, where $s \in [s_0, s_1)$. 
By Corollary \ref{ImprEnerBound} we have that $(u_{s, \lambda})$ is a bounded family in $X^{s_0}_0(B_R)$, and thus, up to a subsequence, there exists $u_{\sigma, \lambda} \in X_0^{s_0}(B_R)$ such that:
\[
\begin{aligned}
&u_{s, \lambda} \rightharpoonup u_{\sigma, \lambda} & &\text{in } X^{s_0}_0(B_R),\\
&u_{s, \lambda} \to u_{\sigma, \lambda} & &\text{in } L^{p}(B_R), \forall p \in (1, 2^*_{s_0}),\\
&u_{s, \lambda} \to u_{\sigma, \lambda} & &\text{a.e. in } \R^N.
\end{aligned}
\] 

On the other hand, thanks to Remark \ref{soave} and Lemma \ref{Linftyboundsc}, it holds that $(u_{s, \lambda})_s$ is a bounded sequence in $C^{0, s_0}(K)$ for every fixed $K \subset \subset B_R$. Then, up to a subsequence, we get that 
\[
\begin{aligned}
&u_{s, \lambda} \to u_{\sigma, \lambda} & &\text{in } C^{0, \alpha}_{loc}(B_R),\\
\end{aligned}
\]
for every fixed $0<\alpha < s_0$, thanks to Remark \ref{gilbcont}.

Exploiting the uniform $L^\infty$-bound given by Lemma \ref{Linftyboundsc} and since $u_{s, \lambda} \equiv 0$ in $\R^N \setminus B_R$, we obtain that
\begin{equation}\label{sconverg1}
\begin{aligned}
&u_{s, \lambda} \to u_{\sigma, \lambda} \text{ in } L^{p}(\R^N), \forall p >1 \\
&|u_{s, \lambda}|_{2^*_s}^{2^*_s} \to |u_{\sigma,\lambda}|_{2^*_\sigma}^{2^*_\sigma}
\end{aligned}
\end{equation}

Arguing as in the proof of Proposition \ref{u0bound} we obtain that $u_{\sigma, \lambda}\in \D^\sigma(\R^N)$, both when $\sigma <1$ or $\sigma =1$ Since $u_{\sigma, \lambda} \in L^{2}(\R^N)$ and $u_{\sigma, \lambda} \equiv 0$ in $\R^N \setminus B_R$ (because $u_{\sigma, \lambda} \in X^{s_0}_0(B_R)$) we conclude that $u_{\sigma,\lambda} \in X^\sigma_0(B_R)$. Alternatively, a simpler way of proving that $u_{\sigma,\lambda} \in X^\sigma_0(B_R)$ is to use the Fourier transform definition of $\|\cdot\|_s$, unifying both cases. In fact, since $u_{s, \lambda}\to u_{\sigma,\lambda}$ in $L^2(\R^N)$, using the characterization via the Fourier transform of the Sobolev spaces and Fatou's Lemma we get that $u \in H^\sigma(\R^N)$ and, as seen before, we have $u_{\sigma, \lambda} \equiv 0$ in $\R^N \setminus B_R$. Therefore $u_{\sigma, \lambda} \in X^{\sigma}_0(B_R)$ and,  as in the proof of Proposition \ref{u0bound}, we get that $u_{\sigma, \lambda}$ weakly satisfies
\begin{equation}\label{technicaleq27}
\begin{cases}
(-\Delta)^\sigma u_{\sigma,\lambda} = \lambda u_{\sigma, \lambda} + |u_{\sigma,\lambda}|^{2^*_\sigma -2}u_{\sigma, \lambda} & \hbox{in}\ B_R \\
u_{\sigma, \lambda} = 0 & \hbox{in}\  \R^N \setminus B_R.
\end{cases}
\end{equation}

Thanks to our choice of $s_0$ it holds that
\[
\frac{s}{N}|u_{s,\lambda}|^{2^*_s}_{2^*_s} = c_{\mathcal{M}_{rad}}(s, \lambda) \geq 2c_\mathcal{N}(s, \lambda) \geq 2\frac{s}{N}\left(1-\frac{\lambda}{\lambda_{1,s}}\right) S_s \geq C,
\]
for some $C >0$ not depending on $s$, and thus, using \eqref{sconverg1} we obtain that $|u_{\sigma,\lambda}|^{2^*_\sigma}_{2^*_\sigma} \geq C$. This implies that $u_{\sigma, \lambda}$ is not trivial and the first part of the proof is complete. 

For the second part, using $u_{\sigma,\lambda}$ as test function in the equation \eqref{technicaleq27}, and using \eqref{sconverg1} we get that 
\[
\|u_{s, \lambda}\|^2_s = \lambda|u_{s, \lambda}|_2^2 + |u_{s, \lambda}|_{2^*_s}^{2^*_s} \to \lambda|u_{\sigma, \lambda}|_2^2 + |u_{\sigma, \lambda}|_{2^*_\sigma}^{2^*_\sigma} = \|u_{\sigma, \lambda}\|^2_\sigma
\] 
which readily implies that
\[
I_{s, \lambda}(u_{s, \lambda}) \to I_{\sigma, \lambda}(u_{\sigma, \lambda}).
\]
The proof is complete.
\end{proof}

\end{section}
\begin{section}{Proof of Theorem \ref{mainteoproba}}\label{7}
Theorem \ref{mainteoproba} is a consequence of the following result.

\begin{prop}\label{nozeroorigin}
Let $s \in \left(\frac{1}{2}, 1\right)$, $N> 6s$. There exist $\overline \lambda \in (0, \lambda_{1,s}]$ and $C>0$ such that for every $\lambda \in (0, \overline \lambda)$ if $u_{s, \lambda}\subset \mathcal{M}_{s, \lambda; rad}$ is a least energy radial sign-changing solution of Problem \eqref{fracBrezis} such that $u_{s, \lambda}(0) \geq 0$, then $u_{s, \lambda}(0) >C$.
\end{prop}
\begin{proof}
Let $s \in \left(\frac{1}{2}, 1\right)$, $N> 6s$. Assume that the thesis is false. Then, there exist two sequences $\lambda_k \to 0^+$ and $C_k \to 0^+$ such that setting $u_k := u_{s, \lambda_k}$ it holds that $0 \leq u_k(0) \leq C_k$, and thus $u_k(0) \to 0$ as $k \to +\infty$. 
Let us set $M_k = |u_k|_\infty$, then, by Lemma \ref{As:energas} $(vi)$, up to a further subsequence, we have $M_k \to \infty$ as $k \to \infty$.

Now consider the rescaled functions
\[
v_k = \frac{1}{M_k} u_k\left( \frac{x}{M_k^{\beta}}\right), \quad x \in M_k^{\beta}B_R,
\] 
where $\beta = \frac{2}{N-2s}$. By construction we observe that $v_k(0) \to 0$ as $k \to +\infty$. Moreover, if $\tilde x_k  \in M_k^{\beta_k}B_R$ is such that $v_k(\tilde x_k) = 1$, then by Proposition \ref{strauss} we obtain that $\tilde x_k$ stays in a compact subset of $\R^N$. Arguing as in Proposition \ref{u0bound} we obtain that there exists $v \in \mathcal{D}^s(\R^N)$ such that $v_k \rightharpoonup v$ in $\mathcal{D}^s(\R^N)$, where $v$ weakly solves
\begin{equation}\label{nozeroorigin17}
(-\Delta)^s v = |v|^{2^*_s-2}v \quad \text{in }\R^N.
\end{equation}

As we have seen in \eqref{nodalenerbound}, if $v$ were a sign-changing solution of \eqref{nozeroorigin17} it will satisfy $|v|^{2^*_s}_{2^*_s} > 2 S_s^{\frac{n}{2s}}$. On the other hand by Fatou's Lemma and Lemma \ref{As:energas} we get that 
\[
|v|^{2^*_s}_{2^*_s} \leq \liminf_{k \to \infty}|v_k|^{2^*_s}_{2^*_s} =  2 S_{s}^{\frac{n}{2s}},
\]
hence the only possibilities are that $v$ is trivial or of constant sign. 

By a standard argument (as seen in Remark \ref{gilbcont}, but here $s$ is fixed), since $|v_k|_\infty \leq 1$, up to a subsequence, we get that $v_k \to v$ in $C^{0,\alpha}_{loc}(\R^N)$ for some $\alpha < s$. In particular, since we have seen that $ \tilde x_k$ stays in a compact subset of $\R^N$, then, up to a subsequence, setting $\tilde x = \lim_{k \to \infty} \tilde x_k$ it follows that $v(\tilde x) = 1$. Hence $v$ is not trivial. Moreover we observe that by construction it holds that $v(0)=0$. 

Therefore $v$ is of constant sign, and without loss of generality let us assume that $v \geq 0$. Since $v$ solves \eqref{nozeroorigin17}, by the strong maximum principle, as stated in \cite[Corollary 4.2]{Musina}, we deduce that $v>0$ in $\R^N$ which gives a contradiction since $v(0) = 0$.

Alternatively, one can argue as follows: since $v \geq0$, we get that that $v_k^+ \to v$ a.e. and then by Fatou's Lemma, Lemma \ref{As:rescaling}, and Lemma \ref{As:energas} $(ii)$, we can infer that 
\[
S_s = \frac{\|v\|^2_s}{|v|^2_{2^*_s}},
\]
i.e. $v$ achieves the infimum in the fractional Sobolev inequality. Hence by Theorem \ref{SobolevEmb}, $v$ is in the form \eqref{eq:bubble} and in particular $v>0$, which once again contradicts the fact that $v(0) = 0$. The proof is complete.
\end{proof}
\end{section}

\begin{section}{Proof of Theorem \ref{mainteocc}}\label{8}
\begin{proof}[Proof of Theorem \ref{mainteocc}]
Let $u_{s, \lambda}$ be a least energy radial sign-changing solution of Problem \eqref{fracBrezis}. The existence of a number $\hat \lambda_s>0$ satisfying the first part of the theorem has been proved in Theorem \ref{twicechange}. Therefore for $0<\lambda < \hat \lambda_s$ we have that $u_{s, \lambda}$ changes sign either once or twice. For the second part of the theorem we begin with proving the following preliminary fact:
\vspace{10pt}

\textbf{Claim:} if there exists $r', r''>0$ such that $u_{s, \lambda}(r') \cdot u_{s, \lambda}(r'')> 0$ and there is no change of sign between $r'$ and $r''$, then $u_{s, \lambda}(r) \neq 0$ for all $r \in [r', r'']$.

\vspace{10pt}

Indeed, assume without loss of generality that $u_{s, \lambda}(r'), u_{s, \lambda}(r'') > 0$ and $u_{s, \lambda}\geq 0$ in $(r', r'')$.  Let  $W_{s, \lambda}$ be the extension of $u_{s,\lambda}$, let $\Omega ^+$, $\Omega ^-$, $P$ and $G$ as in the proof of Lemma \ref{lemmatecnico}. In addition let us recall that under our assumptions, $W_{s, \lambda}$ possesses exactly two nodal domains.

Arguing as in Theorem \ref{twicechange} we get that there exists a Jordan curve $\gamma^+: [0,1] \to \{r\geq 0\} \times \{y \geq 0\}$ which connects $(r', 0)$ and $(r'', 0)$, such that $\gamma^+(t) \in P $ for all $t \in (0, 1)$ and without loss of generality we can assume that $\gamma_+([0,1]) \cap \{r=0\}=\emptyset$.

Then, by Jordan's curve theorem, the curve whose support is $\gamma^+([0,1]) \cup ([r', r'']\times \{0\})$ divides $\{r\geq 0\} \times \{y \geq 0\}$ in two connected regions: a bounded one which we call $A_b$ and a unbounded one $A_u$.

Assume by contradiction that there exists $r_0 \in (r', r'')$ such that $u_{s, \lambda}(r_0) = 0$, and let $B^+_\delta(r_0): = \{(r, y) \in \{r\geq0\} \times \{y > 0\}\ |  \ |(r, y)- (r_0, 0)| < \delta\}$. We claim that $B^+_\delta(r_0) \cap G \neq \emptyset$ for every $\delta >0$. Indeed, assume that this is not the case. Then there exists $\delta>0$ such that $B_\delta^+(r_0) \cap G = \emptyset$, and thus for every $(r, y) \in \overline B_\delta^+(r_0)$ it holds that $W_{s, \lambda}(r, y) \geq 0$. As a consequence of the strong maximum principle (see Proposition \ref{yannick}) we conclude that $u_{s,\lambda}(r)>0$ for every $|r-r_0| < \delta$, and in particular $u_{s, \lambda}(r_0) >0$, which contradicts the assumption on $r_0$.

Therefore, for every $\delta >0$ it holds that $B^+_\delta(r_0) \cap G \neq \emptyset$. On one hand, since $\gamma^+(t) \in P$ for all $t \in (0,1)$ and $r_0 \in (r', r'')$, there exists $\delta$ small enough such that $B^+_\delta(r_0) \subset A_b$. This implies that there exists a point $(r'_-, y'_-) \in A_b \cap G$. 
On the other hand, since $u_{s, \lambda}$ changes sign and $W_{s, \lambda} \geq 0$ in $[r', r''] \times \{0\}$, there exists $r''_- \in \{r \geq 0\} \setminus [r', r'']$ such that $u_{s, \lambda}(r''_-) <0$. Using once again the continuity of $W_{s, \lambda}$ and that $\gamma^+(t) \in P$ for all $t \in (0, 1)$, we get there exists $y''_-$ such that $(r''_-, y''_-) \in G \cap A_u$.

Therefore, being $W_{s, \lambda}$ continuous and $G$ connected, there exists a continuous path joining $(r'_-, y'_-)$ and $(r''_-, y''_-)$ which lies completely in $G$. As a consequence of the Jordan curve theorem, such path must intersect $\gamma^+$, which implies that $P \cap G$ is not empty, which is absurd. The claim is then proved.
\vspace{5pt}

Now let us prove $(a)$. Assume that $u_{s,\lambda}$ changes sign exactly twice, and denote by $0<r_1<r_2<R$ its nodes. In order to prove the result we must show that $u_{s,\lambda}$ cannot vanish in any other point $r \in [0,R)$ different from the nodes. To this end we argue by contradiction. Assume that there exists $r_0 \in [0,R)$, $r_0\neq r_1,r_2$ such that $u_{s,\lambda}(r_0)=0$. Then, there are only three possibilities: $r_0 \in [0,r_1)$, $r_0 \in (r_1,r_2)$ or $r_0 \in (r_2,R)$. Let us show that $r_0=0$ cannot happen.

Indeed, assume by contradiction that $u_{s, \lambda} (0) = 0$. Then there exist $r'$, $r''$ such that $0 < r' <r_1 < r_2 < r'' <R$ and satisfying $u_{s, \lambda}(r') \cdot u_{s, \lambda}(r'') >0$. Without loss of generality we assume that $u_{s, \lambda}(r')>0$. This implies that there exists a simple continuous curve $\gamma^+$ which connects $(r', 0)$ and $(r'', 0)$, lying completely in $P$ except for its ending point $(r', 0)$ and $(r'', 0)$. In addition, without loss of generality, we can assume that $\gamma_+([0,1]) \cap \{r=0\}=\emptyset$. Hence, the closed simple curve whose support is given by $\gamma^+([0,1]) \cup ([r', r'']\times \{0\})$ divides $\{r \geq 0\}\times \{y \geq 0 \}$ in two regions, a bounded one which we call $A_b$ and a unbounded one $A_u$. Since there exists $r_-$ in $(r_1, r_2)$ such that $u_{s, \lambda}(r_-) <0$, thanks to the continuity of $W_{s, \lambda}$ and since $\gamma^+(t) \not \in \{r\geq0\}\times \{0\}$ when $t \in (0, 1)$, we have that there exists $y_->0$ such that $(r_-, y_-) \in A_b \cap G$. But then, since $G$ is connected, the Jordan curve theorem implies that $G \subset A_b$. 

Since $(0, 0) \in A_u$, this implies that there exists $\delta >0$ such that $\overline{B^+_\delta} \subset \R^{n+1}_+$ does not intersect $G$, i.e. $W_{s, \lambda}(x, y) \geq 0$ in $\overline B_\delta^+$. Then, we reach a contradiction as a consequence of the strong maximum principle (see Proposition \ref{yannick}). Therefore $r_0=0$ cannot happen.
\vspace{10pt}

If $r_0 \in (0,r_1)$ we can find two points  $0<r' < r_0<r''<r_1$ such that $u_{s, \lambda}(r') u_{s, \lambda}(r'')>0$ and there is no change of sign between $r'$ and $r''$. In fact, $u_{s,\lambda}$ does not change sign in $(0,r_1)$, and in addition if $u_{s, \lambda}$ were identically zero a subset of positive measure of $(0,r_1)$, then, from \cite[Theorem 1.4]{fellifall}, we would have that $u_{s,\lambda}$ is zero everywhere. Therefore we can find $r'$ and $r''$ satisfying the above  properties and then, by using the Claim, we deduce that $u_{s,\lambda}$ cannot vanish in $(r',r'')$, which leads to a contradiction.
\vspace{10pt}

The proof of the other cases $r_0 \in (r_1,r_2)$ and $r_0 \in (r_2,R)$ is identical, and thus the proof of $(a)$ is complete.
\vspace{10pt}

For the proof of $(b)$, let $r_0 \in [0,R)$ be a zero of $u_{s,\lambda}$ different from the node $r_1$. If $r_0 \neq 0$, then  by using the Claim and arguing as before we get a contradiction and we are done. If $u_{s,\lambda}(0)=0$ we observe that since we are assuming that $u_{s,\lambda}$ changes sign exactly once we cannot exclude this possibility by using a merely topological argument as before. At the end we have only the two possibilities: 
\[
Z = \{0\} \cup \{ x \in \R^N \ | \ |x| = r_1\} \ \hbox{or} \ \ Z = \{ x \in \R^N \ | \ |x| = r_1\},
\]
and the proof of $(b)$ is complete. The final part of the theorem is consequence of Theorem \ref{mainteoproba}. The proof is then complete.
\end{proof}
\end{section}

\begin{section}{Proof of Theorem \ref{mainteorem}}\label{9}
\setcounter{cla}{0}
\begin{proof}[Proof of Theorem \ref{mainteorem}]
Let $N \geq 7$ and $R>0$. Let $\hat \lambda\left(\frac{1}{2}\right)$ be the number given by Theorem \ref{twicechange} for $s_0=\frac{1}{2}$, and let $\tilde\lambda>0$  be such that both $\tilde\lambda \leq \hat \lambda\left(\frac{1}{2}\right)$, and 
\[
\sup_{s \in \left(\frac{1}{2}, 1 \right)}\left| c_{\mathcal{M}_{s, \lambda; rad}} - 2\frac{s}{N}S^{\frac{N}{2s}}_s\right|< \frac{1}{N}S^{\frac{N}{2}}_1 \quad \forall \lambda \in (0, \tilde\lambda),
\] 
are satisfied. The existence of such a number $\tilde\lambda$ is ensured by Lemma \ref{As:energas} by taking $s_0=\frac{1}{2}$.

Let us fix $\lambda \in (0,\tilde\lambda)$. We want to prove that there exists $\overline s \in \left(\frac{1}{2}, 1\right)$ such that for every $s \in (\overline s, 1)$ any least energy radial sign-changing solutions of Problem \eqref{fracBrezis} in $B_R$, changes sign exactly once. Indeed assume by contradiction that this is not the case. Then there exists a sequence $s_k\to 1^-$ and a sequence $(u_{s_k,\lambda})_{s_k}$ of least energy radial solutions which change sign at least twice, for any $k$. For brevity we omit the subscript $k$ in the above sequences. Thanks to definition of $\tilde \lambda$, then Theorem \ref{twicechange} holds and thus $u_{s,\lambda}$ changes sign exactly twice, for any $s \in (\frac{1}{2},1)$. By Theorem \ref{sconverg} we have that $(u_{s,\lambda})_s$ coverges in $C^{0, \alpha}_{loc}(B_R)$ to $u_{1,\lambda}$ for every fixed $0<\alpha < \frac{1}{2}$, where $u_{1, \lambda}$ is a weak non trivial solution of
\[
\begin{cases}
-\Delta u_{1, \lambda} = \lambda u_{1, \lambda} + |u_{1, \lambda}|^{2^*_1-2}u_{1, \lambda} &\hbox{in}\ B_R,\\
u_{1, \lambda} = 0 &\hbox{in}\ \R^N \setminus B_R.
\end{cases}
\] 

On one hand, the definition of $\tilde \lambda$ imply that
\[
I(u_{s,\lambda}) = c_\mathcal{M}(s,\lambda) < \frac{2s}{n}S^{\frac{N}{2s}}_s + \frac{1}{N}S_1^{\frac{N}{2}}.
\] 
Thus, passing to the limit as $s \to 1^-$, we obtain that
\[
I(u_{1,\lambda}) < \frac{3}{N}S^{\frac{N}{2}}_1.
\]
This implies (by arguing as in \cite[Theorem 1.1]{Pacella}) that $u_{1,\lambda}$ changes sign once. 
On the other hand, denoting by $r^1_s$ and $r^2_s$ the nodes of $u_{s, \lambda}$, as $s \to 1^-$, the following holds:
\begin{enumerate}[(i)]
\item $r^1_s \not \to 0$;
\item $r^1_s - r^2_s \not \to 0$;
\item $r^2_s \not \to R$.
\end{enumerate}
This, together with the $C^{0, \alpha}$-convergence in compact subsets of $B_R$, implies that $u_{1, \lambda}$ changes sign at least twice, a contradiction. Let us prove $(i)-(iii)$.

Property $(ii)$ is a consequence of an energetic argument. Indeed, suppose that $r^1_s -r^2_s\to 0$. We readily obtain a contradiction because by Remark \ref{ss:as} and \eqref{fatoulemma22} we have 
\[
0<C \leq \left( 1 - \frac{\lambda}{ \lambda_{1,s}}\right) S_s \leq S_{s, \lambda} \leq \frac{\|u^-_{s, \lambda}\|^2_s-\lambda|u^-_{s,\lambda}|^2_2}{|u^-_{s, \lambda}|^2_{2^*_s}}\leq |u^-_{s,\lambda}|^{2^*_s-2}_{2^*_s},
\]
which, together with Lemma \ref{Linftyboundsc}, implies that
\[
0<C \leq \int_{B_R}|u^-_{s,\lambda}|^{2^*_s}\de x \leq \omega_N|u_{s, \lambda}|^{2^*_s}_\infty\int_{r^1_s}^{r^2_s}r^{N-1}\de r \leq C\left( (r^2_s)^N-(r_s^1)^N\right) \to 0,
\]
which is absurd. We also observe that with the same proof $r^1_s \to 0$ and $r^2_s \to R$ cannot happen at the same time.

For $(i)$, since $u_{s, \lambda} \to u_{1, \lambda}$ in $C^{0, \alpha}_{loc}(B_R)$ for any fixed $0<\alpha < \frac{1}{2}$, it holds that, for a suitable compact $K \subset B_R$ containing the origin as interior point we have
\begin{equation}\label{holdercont}
|(u_{s,\lambda}-u_{1,{\lambda}})(x) - (u_{s,\lambda}-u_{1,{\lambda}})(y)| \leq C_K|x-y|^\alpha
\end{equation}
where $C_K$ is uniformly bounded with respect to $s$ and depends on $K$. 
If we suppose that $r^1_s \to 0$ and evaluate \eqref{holdercont} in $x = r^1_s$ and $y = 0$, we obtain
\[
\left|u_{1,\lambda}(0) - u_{1, \lambda}(r^1_s) - u_{s,\lambda}(0)\right| \to 0 \quad as \quad s \to 1^-.
\]
Since $u_{1, \lambda} \in C^{0,s_0}(K)$ we get $u_{s,\lambda}(0) \to 0$. In addition, since $u_{s,\lambda} \to u_{1,\lambda}$ a.e., this implies also that $u_{1,\lambda}(0) = 0$. As a consequence of \cite[Proposition 2]{Iacopetti}, we get that $0 = |u_{1,\lambda}(0)| = |u_{1, \lambda}|_\infty$ i.e., $u_{1, \lambda} \equiv 0$. This contradicts the non triviality of $u_{1,\lambda}$, which is ensured by Theorem \ref{sconverg}.

To conclude to proof it remains to show that $r_s^2 \to R$ cannot happen. 
Since we are assuming that $u_{s,\lambda}$ changes sign twice, and since without loss of generality we can assume that $u_{s,\lambda}(0)\geq 0$, by Theorem \ref{subsol} in the Appendix, it follows that $u_{s,\lambda} \in C^{0,s}(\R^N)$  is a weak sub-solution of
\[
(-\Delta)^s u_{s,\lambda}\leq \lambda u_{s,\lambda} + |u_{s,\lambda}|^{2^*_s-2}u_{s,\lambda} \quad \text{in }\R^N.
\]

Then, arguing as in the proof of Theorem \ref{sconverg}, taking the limit as $s \to 1^-$, there exists $u_{1,\lambda} \in X_0^1(B_R)$ such that $u_{s,\lambda} \to u_{1,\lambda}$ in $L^2(\R^N)$, $u_{s,\lambda} \to u_{1,\lambda}$ in $C^{0, \alpha}_{loc}(B_R)$, and for every $\varphi \in C^\infty_c(\R^N)$ such that $\varphi \geq 0$, 
\begin{equation}\label{relazsubsol}
\int_{\R^N} \nabla u_{1,\lambda}\cdot \nabla \varphi \de x \leq \int_{\R^N}(\lambda u_{1,\lambda} + |u_{1,\lambda}|^{2^*_1-2}u_{1,\lambda}) \varphi \de x.
\end{equation}

Suppose now that $r^2_s \to R$ as $s \to 1^-$. Then there exists $\delta >0$ such that on the set $\mathcal{I} = \{ x \in \R^N \ |\ R-\delta \leq |x| \leq R+\delta\}$ it holds that $u_{1,\lambda} \leq 0$.
Taking $\varphi \in C^\infty_c(\mathcal{I})$, $\varphi \geq 0$, from \eqref{relazsubsol} we readily get 
\[
\begin{cases}
(-\Delta) u_{1,\lambda} \leq 0 & \hbox{in}\ \mathcal I, \\
u \leq 0 & \hbox{on}\ \partial \mathcal{I},
\end{cases}
\]   
Therefore, by the strong maximum principle either $u <0$ or $u \equiv 0$ in $\mathcal{I}$, but this is absurd since it holds that $u_{1,\lambda}< 0$ in $R-\delta \leq |x|<R$ and $u_{1,\lambda} \equiv 0 $ in $R < |x| \leq R+\delta$. The proof is then complete. 
\end{proof}
\end{section}

\setcounter{cla}{0}

\begin{section}{Proof of Theorem \ref{mainteorem2}}\label{10}
In this section we study the asymptotic behaviour of least energy radial sign-changing solutions of Problem \eqref{fracBrezis} in $B_R$ as $\lambda\to 0^+$. Theorem \ref{mainteorem2} will be a consequence of the following results. 
 
Under the hypotheses of Theorem \ref{mainteorem2} we set $M_{\lambda, \pm} := |u^\pm_\lambda|_\infty$, $\beta := \frac{2}{N-2s}$, and we denote by
\[
\begin{aligned}
t_\lambda&: = \max\{ t \in (0, R) \ | \ u_\lambda(t) = M_{\lambda, +}\}, \\
 r_\lambda &:=\hbox{the node of} \ u_\lambda,\\
\tau_\lambda& := \max\{ t \in (0, R) \ | \ u_\lambda(t) = -M_{\lambda, -}\}.
\end{aligned}
\]
Let us observe that since $u_\lambda$ changes sign once it holds that $t_\lambda < r_\lambda <\tau_\lambda$. Let us consider also the following quantities:
\[
Q_\lambda := \frac{M_{\lambda,+}}{M_{\lambda,-}}, \quad \sigma_\lambda := M_{\lambda,+}^\beta r_\lambda
\]

By Lemma \ref{As:energas} we already know that as $\lambda \to 0^+$, we have $M_{\lambda, \pm} \to +\infty$. The following result states the asymptotic behavior of the quantities $t_\lambda$, $r_\lambda$, $\tau_\lambda$, as $\lambda \to 0^+$. 
\begin{lemma}
We have that $t_\lambda$, $r_\lambda$, $\tau_\lambda \to 0$ as $\lambda \to 0^+$. 
\end{lemma}
\begin{proof}
Since $0< t_\lambda < r_\lambda< \tau_\lambda$, it suffices to prove that $\tau_\lambda\to 0$ as $\lambda \to 0^+$. Evaluating inequality \eqref{straussineq} in a point $x_0$ such that $|x_0| = \tau_\lambda$ we get that
\[
M_{\lambda,-} \leq C\|u_\lambda\|^2\frac{1}{\tau_\lambda^{\frac{N-2s}{2}}} \leq C\frac{1}{\tau_\lambda^{\frac{N-2s}{2}}},
\]
where the uniform bound on the Gagliardo norm is a consequence of Lemma \ref{As:energas}.
Since $M_{\lambda,-}\to \infty$ we obtain the desired result.
\end{proof}

The following result concerns the asymptotic behaviour of $Q_\lambda$ and $\sigma_\lambda$. 

\begin{lemma}\label{asintoticaraggi}
Up to a subsequence, as $\lambda \to 0^+$, we have
\begin{enumerate}[i)]
\item $ Q_\lambda \to +\infty $,
\item $ \sigma_\lambda \to +\infty$. 
\end{enumerate}

\end{lemma}
\begin{proof} The proof is divided in two steps. 
\begin{cla}\label{As:primocaso} The following facts hold:
\begin{enumerate}[(a)]
\item $\sigma_\lambda \to 0$ cannot happen;
\item if either $Q_\lambda \to l \in \R^+\setminus \{0\}$ or $Q_\lambda \to +\infty$, then $\sigma_\lambda \to L \in (0, +\infty)$ cannot happen.
\end{enumerate}
\end{cla}

Property $(a)$ is a straightforward consequence of Lemma \ref{As:energas}. Indeed, assume by contradiction that $\sigma_\lambda \to 0$, then
\[
\begin{aligned}
|u_\lambda^+|_{2^*_s}^{2^*_s} =&\ \int_{B_{r_\lambda}}|u^+_\lambda|^{2^*_s}\de x = \omega_N \int_0^{r_\lambda}|u^+_\lambda(\rho)|^{2^*_s}\rho^{N-1}\de \rho \\
\leq&\ \omega_N (M_{\lambda,+})^{N\beta} \int_0^{r_\lambda}\rho^{N-1}\de \rho = \frac{\omega_N}{N} (M_{\lambda,+}^\beta r_\lambda)^N \to 0,
\end{aligned}
\]
and this is absurd since by Lemma \ref{As:energas} we have $|u_\lambda^+|_{2^*_s}^{2^*_s} \to S_s^{\frac{N}{2s}}$, and $(a)$ is proved.

For $(b)$, let us consider the rescaled functions
\begin{equation}\label{riscalglob}
\tilde u_\lambda (x) = \frac{1}{M_{\lambda,+}} u_\lambda\left(\frac{x}{M_{\lambda,+}^\beta}\right).
\end{equation}
By the assumption on $Q_\lambda$, we have that $\tilde u_\lambda$ is definitely uniformly bounded in $L^\infty$. Moreover, $\tilde u_\lambda$ weakly solves the problem
\[
\begin{cases}
(-\Delta)^s u = \frac{\lambda}{M_{\lambda,+}^{2^*_s-2}}u + |u|^{2^*_s-2}u & \text{in }B_{M^\beta_{\lambda,+}R}\\
u=0 & \text{in }\R^N \setminus B_{M^\beta_{\lambda,+}R}.
\end{cases}
\]
Then, by Remark \ref{soave} we have that there exists $\tilde u_0$ such that $\tilde u_\lambda \to \tilde u_0$ in $C_{loc}^{0,\alpha}(\R^N)$ for every $\alpha <s$. Suppose by contradiction that $\sigma_\lambda \to L$. Since $M^\beta_{\lambda, +}t_\lambda \leq \sigma_\lambda$, we get that, up to subsequences, there exists $\tilde x$ such that $|\tilde x| \leq L$ and $\tilde u(\tilde x) = 1$. This, together with the $C^{0,\alpha}_{loc}$ convergence, implies that $\tilde u \not \equiv 0$.

Let now $\tilde u^+_\lambda$ be the rescaling of $u^+_\lambda$. Since $\sigma_\lambda \to L$, there exists $\bar R$ such that $\tilde u^+_\lambda \in X^s_0(B_{\bar R})$ for all sufficiently small $\lambda>0$. Moreover, Lemma \ref{As:rescaling} and Lemma \ref{As:energas} imply that $\|\tilde u^+_\lambda \|_s^2 = \|u^+_\lambda\|_s^2 \to S_s^{\frac{N}{2s}}$, therefore $(\tilde u_\lambda^+)$ is a bounded sequence in $X^s_0(B_{\bar R})$. Hence, there exists $u_* \in X^s_0(B_{\overline R})$ such that $\tilde u^+_\lambda \rightharpoonup u_*$ in $X^s_0(B_{\overline R})$ and $\tilde u^+_\lambda \to u_*$ almost everywhere.  But then $u_* = u_0^+ \geq 0$ and $u_* \in X^s_0(B_L)$. In addition, there exists $\rho  \in [0, L)$ such that for every $|x| = \rho$ it holds $u_*(x) = 1 $, therefore $u_* \not \equiv 0$.

Since $|\tilde u^+_\lambda|_\infty \leq 1 $ and $\text{supp }\tilde u^+_\lambda \subset B_{\overline R}$, by Lebesgue's dominated convergence theorem we get that $\tilde u^+_\lambda \to u_*$ strongly in $L^{2^*_s}$. Using this and Fatou's Lemma we obtain
\[
S_s \leq \frac{\|u_*\|_s^2}{|u_*|^2_{2^*_s}} \leq \liminf_{\lambda \to 0^+} \frac{\|\tilde u^+_\lambda\|_s^2}{|\tilde u^+_\lambda|^2_{2^*_s}} = S_s,
\]
i.e. $u_*$ realizes the infimum $S_s$ despite being supported on a bounded domain $B_L$, and thus contradicting Theorem \ref{SobolevEmb}.

\begin{cla}\label{As:terzocaso}
The following holds:
\begin{enumerate}[(a)]
\setcounter{enumi}{2}
\item $M_{\lambda_-}^\beta \tau_\lambda \to +\infty$ cannot happen.
\item $M_{\lambda,+}^\beta t_\lambda \to +\infty$ cannot happen.
\end{enumerate}

Since the proofs are identical, we show only $(c)$. Since by Lemma \ref{As:energas} we have that $(u_\lambda)$ is a bounded sequence in $\mathcal{D}^s(\R^N)$, evaluating \eqref{straussineq} in a point $x_0$ such that $|x_0| = \tau_\lambda$ we get that
\[
(M^\beta_{\lambda_-}\tau_\lambda)^{\frac{N-2s}{2}} = \tau_\lambda^{\frac{N-2s}{2}}M_{\lambda,-} = |x_0|^{\frac{N-2s}{2}}|u_\lambda(x_0)|\leq K_{N,s}\|u_\lambda\|_s^2 \leq C,
\]
which proves the claim.
\end{cla}

Now we can prove i). Since $Q_\lambda >0$, up to a subsequence, as $\lambda \to 0^+$ we have that $Q_\lambda \to l \in [0, +\infty]$. Suppose that $Q_\lambda \to 0$. Since by Step \ref{As:terzocaso} we have that $M^\beta_{\lambda, -}\tau_\lambda \not \to +\infty$ we get that
\[
0 \leftarrow \left(Q_\lambda\right)^\beta M^\beta_{\lambda,-}\tau_\lambda = M^\beta_{\lambda,+}\tau_\lambda \geq  \sigma_\lambda \geq 0,
\]
which is impossible by $(a)$. Assume now $Q_\lambda \to l \in (0, +\infty)$. By Step \ref{As:primocaso} this implies that $\sigma_\lambda \to +\infty$, but then
\[
+\infty \leftarrow \left(\frac{1}{Q_\lambda}\right)^\beta \sigma_\lambda = M^\beta_{\lambda,-}r_\lambda \leq M^\beta_{\lambda,-}\tau_\lambda,
\]
and this is impossible by Step \ref{As:terzocaso}. Therefore the only possibility is $Q_\lambda \to +\infty$, and $i)$ is proved. For $ii)$, we observe that i) and Step \ref{As:primocaso} imply that, up to a subsequence, the only possibility is $\sigma_\lambda \to +\infty$, as $\lambda \to 0^+$. The proof is complete.
\end{proof}

\begin{prop}
Under the hypotheses of Theorem \ref{mainteorem2}, up to a subsequence, as $\lambda \to 0^+$ we have that the function
\[
\tilde u^+_\lambda (x) = \frac{1}{M_{\lambda,+}}u^+_\lambda\left( \frac{x}{M_{\lambda,+}^\beta}\right), 
\]
converges to $U(x) = k_{\hat \mu} \frac{\hat \mu^{N-2s}}{(\hat \mu^2 + |x|^2)^{\frac{N-2s}{2}}}$ in $C^{0, \alpha}_{loc}(\R^N)$, for every fixed $\alpha \in (0,s)$, and strongly in $\mathcal{D}^s(\R^N)$, where
\[
\hat \mu = S_s^{\frac{1}{2s}}\left(\int_{\R^N}\frac{1}{(1+|x|^2)^N}\de x\right)^{-\frac{1}{N}}. 
\]
\end{prop}
\begin{proof}
Let $\tilde u_\lambda$ be the rescaling defined in \eqref{riscalglob}. Since $(\tilde u_\lambda)$ is a bounded sequence in $\mathcal{D}^s(\R^N)$ by Lemma \ref{As:rescaling} and Lemma \ref{As:energas}, up to a subsequence, $\tilde u_\lambda$ weakly converges to $\tilde u_0$ in $\mathcal{D}^s(\R^N)$, strongly in $L^p_{loc}(\R^N)$ for every $p \in (1, 2^*_s)$ and also almost everywhere in $\R^N$. The same holds for $\tilde u_\lambda^\pm$, and in particular $\tilde u_\lambda^\pm \to \tilde u^\pm_0$ a.e. As a consequence of Lemma \ref{asintoticaraggi} we have that $\tilde u^+_\lambda \to \tilde u_0$ and $\tilde u^-_\lambda \to 0$ almost everywhere, thus $\tilde u_0 \geq 0$. On the other hand the function $\tilde u_\lambda$ weakly satisfies 
\begin{equation}\label{profilobubble1}
\begin{cases}
(-\Delta)^s \tilde u_\lambda = \frac{\lambda}{M_{\lambda,+}^{2^*_s-2}}\tilde u_\lambda + |\tilde u_\lambda|^{2^*_s-2}\tilde u_\lambda & \hbox{in} \ B_{M^\beta_{\lambda,+}R},\\
\tilde u_\lambda =0 & \hbox{in} \ \R^N \setminus B_{M^\beta_{\lambda,+}R},
\end{cases}
\end{equation}
and  thanks to Proposition \ref{strauss} the point where the maximum of $\tilde u_\lambda$ is achieved stays in a compact subset $K \subset \subset \R^N$.
By Lemma \ref{asintoticaraggi} and the definition of $\tilde u_\lambda$ we have $|\tilde u_\lambda|_\infty \leq 1$, and hence by a standard argument (as seen in Remark \ref{gilbcont}, but here $s$ is fixed) 
we obtain that $\tilde u_\lambda \to u_*$ in $C^{0,\alpha}_{loc}(\R^N)$ for every fixed $\alpha \in (0,s)$, and $u_* = \tilde u_0$ thanks to the a.e. convergence. Therefore, since the maximum of $\tilde u_\lambda$ definitely stay in compact subset $K$ of $\R^N$, there exists $\overline x \in K$ such that $\tilde u_0(\overline x) =1$, hence $\tilde u_0$ is not trivial.
Passing to the limit in \eqref{profilobubble1} we deduce that $\tilde u_0$ weakly solves 
\begin{equation}\label{sasasa}
(-\Delta)^s \tilde u_0 = |\tilde u_0|^{2^*_s-2}\tilde u_0 \quad \hbox{in}\ \R^N.
\end{equation}
Since $\tilde u_0$ is a non trivial solution of \eqref{sasasa} we obtain 
\[
S_s \leq \frac{\|\tilde u_0\|_s^2}{|\tilde u_0|^2_{2^*_s}} = |\tilde u_0|^{2^*_s-2}_{2^*_s}.
\]
On the other hand by Fatou's lemma we have that $|\tilde u_0 |^{2^*_s}_{2^*_s} \leq \liminf_{\lambda \to 0^+}|\tilde u^+_\lambda|^{2^*_s}_{2^*_s} = S_s^{\frac{N}{2s}}$.

Therefore $|\tilde u_\lambda|^{2^*_s}_{2^*_s} \to |\tilde u_0|^{2^*_s}_{2^*_s}$, thus obtaining that $\tilde u^+_\lambda \to \tilde u_0$ strongly in $L^{2^*_s}(\R^N)$ thanks to the Brezis-Lieb's Lemma. 
Hence, we infer that $\tilde u_0$ is a minimizer for $S_s$ and a solution of \eqref{sasasa}, which implies that it has to be of the form \eqref{eq:bubble}, for some $\mu \in \R$, $x_0 \in \R^N$. Since $\tilde u$ is the limit of radial functions, then $\tilde u$ is radial and $x_0 = 0$. Then, by construction we get that $u(0) = 1$ and $\mu = S_s^{\frac{1}{2s}}\left(\int_{\R^N}\frac{1}{(1+|x|^2)^N}\de x\right)^{-\frac{1}{N}} $. The proposition is proved.  
\end{proof}
\end{section}

\section*{Appendix: Some technical results}
\renewcommand\thesection{\Alph{section}}
\renewcommand\theequation{\thesection.\arabic{equation}}
\setcounter{section}{1}
\setcounter{equation}{0}
\setcounter{teo}{0}
\addcontentsline{toc}{section}{Appendix: Some technical results}

Let $u_{s, \lambda}$ be a solution of Problem \eqref{fracBrezis}, and let $W_{s, \lambda} = E_{s, \lambda}u_{s, \lambda}$
be its extension, i.e.
\[
W_{s, \lambda}(x, y) = \int_{\R^N}P_{N,s}(x- \xi, y) u_{s, \lambda}(\xi) \de \xi = p_{N,s}\int_{\R^N}\frac{y^{2s}}{(y^2 + |x-\xi|^2)^{\frac{N+2s}{2}}}u_{s,\lambda}(\xi) \de \xi,
\]
where $p_{N,s}$ is such that $p_{N,s}\int_{\R^N}P_{N,s}(x- \xi, y) \de \xi = 1$, for any $y>0$. 

\begin{lemma}\label{app1}
Let $s_0 > \frac{2}{3}$, let $s \in (s_0, 1)$, let $\lambda \in (0, \lambda_{1,s})$ and let $R>0$. Let $u_{s, \lambda}$ be a least energy radial sign-changing solution of Problem \eqref{fracBrezis} in $B_R$. Let $\delta >0$ and define $A_\delta$ as the set 
\begin{equation}\label{adelta}
A_\delta = \{ x \in \R^N \ |\ |R - |x||>\delta \}.
\end{equation}
Then, there exists $C>0$ which depends on $N$, $s$, $\lambda_{1,s}$ and $\delta$ such that
\[
\sup_{x \in A_\delta}|(-\Delta)^s u_{s, \lambda}(x)| \leq C. 
\]
\end{lemma}
\begin{proof}Let $x \in A_\delta$ be such that $|x|>R$. Since $u_{s, \lambda}(x) = 0$ when $|x| \geq R$  we get that
\[
\begin{aligned}
\left|C_{N,s}\int_{\R^N} \frac{u_{s, \lambda}(x) - u_{s, \lambda}(y)}{|x-y|^{N+2s}}\de y \right| &\leq C_{N,s}\int_{\R^N} \frac{|u_{s, \lambda}(x) - u_{s, \lambda}(y)|}{|x-y|^{N+2s}}\de y \\
&= C_{N,s}\int_{|x-y|> \delta}\frac{|u_{s, \lambda}(y)|}{|x-y|^{N+2s}}\de y \leq |u|_\infty \omega_N\left( \frac{C_{N,s}}{2s\delta^{2s}}\right).
\end{aligned}
\] 
As a consequence we obtain that
\[
(-\Delta)^s u_{s, \lambda}(x) = C_{N,s}P.V.\int_{\R^N} \frac{u_{s, \lambda}(x) - u_{s, \lambda}(y)}{|x-y|^{N+2s}}\de y = C_{N,s}\int_{\R^N} \frac{u_{s, \lambda}(x) - u_{s, \lambda}(y)}{|x-y|^{N+2s}}\de y < + \infty.
\]

When $|x| < R$, since $s > \frac{2}{3}$, by \cite[Corollary 1.6, (a)]{HolderReg}, for every $K \subset \subset \R^N$ we have that $u_{s, \lambda} \in C^{2, 3s-2}(K)$ with $|D^2 u_{s, \lambda}|_{\infty; K} \leq C \text{dist}(\partial K, \partial B_R)^{-2s}$. In particular, let us consider $K$ such that $x \pm y \in K$ for every $x \in A_\delta$ and $|y| \leq \frac{\delta}{2}$. Then $|D^2 u_{s, \lambda}|_{\infty; K} \leq C \left( \frac{\delta}{2}\right)^{-2s}$. Using the alternative form of the fractional Laplacian
\[
(-\Delta)^s u(x) = -\frac{C_{N,s}}{2}\int_{\R^N}\frac{u(x+y) + u(x-y) - 2u(x)}{|y|^{N+2s}}\de y,
\]
and arguing as before, we obtain that 
\[
\begin{aligned}
&\left|\frac{C_{N,s}}{2}\int_{\R^N}\frac{|u_{s, \lambda}(x+y) + u_{s, \lambda}(x-y) - 2u_{s, \lambda}(x)|}{|y|^{N+2s}}\de y\right|  \\
\leq&\ \frac{C_{N,s}}{2}\int_{|y|< \frac{\delta}{2}}\frac{|u_{s, \lambda}(x+y) + u_{s, \lambda}(x-y) - 2u_{s, \lambda}(x)|}{|y|^{N+2s}}\de y + |u_{s, \lambda}|_\infty \omega_N2^{2s+2}\left( \frac{C_{N,s}}{2s\delta^{2s}}\right) \\
\leq &\ |D^2 u_{s, \lambda}|_{\infty; K} \omega_N \frac{C_{N,s}}{2(1-s)}\left(\frac{\delta}{2}\right)^{2(1-s)}+ |u_{s, \lambda}|_\infty \omega_N2^{2s+2}\left( \frac{C_{N,s}}{2s\delta^{2s}}\right)\\
\leq &\ C \delta^{-2s},
\end{aligned}
\]
where $C>0$ depends only on $N$, $s$, $|u_{s, \lambda}|_\infty$. The proof is then complete.
\end{proof}

\begin{lemma}\label{app2}
Let $u_{s, \lambda}$ be a solution of Problem \eqref{fracBrezis} in $B_R$. For any $x \in \R^N$ such that $|x| \neq R$, the following pointwise relations hold:
\[
\lim_{\varepsilon \to 0^+}-d_s \varepsilon^{1-2s}\frac{\partial W_{s, \lambda}}{\partial y}(x, \varepsilon) = \lim_{\varepsilon \to 0^+}-2sd_s\frac{W_{s, \lambda}(x, \varepsilon) - W_{s, \lambda}(x, 0)}{\varepsilon^{2s}} = (-\Delta)^s u_{s, \lambda}(x).
\]
\end{lemma}
\begin{proof}
For the first equality, applying the Cauchy mean value theorem we get that 
\[
\frac{W_{s, \lambda}(x, \varepsilon) - W_{s, \lambda}(x, 0)}{\varepsilon^{2s}} = \frac{\frac{\partial W_{s, \lambda}}{\partial y}(x, \tau)}{2s \tau^{2s-1}}, 
\]
where $\tau= \tau(x) \in (0, \varepsilon)$, and the desired result follows by passing to the limit and relabeling the variable at the right-hand side.

For the other equality, by definition of $W_{s, \lambda}$, we have
\begin{equation}\label{wonderlustking}
-\frac{W_{s, \lambda}(x, \varepsilon) - W_{s, \lambda}(x, 0)}{\varepsilon^{2s}} = p_{N,s} \int_{\R^N}\frac{u_{s, \lambda}(x) - u_{s, \lambda}(y)}{(\varepsilon^2 + |x-y|^2)^{\frac{N+2s}{2}}}\de y.
\end{equation}
Recalling that $2sd_s = \frac{C_{N,s}}{p_{N,s}}$ one get the desired result passing to the limit $\varepsilon \to 0^+$ and 
arguing as in Lemma \ref{app1} for the right-hand side of \eqref{wonderlustking}.
\end{proof}

\begin{lemma}\label{app3}
Let $\delta >0$ and $A_\delta$ be as in \eqref{adelta}. For every $\varphi \in L^2(A_\delta)$ such that $\text{supp } \varphi$ is bounded in $A_\delta$, it holds that 
\[
\begin{aligned}
\lim_{\varepsilon \to 0^+}&\int_{A_\delta}-2sd_s\frac{W_{s, \lambda}(x, \varepsilon) - W_{s, \lambda}(x, 0)}{\varepsilon^{2s}}\varphi(x) \de x = \int_{A_\delta} (-\Delta)^s u_{s, \lambda}(x)\varphi(x) \de x.
\end{aligned}
\]
\end{lemma}
\begin{proof}
This is a consequence of previous lemmas. Indeed, by Lemma \ref{app1} and \ref{app2} we get that for every $x \in A_\delta$ 
\[
\left|-\frac{W_{s, \lambda}(x, \varepsilon) - W_{s, \lambda}(x, 0)}{\varepsilon^{2s}}\right| \leq |(-\Delta)^s u_{s, \lambda}(x)| \leq C,
\]
hence is sufficient to apply the Lebesgue's dominated convergence theorem to get the result.
\end{proof}

\begin{lemma}\label{app4}
Let $\delta >0$. It holds that $(-\Delta)^s u_{s, \lambda} = \lambda u_{s, \lambda} + |u_{s, \lambda}|^{2^*_s-2}u_{s, \lambda}$ for almost every $x \in B_{R-\delta}$.
\end{lemma}

\begin{proof}
Let $\varphi \in C^\infty_c(B_{R-\delta})$. Since $B_{R-\delta}\subset A_\delta$, then Lemma \ref{app3} and formula \eqref{delhopital} imply that 
\[
(u_{s, \lambda}, \varphi) = \int_{\R^N}(-\Delta)^s u_{s, \lambda} \varphi \de x. 
\] 
Since $u_{s, \lambda}$ solves Problem \eqref{fracBrezis} we obtain 
\[
\int_{\R^N} \left( (-\Delta)^s u_{s, \lambda} -\lambda u_{s, \lambda} + |u_{s, \lambda}|^{2^*_s-2}u_{s, \lambda}\right) \varphi \de x = 0, 
\]
and the result follows from the arbitrariety of $\varphi$. 
\end{proof}

\begin{lemma}\label{app5}
Let $A_\delta$ be as in \eqref{adelta} and let $\varphi \in L^2(A_\delta)$ be such that $\text{supp } \varphi$ is bounded in $A_\delta$. Then
\[
\lim_{\varepsilon \to 0^+}\int_{A_\delta} - d_s\varepsilon^{1-2s} \frac{\partial W_{s, \lambda}}{\partial y}(x, \varepsilon) \varphi(x) \de x = \int_{A_\delta}(-\Delta)^s u_{s, \lambda}(x) \varphi(x) \de x. 
\]
\end{lemma}
\begin{proof}
Let us set $F_\delta := A_\delta \cap \text{supp }\varphi$. Let $\varepsilon>0$ and $x \in F_\delta$ be fixed. We claim that
\[
\left|-d_s\varepsilon^{1-2s}\frac{\partial W_{s, \lambda}}{\partial y}(x, \varepsilon)\right| \leq \frac{C_{N,s}}{4s}\max\{2s, N\}\int_{\R^N}\frac{|u_{s, \lambda}(x+z)+u_{s, \lambda}(x-z)-2u_{s, \lambda}(x)|}{|z|^{N+2s}}\de z.
\]
Indeed, by the definition of $W_{s, \lambda}$ as a convolution, after some computations and the change of variable $\xi = x \pm y\eta$ we obtain that
\[
\frac{\partial W_{s, \lambda}}{\partial y}(x, y) = p_{N,s}\int_{\R^N}\frac{2s|\eta|^2 - N}{(1 + |\eta|^2)^{\frac{N+2s+2}{2}}} \frac{u_{s, \lambda}(x \pm y\eta)}{y} \de \eta.
\]
Moreover, differentiating with respect to $y$ the relation 
\[
p_{N,s}\int_{\R^N}\frac{y^{2s}}{(y^2 + |x-\xi|^2)^{\frac{N+2s}{2}}}\de \xi = 1,
\]
after the same change of variables we get that
\[
p_{N,s}\int_{\R^N}\frac{2s |\eta|^2-N}{(1 + |\eta|^2)^{\frac{N+2s+2}{2}}}\de \eta = 0,
\]
for every $N \in \N$, $s >0$.
As a consequence, using also that $d_s = \frac{C_{N,s}}{2s p_{N,s}}$, we obtain 
\[
\begin{aligned}
&-d_s y^{1-2s}\frac{\partial W_{s, \lambda}}{\partial y}(x, y) \\
= &\ -\frac{C_{N,s}}{2} \frac{1}{2s} \int_{\R^N}\frac{2s|\eta|^2-N}{(1 + |\eta|^2)^{\frac{N+2s+2}{2}}} \frac{u_{s, \lambda}(x+y\eta)+u_{s, \lambda}(x-y\eta) - 2u_{s, \lambda}(x)}{y^{2s}}\de \eta.
\end{aligned}
\]
 Then, after the change of variables $y\eta = z$ we deduce that
\[
-d_s y^{1-2s}\frac{\partial W_{s, \lambda}}{\partial y}(x, y) = -\frac{C_{N,s}}{2} \frac{1}{2s} \int_{\R^N}\frac{2s|z|^2-Ny^2}{(y^2 + |z|^2)^{\frac{N+2s+2}{2}}} (u_{s, \lambda}(x+z)+u_{s, \lambda}(x-z) - 2u_{s, \lambda}(x))\de z,
\]
and the claim easily follows. At the end, as a consequence of the claim we get that for any $x \in A_\delta$ (see Lemma \ref{app1}), 
\[
\begin{aligned}
&-\lim_{\varepsilon \to 0^+}d_s \varepsilon^{1-2s}\frac{\partial W_{s, \lambda}}{\partial y}(x, \varepsilon) = \\
=&\ -\frac{C_{N,s}}{2} \int_{\R^N}\frac{u_{s, \lambda}(x+z)+u_{s, \lambda}(x-z) - 2u_{s, \lambda}(x)}{|z|^{N+2s+2}} \de z = (-\Delta)^s u_{s, \lambda}(x),
\end{aligned}
\]
and the thesis follows from the Lebesgue's dominated convergence theorem. 
\end{proof}

\begin{teo}\label{subsol}
Let $u_{s, \lambda}$ be a least energy radial sign-changing solution of Problem \eqref{fracBrezis} in $B_R$ which changes sign exactly twice and such that $u_{s, \lambda}(0) \geq 0$. Then $u_{s,\lambda}$ is a weak sub-solution of
\[
(-\Delta)^s u_{s, \lambda} \leq \lambda u_{s, \lambda} + |u_{s, \lambda}|^{2^*_s-2}u_{s, \lambda} \quad \text{ in }\R^N
\] 
i.e., for every $\varphi \in C^\infty_c(\R^N)$ such that $\varphi \geq 0$ it holds that
\[
(u_{s, \lambda}, \varphi)_s \leq \int_{\R^N}(\lambda u_{s, \lambda} + |u_{s, \lambda}|^{2^*_s-2}u_{s, \lambda}) \varphi \de x.
\]
\end{teo}
\begin{proof}
Let $\varphi \in C^\infty_c(\R^N)$, $\varphi \geq 0$. We observe that there always exists a function $\phi \in C^\infty_c(\overline{\R^{N+1}_+})$ such that $\phi \geq 0$ and $\phi(x, 0) = \varphi (x)$. By Lemma \ref{extreg} we have that 
\begin{equation}\label{equazagg1}
(u_{s, \lambda}, \varphi)_s = \int_{\R^{N+1}_+} y^{1-2s} \nabla W_{s, \lambda}\cdot \nabla \phi \de x \de y. 
\end{equation}
Let $0<r_s^1<r_s^2<R$ be the nodes of $u_{s,\lambda}$. As proved in Lemma \ref{lemmatecnico} (see \eqref{asintotre1}), for every $\overline \rho >r^2_s$ such that $u_{s, \lambda}(\overline \rho) >0$ there exists $\overline \varepsilon$ such that
\begin{equation}\label{segno}
W_{s, \lambda}(x, y) \geq 0 \quad \text{ in } \{x \ |\ |x|>\overline\rho\} \times (0, \overline \varepsilon).
\end{equation}
Let us fix $\overline \rho \in (r_s^2,R)$ and let $\delta := R-\overline \rho$. For $\rho \in \left(0, \frac{\delta}{2}\right)$, let $T_\rho$ be the semitoroidal open set defined by
\[
T_\rho = \{ (x, y) \in \R^{N+1}_+ \ | \  \sqrt{(R-|x|)^2 + y^2} < \rho\}.
\]
Since $W_{s, \lambda}$, $\phi \in \mathcal{D}^{1,s}(\R^{N+1}_+)$, by the Lebesgue's dominated convergence theorem we deduce that
\begin{equation}\label{opiccolorho}
\lim_{\rho \to 0^+} \int_{T_\rho} y^{1-2s} \nabla W_{s, \lambda}\cdot \nabla \phi \de x \de y = 0.
\end{equation}
Let $\varepsilon < \min\left\{\frac{\rho}{2}, \overline \varepsilon\right\}$. We have that 
\[
\begin{aligned}
&\int_{\R^{n+1}_+\setminus T_\rho} y^{1-2s} \nabla W_{s, \lambda}\cdot \nabla \phi \de x \de y \\
=&\ \int_{(\R^{n+1}_+ \setminus T_\rho)\cap \{y\geq  \varepsilon\}} y^{1-2s} \nabla W_{s, \lambda}\cdot \nabla \phi \de x \de y + \int_{(\R^{n+1}_+ \setminus T_\rho)\cap \{y<\varepsilon\}} y^{1-2s} \nabla W_{s, \lambda}\cdot \nabla \phi \de x \de y \\
=&\ \int_{(\R^{n+1}_+ \setminus T_\rho)\cap \{y \geq \varepsilon\}} y^{1-2s} \nabla W_{s, \lambda}\cdot \nabla \phi \de x \de y + o_\rho(1)
\end{aligned}
\]
where the last equality is obtained arguing as in \eqref{opiccolorho}, and $o_\rho(1)$ is a function of $\rho$ and $\varepsilon$ such that $\lim_{\varepsilon \to 0^+}o_\rho(1)=0$ for any $\rho \in (0,\frac{\delta}{2})$. Integrating by parts, using that $\phi \in C^\infty_c(\overline{\R^{n+1}_+})$ we obtain
\[
\begin{aligned}
&\int_{\R^{n+1}_+\setminus T_\rho} y^{1-2s} \nabla W_{s, \lambda}\cdot \nabla \phi \de x \de y \\
=& \int_{|x|\leq R - \sqrt{\rho^2 - \varepsilon^2} } -\varepsilon^{1-2s} \frac{\partial W_{s, \lambda}}{\partial y}(x, \varepsilon) \phi(x, \varepsilon) \de x + \ \int_{\partial T_{\rho, \varepsilon}} y^{1-2s} \nabla W_{s, \lambda}\cdot \nu \phi \de \sigma  \\
+&\ \int_{|x|\geq R + \sqrt{\rho^2 - \varepsilon^2} } -\varepsilon^{1-2s} \frac{\partial W_{s, \lambda}}{\partial y}(x, \varepsilon) \phi(x, \varepsilon) \de x  + o_\rho(1), 
\end{aligned}
\] 
where $T_{\rho, \varepsilon} = T_{\rho}\cap \{y \geq \varepsilon\}$, $\nu$ is the exterior normal to $\partial T_{\rho, \varepsilon}$, 
and $d\sigma$ is the surface measure of $\partial T_{\rho, \varepsilon}$. Notice that thanks to the choice of $\varepsilon$ we have $T_{\rho, \varepsilon} \neq \emptyset$. Moreover, arguing as in the proof of Lemma \ref{extreg}, we deduce that 
\[
\begin{aligned}
\int_{|x|\leq R - \sqrt{\rho^2 - \varepsilon^2} }\! -\varepsilon^{1-2s} \frac{\partial W_{s, \lambda}}{\partial y}(x, \varepsilon) \phi(x, \varepsilon) \de x = 
\int_{|x|\leq R - \sqrt{\rho^2 - \varepsilon^2} }\! -\varepsilon^{1-2s} \frac{\partial W_{s, \lambda}}{\partial y}(x, \varepsilon) \varphi(x) \de x + o_\rho(1), \\
 \int_{|x|\geq R + \sqrt{\rho^2 - \varepsilon^2} }\! -\varepsilon^{1-2s} \frac{\partial W_{s, \lambda}}{\partial y}(x, \varepsilon) \phi(x, \varepsilon) \de x =  \int_{|x|\geq R + \sqrt{\rho^2 - \varepsilon^2} }\! -\varepsilon^{1-2s} \frac{\partial W_{s, \lambda}}{\partial y}(x, \varepsilon) \varphi(x) \de x + o_\rho(1).
\end{aligned}
\]
As a consequence we get that
\begin{equation}\label{equazagg}
\begin{aligned}
&\int_{\R^{n+1}_+\setminus T_\rho} y^{1-2s} \nabla W_{s, \lambda}\cdot \nabla \phi \de x \de y \\
=&\ \int_{|x|\leq R - \sqrt{\rho^2 - \varepsilon^2} } -\varepsilon^{1-2s} \frac{\partial W_{s, \lambda}}{\partial y}(x, \varepsilon) \varphi(x) \de x + \int_{\partial T_{\rho, \varepsilon}} y^{1-2s} \nabla W_{s, \lambda}\cdot \nu \phi \de \sigma  \\
+&\ \int_{|x|\geq R + \sqrt{\rho^2 - \varepsilon^2} } -\varepsilon^{1-2s} \frac{\partial W_{s, \lambda}}{\partial y}(x, \varepsilon) \varphi(x) \de x  + o_\rho(1) \\
=&\ (I) + (II) + (III) + o_\rho(1).  
\end{aligned}
\end{equation}

For the term $(I)$, by Lemma \ref{app4} and Lemma \ref{app5}, we obtain that 
\begin{equation}\label{pezzoI}
\lim_{\varepsilon \to 0^+} (I) = \int_{B_{R-\rho}}(-\Delta)^s u_{s, \lambda}\varphi \de x = \int_{B_{R-\rho}}\left(\lambda u_{s, \lambda} + |u_{s, \lambda}|^{2^*_s-2}u_{s, \lambda}\right) \varphi \de x.
\end{equation}

For $(III)$, by Lemma \ref{app3} and Lemma \ref{app5}, we have
\[
\lim_{\varepsilon \to 0^+}(III) = \lim_{\varepsilon \to 0^+} \int_{\R^N \setminus B_{R+\sqrt{\rho^2-\varepsilon^2}}}-2sd_s \frac{W_{s, \lambda}(x, \varepsilon) - W_{s, \lambda}(x, 0)}{\varepsilon^{2s}}\varphi(x) \de x.
\]
We observe that, as a consequence of \eqref{segno}, when $x \in \R^N \setminus B_{R+\sqrt{\rho^2-\varepsilon^2}}$ and $\varepsilon < \overline \varepsilon$ we have $W_{s, \lambda}(x, \varepsilon) = W_{s, \lambda}(x, \varepsilon) - W_{s, \lambda}(x, 0) \geq 0$. Moreover, by Lemma \ref{app1} and Lemma \ref{app2}, we infer that the limit $\lim_{\varepsilon \to 0^+}(III)$ exists and it is finite. 
In particular we have
\begin{equation}\label{pezzoII}
\lim_{\varepsilon \to 0^+}(III) = h(\rho) \leq 0,
\end{equation}
where $h$ is a non-positive function which depends on $\rho$.

As a consequence of \eqref{equazagg}, \eqref{pezzoI} and \eqref{pezzoII} there exists $q = q(\rho)$ such that 
\begin{equation}\label{limfinito}
\lim_{\varepsilon \to 0^+}(II) = q(\rho).
\end{equation}
In particular, $q$ does not depends on $\varepsilon$. In addition, in Lemma \ref{stimaintegrale} we will show that there exists $C>0$ which does not depends on $\rho$ such that 
\begin{equation}\label{pezzoIII}
-C\rho^{1-s} \leq \liminf_{\varepsilon \to 0^+}(II) \leq \limsup_{\varepsilon \to 0^+}(II) \leq C \rho^{1-s}.
\end{equation}

We can now conclude the proof of the Theorem. Taking in account of \eqref{equazagg1}, \eqref{opiccolorho}, \eqref{pezzoI}, \eqref{pezzoII} and \eqref{pezzoIII} we get that
\begin{eqnarray*}
(u_{s, \lambda}, \varphi)_s &=& \limsup_{\varepsilon \to 0^+}((I) + (II) + (III) + o_\rho(1)) + o(\rho)\\
& \leq& \int_{B_{R-\rho}}\left(\lambda u_{s, \lambda} + |u_{s, \lambda}|^{2^*_s-2}u_{s, \lambda}\right) \varphi \de x + C\rho^{1-s}\\
&\leq&  \int_{B_{R}}\left(\lambda u_{s, \lambda} + |u_{s, \lambda}|^{2^*_s-2}u_{s, \lambda}\right) \varphi \de x + C\rho^{1-s},
\end{eqnarray*}
because $u_{s,\lambda}\geq0$ in $(r_s^2,R)$ and $\varphi\geq0$. Then, passing to the limit as $\rho \to 0^+$ we get the desired result and the proof is complete.

\end{proof}

\begin{lemma}\label{stimaintegrale}
There exists a constant $C>0$ such that for all sufficiently small $\rho>0$
\[
-C\rho^{1-s}\leq \liminf_{\varepsilon \to 0^+}(II) \leq \limsup_{\varepsilon \to 0^+}(II) \leq C \rho^{1-s}.
\]
\end{lemma}
\begin{proof}
Let $\Sf^{N-1}$ be the unit sphere in $\R^N$. For any $\mathbf{e} \in \Sf^{N-1}$, since $W_{s, \lambda}$ is cylindrical symmetric (see the proof of Theorem \ref{twicechange}), we have $W_{s, \lambda}(x, y)  = W_{s, \lambda}(\mathbf{e}|x|, y),\ \hbox{for any}\ x \in \R^N, y\geq 0.$

In particular, without loss of generality, $W_{s, \lambda}(x, y)  = W_{s, \lambda}(e_1|x|, y),\ \hbox{for any}\ x \in \R^N, y\geq 0$, where $e_1= (1, 0, \ldots, 0)$.
Let $\delta >0$, $\rho \in \left(0, \frac{\delta}{2}\right)$ and $\varepsilon \in \left(0, \min\{\frac{\rho}{2}, \overline \varepsilon\}\right)$ as in the proof of Theorem \ref{subsol}.
Since $\rho < R$, we can express the set $\partial T_{\rho, \varepsilon}$ in the following way
\begin{equation}\label{parametrizT}
\partial T_{\rho, \varepsilon} = \left \{\left((R-\rho \cos \theta)\mathbf{e}, \rho \sin \theta \right) \ | \ \theta \in (\theta_\rho(\varepsilon), \pi - \theta_\rho(\varepsilon)), \mathbf{e} \in \Sf^{N-1}\right\},
\end{equation}
where $\theta_\rho(\varepsilon)= \arcsin \frac{\varepsilon}{\rho}$. We notice that since $\varepsilon \in \left(0, \frac{\rho}{2}\right)$, then $\theta_\rho(\varepsilon) \in \left(0, \frac{\pi}{6}\right)$. Moreover $\varepsilon\mapsto\theta_\rho(\varepsilon)$ is continuous, monotone and  $\lim_{\varepsilon\to 0^+}\theta_\rho(\varepsilon) = 0$, for any fixed $\rho \in (0,\frac{\delta}{2})$. Since all the estimates that we are going to prove will be uniform with respect to $\theta \in (0,\pi)$ we drop for brevity the subscript $\rho$ in $\theta_\rho(\epsilon)$.

Exploiting the cylindrical symmetry of $W_{s, \lambda}$ we notice that when $(x, y) \in \partial T_{\rho, \varepsilon}$ we can express $W_{s, \lambda}$ just using using the coordinates $\rho$, $\theta$, obtaining that
\[
W_{s, \lambda}(\rho, \theta) = \int_{\R^N}\frac{(\rho \sin \theta)^{2s}}{(\rho \sin \theta)^2 + |(R- \rho \cos \theta)e_1- \xi^2|)^{\frac{N+2s}{2}}}u_{s, \lambda}(\xi) \de \xi.
\]

Now, denoting by $\nu$ the exterior normal to the surface $\partial T_{\rho, \varepsilon}$, and taking account of the orientations, by a simple computation we have that 
\[
\nabla W_{s, \lambda}(x, y) \cdot \nu(x, y)\big|_{\partial T_{\rho, \varepsilon}} = - \frac{\partial W_{s, \lambda}}{\partial \rho}(\rho, \theta).
\]  

Therefore, by a slight abuse of notation, parametrizing the hypersurface $\partial T_{\rho, \varepsilon}$ as in \eqref{parametrizT} with the coordinates $(\theta, \mathbf{e})$, we obtain 
\[
(II) = \int_{\Sf^{N-1}} \int_{\theta(\varepsilon)}^{\pi - \theta(\varepsilon)}-(\rho \sin \theta)^{1-2s}\frac{\partial W_{s, \lambda}}{\partial \rho}(\rho, \theta) \phi(\rho, \theta, \mathbf{e})\rho \Psi(\rho, \theta, \mathbf{e}) \de \theta \de \mathbf{e},
\]
where $\rho \Psi(\rho, \theta, \mathbf{e})$ is a positive factor coming from the definition of surface measure. In particular, $\Psi$ is uniformly bounded when $\rho$ is small, and Lipschitz continuous with respect to $\theta$, uniformly in $\rho$ and $\mathbf{e}$.
 More precisely, using for $\mathbb{S}^{N-1}$ the atlas $\{(U_l,\psi_l^{-1})\}_{l=1,\ldots,N}$ whose charts are the spherical coordinates 
$\psi_l^{-1}:U_l \to V$ (see \cite{AbateTovena}, Example 2.1.29), where 
\[
V :=\{(\theta_1,...,\theta_{N-1})\in \R^{N-1} | \ 0<\theta_1<2\pi,\ 0<\theta_h <\pi\ \hbox{for}\  h=2,\ldots,N-1\},
\]
then local parametrizations for $\partial T_{\rho, \varepsilon}$ are given by the maps
\[
\begin{aligned}
&\tilde \psi_l:V\times (\theta(\varepsilon),\pi-\theta(\varepsilon)) \to \partial T_{\rho, \varepsilon},\\
&\tilde \psi_l (\theta_1,\ldots,\theta_{n-1},\theta):=((R-\rho\cos\theta)\psi_l(\theta_1,\ldots,\theta_{n-1}),\rho \sin \theta).
\end{aligned}
\]
Now, since the matrix $(g_{ij})_{i,j=1,\ldots,N-1}$ of the induced metric on $\mathbb{S}^{N-1}$ by the spherical coordinates is diagonal and given by $g_{ij}=\delta_{ij} (\sin \theta_{i+1}\cdots\sin \theta_{N-1})^2$ (see \cite{AbateTovena}, Example 6.5.22), then, for each parametrization $\tilde \psi_l$ the determinant of the matrix $(\tilde g_{ij})_{i,j=1,\ldots,N}$ of the induced metric on $ \partial T_{\rho, \varepsilon}$ is $\rho^2 (\cos\theta)^2 (R-\rho\cos\theta)^{2(N-1)}\Pi_{i=1,\ldots,N-1}  (\sin \theta_{i+1}\cdots\sin \theta_{N-1})^2$ and thus its square root is $\rho |(\cos\theta)| (R-\rho\cos\theta)^{(N-1)}\Pi_{i=1,\ldots,N-1}  |(\sin \theta_{i+1}\cdots\sin \theta_{N-1})|$. Therefore $\Psi(\rho, \theta,\theta_1,\ldots,\theta_{N-1})$ has the desired properties. For brevity we will use the more compact notation $\Psi(\rho, \theta, \mathbf{e})$.

Now, by an elementary computation we obtain that
\begin{equation}\label{derivrho}
\begin{aligned}
&-(\rho \sin \theta)^{1-2s}\frac{\partial W_{s, \lambda}}{\partial \rho}(\rho, \theta) \rho\\
=&\ -p_{N,s}2s \rho \sin \theta \int_{B_R} \frac{u_{s, \lambda}(\xi)}{((\rho \sin \theta)^2 + |(R- \rho \cos \theta)e_1- \xi|^2)^{\frac{N+2s}{2}}} \de \xi \\
+&\ p_{N,s}(N+2s) \rho^2 \sin \theta \int_{B_R} \frac{u_{s, \lambda}(\xi) ( \rho - R\cos \theta + \cos \theta (e_1, \xi))}{((\rho \sin \theta)^2 + |(R- \rho \cos \theta)e_1- \xi|^2)^{\frac{N+2s+2}{2}}}\de \xi,
\end{aligned}
\end{equation}
where we used that $u_{s, \lambda}(\xi)  = 0$ when $\xi \in \R^N \setminus B_R$. Let us define the set
\[
C_\delta = \{ \xi \in B_R \ |\  |\xi - Re_1| < \delta\}. 
\]
We can split the integrals appearing in \eqref{derivrho} taking as domains of integrations $C_\delta$ and $B_R \setminus C_\delta$. Since $\rho \in \left(0, \frac{\delta}{2} \right)$, when $\xi \in B_R \setminus C_\delta$ the relation $|(R-\rho \cos \theta)e_1 - \xi| > \frac{\delta}{2}$ holds, and thus all the quantities appearing in the integrals over $B_R \setminus C_\delta$ are bounded from above and below, respectively, by $\pm C$, where $C>0$ is a constant which depends only on $N$, $s$, $R$, $|u_{s, \lambda}|_\infty$ and $\delta$ (but not on $\rho$, $\theta$, and hence neither on $\varepsilon$). 

Taking this into account and performing the change of variable $\xi = \eta + Re_1$, after some computations we can write
\begin{equation}\label{integrale1}
\begin{aligned}
&-(\rho \sin \theta)^{1-2s}\frac{\partial W}{\partial \rho}(\rho, \theta) \rho\\
=&\ O (\rho) + p_{N,s}\rho \sin \theta \int_{C_\delta-Re_1} \frac{u_{s, \lambda}(\eta +Re_1)(N\rho^2 - 2s|\eta|^2 + (N-2s)\rho|\eta|\cos \theta (e_1, \hat \eta))}{(\rho^2 + |\eta|^2 + 2 \rho|\eta|\cos \theta(e_1, \hat \eta))^{\frac{N+2s+2}{2}}}\de \eta,
\end{aligned}
\end{equation}
where $\hat \eta$ is such that $\eta = |\eta|\hat \eta$, and $O(\rho)$ does not depends on $\theta$. 
We notice that, if $\eta \in C_\delta - R e_1$, then $(e_1, \hat \eta) <0$ and in particular $(e_1, \hat \eta) = -|(e_1, \hat \eta)|$.  

Since $\rho < \frac{\delta}{2}$, it holds that for every fixed $\tau \in (0,1)$ we have $B_{\tau \rho}((R-\rho)e_1) \subset C_\delta$ or equivalently $B_{\tau \rho}(-\rho e_1) \subset C_\delta - Re_1$. Moreover, when $\eta \in B_{\tau \rho}(-\rho e_1)$ the following inequalities hold:
\begin{equation}\label{stimepalletta}
\begin{aligned}
&(1-\tau)\rho < |\eta| < (1+ \tau)\rho; \\
&|(e_1, \hat \eta)| \geq \frac{1-\tau}{1+\tau};\\
&|\eta+\rho e_1| \leq \sqrt{\rho^2 + |\eta|^2 -2\rho|\eta|\cos \theta|(e_1, \hat \rho)|}. 
\end{aligned}
\end{equation}

Writing  $u_{s, \lambda}(x) = \frac{u_{s, \lambda}(x)}{\gamma^s(x)}\gamma^s(x) =: g_{s, \lambda}(x) \gamma^s(x)$, where $\gamma(x): = \text{dist}(x, \partial B_R)$, by Theorem \ref{boundregRO} we have that  $g_{s, \lambda}(x)$ is bounded in $C_\delta$. Moreover we also have $g_{s, \lambda}\geq 0$ in $C_\delta$ because $u_{s, \lambda} \geq 0$ in $C_\delta$.  
As a consequence when $\eta \in C_\delta -Re_1$ we get that 
\begin{equation}\label{stimau}
0 \leq u_{s, \lambda}(\eta + Re_1) \leq \sup_{C_\delta}|g_{s, \lambda}| |\eta|^s.
\end{equation}

Let us estimate the integral in the right hand side of \eqref{integrale1}. To this end we divide the domain of integration $C_\delta- Re_1$ in two parts. Let us fix $\tau \in (0,1)$. In the set $(C_\delta - R e_1) \setminus B_{\tau \rho}(-\rho e_1)$ it holds that 
\[
\sqrt{\rho^2 + |\eta|^2 - 2 \rho|\eta|\cos \theta |(e_1, \hat \eta)|} \geq |\eta + \rho e_1| \geq \tau \rho.
\]
Hence, performing the change of variables $\eta = \rho k$, we get
\[
\begin{aligned}
&\left| p_{N,s}\rho \sin \theta \int_{(C_\delta-Re_1)\setminus B_{\tau \rho}(-\rho e_1)} \frac{u_{s, \lambda}(\eta +Re_1)(N\rho^2 - 2s|\eta|^2 - (N-2s)\rho|\eta|\cos \theta |(e_1, \hat \eta)|)}{(\rho^2 + |\eta|^2 - 2 \rho|\eta|\cos \theta|(e_1, \hat \eta)|)^{\frac{N+2s+2}{2}}}\de \eta \right| \\
\leq&\ C\rho \int_{\R^N\setminus B_{\tau \rho}(-\rho e_1)} \frac{|\eta|^s (N\rho^2 + 2s|\eta|^2 + (N-2s)\rho|\eta|)}{|\eta + \rho e_1|^{N+2s+2}}\de \eta \\
=& \ C \rho^{1-s} \int_{\R^N \setminus B_\tau(-e_1)}\frac{|k|^s(N +2s|k|^2 + (N-2s)|k|)}{|k + e_1|^{N+2s+2}}\de k \leq C \rho^{1-s},
\end{aligned}
\]
where $C>0$ depends on $N$, $s$, $|g_{s, \lambda}|$ and $\tau$, but does not depends on $\rho$ and $\theta$. Therefore we obtain
\[
\begin{aligned}
&-(\rho \sin \theta)^{1-2s}\frac{\partial W}{\partial \rho}(\rho, \theta) \rho \\
=&\ O (\rho^{1-s}) + p_{N,s}\rho^{1-s} \sin \theta \int_{|k + e_1|< \tau} \frac{\frac{u_{s, \lambda}(\rho k +Re_1)}{\rho^s}(N - 2s|k|^2 - (N-2s)|k|\cos \theta |(e_1, \hat k)|)}{(1 + |k|^2 - 2 |k|\cos \theta|(e_1, \hat k)|)^{\frac{N+2s+2}{2}}}\de k.
\end{aligned}
\]
Using the relation 
\[
\begin{aligned}
&N-2s |k|^2 - (N-2s)|k| \cos \theta |(e_1, \hat k)| \\
=&\ \frac{N+2s}{2}(1 - |k|)(1+|k|) + \frac{N-2s}{2}(1 + |k|^2 - 2|k| \cos \theta |(e_1, \hat k)|),
\end{aligned}
\]
we then deduce that 
\[
\begin{aligned}
&(II)=\\
&p_{N,s}\!\int_{\Sf^{N-1}}\! \int_{\theta(\varepsilon)}^{\pi - \theta(\varepsilon)}\!\! \!\! \int_{|k + e_1|< \tau}\!\!\!\!\!\! \frac{\frac{N+2s}{2}\rho^{1-s}\sin \theta \frac{u_{s, \lambda}(\rho k +Re_1)}{\rho^s}(1-|k|)(1+|k|)\phi(\rho, \theta, \mathbf{e}) \Psi(\rho, \theta, \mathbf{e})}{(1 + |k|^2 - 2 |k|\cos \theta|(e_1, \hat k)|)^{\frac{N+2s+2}{2}}} \de k\de \theta \de \mathbf{e}  \\
&+ p_{N,s}\!\int_{\Sf^{N-1}}\! \int_{\theta(\varepsilon)}^{\pi - \theta(\varepsilon)}\!\!\!\! \int_{|k + e_1|< \tau}\!\!\!\!\!\! \frac{\frac{N - 2s}{2}\rho^{1-s} \sin \theta\frac{u_{s, \lambda}(\rho k +Re_1)}{\rho^s}\phi(\rho, \theta, \mathbf{e}) \Psi(\rho, \theta, \mathbf{e}) }{(1 + |k|^2 - 2 |k|\cos \theta|(e_1, \hat k)|)^{\frac{N+2s}{2}}}\de k \de \theta \de \mathbf{e} +O(\rho^{1-s})  \\
&= (i) + (ii) + O(\rho^{1-s}), 
\end{aligned}
\]
where $O(\rho^{1-s})$ is uniform with respect to $\varepsilon$.
We start showing that $|(ii)| \leq C\rho^{1-s}$, where $C$ does not depends on $\rho$ and $\varepsilon$. Indeed, applying Fubini-Tonelli's theorem and integrating by parts we get that 
\[
\begin{aligned}
&\left|\int_{\theta(\varepsilon)}^{\pi - \theta(\varepsilon)}\frac{ \sin \theta\phi(\rho, \theta, \mathbf{e})\Psi(\rho, \theta, \mathbf{e})}{(1 + |k|^2 - 2 |k|\cos \theta|(e_1, \hat k)|)^{\frac{N+2s}{2}}} \de \theta \right| \\
=&\ \bigg|\left[ - \frac{2}{N+2s-2} \frac{1}{2|k||(e_1, \hat k)|}\frac{\phi(\rho, \theta, \mathbf{e})\Psi(\rho, \theta, \mathbf{e})}{(1 + |k|^2 - 2|k|\cos \theta|(e_1, \hat k)|)^{\frac{N+2s-2}{2}}}\right]^{\pi - \theta(\varepsilon)}_{\theta(\varepsilon)}  \\
+&\ \int_{\theta(\varepsilon)}^{\pi- \theta(\varepsilon)}\frac{2}{N+2s-2} \frac{1}{2|k||(e_1, \hat k)|}\frac{\partial_\theta[\phi(\rho, \theta, \mathbf{e})\Psi(\rho, \theta, \mathbf{e})]}{(1 + |k|^2 - 2|k|\cos \theta|(e_1, \hat k)|)^{\frac{N+2s-2}{2}}}\de \theta \bigg| \\
\leq&\ C \frac{1}{|k+e_1|^{N+2s-2}}
\end{aligned}
\]
where $C>0$ depends only on $N$, $s$, $\phi$, $\Psi$ and $\tau$ but not on $\rho$ and $\varepsilon$, and where we used the inequalities \eqref{stimepalletta} evaluated at $\eta = \rho k$ and the fact that $\frac{\partial (\phi \Psi)}{\partial \theta}$ is uniformly bounded. Since by \eqref{stimau} we get that $\frac{u_{s, \lambda}(\rho k +Re_1)}{\rho^s}\leq \sup_{C_\delta}|g_{s,\lambda}||k|^s \leq \sup_{C_\delta}|g_{s,\lambda}|(1+ \tau)^s$, we infer that
\[
|(ii)| \leq C\rho^{1-s} \int_{|k+e_1|< \tau} \frac{1}{|k + e_1|^{N+2s-2}} \de k \leq C_1 \rho^{1-s}.
\]
where $C_1>0$ depends on $N$, $s$, $\phi$, $\Psi$ and $\tau$ but not on $\rho$ and $\varepsilon$. 

Now, we can write 
\begin{equation}\label{calcolo12}
\begin{aligned}
(i)=& \ \int_{\Sf^{N-1}} \int_{\theta(\varepsilon)}^{\pi - \theta(\varepsilon)}p_{N,s}\frac{N+2s}{2}\sin \theta\bigg[\rho^{1-s}  \int_{|k + e_1|< \tau} \frac{\frac{u_{s, \lambda}((R-\rho) e_1)}{\rho^s}(1-|k|)(1+|k|)}{(1 + |k|^2 - 2 |k|\cos \theta|(e_1, \hat k)|)^{\frac{N+2s+2}{2}}}\de k \\
+&\rho^{1-s} \int_{|k + e_1|< \tau} \frac{\frac{u_{s, \lambda}(\rho k +Re_1)- u_{s, \lambda}((R-\rho)e_1)}{\rho^s}(1-|k|)(1+|k|)}{(1 + |k|^2 - 2 |k|\cos \theta|(e_1, \hat k)|)^{\frac{N+2s+2}{2}}}\de k\bigg]\phi(\rho, \theta, \mathbf{e}) \Psi(\rho, \theta, \mathbf{e}) \de \theta \de \mathbf{e}
\end{aligned}
\end{equation}

Thanks to \cite[Corollary 1.6, (a)]{HolderReg}, we have that there exists $C>0$ which depends on $N$, $s$ and $u_{s, \lambda}$ such that for $\alpha \in [s,1+2s)$
\[
[u_{s, \lambda}]_{0, \alpha; K} \leq C \text{dist}(K, \partial B_R)^{s - \alpha}.
\] 
Taking $\alpha = 1$ and $K = B_{\tau \rho}((R-\rho) e_1)$, since $\text{dist }(K, \partial B_R) = \rho(1-\tau)$, we get that 
\[
|u_{s, \lambda}(x) - u_{s, \lambda}(y)| \leq C \rho^{s-1}(1-\tau)^{s-1}|x-y|.
\]
for every $x, y$ in $B_{\tau \rho}((R-\rho)e_1)$. 
Therefore, when $x = \rho k + Re_1$ and $y = (R-\rho)e_1$ we obtain 
\[
|u_{s, \lambda}(\rho k +Re_1)- u_{s,\lambda}((R-\rho)e_1)| \leq C \rho^s |k + e_1|
\]
where $C>0$ depends on $N$, $s$ and $\tau$ but not on $\rho$ nor on $\theta$. 
As a consequence we get the extimate
\[
\left| \frac{u_{s, \lambda}(\rho k +Re_1)- u_{s, \lambda}((R-\rho)e_1)}{\rho^s}\right|\leq C|k+e_1| \leq C\sqrt{1 + |k|^2 - 2|k| \cos \theta |(e_1, \hat k)|},
\] 
where $C>0$ does not depends on $\rho$ and $\theta$, but only on $N$, $s$ and $\tau$. 
Moreover, since 
\begin{equation}\label{calcolo2}
|(1-|k|)(1 + |k|)| \leq (2+\tau)\sqrt{1 + |k|^2 - 2|k| \cos \theta |(e_1, \hat k)|}, 
\end{equation}
arguing as we have done for $(ii)$ we get that the second term in \eqref{calcolo12} is $O(\rho^{1-s})$ uniformly in $\varepsilon$. Now we can further refine the estimate noticing that 
\[
\begin{aligned}
&|\phi(\rho, \theta, \mathbf{e}) - \phi(\rho, 0, \mathbf{e})|\\
=&\ |\phi((R-\rho \cos \theta)\mathbf{e}, \rho \sin \theta) - \phi ((R-\rho)\mathbf{e}, 0)| \\
\leq&\ C\sqrt{|\rho(1-\cos \theta)|^2 + \rho^2 (\sin \theta)^2} = C \rho \sqrt{2(1-\cos \theta)}.
\end{aligned}
\]
Moreover, since in $|k + e_1|< \tau$ it holds $|k| > (1-\tau)$ (see \eqref{stimepalletta}), we get that
\begin{equation}\label{stimavarphi}
\begin{aligned}
|\phi(\rho, \theta, \mathbf{e}) - \phi(\rho, 0, \mathbf{e})|& \leq C \sqrt{2|k|(1-\cos \theta)} \leq C\sqrt{(|k|-1)^2+ 2|k|(1-\cos \theta)}\\
& \leq C\sqrt{1 + |k|^2 - 2 |k| \cos \theta|(e_1, \hat k)|}.
\end{aligned}
\end{equation}
Arguing as before, using again \eqref{calcolo2}, we get that
\[
\begin{aligned}
&\left|\int_{\theta(\varepsilon)}^{\pi - \theta(\varepsilon)}\frac{ \sin \theta(\phi(\rho, \theta, \mathbf{e})- \phi (\rho, 0, \mathbf{e})) (1- |k|)(1+ |k|)\Psi(\rho, \theta, \mathbf{e})}{(1 + |k|^2 - 2 |k|\cos \theta|(e_1, \hat k)|)^{\frac{N+2s+2}{2}}}\de \theta\right|  \\
\leq&\ \int_{\theta(\varepsilon)}^{\pi - \theta(\varepsilon)}\frac{ \sin \theta|\phi(\rho, \theta, \mathbf{e})- \phi (\rho, 0, \mathbf{e})| |1- |k||(1+ |k|)\Psi(\rho, \theta, \mathbf{e})}{(1 + |k|^2 - 2 |k|\cos \theta|(e_1, \hat k)|)^{\frac{N+2s+2}{2}}}\de \theta\\
\leq &\ C \frac{1}{|k+e_1|^{N+2s-2}}.
\end{aligned}
\]

As a consequence we have a further negligible term, and thus
\[
\begin{aligned}
(II) = O(\rho^{1-s}) &+ p_{N,s}\frac{N+2s}{2}\rho^{1-s}g_{s,\lambda}(\rho e_1) \cdot \\
&\cdot\int_{\Sf^{N-1}}\int_{\theta(\varepsilon)}^{\pi - \theta(\varepsilon)} \int_{|k + e_1|< \tau} \frac{\sin \theta(1-|k|)(1+|k|)\phi (\rho, 0, \mathbf{e})\Psi(\rho, \theta, \mathbf{e})}{(1 + |k|^2 - 2 |k|\cos \theta|(e_1, \hat k)|)^{\frac{N+2s+2}{2}}} \de k \de \theta \de \mathbf{e}.
\end{aligned}
\]

To conclude we have to get rid of the dependence from $\theta$ in the function $\Psi$. Arguing as in \eqref{stimavarphi} and since $\Psi$ is Lipschitz continuous in $\theta \in [0, \pi]$, uniformly in $\rho$ and $\mathbf{e}$, we get that  
\[
|\Psi(\rho, \theta, \mathbf{e}) - \Psi(\rho, 0, \mathbf{e})| \leq C \theta \leq 4C \sqrt{1 - \cos \theta} \leq 4C\sqrt{1 + |k|^2 -2|k| \cos \theta |(e_1, \hat \eta)|},
\]
where $C$ does not depend on $\theta$ and $\rho$. Hence, arguing as before, we obtain 
\[
\begin{aligned}
(II) = O(\rho^{1-s}) + &p_{N,s}\frac{N+2s}{2}\rho^{1-s}g_{s,\lambda}(\rho e_1) \left(\int_{\Sf^{N-1}}\phi(\rho, 0, \mathbf{e})\Psi(\rho, 0, \mathbf{e})\de \mathbf{e}\right)\cdot \\
&\cdot\int_{\theta(\varepsilon)}^{\pi - \theta(\varepsilon)} \int_{|k + e_1|< \tau} \frac{\sin \theta(1-|k|)(1+|k|)}{(1 + |k|^2 - 2 |k|\cos \theta|(e_1, \hat k)|)^{\frac{N+2s+2}{2}}} \de k \de \theta,
\end{aligned}
\]
where $O(\rho^{1-s})$ is uniform with respect to $\varepsilon$. Let us set $F(\theta,k):=\frac{\sin \theta(1-|k|)(1+|k|)}{(1 + |k|^2 - 2 |k|\cos \theta|(e_1, \hat k)|)^{\frac{N+2s+2}{2}}}$.
Thanks to the previous equality and \eqref{limfinito} we infer that both $$ \liminf_{\varepsilon \to 0^+}\int_{\theta(\varepsilon)}^{\pi - \theta(\varepsilon)} \int_{|k + e_1|< \tau} F(\theta,k)\de k \de \theta =  \liminf_{t \to 0^+} \int_{t}^{\pi - t} \int_{|k + e_1|< \tau}F(\theta,k) \de k \de \theta,$$ and $$\limsup_{\varepsilon \to 0^+}\int_{\theta(\varepsilon)}^{\pi - \theta(\varepsilon)} \int_{|k + e_1|< \tau} F(\theta,k) \de k \de \theta =  \limsup_{t \to 0^+} \int_{t}^{\pi - t} \int_{|k + e_1|< \tau}F(\theta,k)\de k \de \theta,$$ 
are finite and they do not depend on $\rho$, where for the last equality we used the properties of $\theta(\varepsilon)=\theta_\rho(\varepsilon)$ and the definition of $\liminf$, $\limsup$.

 At the end, since the quantity $g_{s, \lambda}(\rho e_1) \left(\int_{\Sf^{N-1}}\phi(\rho, 0, \mathbf{e})\Psi(\rho, 0, \mathbf{e})\de \mathbf{e}\right)$ is uniformly bounded with respect to $\rho$, we conclude that
\[
-C\rho^{1-s}\leq \liminf_{\varepsilon \to 0^+}(II) \leq \limsup_{\varepsilon \to 0^+}(II) \leq C \rho^{1-s},
\]
for some constant $C>0$ which does not depends on $\rho$. The proof is complete.
\end{proof}

\end{chapter}

\end{document}